\documentclass[12pt, leqno]{article}

\usepackage{amsmath, amsthm, amssymb}
\usepackage{color}
\usepackage[cmtip,all]{xy}

\usepackage[top=26truemm,bottom=26truemm,left=30truemm,right=30truemm]{geometry}

\makeatletter
    
    \@addtoreset{equation}{section}
  \makeatother

\theoremstyle{plain}
\newtheorem{thm}{Theorem}[section]
\newtheorem{prop}[thm]{Proposition}
\newtheorem{lem}[thm]{Lemma}
\newtheorem{cor}[thm]{Corollary}
\newtheorem{conj}[thm]{Conjecture}
\theoremstyle{definition}
\newtheorem{defn}[thm]{Definition}
\newtheorem{exa}[thm]{Example}
\newtheorem{rem}[thm]{Remark}
\newtheorem*{conv}{Conventions}

\newcommand{\FRAC}[2]{\leavevmode\kern.1em\raise.5ex\hbox{\the\scriptfont0 #1}\kern-.1em/\kern-.15em\lower.25ex\hbox{\the\scriptfont0 #2}}

\newcommand{\BQ}{{\mathbf{Q}}}

\newcommand{\BZ}{{\mathbf{Z}}}
\newcommand{\dvr}{\mathcal{O}}
\newcommand{\et}{\mathrm{\acute{e}t}}
\DeclareMathOperator{\Hom}{Hom}

\newcommand{\sw}{\mathrm{sw}}
\newcommand{\ab}{\mathrm{ab}}
\newcommand{\mf}{\mathcal{F}}

\newcommand{\id}{\mathrm{id}}
\newcommand{\pr}{\mathrm{pr}}
\newcommand{\cont}{\mathrm{cont}}

\newcommand{\rsw}{\mathrm{rsw}}
\DeclareMathOperator{\cform}{char}

\newcommand{\fil}{\mathrm{fil}'}
\newcommand{\fillog}{\mathrm{fil}}

\DeclareMathOperator{\ord}{ord}

\newcommand{\gr}{\mathrm{gr}'}
\newcommand{\grlog}{\mathrm{gr}}

\newcommand{\dlog}{d\log}
\DeclareMathOperator{\Supp}{Supp}

\DeclareMathOperator{\dimtot}{dimtot}

\newcommand{\mII}{\mathrm{II}}
\newcommand{\mI}{\mathrm{I}}
\newcommand{\mT}{\mathrm{T}}
\newcommand{\mW}{\mathrm{W}}

\newcommand{\dt}{\mathrm{dt}}
\DeclareMathOperator{\res}{res}

\newcommand{\red}{\mathrm{red}}
\newcommand{\mg}{\mathcal{G}}

\newcommand{\mh}{\mathcal{H}}

\DeclareMathOperator{\Spec}{Spec}
\DeclareMathOperator{\Frac}{Frac}

\DeclareMathOperator{\Ker}{Ker}
\DeclareMathOperator{\Image}{Im}

\makeindex

\begin{document}
\title
{Singular support and Characteristic cycle \\ of a rank one sheaf in codimension two}
\author{YURI YATAGAWA}
\date{}
\maketitle

\begin{abstract}
We compute the singular support and the characteristic cycle of a rank 1 sheaf on a smooth variety in codimension 2 using ramification theory, 
when the ramification of the sheaf is clean.
We develop a general theory, called the partially logarithmic ramification theory, 
and define an algebraic cycle on a logarithmic cotangent bundle with partial logarithmic poles along the boundary.
We prove that the inverse image of the support of the cycle 
and the pull-back of the cycle
to the cotangent bundle are equal to the singular support and
the characteristic cycle, respectively, outside a closed subset of the variety of codimension $\ge 3$ under a mild assumption.
\end{abstract}

\begin{center}
\end{center}
\section*{Introduction}

Motivated by the observation in \cite{de},
Beilinson (\cite{be}) and Saito (\cite{sacc}) have defined the singular support $SS(\mg)$ and the characteristic cycle $CC(\mg)$ of an \'{e}tale constructible sheaf $\mg$ on a smooth algebraic variety $X$ in any characteristic.
Similarly as in the theory in \cite{ks} for complex manifolds,
the singular support $SS(\mg)$ is a closed conical subset of the cotangent bundle $T^{*}X$ of $X$ of dimension $\dim X$.
The characteristic cycle $CC(\mg)$ is an algebraic cycle of dimension $\dim X$ on $T^{*}X$ whose support is contained in $SS(\mg)$.
One of the differences from the theory for complex manifolds is
that the vanishing cycles are mainly used to construct $SS(\mg)$ and $CC(\mg)$,
which makes it difficult to compute $SS(\mg)$ and $CC(\mg)$.

In this article, we consider the zero extension $j_{!}\mf$
of a smooth sheaf $\mf$ of $\Lambda$-modules of rank 1 on an open subvariety $U$ of 
a smooth scheme $X$ over a perfect field $k$ of positive characteristic 
such that the boundary
$D=X-U$ is a divisor on $X$ with simple normal crossings.
Here $\Lambda$ is a finite field of characteristic invertible in $k$ and
$j\colon U\rightarrow X$ denotes the canonical open immersion.
We assume that the ramification of $\mf$ is clean along $D$ in the sense of \cite{kalog},
and consider computations of the singular support $SS(j_{!}\mf)$
and the characteristic cycle $CC(j_{!}\mf)$ of $j_{!}\mf$ in codimension $2$,
namely outside a closed subscheme of $X$ of codimension $\ge 3$,
in terms of ramification theory.

As a previous result on a computation of
the characteristic cycle $CC(j_{!}\mf)$,
Saito (\cite[Proposition 4.13, Theorem 7.14]{sacc}) has given 
a computation of $CC(j_{!}\mf)$ in codimension $1$.
In the case where $X$ is a surface, a computation of $CC(j_{!}\mf)$ has been
given in \cite[Theorem 6.1]{yacc} as a complement of Saito's computation.
Once we obtain a computation of $CC(j_{!}\mf)$, the singular support $SS(j_{!}\mf)$ of $j_{!}\mf$ is computed as the support of $CC(j_{!}\mf)$ (Corollary \ref{corssccsm} (2)).
The condition that the ramification of $\mf$ is clean along $D$,
which appears in our setting, is satisfied by taking 
a blow-up of $X$ in the case where $X$ is a surface (\cite[Theorem 4.1]{kalog}) and it is expected to be satisfied even in general. 

For a computation of $CC(j_{!}\mf)$, we introduce a general theory
called the partially logarithmic ramification theory.
The theory generalizes the logarithmic and non-logarithmic ramification theories developed by Brylinski (\cite{br}), Kato (\cite{kasw}, \cite{kalog}),
Matsuda (\cite{ma}), and Abbes-Saito (\cite{as1}, \cite{aschar}, \cite{asan}, \cite{ascl}), and Saito (\cite{sacc}).
The adjective ``partially logarithmic" means that log poles are imposed on
any union $D'$ of irreducible components of $D$ and
the theory is a refinement of the theory
on the mixed refined Swan conductor introduced by Matsuda (\cite{ma}).
Similarly as in the logarithmic and non-logarithmic ramification theories,
we introduce a notion the $\log$-$D'$-cleanliness for the ramification of $\mf$ along $D$ for a union $D'$ of irreducible components of $D$.
The $\log$-$D'$-cleanliness 
is the cleanliness introduced in \cite{kalog} when $D'=D$ and the non-degeneration introduced in \cite{sacot} when $D'=\emptyset$.
Following the construction of the logarithmic characteristic cycle given by Kato 
(\cite{kalog}) in the logarithmic ramification theory, we construct the $\log$-$D'$-characteristic cycle $CC_{D'}^{\log}(j_{!}\mf)$ 
as an algebraic cycle
on the logarithmic cotangent bundle $T^{*}X(\log D')$ with logarithmic poles along $D'$,
when the ramification of $\mf$ is $\log$-$D'$-clean along $D$,
in the partially logarithmic ramification theory.
We compute $CC(j_{!}\mf)$ in terms of ramification theory 
by proving that $CC(j_{i}\mf)$ is equal to the
pull-back of $CC_{D'}^{\log}(j_{!}\mf)$ to the cotangent bundle $T^{*}X$.

We explain the strategy of a computation of $CC(j_{!}\mf)$ 
in this article more precisely.
Under the assumption that the ramification of $\mf$ is clean along $D$,
we can take a union $D'$ of irreducible components of $D$
such that the ramification of $\mf$ is $\log$-$D'$-clean along $D$ (Proposition \ref{propdpdtdi})
and that $\mf$ is wildly ramified along any irreducible component of $D$ not contained in $D'$.
If the inverse image $\tau_{D'}^{-1}(S_{D'}^{\log}(j_{!}\mf))$ of the support $S_{D'}^{\log}(j_{!}\mf)$ of $CC_{D'}^{\log}(j_{!}\mf)$ 
by the canonical morphism $\tau_{D'}\colon T^{*}X\rightarrow T^{*}X(\log D')$
is of dimension $d=\dim X$,
then we obtain computations of $SS(j_{!}\mf)$ and $CC(j_{!}\mf)$
in codimension $2$:

\begin{thm}[{Theorem \ref{mainthmosct}, Corollary \ref{corosct}}]
\label{thmintro}
Suppose that $X$ is purely of dimension $d$. 
Let $D'$ be a union of irreducible components of $D$ such that
$\mf$ is wildly ramified along any irreducible component of $D$ not contained in $D'$.
Assume that the ramification of $\mf$ is $\log$-$D'$-clean along $D$ 
and that
the inverse image $\tau_{D'}^{-1}(S_{D'}^{\log}(j_{!}\mf))$ of the support $S_{D'}^{\log}(j_{!}\mf)$ of $CC_{D'}^{\log}(j_{!}\mf)$ 
by the canonical morphism $\tau_{D'}\colon T^{*}X\rightarrow T^{*}X(\log D')$ is of dimension $d$.
Let 
\begin{equation}
\tau_{D'}^{!}\colon Z_{d}(S_{D'}^{\log}(j_{!}\mf))
=CH_{d}(S_{D'}^{\log}(j_{!}\mf))\rightarrow CH_{d}(\tau_{D'}^{-1}(S_{D'}^{\log}(j_{!}\mf)))=Z_{d}(\tau_{D'}^{-1}(S_{D'}^{\log}(j_{!}\mf))) \notag
\end{equation}
denote the Gysin homomorphism (\cite[6.6]{ful}) defined by the l.c.i.\ morphism $\tau_{D'}$.
If the bases of irreducible components of $\tau_{D'}^{-1}(S_{D'}^{\log}(j_{!}\mf))\subset T^{*}X$
are of codimension $\le 2$ in $X$,
then we have
\begin{equation}
\label{eqintross}
SS(j_{!}\mf)=\tau_{D'}^{-1}(S_{D'}(j_{!}\mf))
\end{equation}
and we have
\begin{equation}
\label{eqintrocc}
CC(j_{!}\mf)=\tau_{D'}^{!}CC_{D'}^{\log}(j_{!}\mf)
\end{equation}
in $Z_{d}(\tau_{D'}^{-1}(S_{D'}^{\log}(j_{!}\mf)))$.
\end{thm}

\noindent
We can remove the assumption on the dimension of $\tau_{D'}^{-1}(S_{D'}^{\log}(j_{!}\mf))$ in Theorem \ref{thmintro}
by admitting blowing up $X$ along a closed subscheme of $D$:

\begin{prop}[{Corollary \ref{corsddim}}]
\label{propintroblup}
Suppose that $X$ is purely of dimension $d$ and
that the ramification of $\mf$ is $\log$-$D'$-clean along $D$
for some union $D'$ of irreducible components of $D$.
Then there exist a blow-up 
$f\colon X'\rightarrow X$ of $X$ along a closed subscheme of $D$
and a union $E'$ of irreducible components of $(f^{*}D)_{\red}$
such that the ramification of $f^{*}\mf$ is $\log$-$E'$-clean along $(f^{*}D)_{\red}$,
that $f^{*}\mf$ is  
wildly ramified along any irreducible component of $(f^{*}D)_{\red}$ not contained in $E'$,
and that the inverse image $\tau_{E'}^{-1}(S_{E'}^{\log}(j'_{!}f^{*}\mf))$ is purely of dimension $d$,
where $j'\colon f^{*}U\rightarrow X'$ is the base change of $j$ by $f$.
\end{prop}

\noindent
We can expect Theorem \ref{thmintro} holds
without the assumption
on the codimensions of the bases of $\tau_{D'}^{-1}(S_{D'}^{\log}(j_{!}\mf))$.
In fact, under the assumptions of Theorem \ref{thmintro} without the assumption
on the codimensions of the bases of $\tau_{D'}^{-1}(S_{D'}^{\log}(j_{!}\mf))$,
the pull-back $\tau_{D'}^{!}CC_{D'}^{\log}(j_{!}\mf)$ 
determines the characteristic cycle $CC(j_{!}\mf)$ 
in the following sense:

\begin{prop}[{Proposition \ref{prophicc}}]
\label{propintrohi}
Suppose that $X$ is purely of dimension $d$. 
Let $D'$ be a union of irreducible components of $D$. 
Let $\mf_{i}$ for $i=0,1$ be smooth sheaves of $\Lambda$-modules of
rank $1$ on $U$ such that the ramifications of $\mf_{i}$ are $\log$-$D'$-clean along $D$ and that
$\mf_{i}$ are wildly ramified along any irreducible component of $D$ not contained in $D'$.
If $\tau_{D'}^{-1}(S_{D'}^{\log}(j_{!}\mf_{i}))$ are of dimension $d$ for $i=0,1$ 
and if $\tau_{D'}^{!}CC_{D'}^{\log}(j_{!}\mf_{0})=\tau_{D'}^{!}CC_{D'}^{\log}(j_{!}\mf_{1})$,
then we have $CC(j_{!}\mf_{0})=CC(j_{!}\mf_{1})$.
\end{prop}

We give a brief outline of the proof of Theorem \ref{thmintro}.
If $X$ is a surface, then the equality (\ref{eqintrocc}) is obtained by an explicit 
computation of $\tau_{D'}^{!}CC_{D'}^{\log}(j_{!}\mf)$ and the comparison 
with the computation of $CC(j_{!}\mf)$ given in \cite[Theorem 6.1]{yacc}.
In general, the proof of the equality (\ref{eqintrocc}) is
reduced to the case where $X$ is a surface, and
we obtain the equality (\ref{eqintross}) by taking the supports of the
both sides of the equality (\ref{eqintrocc}). 
To reduce the proof of the equality (\ref{eqintross}) to the surface case, we use the equality (\ref{eqintrocc}) in codimension $1$,
which is obtained by the computation of $CC(j_{!}\mf)$ given in \cite[Proposition 4.13, Theorem 7.14]{sacc},
and use the pull-back formulas for $CC(j_{!}\mf)$ (\cite[Theorem 7.6]{sacc}) and $CC_{D'}^{\log}(j_{!}\mf)$ (Proposition \ref{proppblogcc})
with taking an appropriate immersion from a smooth surface.
As a key of the reduction, we
prove that $SS(j_{!}\mf)$ locally has at most one irreducible component
whose base is of codimension $2$ in $X$.
The partially logarithmic ramification theory enables us to make
the reduction by allowing
the divisor $D$ on $X$ not to have simple normal crossings
as long as $D$ has smooth irreducible components. 
We devote almost a half of this article to develop the partially logarithmic ramification theory where the divisor $D$ on $X$ has smooth irreducible components but does not
necessarily have simple normal crossings.

We briefly explain the content of each section.
We devote from Section \ref{ariram} to Section \ref{secdil}
to the study of ramification theory.
In Subsection \ref{ariramloc}, we recall the arithmetic theories of logarithmic and non-logarithmic ramifications of complete discrete
valuation fields possibly with imperfect residue fields
introduced by Brylinski (\cite{br}), Kato (\cite{kasw}, \cite{kalog}), Matsuda (\cite{ma}), 
Abbes-Saito (\cite{as1}, \cite{asan}), and Saito (\cite{sacot}).
From Subsection \ref{ariramgl} to Section \ref{secdil},
we develop the partially logarithmic ramification theory, 
which admits the log poles to be along a
divisor $D'$ on $X$ with simple normal crossings contained in the boundary $D$.
We construct a conductor $\cform^{D'}(\mf)$ called $\log$-$D'$-characteristic form in Subsection \ref{ariramgl} and
compare the characteristic forms in various settings in Subsection \ref{sscompcf}.
We introduce the notion that the ramification of $\mf$ is $\log$-$D'$-clean along $D$ by using $\cform^{D'}(\mf)$ in Subsection \ref{sslogdpcl}.
Subsection \ref{sspresecfive} is just preparation 
for Subsections \ref{sscandss}, \ref{sshicc}, and Section \ref{ssrk1} and is not used in the other sections. 
As a refinement of the corresponding result \cite[Proposition 2.34]{yacc} in logarithmic ramification theory,
we prove that the canonical morphism $j_{!}\mf\rightarrow Rj_{*}\mf$
is an isomorphism when the ramification of $\mf$ is $\log$-$D'$-clean along $D$ and when $\mf$ is wildly ramified along all irreducible components of $D$ 
in Subsection \ref{seccldim} by using the dilatations constructed in Subsection \ref{dil}.

The singular support and the characteristic cycle of a constructible sheaf on $X$ are recalled in Subsections \ref{ssss} and \ref{sscc}, respectively.
We devote Section \ref{seclogtrans} to construct and study
an algebraic cycle $CC_{D'}^{\log}(j_{!}\mf)$ on $T^{*}X(\log D')$ called the $\log$-$D'$-characteristic cycle.
We construct $CC_{D'}^{\log}(j_{!}\mf)$ by using the conductor $\cform^{D'}(\mf)$ in Subsection \ref{sslogtrans}
when the ramification of $\mf$ is $\log$-$D'$-clean along $D$.
We construct a closed conical subset $S_{D'}(j_{!}\mf)$ of $T^{*}X$ 
in Subsection \ref{ssinvimlogss} by using the inverse image of the support $S_{D'}^{\log}(j_{!}\mf)$ 
of $CC_{D'}^{\log}(j_{!}\mf)$ by the canonical morphism $\tau_{D'}\colon T^{*}X\rightarrow T^{*}X(\log D')$
and prove that $S_{D'}(j_{!}\mf)$ contains the singular support $SS(j_{!}\mf)$.
For the proof of the inclusion, 
we a proposition prepared in Subsection \ref{ssredtw}
to reduce the proof to the case where $\mf$ is wildly ramified along
any irreducible component of $D$, and use
the the isomorphism $j_{!}\mf\rightarrow Rj_{*}\mf$ given in Section \ref{seccldim}.
In Subsection \ref{sscandss}, we refine the result.
More precisely, we prove that $SS(j_{!}\mf)$ is a union of irreducible components of the inverse image $\tau_{D'}^{-1}(S_{D'}^{\log}(j_{!}\mf))
\subset S_{D'}(j_{!}\mf)$ 
when $\tau_{D'}^{-1}(S_{D'}^{\log}(j_{!}\mf))$ is of dimension $\dim X$.
We prove Proposition \ref{propintroblup} in Subsection \ref{sscandss}
and we prove Proposition \ref{propintrohi} in Subsection \ref{sshicc}.
In Section \ref{ssrk1}, which is the last section of this article,
we consider computations of $SS(j_{!}\mf)$ and $CC(j_{!}\mf)$ in terms of ramification theory.
We first recall the computations of $CC(j_{!}\mf)$ given in \cite{sacc} 
in the case where the ramification of $\mf$ is non-degenerate along $D$ 
and given in \cite{yacc} in the case where $X$ is a surface in Subsection \ref{sscompndeg}.
We then prove Theorem \ref{thmintro} 
in Subsection \ref{sscompccvar} by reducing the proof to
the case where $X$ is a surface.
Finally, we give a few comments in the exceptional case in 
Subsection \ref{sscompccvarrem}.

The author would like to express her sincere gratitude to Takeshi Saito for asking 
about computations of singular supports and characteristic cycles of rank 1 sheaves 
in codimension 2.
She is very grateful to him for his careful reading of the manuscript,
for his pointing out mistakes,  
for a lot of comments on earlier versions, and for discussions.
The research is supported by JSPS KAKENHI Grant Numbers JP 19K21020
and 21K13769.

\begin{conv}
In this article, the symbol $k$ denotes a perfect field of characteristic $p>0$ and $X$ denotes a smooth scheme over $k$.
Let $U$ be an open subscheme of $X$ and let $D=X-U$ be the complement of $U$ in $X$.
We assume that $D$ is a divisor on $X$ whose irreducible components
$\{D_{i}\}_{i\in I}$ are smooth.
The symbol $\Lambda$ denotes a finite field of characteristic prime to $p$
and the symbol $\mf$ denotes a smooth sheaf of $\Lambda$-modules of rank 1 on $U$,
which is the sheaf we mainly consider in this article.
For a subset $I'\subset I$ of the index set $I$ of irreducible components of $D$,
we say that $D'=\bigcup_{i\in I'}D_{i}$ {\it has simple normal crossings}
(or $D'$ {\it is a divisor on $X$ with simple normal crossings}) if $I'=\emptyset$ or if $D'$ is a divisor on $X$ with simple normal crossings.
The cardinality of a subset $I'\subset I$ is sometimes denoted by $\sharp I'$.
Except for Section \ref{ss}, we denote the canonical open immersion from $U$ to $X$ by $j\colon U\rightarrow X$.

For a point $\mathfrak{p}\in X$ of codimension $1$, the local ring $\dvr_{X,\mathfrak{p}}$ is a discrete valuation ring. 
We call the fraction field of the completion $\hat{\dvr}_{X,\mathfrak{p}_{i}}$ of $\dvr_{X,\mathfrak{p}}$ the {\it local field} at $\mathfrak{p}$.
We denote the local field at the generic point $\mathfrak{p}_{i}$ of $D_{i}$ for $i\in I$ by $K_{i}$
and the valuation ring of $K_{i}$ by $\dvr_{K_{i}}$.
For a field $K$, the absolute Galois group of $K$ is denoted by $G_{K}$.
Then there is a canonical morphism $G_{K_{i}}^{\ab}\rightarrow \pi_{1}^{\ab}(U)$ of abelianized fundamental groups.
Except in Subsection  \ref{ariramloc}, we denote
by $\chi\colon \pi^{\ab}(U)\rightarrow \Lambda^{\times}$
the character corresponding to $\mf$.
For the character $\chi\colon \pi^{\ab}(U)\rightarrow \Lambda^{\times}$,
let $\chi|_{K_{i}}\colon G_{K_{i}} \rightarrow \Lambda^{\times}$
denote the composition
$G_{K_{i}}\rightarrow G_{K_{i}}^{\ab}\rightarrow \pi_{1}^{\ab}(U) \xrightarrow{\chi} \Lambda^{\times}$ of the canonical morphisms and $\chi$. 
We fix an inclusion $\psi\colon \Lambda\rightarrow \mathbf{Q}/\mathbf{Z}$ and often regard $\chi|_{K_{i}}$ as an element of $H^{1}(K,\mathbf{Q}/\mathbf{Z})=\Hom_{\cont}(G_{K},\mathbf{Q}/\mathbf{Z})$.

For a smooth scheme $Y$ over $k$, the cotangent bundle of $Y$
is denoted by $T^{*}Y$
and the zero section of $T^{*}Y$ is denoted 
by $T^{*}_{Y}Y$.
The symbol $T^{*}_{y}Y$ denotes the fiber of $T^{*}Y$ at a closed point $y\in Y$.
If $E$ is a divisor on $Y$ with simple normal crossings,
let $T^{*}Y(\log E)$ denote the logarithmic cotangent bundle with logarithmic poles along $E$.
The zero section and the fiber at $y\in Y$ of $T^{*}Y(\log E)$ are denoted by $T^{*}_{Y}Y(\log E)$ and $T^{*}_{y}Y(\log E)$, respectively.
For a closed subset $T$ of a scheme $S$ over $k$ of finite type, let $Z_{r}(T)$ denote the group of $r$-cycles on $T$  
and let $CH_{r}(T)$ denote the Chow group of $r$-cycles on $T$,
where we endow $T$ with the reduced subscheme structure.
If $f\colon S\rightarrow T$ is a morphism of schemes of finite types over $k$
and if $C$ is a closed subscheme of $T$, then we regard the inverse image $f^{-1}(C)$
as the fiber product $C\times_{T}S$.
We denote the algebraic cycle on a scheme $S$ of finite type over $k$ defined by a closed subscheme $C$ of $S$ by $[C]$.
\end{conv}

\tableofcontents

\section{Conductors}
\label{ariram}
See Conventions for the notation.
In this section, we recall and unify the arithmetic logarithmic and non-logarithmic ramification theories introduced and developed by Brylinski (\cite{br}), Kato (\cite{kasw}, \cite{kalog}), Matsuda (\cite{ma}), 
Abbes-Saito (\cite{as1}, \cite{asan}), and Saito (\cite{sacot}).
Let $D'$ be a divisor on $X$ with simple normal crossings
contained in $D$.
We introduce an invariant $\cform^{D'}(\mf)$ called the $\log$-$D'$-characteristic form of $\mf$ in Definition \ref{defcform}, and
a notion that the ramification of $\mf$ is $\log$-$D'$-clean 
along $D$ by using $\cform^{D'}(\mf)$ in Definition \ref{deflogdcl}.
The $\log$-$D'$-characteristic form $\cform^{D'}(\mf)$ is a key ingredient of computations of  
the singular support $SS(j_{!}\mf)$ and the characteristic cycle $CC(j_{!}\mf)$ of the zero extension $j_{!}\mf$ of $\mf$ in Subsection \ref{sscompccvar}.
We recall the definitions of $SS(j_{!}\mf)$ and $CC(j_{!}\mf)$ in Section \ref{ss}.
\subsection{Local definition}
\label{ariramloc}

We briefly recall the logarithmic and non-logarithmic 
ramification theories (\cite{br},  
\cite{kasw}, \cite{ma}, \cite{asan}, \cite{yafil}) for a complete discrete valuation field.

Let $K$ be a complete discrete valuation field of characteristic $p>0$.
For the Witt ring $W_{s}(K)$ of length $s\in\mathbf{Z}_{\ge 0}$, let $V\colon W_{s}(K)\rightarrow W_{s}(K)$  
denote the Verschiebung:
\begin{equation}
V\colon W_{s}(K)\rightarrow W_{s+1}(K)\; ;\; (a_{s-1},a_{s-2},\ldots,a_{0})\mapsto (0,a_{s-1},a_{s-2},\ldots,a_{0}). \notag
\end{equation}
By convention, we have $W_{0}(K)=0$ and $W_{1}(K)=K$.
Let $\ord_{K}\colon K\rightarrow \mathbf{Z}\cup\{\infty\}$ denote the normalized valuation of $K$.

\begin{defn}[{\cite[Proposition 1]{br}, \cite[Definition 1.12]{yafil}}, {cf.\ \cite[3.1]{ma}}]
\label{deffilw}

Let $K$ be a complete discrete valuation field of characteristic $p>0$.
Let $s\in \mathbf{Z}_{\ge 0}$.

\begin{enumerate}
\item We define the {\it normalized valuation} $\ord_{K}\colon W_{s}(K)\rightarrow \mathbf{Z}\cup\{\infty\}$ of $W_{s}(K)$ by
\begin{equation}
\ord_{K}(a)=\min_{0\le i \le s-1}\{p^{i}\ord_{K}(a_{i})\} \notag
\end{equation}
for $a=(a_{s-1},\ldots,a_{0})\in W_{s}(K)$.
\item We define an increasing filtration $\{\fillog_{n}W_{s}(K)\}_{n\in \mathbf{Z}}$
of $W_{s}(K)$ by
\begin{equation}
\fillog_{n}W_{s}(K)= \{a\in W_{s}(K)\; | \; \ord_{K}(a)\ge -n\}. \notag
\end{equation}
\item We define an increasing filtration $\{\fil_{m}W_{s}(K)\}_{m\in \mathbf{Z}_{\ge 1}}$ of $W_{s}(K)$ by
\begin{equation}
\fil_{m}W_{s}(K)=\fillog_{m-1}W_{s}(K)+V^{s-s'}(\fillog_{m}W_{s'}(K)), \notag
\end{equation}
where $s'=\min\{s,\ord_{p}(m)\}$ and $\ord_{p}$ denotes the $p$-adic valuation.
\end{enumerate}
\end{defn}

\begin{rem}
\label{remfilwsloc}
Let $K$ be a complete discrete valuation field of characteristic $p>0$.
By (3) in Definition \ref{deffilw}, the following relations hold between the two filtrations $\{\fillog_{n}W_{s}(K)\}_{n\in \mathbf{Z}}$ and
$\{\fil_{m}W_{s}(K)\}_{m\in \mathbf{Z}_{\ge 1}}$:
For $m=1$, we have \begin{equation}
\label{fileqzows}
\fil_{1}W_{s}(K)=\fillog_{0}W_{s}(K),
\end{equation} 
since $\ord_{p}(1)=0$.
For $m\in \mathbf{Z}_{\ge 1}$,
we have 
\begin{equation}
\label{fileqmnws}
\fillog_{m-1}W_{s}(K)\subset \fil_{m}W_{s}(K)\subset \fillog_{m}W_{s}(K), 
\end{equation}
since $V^{s-s'}(\fillog_{m}W_{s'}(K))\subset \fillog_{m}W_{s}(K)$
for all $m\in \mathbf{Z}$ and all $s'\in \mathbf{Z}_{\ge 0}$
such that $s'\le s$.
\end{rem}

For a complete discrete valuation field $K$ of characteristic $p>0$,
let $H^{1}(K,\mathbf{Z}/n\mathbf{Z})$
denote the Galois cohomology group 
\begin{equation}
H^{1}(G_{K},\mathbf{Z}/n\mathbf{Z})=\Hom_{\mathrm{cont}}(G_{K},\mathbf{Z}/n\mathbf{Z}) \notag
\end{equation}
for $n\in \mathbf{Z}_{\ge 1}$.
We regard $H^{1}(K,\mathbf{Z}/n\mathbf{Z})$ 
as a subgroup of 
\begin{equation}
H^{1}(K,\mathbf{Q}/\mathbf{Z})=\varinjlim_{m} H^{1}(K,\mathbf{Z}/m\mathbf{Z}).  \notag
\end{equation}
By the Artin-Schreier-Witt theory, for $s\in \mathbf{Z}_{\ge 0}$, there is an exact sequence
\begin{equation}
\label{esdelta}
0\rightarrow W_{s}(\mathbf{F}_{p})\rightarrow W_{s}(K) \xrightarrow{F-1}
W_{s}(K)\xrightarrow{\delta_{s}} H^{1}(K,\mathbf{Z}/p^{s}\mathbf{Z})\rightarrow 0. 
\end{equation}
Each of filtrations $\{\fillog_{n}W_{s}(K)\}_{n\in \mathbf{Z}}$ and
$\{\fil_{m}W_{s}(K)\}_{m\in \mathbf{Z}_{\ge 1}}$ of $W_{s}(K)$ defines
a filtration of $H^{1}(K,\mathbf{Q}/\mathbf{Z})$ as follows:

\begin{defn}[{\cite[Corollary (2.5), Theorem (3.2) (1)]{kasw}, \cite[Definition 1.13]{yafil}}, {cf.\ \cite[Definition 3.1.1]{ma}}]
\label{deffilh}

Let $K$ be a complete discrete valuation field of characteristic $p>0$.
Let $H^{1}(K,\mathbf{Q}/\mathbf{Z})'$ denote the prime to $p$-part of the Galois cohomology group $H^{1}(K,\mathbf{Q}/\mathbf{Z})$.
\begin{enumerate}
\item We define an increasing filtration $\{\fillog_{n}H^{1}(K,\mathbf{Q}/\mathbf{Z})\}_{n\in \mathbf{Z}}$
of $H^{1}(K,\mathbf{Q}/\mathbf{Z})$ by
\begin{equation}
\fillog_{n}H^{1}(K,\mathbf{Q}/\mathbf{Z})=
H^{1}(K,\mathbf{Q}/\mathbf{Z})'+\bigcup_{s\ge 0}\delta_{s}(\fillog_{n}W_{s}(K)). \notag
\end{equation}
\item We define an increasing filtration $\{\fil_{m}H^{1}(K,\mathbf{Q}/\mathbf{Z})\}_{m\in \mathbf{Z}_{\ge 1}}$
of $H^{1}(K,\mathbf{Q}/\mathbf{Z})$ by
\begin{equation}
\fil_{m}H^{1}(K,\mathbf{Q}/\mathbf{Z})=
H^{1}(K,\mathbf{Q}/\mathbf{Z})'+\bigcup_{s\ge 0}\delta_{s}(\fil_{m}W_{s}(K)). \notag
\end{equation}
\end{enumerate}
\end{defn}

\begin{rem}
\label{remfilhoneloc}
By the relations (\ref{fileqzows}) and (\ref{fileqmnws}) of filtrations of $W_{s}(K)$, respectively,
the following relations hold between the two filtrations 
$\{\fillog_{n}H^{1}(K,\mathbf{Q}/\mathbf{Z})\}_{n\in \mathbf{Z}}$ and $\{\fil_{m}H^{1}(K,\mathbf{Q}/\mathbf{Z})\}_{m\in \mathbf{Z}_{\ge 1}}$:
For $m=1$, we have
\begin{equation}
\label{filheq}
\fil_{1}H^{1}(K,\mathbf{Q}/\mathbf{Z})
=\fillog_{0}H^{1}(K,\mathbf{Q}/\mathbf{Z}). 
\end{equation}
For $m\in \mathbf{Z}_{\ge 1}$, we have
\begin{equation}
\label{filcont}
\fillog_{m-1}H^{1}(K,\BQ/\BZ)\subset \fil_{m}H^{1}(K,\BQ/\BZ)
\subset \fillog_{m}H^{1}(K,\BQ/\BZ).
\end{equation}
\end{rem}

We recall two kinds of conductors $\sw(\chi)$ and $\dt(\chi)$ for a character $\chi\in H^{1}(K,\mathbf{Q}/\mathbf{Z})$ of 
$G_{K}$ for a complete discrete valuation field $K$ of characteristic $p>0$.
The conductors $\sw(\chi)$ and $\dt(\chi)$ are defined by
using the filtrations $\{\fillog_{n}H^{1}(K,\mathbf{Q}/\mathbf{Z})\}_{n\in \mathbf{Z}}$ and $\{\fil_{m}H^{1}(K,\mathbf{Q}/\mathbf{Z})\}_{m\in \mathbf{Z}_{\ge 1}}$ of $H^{1}(K,\mathbf{Q}/\mathbf{Z})$
(Definition \ref{deffilh}), respectively.

\begin{defn}[{\cite[Definition (2.2)]{kasw}, \cite[Definition 1.14]{yafil}}]
\label{defwedt}
Let $K$ be a complete discrete valuation field of characteristic $p>0$
and let $\chi\in H^{1}(K,\mathbf{Q}/\mathbf{Z})$
be a character of $G_{K}$. 
\begin{enumerate}
\item We define the {\it Swan conductor} $\sw(\chi)$ of $\chi$ to be the smallest non-negative integer $n$ such that
$\chi\in \fillog_{n}H^{1}(K,\mathbf{Q}/\mathbf{Z})$.
\item We define the {\it total dimension} $\dt(\chi)$ of $\chi$ to be the smallest positive integer $m$ such that
$\chi\in \fil_{m}H^{1}(K,\mathbf{Q}/\mathbf{Z})$.
\end{enumerate}
\end{defn}

\begin{rem}
\label{remtameloc}
Let $K$ be a complete discrete valuation field of characteristic $p>0$.
\begin{enumerate}
\item 
By the construction of $\fillog_{n}H^{1}(K,\mathbf{Q}/\mathbf{Z})$ (resp.\ $\fil_{m}H^{1}(K,\mathbf{Q}/\mathbf{Z})$) as the sum of
the prime to $p$-part $H^{1}(K,\mathbf{Q}/\mathbf{Z})'$ of $H^{1}(K,\mathbf{Q}/\mathbf{Z})$ and a submodule of the $p$-part of 
$H^{1}(K,\mathbf{Q}/\mathbf{Z})$,
the Swan conductor $\sw(\chi)$ (resp.\ the total dimension $\dt(\chi)$) for
$\chi\in H^{1}(K,\mathbf{Q}/\mathbf{Z})$ is dependent only on the $p$-part of $\chi$.

\item 
Let $\chi\in H^{1}(K,\mathbf{Q}/\mathbf{Z})$. 
By (\ref{filheq}), the Swan conductor $\sw(\chi)$ is $0$ if and only if 
the total dimension $\dt(\chi)$ is $1$.
Generally,
the total dimension
$\dt(\chi)$ is equal to $\sw(\chi)$ or $\sw(\chi)+1$ 
by (\ref{filcont}).
By \cite[Propositions 3.7 (3), 3.15 (4)]{as1}, \cite[Corollaire 9.12]{asan},
and \cite[Theorem 3.1]{yafil},
if the residue field $F_{K}$ of $K$ is perfect, then 
$\fillog_{n}H^{1}(K,\mathbf{Q}/\mathbf{Z})=\fil_{n+1}H^{1}(K,\mathbf{Q}/\mathbf{Z})$ for all $n\in\mathbf{Z}_{\ge 0}$ and
we have $\dt(\chi)=\sw(\chi)+1$ for all $\chi \in H^{1}(K,\mathbf{Q}/\mathbf{Z})$.

\item 
Both of the filtrations $\{\fillog_{n}H^{1}(K,\mathbf{Q}/\mathbf{Z})\}_{n\in \mathbf{Z}}$ and $\{\fil_{m}H^{1}(K,\mathbf{Q}/\mathbf{Z})\}_{m\in \mathbf{Z}_{\ge 1}}$ of $H^{1}(K,\mathbf{Q}/\mathbf{Z})$
(Definition \ref{deffilh}) measure the ramification of a character $\chi\in H^{1}(K,\mathbf{Q}/\mathbf{Z})$ of $G_{K}$.
Indeed,
by \cite[Proposition 3.1.5 (2)]{as1} and \cite[Corollaire 9.12]{asan},
the ramification of $\chi\in H^{1}(K,\mathbf{Q}/\mathbf{Z})$ is
tamely ramified if and only if the Swan conductor $\sw(\chi)$ is $0$,
or equivalently the total dimension $\dt(\chi)$ is $1$ by (2).
\end{enumerate}
\end{rem}

According to Remark \ref{remtameloc} (2),
we introduce the types of $\chi$ for a character $\chi\in H^{1}(K,\mathbf{Q}/\mathbf{Z})$ of $G_{K}$ that is wildly ramified.

\begin{defn}[{cf. {\cite[Definition 3.2.10]{ma}}}]
\label{deftypes}
Let $K$ be a complete discrete valuation field of characteristic $p>0$.
Let $\chi\in H^{1}(K,\mathbf{Q}/\mathbf{Z})$ be a character of $G_{K}$ such that $\sw(\chi)>0$ (or equivalently $\dt(\chi)>1$ by Remark \ref{remtameloc} (2)).
\begin{enumerate}
\item We say that a character $\chi\in H^{1}(K,\mathbf{Q}/\mathbf{Z})$ is of
{\it type} $\mI$ if  $\dt(\chi)=\sw(\chi)+1$.
\item We say that a character $\chi\in H^{1}(K,\mathbf{Q}/\mathbf{Z})$ is of
{\it type} $\mII$ if $\dt(\chi)=\sw(\chi)$.
\end{enumerate}
\end{defn}

Let $K$ be a complete discrete valuation field of characteristic $p>0$.
We recall two kinds of refined conductors $\rsw(\chi)$ and $\cform(\chi)$ for a character $\chi\in H^{1}(K,\mathbf{Q}/\mathbf{Z})$ of $G_{K}$, which 
will appear as stalks of the $\log$-$D'$-characteristic form $\cform^{D'}(\mf)$ defined in the next subsection.
The refined conductors $\rsw(\chi)$ and $\cform(\chi)$ are 
defined to be elements of graded quotients of the filtrations $\{\fillog_{n}\Omega^{1}_{K}\}_{n\in \mathbf{Z}_{\ge 0}}$ and 
$\{\fil_{m}\Omega^{1}_{K}\}_{m\in \mathbf{Z}_{\ge 1}}$
of the module $\Omega_{K}^{1}$ of differential forms, 
which are recalled in the following, respectively.

For a perfect field $k$ of characteristic $p>0$ and a $k$-algebra $A$,
let $\Omega^{1}_{A}$ denote the module of differential forms:
\begin{equation}
\Omega^{1}_{A}=\Omega^{1}_{A/A^{p}}=\Omega^{1}_{A/k}. \notag
\end{equation}
Let $\dvr_{K}$ denote the valuation ring of $K$
and $\mathfrak{m}_{K}$
the maximal ideal of $\dvr_{K}$.
Let $\Omega^{1}_{\dvr_{K}}(\log)$ be the module of logarithmic differential forms:
\begin{equation}
\Omega^{1}_{\dvr_{K}}(\log)=(\Omega^{1}_{\dvr_{K}}\oplus (\dvr_{K}\otimes_{\mathbf{Z}} K^{\times}))/B, \notag
\end{equation}
where $B$ is the sub $\dvr_{K}$-module of $\Omega^{1}_{\dvr_{K}}\oplus 
(\dvr_{K}\otimes_{\mathbf{Z}} K^{\times})$ generated by 
$(da,0)-(0,a\otimes a)$ for all $a\in \dvr_{K}-\{0\}$.
Let $\dlog a\in \Omega^{1}_{\dvr_{K}}(\log)$ denote the class of
$(0,1\otimes a)$ for $a\in K^{\times}$.
Then there is a canonical isomorphism
\begin{equation}
\label{isomdifk}
\Omega^{1}_{\dvr_{K}}\otimes_{\dvr_{K}}K\xrightarrow{\sim}\Omega^{1}_{\dvr_{K}}(\log)\otimes_{\dvr_{K}}K\; ;\; da\otimes b \mapsto da\otimes b,
\end{equation}
where the image of 
$da\otimes a^{-1}\in \Omega_{\dvr_{K}}^{1}\otimes_{\dvr_{K}}K$ is $\dlog a\in \Omega^{1}_{\dvr_{K}}(\log)\otimes_{\dvr_{K}}K$ for $a\in \dvr_{K}-\{0\}$.
We define $\fillog_{n}\Omega^{1}_{K}$ for $n\in \mathbf{Z}_{\ge 0}$
and $\fil_{m}\Omega^{1}_{K}$ for $m\in \mathbf{Z}_{\ge 1}$ to be the images of the compositions
\begin{align}
\Omega^{1}_{\dvr_{K}}(\log)\otimes_{\dvr_{K}}&\mathfrak{m}_{K}^{-n}
\rightarrow \Omega^{1}_{\dvr_{K}}(\log)\otimes_{\dvr_{K}}K \underset{(\ref{isomdifk})}{\simeq}
\Omega^{1}_{\dvr_{K}}\otimes_{\dvr_{K}}K\xrightarrow{\sim}\Omega^{1}_{K} \notag 
\end{align}
and
\begin{align}
&\quad \Omega^{1}_{\dvr_{K}}\otimes_{\dvr_{K}}\mathfrak{m}_{K}^{-m}
\rightarrow \Omega^{1}_{\dvr_{K}}\otimes_{\dvr_{K}}K\xrightarrow{\sim}\Omega^{1}_{K}\notag
\end{align}
of the canonical morphisms, respectively.
These two filtrations $\{\fillog_{n}\Omega^{1}_{K}\}_{n\in \mathbf{Z}_{\ge 0}}$ and 
$\{\fil_{m}\Omega^{1}_{K}\}_{m\in \mathbf{Z}_{\ge 1}}$ of $\Omega_{K}^{1}$ 
are increasing filtrations, 
since $\mathfrak{m}_{K}^{-m'}\subset \mathfrak{m}_{K}^{-n'}$
if $m'\le n'$. 
If $\pi$ is a uniformizer of $K$, then we have
\begin{equation}
\fillog_{n}\Omega^{1}_{K}=\{(\alpha+\beta\dlog \pi)/\pi^{n}\; |\;\alpha \in \Omega^{1}_{\dvr_{K}}, \beta\in \dvr_{K}\}\subset \Omega^{1}_{K} \notag
\end{equation}
for $n\in \mathbf{Z}_{\ge 0}$ and 
\begin{equation}
\fil_{m}\Omega^{1}_{K}=\{\alpha/\pi^{m}\;|\; \alpha\in \Omega^{1}_{\dvr_{K}}\}
\subset \Omega^{1}_{K} \notag
\end{equation}
for $m\in \mathbf{Z}_{\ge 1}$.

We put $\grlog_{n}=\fillog_{n}/\fillog_{n-1}$ for $n\in \mathbf{Z}_{\ge 1}$
and $\gr_{m}=\fil_{m}/\fil_{m-1}$ for $m\in \mathbf{Z}_{\ge 2}$.
Then the morphism $\delta_{s}$ (\ref{esdelta}) for $s\in \mathbf{Z}_{\ge 0}$ induces the two morphisms
\begin{align}
\label{deltasloclog}
\delta_{s}^{(n)}&\colon \grlog_{n}W_{s}(K)\rightarrow \grlog_{n}H^{1}(K,\mathbf{Q}/\mathbf{Z}), 
\end{align}
for $n\in \mathbf{Z}_{\ge 1}$ and 
\begin{align}
\label{deltaslocnlog}
\delta'^{(m)}_{s}&\colon \gr_{m}W_{s}(K)\rightarrow \gr_{m}H^{1}(K,\mathbf{Q}/\mathbf{Z})
\end{align}
for $m\in \mathbf{Z}_{\ge 2}$.
Let $-F^{s-1}d\colon W_{s}(K)\rightarrow \Omega_{K}^{1}$ be 
the morphism
\begin{equation}
-F^{s-1}d\colon W_{s}(K)\rightarrow \Omega^{1}_{K}\; ;\; (a_{s-1},\ldots,a_{0})
\mapsto -\sum_{i=0}^{s-1}a_{i}^{p^{i}-1}da_{i}.  \notag
\end{equation} 
Then the morphism $-F^{s-1}d$ induces for $n\in \mathbf{Z}_{\ge 1}$
the morphism
\begin{equation}
\label{varphilogloc}
\varphi^{(n)}_{s}\colon \grlog_{n}W_{s}(K)\rightarrow \grlog_{n}\Omega^{1}_{K}.
\end{equation}
If $F_{K}$ denotes the residue field of $K$, then there exists for $m\in \mathbf{Z}_{\ge 2}$ a unique morphism
\begin{equation}
\label{varphinlogloc}
\varphi'^{(m)}_{s} \colon \gr_{m}W_{s}(K)\rightarrow \gr_{m}\Omega^{1}_{K}\otimes_{F_{K}}F_{K}^{1/p}  
\end{equation}
satisfying the equalities
\begin{equation}
\varphi'^{(m)}_{s}(\bar{a})=\begin{cases}
\overline{-F^{s-1}da}\otimes \overline{1} & ((m,p)\neq (2,2)), \\
\overline{-F^{s-1}da}\otimes \overline{1}+\overline{d\pi/\pi^{2}}\otimes\sqrt{\overline{a_{0}\pi^{2}}} & ((m,p)= (2,2)) 
\end{cases} \notag
\end{equation}
for every $a=(a_{s-1},\ldots,a_{0})\in \fil_{m}W_{s}(K)$ and every uniformizer $\pi$ of $K$.
The existence of $\varphi_{s}'^{(m)}$ (\ref{varphinlogloc}) follows from \cite[3.2]{ma} 
for $(p,m)\neq (2,2)$ and \cite[Proposition 1.17 (i)]{yafil} for $(p,m)=(2,2)$.

\begin{prop}[{\cite[Proposition 10.7]{asan}, \cite[Proposition 3.2.3]{ma},
\cite[Proposition 1.17 (ii)]{yafil}}] 
\label{swgldef}
Let the notation be as above.
\begin{enumerate}
\item For $n\in \mathbf{Z}_{\ge 1}$, there exists a unique injection
\begin{equation}
\phi^{(n)}\colon \grlog_{n}H^{1}(K,\mathbf{Q}/\mathbf{Z})\rightarrow
\grlog_{n}\Omega^{1}_{K} \notag
\end{equation}
such that the following diagram is commutative for every $s\in\mathbf{Z}_{\ge 0}$:
\begin{equation}
\xymatrix{
\grlog_{n}W_{s}(K) \ar[rd]_{\delta_{s}^{(n)}} \ar[rr]^-{\varphi_{s}^{(n)}} & & \grlog_{n}\Omega^{1}_{K}\\
& \grlog_{n}H^{1}(K,\BQ/\BZ) \ar[ru]_-{\phi^{(n)}}. &
} \notag
\end{equation}
\item Let $F_{K}$ denote the residue field of $K$.
For $m\in \mathbf{Z}_{\ge 2}$, there exists a unique injection
\begin{equation}
\phi'^{(m)}\colon \gr_{m}H^{1}(K,\mathbf{Q}/\mathbf{Z})\rightarrow
\gr_{m}\Omega^{1}_{K}\otimes_{F_{K}}F_{K}^{1/p} \notag
\end{equation}
such that the following diagram is commutative for every $s\in\mathbf{Z}_{\ge 0}$:
\begin{equation}
\xymatrix{
\gr_{m}W_{s}(K) \ar[rd]_{\delta'^{(m)}_{s}} \ar[rr]^-{\varphi_{s}'^{(m)}} & & \gr_{m}\Omega^{1}_{K}\otimes_{F_{K}}
F_{K}^{1/p}\\
& \gr_{m}H^{1}(K,\BQ/\BZ)\ar[ru]_-{\phi^{\prime (m)}}. &
}\notag
\end{equation}
\end{enumerate}
\end{prop}

The following are the definitions of the refined conductors:

\begin{defn}[{\cite[(3.4.2)]{kasw}, \cite[Definition 1.18]{yafil}}]
\label{defrefcond}
Let $K$ be a complete discrete valuation field of characteristic $p>0$.
Let $\chi\in H^{1}(K,\mathbf{Q}/\mathbf{Z})$ be a character of $G_{K}$.
\begin{enumerate}
\item Suppose that $\sw(\chi)=n\ge 1$ (Definition \ref{defwedt} (1)).
Let $\bar{\chi}$ denote the image of $\chi$ in $\grlog_{n} H^{1}(K,\mathbf{Q}/\mathbf{Z})$.
Then we define the {\it refined Swan conductor} $\rsw(\chi)$ of $\chi$ to be
the image of $\bar{\chi}$ by the morphism $\phi^{(n)}$ in Proposition \ref{swgldef} (1).
\item Suppose that $\dt(\chi)=m\ge 2$ (Definition \ref{defwedt} (2)). 
Let $\bar{\chi}$ denote the image of $\chi$ in $\gr_{m} H^{1}(K,\mathbf{Q}/\mathbf{Z})$.
Then we define the {\it characteristic form} 
$\cform(\chi)$ of $\chi$ to be the image of 
$\bar{\chi}$ by the morphism $\phi'^{(m)}$ in Proposition \ref{swgldef} (2).
\end{enumerate}
\end{defn}

The next lemma is useful to compute the refined conductors.

\begin{lem}[{\cite[Lemmas 2.3, 2.11]{yacc}}]
\label{lemrsw}
Let $K$ be a complete discrete valuation field of characteristic $p>0$
and let $\chi\in H^{1}(K,\mathbf{Q}/\mathbf{Z})$ be a character of $G_{K}$. 
Let $a\in W_{s}(\dvr_{K})$ be an element of the Witt ring of the valuation ring $\dvr_{K}$ of $K$ of length $s\in \mathbf{Z}_{\ge 0}$
whose image by $\delta_{s}$ (\ref{esdelta}) 
is the $p$-part of $\chi$.
\begin{enumerate}
\item Suppose that $a\in \fillog_{n}W_{s}(K)$ for $n\in \mathbf{Z}_{\ge 1}$
(Definition \ref{deffilw} (2)). 
Let $\bar{a}$ denote the image of $a$ in $\grlog_{n}W_{s}(K)$.
Then the following three conditions are equivalent:
\begin{enumerate}
\item $\sw(\chi)=n$.
\item $\rsw(\chi)=\varphi_{s}^{(n)}(\bar{a})$,
where the morphism $\varphi_{s}^{(n)}$ is as in (\ref{varphilogloc}).
\item $\varphi_{s}^{(n)}(\bar{a})\neq 0$ in $\grlog_{n}\Omega_{K}^{1}$.
\end{enumerate}
\item Suppose that $a\in \fil_{m}W_{s}(K)$ for $m\in \mathbf{Z}_{\ge 2}$ (Definition \ref{deffilw} (3)).
Let $\bar{a}$ denote the image of $a$ in $\gr_{m}W_{s}(K)$.
Then the following three conditions are equivalent:
\begin{enumerate}
\item $\dt(\chi)=m$.
\item $\cform(\chi)=\varphi_{s}'^{(m)}(\bar{a})$, where the morphism $\varphi_{s}'^{(m)}$ is as in (\ref{varphinlogloc}).
\item $\varphi_{s}'^{(m)}(\bar{a})\neq 0$ in $\gr_{m}\Omega_{K}^{1}$.
\end{enumerate}
\end{enumerate}
\end{lem}

\begin{rem}
\label{remrsw}
Let $K$ be a complete discrete valuation field of characteristic $p>0$
and let $\chi\in H^{1}(K,\mathbf{Q}/\mathbf{Z})$ be a character of $G_{K}$.
If the refined conductors of $\chi$ are defined, namely if 
$\sw(\chi)\ge 1$
(or equivalently $\dt(\chi)\ge 2$ by Remark \ref{remtameloc} (2)), then
we have $\rsw(\chi)\neq 0$ and $\cform(\chi)\neq 0$ by Lemma \ref{lemrsw}.
\end{rem}

\begin{exa}
Let $\pi$ be a uniformizer of a complete discrete valuation field $K$ 
of characteristic $p>0$.
Let $u$ be a unit of the valuation ring $\dvr_{K}$ of $K$.
Suppose that the $p$-part of $\chi\in H^{1}(K,\mathbf{Q}/\mathbf{Z})$ is
the image of $a=(u/\pi, u, u/\pi^{p^{2}})\in W_{3}(K)$ by 
$\delta_{3}$ (\ref{esdelta})
and that the images $\bar{1}$ and $\bar{u}$ of $1\in \dvr_{K}$ and $u\in \dvr_{K}$ in the residue field $F_{K}$ of $K$ are 
a part of a $p$-basis of $F_{K}$ 
over a perfect subfield of $F_{K}$.
Then the smallest $n\in \mathbf{Z}_{\ge 0}$ and $m\in \mathbf{Z}_{\ge 1}$ such that
$a\in \fillog_{n}W_{3}(K)$ and that $a\in \fil_{m}W_{3}(K)$ are
$p^{2}$ and $p^{2}+1$, respectively.
By the assumption that $\bar{1}$ and $\bar{u}$ are linearly independent over $\mathbf{F}_{p}$,
we have $\sw(\chi)=p^{2}$ and $\dt(\chi)=p^{2}+1$.
Thus, by Lemma \ref{lemrsw}, we have $\rsw(\chi)=\varphi_{3}^{(p^{2})}(a)=(u^{p^{2}}\dlog \pi-(u^{p^{2}-1}+1)du)/\pi^{p^{2}}$
and $\cform(\chi)=\varphi_{3}'^{(p^{2}+1)}(a)=u^{p^{2}}d\pi/\pi^{p^{2}+1}$.
\end{exa}

\subsection{Global definition}
\label{ariramgl}
We define a refinement $\cform^{D'}(\mf)$ of the mixed refined Swan conductor of $\mf$ introduced  in \cite[4.1]{ma} in Definition \ref{defcform}.
The refinement $\cform^{D'}(\mf)$ is defined for $\mf$ and a divisor $D'$
on $X$ with simple normal crossings contained in $D$ 
(possibly $D'=\emptyset$ as is explained in Conventions).
The conductor $\cform^{D'}(\mf)$ is also a refinement of
the refined Swan conductor of $\mf$ and the characteristic form of $\mf$ introduced in \cite{kasw} and \cite{sacot}, respectively
(see Remark \ref{remcfnormal} (4)).
We construct the conductor $\cform^{D'}(\mf)$ by using the sheaf theory 
and following the constructions of conductors $\rsw(\chi)$ and $\cform(\chi)$ in the previous subsection. 

Let $W_{s}(\dvr_{U_{\et}})$ (resp.\ $W_{s}(\dvr_{U})$) denote the \'{e}tale 
(resp.\ Zariski) sheaf of Witt vectors of length $s\in \mathbf{Z}_{\ge 0}$.
Let $F\colon W_{s}(\dvr_{U_{\et}})\rightarrow W_{s}(\dvr_{U_{\et}})$ 
(resp.\ $F\colon W_{s}(\dvr_{U})\rightarrow W_{s}(\dvr_{U})$) denote
the Frobenius morphism:
\begin{align}
F\colon W_{s}(\dvr_{U_{\et}})\rightarrow W_{s}(\dvr_{U_{\et}})\; &;\;
(a_{s-1},a_{s-2},\ldots,a_{0})\mapsto (a_{s-1}^{p}, a_{s-2}^{p}, \ldots, a_{0}^{p}) \notag \\
(\text{resp.\ }F\colon W_{s}(\dvr_{U})\rightarrow W_{s}(\dvr_{U})\; &;\; 
(a_{s-1},a_{s-2},\ldots,a_{0})\mapsto (a_{s-1}^{p}, a_{s-2}^{p}, \ldots, a_{0}^{p})\,).
\notag
\end{align}
Let $j\colon U=X-D\rightarrow X$ be the canonical open immersion
and let $\varepsilon\colon X_{\et}\rightarrow X_{\mathrm{Zar}}$ be the canonical mapping from the \'{e}tale site of $X$ to the Zariski site of $X$.
Applying $(\varepsilon\circ j)_{*}$ to the exact sequence
\begin{equation}
0\rightarrow W_{s}(\mathbf{F}_{p})\rightarrow W_{s}(\dvr_{U_{\et}})
\xrightarrow{F-1} W_{s}(\dvr_{U_{\et}}) \rightarrow 0 \notag
\end{equation}
of \'{e}tale sheaves on $U$,
we obtain the exact sequence
\begin{equation}
\label{deltassh}
0\rightarrow W_{s}(\mathbf{F}_{p})\rightarrow j_{*}W_{s}(\dvr_{U})
\xrightarrow{F-1} j_{*}W_{s}(\dvr_{U}) \xrightarrow{\delta_{s,j}} R^{1}(\varepsilon\circ j)_{*}\mathbf{Z}/p^{s}\mathbf{Z} \rightarrow 0
\end{equation}
of Zariski sheaves on $X$ by $R^{1}(\varepsilon\circ j)_{*}W_{s}(\dvr_{U_{\et}})=0$.
If $V$ is an open subscheme of $X$, then the canonical morphism
\begin{equation}
H^{1}_{\et}(U\cap V,\BZ/p^{s}\BZ)\rightarrow \Gamma(V, R^{1}(\varepsilon\circ j)_{*}\BZ/p^{s}\BZ) \notag
\end{equation}
is an isomorphism, since $E^{1,0}_{2}=E_{2}^{2,0}=0$ for the spectral sequence
\begin{equation}
E^{a,b}_{2}=H^{a}_{\mathrm{Zar}}(V,R^{b}(\varepsilon\circ j)_{*}\BZ/p^{s}\BZ)
\Rightarrow H^{a+b}_{\et}(U\cap V,\BZ/p^{s}\BZ). \notag
\end{equation}

If $D'$ is a divisor on $X$ with simple normal crossings contained in $D$
and if $R=\sum_{i\in I}n_{i}D_{i}$ with $n_{i}\in \mathbf{Z}$, 
then there is a canonical injection
\begin{equation}
\Omega^{1}_{X}(\log D')(R)\rightarrow j_{*}j^{*}\Omega_{X}^{1}(\log D'). \notag
\end{equation}
Identifying $j_{*}\Omega_{U}^{1}$ with $j_{*}j^{*}\Omega_{X}^{1}(\log D')$ by the canonical isomorphisms
\begin{equation}
j_{*}\Omega_{U}^{1}\xrightarrow{\sim} j_{*}j^{*}\Omega^{1}_{X}\xrightarrow{\sim}
j_{*}j^{*}\Omega_{X}^{1}(\log D'), \notag
\end{equation}
we regard $\Omega_{X}^{1}(\log D')(R)$ as a subsheaf of $j_{*}\Omega^{1}_{U}$.

\begin{defn}[cf. {\cite[Definitions 1.25, 1.34]{yafil}}]
\label{deffilsh}
Let $I'$ be a subset of $I$ and let $D'=\bigcup_{i\in I'}D_{i}$.
Let $R=\sum_{i\in I}n_{i}D_{i}$ with $n_{i}\in \mathbf{Z}_{\ge 0}$ for $i\in I'$ and
$n_{i}\in \mathbf{Z}_{\ge 1}$ for $i\in I-I'$. 
Let $j_{i}\colon \Spec K_{i}\rightarrow X$ denote the canonical morphism from the spectrum of 
the local field $K_{i}=\Frac\hat{\dvr}_{D_{i},\mathfrak{p}_{i}}$ at the generic point $\mathfrak{p}_{i}$ of $D_{i}$ for $i\in I$.

\begin{enumerate}
\item We define a Zariski subsheaf $\fillog^{D'}_{R}j_{*}W_{s}(\dvr_{U})$ of $j_{*}W_{s}(\dvr_{U})$ to be the inverse image of 
\begin{equation}
\label{filorig}
\bigoplus_{i\in I}j_{i*}\fillog_{n_{i}}W_{s}(K_{i})\oplus\bigoplus_{i\in I-I'}j_{i*}\fil_{n_{i}}W_{s}(K_{i}) \subset \bigoplus_{i\in I}j_{i*}W_{s}(K_{i}) 
\end{equation} 
by the canonical morphism
\begin{equation}
\label{canmorwitt}
j_{*}W_{s}(\dvr_{U})\rightarrow \bigoplus_{i\in I}j_{i*}W_{s}(K_{i}). 
\end{equation}
\item We define a Zariski subsheaf $\fillog^{D'}_{R}R^{1}(\varepsilon\circ j)_{*}\BZ/p^{s}\BZ$ of $R^{1}(\varepsilon\circ j)_{*}\BZ/p^{s}\BZ$ to be 
the image of $\fillog_{R}^{D'}j_{*}W_{s}(\dvr_{U})\subset j_{*}W_{s}(\dvr_{U})$ by $\delta_{s,j}$ (\ref{deltassh}).
\item Assume that $D'$ has simple normal crossings.
We define a Zariski subsheaf $\fillog_{R}^{D'}j_{*}\Omega^{1}_{U}$ of $j_{*}\Omega^{1}_{U}$
to be $\Omega^{1}_{X}(\log D')(R)$,
which has been regarded as a subsheaf of $j_{*}\Omega_{U}^{1}$.
\end{enumerate}
\end{defn}

Under the notation in Definition \ref{deffilsh},
let $Z=\Supp(R+D'-D)$ be the support of $R+D'-D$.
Then we have $\fillog_{R-Z}^{D'}\subset \fillog_{R}^{D'}$ 
and we put 
\begin{equation}
\grlog_{R}^{D'}=\fillog_{R}^{D'}/\fillog_{R-Z}^{D'}. \notag
\end{equation}

\begin{rem}
\label{remsncdfil}
Let the notation be as in Definition \ref{deffilsh}.
\begin{enumerate}
\item Suppose that $D$ has simple normal crossings.
Then $\fillog_{R}^{D}$ is the same as $\fillog_{R}$ in
\cite[Subsection 1.3]{yafil}
and $\fillog_{R}^{\emptyset}$ is the same as $\fil_{R}$ in \cite[Subsection 1.4]{yafil}.
\item Let $D''=\bigcup_{i\in I''}D_{i}$ for $I''\subset I'$.
Since $\fil_{m}W_{s}(K)\subset \fillog_{m}W_{s}(K)$ for any $m\in \mathbf{Z}_{\ge 1}$ and for any complete discrete valuation field $K$ of characteristic $p$ by (\ref{fileqmnws}),
we have 
\begin{equation}
\label{incjwsdppdp}
\fillog_{R}^{D''}j_{*}W_{s}(\dvr_{U})\subset \fillog_{R}^{D'}j_{*}W_{s}(\dvr_{U}).
\end{equation}
The inclusion (\ref{incjwsdppdp}) induces the canonical morphism
\begin{equation}
\label{morgrjwsdppdp}
\grlog_{R}^{D''}j_{*}W_{s}(\dvr_{U})\rightarrow \grlog_{R}^{D'}j_{*}W_{s}(\dvr_{U}). 
\end{equation}
\item For a morphism $h\colon W\rightarrow X$ of smooth schemes over $k$ such that 
$h^{*}D_{i}$ for $i\in I'$ and
$(h^{*}D')_{\red}=(D'\times_{X}W)_{\red}$ are divisors on $W$ with simple normal crossings
and that $h^{*}D_{i}$ for $i\in I-I'$ are smooth divisors on $W$,
we have 
\begin{equation}
\label{inchfilws}
h^{*}\fillog_{R}^{D'}j_{*}W_{s}(\dvr_{U})\subset \fillog_{h^{*}R}^{(h^{*}D')_{\red}}j'_{*}W_{s}(\dvr_{h^{*}U}), 
\end{equation} 
where $j'\colon h^{*}U=U\times_{X}W\rightarrow W$ denotes the base change of $j\colon U\rightarrow X$ by $h$. 
Actually, let $\{E_{\theta}\}_{\theta\in \Theta}$ be the irreducible components
of $(h^{*}D)_{\red}$ and let $\Theta'\subset \Theta$ be the index set 
of irreducible components of $(h^{*}D')_{\red}$. 
We put $h^{*}R=\sum_{\theta\in \Theta}m_{\theta}E_{\theta}$,
where $m_{\theta}\in \mathbf{Z}_{\ge 0}$.
Then $m_{\theta}$ for $\theta\in \Theta-\Theta'$ is positive, 
since so is $n_{i}$ for $i\in I-I'$.
Let $L_{\theta}$ denote the local field at the generic point of $E_{\theta}$ 
for $\theta\in \Theta$ 
and let $j_{\theta}'\colon \Spec L_{\theta}\rightarrow W$ 
be the canonical morphism.
Since $\fil_{m}W_{s}(K)\subset \fillog_{m}W_{s}(K)$ for any $m\in \mathbf{Z}_{\ge 1}$ and for any complete discrete valuation field $K$ of characteristic $p$ by (\ref{fileqmnws}),
the image of the restriction of the canonical morphism
\begin{equation}
\label{arwsltheta}
j'_{*}W_{s}(\dvr_{h^{*}U})\rightarrow j'_{\theta *}W_{s}(L_{\theta}). 
\end{equation}
to $h^{*}\fillog_{R}^{D'}j_{*}W_{s}(\dvr_{U})$ is contained in $j'_{\theta *}\fillog_{m_{\theta}}W_{s}(L_{\theta})$ for $\theta\in \Theta'$.
Let $\theta \in \Theta-\Theta'$ and
let $I_{\theta}$ be the subset of $I$ consisting of $i\in I$ such that $E_{\theta}\subset h^{*}D_{i}$.
Then we have $I_{\theta}\subset I-I'$ and $m_{\theta}=\sum_{i\in I_{\theta}}n_{i}$.
We put $s_{i}'=\min\{s,\ord_{p}(n_{i})\}$ for $i\in I_{\theta}$
and $s'=\min\{s,\ord_{p}(m_{\theta})\}$.
Since $\min\{s_{i}'\; |\; i\in I_{\theta}\}\le s'$, 
there is at least one $i\in I_{\theta}$ such that $s_{i}'\le s'$,
and the image of the restriction of (\ref{arwsltheta}) to $h^{*}\fillog_{R}^{D'}j_{*}W_{s}(\dvr_{U})$ is contained in
\begin{equation}
j'_{\theta *}\fillog_{m_{\theta}-1}W_{s}(L_{\theta})+j'_{\theta *}V^{s-s'}\fillog_{m_{\theta}}W_{s'}(L_{\theta})=
j'_{\theta *}\fil_{m_{\theta}}W_{s}(L_{\theta}). \notag
\end{equation}
Thus we have (\ref{inchfilws}).
\end{enumerate}
\end{rem}

For a subset $I'\subset I$ such that $D'=\bigcup_{i\in I'}D_{i}$ has simple normal crossings 
and for $R=\sum_{i\in I}n_{i}D_{i}$, where $n_{i}\in \mathbf{Z}_{\ge 0}$ for $i\in I'$ 
and $n_{i}\in \mathbf{Z}_{\ge 1}$ for $i\in I-I'$,
we construct a morphism
\begin{equation}
\label{mordefcf}
\phi_{s}^{(D'\subset D,R)}\colon 
\grlog_{R}^{D'}R^{1}(\varepsilon\circ j)_{*}\BZ/p^{s}\BZ
\rightarrow \grlog^{D'}_{R}j_{*}\Omega^{1}_{U}\otimes_{\dvr_{Z}}\dvr_{Z^{1/p}}
=\Omega^{1}_{X}(\log D')(R)|_{Z^{1/p}} 
\end{equation}
to define the conductor $\cform^{D'}(\mf)$.
Here the symbol $Z^{1/p}$ in (\ref{mordefcf}) denotes the radicial covering
of $Z=\Supp(R+D'-D)$ defined as follows.

Let $f\colon S\rightarrow \Spec k$ be the structure morphism of a scheme $S$ over $k$.
We define the {\it radicial covering} $S^{1/p}$ of $S$ by the cartesian diagram
\begin{equation}
\label{diagradcov}
\xymatrix{
S^{1/p}\ar[r] \ar[d] \ar@{}[dr] | {\square} & S \ar[d]^-{f} \\
\Spec k \ar[r]_{F_{k}^{-1}} & \Spec k, }
\end{equation}
where the symbol $F_{k}$ denotes the absolute Frobenius morphism of $\Spec k$.
We regard $S^{1/p}$ as a scheme over $S$ by the composition
\begin{equation}
\label{radicialcvmap}
S^{1/p}\rightarrow S\xrightarrow{F_{S}} S  
\end{equation}
of the upper horizontal arrow in (\ref{diagradcov}) and
the absolute Frobenius $F_{S}$ of $S$.
If $S$ is locally of finite type over $k$, then the composition
(\ref{radicialcvmap}) is a finite covering.
If $T$ is a subscheme of $S$, then $T^{1/p}$ is the fiber product of
$T$ and $S^{1/p}$ over $S$.
If $\{S_{h}\}_{h}$ are the irreducible components of $S$,
then $\{S_{h}^{1/p}\}_{h}$ are the irreducible components of $S^{1/p}$.

For $s\in \mathbf{Z}_{\ge 0}$, let 
\begin{equation}
\label{fdsh}
-F^{s-1}d\colon j_{*}W_{s}(\dvr_{U})
\rightarrow j_{*}\Omega^{1}_{U} 
\end{equation} 
be the morphism locally defined by
\begin{equation}
-F^{s-1}da=-\sum_{i=0}^{s-1}a_{i}^{p^{i}-1}da_{i} \notag
\end{equation} 
for $a=(a_{s-1},\ldots,a_{0})\in j_{*}W_{s}(\dvr_{U})$.
Then the morphism $-F^{s-1}d$ (\ref{fdsh})
induces the morphism
\begin{equation}
\label{morfilwsshomsh}
\fillog_{R}^{D'}j_{*}W_{s}(\dvr_{U})\rightarrow \fillog_{R}^{D'}j_{*}\Omega_{U}^{1}, 
\end{equation}
and the morphism (\ref{morfilwsshomsh}) induces the morphism
\begin{equation}
\label{morgrrdpwo}
\grlog_{R}^{D'}j_{*}W_{s}(\dvr_{U})\rightarrow \grlog_{R}^{D'}j_{*}\Omega_{U}^{1}=\Omega_{X}^{1}(\log D')\otimes_{X}\dvr_{Z}.
\end{equation}
In the following, we often denote $\Omega_{X}^{1}(\log D')\otimes_{\dvr_{X}}\dvr_{Z}$
and $\Omega_{X}^{1}(\log D')\otimes_{\dvr_{X}}\dvr_{Z^{1/p}}$
by $\Omega_{X}^{1}(\log D')|_{Z}$ and
$\Omega_{X}^{1}(\log D')|_{Z^{1/p}}$, respectively.

\begin{lem}
\label{lemcfdef}
Let $I'\subset I$ be a subset such that $D'=\bigcup_{i\in I'}D_{i}$
has simple normal crossings and
let $R=\sum_{i\in I}n_{i}D_{i}$, where $n_{i}\in \mathbf{Z}_{\ge 0}$ for $i\in I'$
and $n_{i}\in \mathbf{Z}_{\ge 1}$ for $i\in I-I'$.
We put $Z=\Supp(R+D'-D)$.
For $s\in \mathbf{Z}_{\ge 0}$,
there exists a unique morphism
\begin{equation}
\varphi_{s}^{(D'\subset D,R)}\colon \grlog^{D'}_{R}j_{*}W_{s}(\dvr_{U})
\rightarrow \grlog^{D'}_{R}j_{*}\Omega^{1}_{U}\otimes_{\dvr_{Z}}\dvr_{Z^{1/p}}
=\Omega^{1}_{X}(\log D')(R)|_{Z^{1/p}} \notag
\end{equation}
such that the equations
\begin{equation}
\label{varphidpdr}
\varphi_{s}^{(D'\subset D,R)}(\bar{a})=
\begin{cases} \overline{-F^{s-1}da}\otimes \overline{1} & (p\neq 2), \\ 
\overline{-F^{s-1}da}\otimes \overline{1}+
\sum_{\substack{i\in I-I'\\ n_{i}=2}} \left(\overline{dt_{i}/t_{i}^{2}}\otimes \sqrt{\overline{a_{0}t_{i}^{2}}}\right)& (p=2)
\end{cases}  
\end{equation}
locally hold for the image $\bar{a}$ of every $a=(a_{s-1},a_{s-2},\ldots,a_{0})\in \fillog_{R}^{D'}j_{*}W_{s}(\dvr_{U})$ in $\grlog^{D'}_{R}j_{*}W_{s}(\dvr_{U})$, and for every local equation $t_{i}$ of $D_{i}$ for $i\in I-I'$ such that $n_{i}=2$ in the case where $p=2$.
\end{lem}

\begin{proof}
We consider the composition
\begin{equation}
\label{compcfdef}
\grlog_{R}^{D'}j_{*}W_{s}(\dvr_{U})
\xrightarrow{(\ref{morgrrdpwo})}
\grlog_{R}^{D'}j_{*}\Omega_{U}^{1}\rightarrow 
\grlog_{R}^{D'}j_{*}\Omega_{U}^{1}\otimes_{\dvr_{Z}}\dvr_{Z^{1/p}}
\end{equation}
of the morphism (\ref{morgrrdpwo}) and the canonical injection.
If $p\neq 2$, then the composition (\ref{compcfdef})
satisfies the desired condition.

Suppose that $p=2$.
Let $t_{i}$ be a local equation of $D_{i}$ for $i\in I$ and
let $a=(a_{s-1},a_{s-2},\ldots,a_{0})\in \fillog_{R}^{D'}j_{*}W_{s}(\dvr_{U})$.
Take $i\in I-I'$ such that $n_{i}=2$.
Then we locally have 
\begin{equation}
\overline{dt_{i}/t_{i}^{2}}\otimes \sqrt{\overline{a_{0}t_{i}^{2}}}
=\overline{dt_{i}/\prod_{i'\in I}t_{i'}^{n_{i'}}}\otimes \sqrt{\overline{a_{0}\prod_{i'\in I}t_{i'}^{n_{i'}}\prod_{i'\in I-\{i\}}t_{i'}^{n_{i'}}}}\in \Omega_{X}^{1}(\log D')(R)|_{Z^{1/2}}, \notag
\end{equation}
where $\overline{a_{0}\prod_{i'\in I}t_{i'}^{n_{i'}}}\in \dvr_{Z}$.
If $u\in \dvr_{X}$ is locally a unit, then we locally have 
\begin{align}
\overline{d(ut_{i})/(ut_{i})^{2}}\otimes \sqrt{\overline{a_{0}(ut_{i})^{2}}}
&=\overline{dt_{i}/t_{i}^{2}}\otimes \sqrt{\overline{a_{0}t_{i}^{2}}}+
\overline{t_{i}du/ut_{i}^{2}}\otimes \sqrt{\overline{a_{0}t_{i}^{2}}} \notag \\
&=\overline{dt_{i}/t_{i}^{2}}\otimes \sqrt{\overline{a_{0}t_{i}^{2}}} \notag
\end{align}
in $\Omega_{X}^{1}(\log D')(R)|_{Z^{1/2}}$,
where we applied $t_{i}\sqrt{\prod_{i'\in I-\{i\}}t_{i'}^{n_{i'}}}=0$ in $\dvr_{Z^{1/2}}$ at the last equality.
Therefore we can glue $\overline{dt_{i}/t_{i}^{2}}\otimes \sqrt{\overline{a_{0}t_{i}^{2}}}$,
and there exists a unique mapping
\begin{equation}
\varphi_{i}\colon\grlog^{D'}_{R}j_{*}W_{s}(\dvr_{U})
\rightarrow \grlog^{D'}_{R}j_{*}\Omega^{1}_{U}\otimes_{\dvr_{Z}}\dvr_{Z^{1/2}}  \notag
\end{equation}
locally defined by
\begin{equation}
(\overline{a_{s-1}},\overline{a_{s-2}},\ldots,\overline{a_{0}})\mapsto
\overline{dt_{i}/t_{i}^{2}}\otimes \sqrt{\overline{a_{0}t_{i}^{2}}}. \notag
\end{equation}
Since $p=2$ and since we have
\begin{equation}
(\overline{a_{s-1}},\overline{a_{s-2}},\ldots,\overline{a_{0}})
=(0, 0, \ldots, \overline{a_{0}}) \notag
\end{equation}
in $j_{i*}\gr_{2}W_{s}(K_{i})$ for $(\overline{a_{s-1}},\overline{a_{s-2}},\ldots,\overline{a_{0}})\in \grlog^{D'}_{R}j_{*}W_{s}(\dvr_{U})$, 
the mapping $\varphi_{i}$ is a morphism.
Then the sum of the composition (\ref{compcfdef}) and $\varphi_{i'}$ for all $i'\in I-I'$ such that $n_{i'}=2$
is the desired morphism.
\end{proof}

By the definitions of $\fillog_{R-Z}^{D'}R^{1}(\varepsilon\circ j)_{*}\BZ/p^{s}\BZ$ and $\fillog_{R}^{D'}R^{1}(\varepsilon\circ j)_{*}\BZ/p^{s}\BZ$ as
the images of $\fillog_{R-Z}^{D'}j_{*}W_{s}(\dvr_{U})$ and $\fillog_{R}^{D'}j_{*}W_{s}(\dvr_{U})$ by $\delta_{s,j}$ (\ref{deltassh}), respectively,
the morphism $\delta_{s,j}$ for $s\in \mathbf{Z}_{\ge 0}$ induces a surjection
\begin{equation}
\label{deltasdpdr}
\delta_{s}^{(D'\subset D, R)}\colon \grlog^{D'}_{R}j_{*}W_{s}(\dvr_{U})\rightarrow
\grlog^{D'}_{R}R^{1}(\varepsilon\circ j)_{*}\BZ/p^{s}\BZ. 
\end{equation}

We construct the morphism (\ref{mordefcf}) as follows:

\begin{prop}
\label{propcf}
Let $I'\subset I$ be a subset such that $D'=\bigcup_{i\in I-I'}D_{i}$ has
simple normal crossings and
let $R=\sum_{i\in I}n_{i}D_{i}$ where $n_{i}\in \mathbf{Z}_{\ge 0}$ for $i\in I'$
and $n_{i}\in\mathbf{Z}_{\ge 1}$ for $i\in I-I'$.
We put $Z=\Supp(R+D'-D)$.
For each $s\in \mathbf{Z}_{\ge 0}$, 
there exists a unique morphism
\begin{equation}
\phi_{s}^{(D'\subset D,R)}\colon \grlog_{R}^{D'}R^{1}(\varepsilon\circ j)_{*}\BZ/p^{s}\BZ\rightarrow \grlog_{R}^{D'}j_{*}\Omega^{1}_{U}\otimes_{\dvr_{Z}}\dvr_{Z^{1/p}}
=\Omega^{1}_{X}(\log D')(R)|_{Z^{1/p}} \notag
\end{equation}
such that the following diagram is commutative:
\begin{equation}
\label{comdiagcfshf}
\begin{xy}
(0,0)="A"*{\grlog^{D'}_{R}R^{1}(\varepsilon\circ j)_{*}\mathbf{Z}/p^{s}\mathbf{Z},},
"A"+<3.6cm,1.5cm>="C"*{\grlog^{D'}_{R}j_{*}\Omega_{U}^{1}\otimes_{\dvr_{Z}}\dvr_{Z^{1/p}}},
"C"+/l7.2cm/="B"*{\grlog^{D'}_{R}j_{*}W_{s}(\dvr_{U})}
\ar^-{\varphi_{s}^{(D'\subset D,R)}}"B"+/r1.4cm/;"C"+/l2cm/
\ar_-{\delta_{s}^{(D'\subset D,R)}} "B"+/d.3cm/;"A"+/u.4cm/+/l1.3cm/
\ar_-{\phi_{s}^{(D'\subset D,R)}} "A"+<1.3cm,0.3cm>;"C"+/d0.4cm/
\end{xy} 
\end{equation}
where $\delta_{s}^{(D'\subset D,R)}$ is as in (\ref{deltasdpdr}) and
$\varphi_{s}^{(D'\subset D,R)}$ is as in Lemma \ref{lemcfdef}.
\end{prop}

\begin{proof}
We put 
\begin{align}
\mg&=\bigoplus_{\substack{i\in I\\ n_{i}>0}}j_{i*}\grlog_{n_{i}}H^{1}(K_{i},\mathbf{Q}/\mathbf{Z})
\oplus \bigoplus_{\substack{i\in I-I' \\ n_{i}>1}} j_{i*}\gr_{n_{i}}H^{1}(K_{i},\mathbf{Q}/\mathbf{Z}), \notag \\
\mh&=\bigoplus_{\substack{i\in I\\ n_{i}>0}}j_{i*}\grlog_{n_{i}}W_{s}(K_{i})
\oplus \bigoplus_{\substack{i\in I-I' \\ n_{i}>1}} j_{i*}\gr_{n_{i}}W_{s}(K_{i}), \notag \\
\mathcal{I}&=\bigoplus_{\substack{i\in I\\ n_{i}>0}}j_{i*}(\grlog_{n_{i}}\Omega^{1}_{K_{i}}\otimes_{F_{K_{i}}}F_{K_{i}}^{1/p})
\oplus \bigoplus_{\substack{i\in I-I' \\ n_{i}>1}} j_{i*}(\gr_{n_{i}}\Omega^{1}_{K_{i}}\otimes_{F_{K_{i}}}F_{K_{i}}^{1/p}), \notag
\end{align}
where $j_{i}\colon \Spec K_{i}\rightarrow X$ are the canonical morphisms.
Let $\oplus \delta_{s}^{(n_{i})}\colon \mh\rightarrow \mg$ denote the direct sum of the morphisms 
defined by $\delta_{s}^{(n_{i})}\colon \grlog_{n_{i}}W_{s}(K_{i})\rightarrow \grlog_{n_{i}}H^{1}(K_{i},\mathbf{Q}/\mathbf{Z})$ (\ref{deltasloclog})
for $i\in I'$ such that $n_{i}>0$ 
and by $\delta_{s}'^{(n_{i})}\colon \gr_{n_{i}}W_{s}(K_{i})\rightarrow \gr_{n_{i}}H^{1}(K_{i},\mathbf{Q}/\mathbf{Z})$ (\ref{deltaslocnlog})
for $i\in I-I'$ such that $n_{i}>1$.
Let $\bigoplus \varphi_{s}^{(n_{i})}\colon \mh\rightarrow \mathcal{I}$ be 
the direct sum of the morphisms defined by the compositions
\begin{equation}
\grlog_{n_{i}}W_{s}(K_{i})\xrightarrow{\varphi_{s}^{(n_{i})}} \grlog_{n_{i}}\Omega^{1}_{K}
\rightarrow \grlog_{n_{i}}\Omega^{1}_{K_{i}}\otimes_{F_{K_{i}}}F_{K_{i}}^{1/p} \notag
\end{equation}
of $\varphi_{s}^{(n_{i})}$ (\ref{varphilogloc}) and the canonical injections
for $i\in I'$ such that $n_{i}>0$ and 
those defined by the morphisms
$\varphi_{s}'^{(n_{i})}\colon \gr_{n_{i}}W_{s}(K_{i})\rightarrow \gr_{n_{i}}\Omega^{1}_{K_{i}}\otimes_{F_{K_{i}}}F_{K_{i}}^{1/p}$ (\ref{varphinlogloc}) for $i\in I-I'$ such that $n_{i}>1$.
We consider the commutative diagram
\begin{equation}
\label{cdchargl}
\xymatrix{
\grlog_{R}^{D'}R^{1}(\varepsilon\circ j)_{*}\mathbf{Z}/p^{s}\mathbf{Z} \ar[d] & &
\grlog_{R}^{D'}j_{*}W_{s}(\dvr_{U}) \ar[ll]_-{\delta_{s}^{(D'\subset D,R)}}
\ar[rr]^-{\varphi_{s}^{(D'\subset D,R)}} \ar[d] & &
\grlog_{R}^{D'}j_{*}\Omega_{U}^{1}\otimes_{\dvr_{Z}}\dvr_{Z^{1/p}} \ar[d]\\
\mg &  & \mh \ar[ll]^-{\oplus\delta_{s}^{(n_{i})}} \ar[rr]_-{\oplus \varphi_{s}^{(n_{i})}} & & \mathcal{I},
} 
\end{equation}
where the vertical arrows are canonical morphisms.
Since $\grlog_{n_{i}}\Omega_{K_{i}}^{1}\otimes_{F_{K_{i}}}F_{K_{i}}^{1/p}$
for $i\in I'$ such that $n_{i}>0$
and $\gr_{n_{i}}\Omega_{K_{i}}^{1}\otimes_{F_{K_{i}}}F_{K_{i}}^{1/p}$
for $i\in I-I'$ such that $n_{i}>1$
are the stalks of the locally free $\dvr_{Z^{1/p}}$-module $\Omega_{X}^{1}(\log D')(R)|_{Z^{1/p}}$ 
at the generic points of $D_{i}^{1/p}$, respectively,
the right vertical arrow in (\ref{cdchargl}) is injective.
By Proposition \ref{swgldef}, the kernel of $\oplus\delta_{s}^{(n_{i})}$
is equal to that of $\bigoplus \varphi_{s}^{(n_{i})}$.
By the commutativity of (\ref{cdchargl}) and the injectivity of the right vertical arrow, 
the kernel of $\delta_{s}^{(D'\subset D,R)}$ is contained in the kernel of $\varphi_{s}^{(D'\subset D,R)}$.
Hence the morphism
$\varphi_{s}^{(D'\subset D,R)}$ factors as the canonical surjection
\begin{equation}
\grlog_{R}^{D'}j_{*}W_{s}(\dvr_{U})
\rightarrow \grlog_{R}^{D'}j_{*}W_{s}(\dvr_{U})/\Ker \delta_{s}^{(D'\subset D,R)} \notag
\end{equation}
followed by 
the composition 
\begin{equation}
\label{compdvp}
\grlog_{R}^{D'}j_{*}W_{s}(\dvr_{U})/\Ker \delta_{s}^{(D'\subset D,R)}
\rightarrow 
\grlog_{R}^{D'}j_{*}W_{s}(\dvr_{U})/\Ker \varphi_{s}^{(D'\subset D,R)}
\rightarrow
\grlog_{R}^{D'}j_{*}\Omega^{1}_{U}\otimes_{\dvr_{Z}}\dvr_{Z^{1/p}}
\end{equation}
of the canonical surjection and the injection induced by $\varphi_{s}^{(D'\subset D,R)}$.
Since $\delta_{s}^{(D'\subset D,R)}$ is surjective,
the morphism $\delta_{s}^{(D'\subset D,R)}$ induces an isomorphism 
\begin{equation}
\label{isomdsdr}
\grlog_{R}^{D'}j_{*}W_{s}(\dvr_{U})/\Ker \delta_{s}^{(D'\subset D,R)}
\xrightarrow{\sim} \grlog_{R}^{D'}R^{1}(\varepsilon\circ j)_{*}\mathbf{Z}/p^{s}\mathbf{Z}. 
\end{equation}
Then the composition of the inverse of (\ref{isomdsdr}) and the composition (\ref{compdvp})
is the desired morphism.
\end{proof}

We construct the conductor $\cform^{D'}(\mf)$.
In the rest of this article, we denote by $\chi\colon \pi^{\ab}(U)\rightarrow \Lambda^{\times}$ the character corresponding to $\mf$,
as is explained in Conventions.
We fix an inclusion $\psi\colon \Lambda^{\times}\rightarrow \BQ/\BZ$
and regard $\chi|_{K_{i}}\colon G_{K_{i}}\rightarrow \Lambda^{\times}$
(the second paragraph in Conventions)
as an element of $H^{1}(K_{i},\mathbf{Q}/\mathbf{Z})$ by $\psi$ for $i\in I$.
Then we can consider the conductors defined in the previous subsection for $\chi|_{K_{i}}$.

\begin{defn}
\label{defindsub}
\begin{enumerate}
\item We define a subset $I_{\mT,\mf}$ of $I$ by
\begin{align}
I_{\mT,\mf}&=\{i\in I\; |\; \sw(\chi|_{K_{i}})=0\}  \notag 
\end{align}
and define a subset $I_{\mW,\mf}$ of $I$ by
\begin{align} 
I_{\mW,\mf}&=I-I_{\mT,\mf}=\{i\in I\; |\; \sw(\chi|_{K_{i}})\ge 1\}. \notag
\end{align}
\item We define a subset $I_{\mI,\mf}$ of $I_{\mW,\mf}$ by
\begin{align}
I_{\mI,\mf}&=\{i\in I_{\mW,\mf}\; |\; \chi|_{K_{i}} \text{ is of type }
\mI\},  \notag 
\end{align}
and define a subset $I_{\mII,\mf}$ of $I$ by 
\begin{align}
I_{\mII,\mf}&=I_{\mW,\mf}-I_{\mI,\mf}=\{i\in I_{\mW,\mf}\; |\; \chi|_{K_{i}} \text{ is of type }
\mII\} \notag
\end{align}
(see Definition \ref{deftypes} for the notion of the types of $\chi|_{K_{i}}$).
\end{enumerate}
\end{defn}

If there is no risk of confusion, then
we use $I_{*,\mf}$ for $*=\mT,\mW, \mI,\mII$ 
independently on the symbol denoting the
index set of irreducible components of a divisor.

\begin{rem}
\label{remdefitw}
Let the notation be as in Definition \ref{defindsub}.
\begin{enumerate}
\item We have $I=I_{\mT,\mf} \sqcup I_{\mW,\mf}$ and
$I_{\mW,\mf}=I_{\mI,\mf}\sqcup I_{\mII,\mf}$.
\item By Remark \ref{remtameloc} (1), 
the subset $I_{*,\mf}$ of $I$ is dependent only on the $p$-part of $\chi$ for every $*=\mT,\mW,\mI,\mII$.
\end{enumerate}
\end{rem}

Let $I'\subset I$ be a subset and let $D'=\bigcup_{i\in I'}D_{i}$.
Then we put
\begin{equation}
\label{defswdp}
\sw^{D'}(\chi|_{K_{i}})=\begin{cases}
\sw(\chi|_{K_{i}}) & (i\in I'), \\
\dt(\chi|_{K_{i}}) & (i\in I-I')
\end{cases}
\end{equation}
for $i\in I$.

\begin{defn}
\label{defconddiv}
\begin{enumerate}
\item Let $I'$ be a subset of $I$ and let $D'=\bigcup_{i\in I'}D_{i}$.
We define a divisor $R_{\mf}^{D'}$ on $X$ by
\begin{align}
R_{\mf}^{D'}&=\sum_{i \in I}\sw^{D'}(\chi|_{K_{i}})D_{i}, 
\notag
\end{align}
where $\sw^{D'}(\chi|_{K_{i}})$ is as in (\ref{defswdp}).
\item We define a closed subscheme $Z_{\mf}$ of $X$ 
to be the union
of $D_{i}$ for $i\in I_{\mW,\mf}$ (Definition \ref{defindsub} (1))
with the reduced subscheme structure:
\begin{equation}
Z_{\mf}=\bigcup_{i\in I_{\mW,\mf}}D_{i}. \notag
\end{equation}
\end{enumerate}
\end{defn}

\begin{rem}
\label{remconddiv}
Let the notation be as in Definition \ref{defconddiv}.
The divisor $R_{\mf}^{D'}$ and the closed subscheme $Z_{\mf}$ of $X$ are dependent only on the $p$-part of $\chi$,
since $\sw^{D'}(\chi|_{K_{i}})$ for $i\in I$ 
and the subset $I_{\mW,\mf}\subset I$ are dependent only on the $p$-part of $\chi$ by Remarks \ref{remtameloc}
(1) and \ref{remdefitw} (2), respectively.
\end{rem}

\begin{lem}
\label{lemzf}
Let $I'$ be a subset of $I$ and let $D'=\bigcup_{i\in I'}D_{i}$.
Then the closed subscheme $Z_{\mf}$ of $X$ 
(Definition \ref{defconddiv} (2)) is equal to
the support of
$R_{\mf}^{D'}+D'-D$ (Definition \ref{defconddiv} (1)).
Consequently, the support of $R_{\mf}^{D'}+D'-D$ is independent of the choice of a subset $I'$ of $I$.
\end{lem}

\begin{proof}
Since $i\in I_{\mW,\mf}$ (Definition \ref{defindsub} (1)) if and only if $\sw(\chi|_{K_{i}})>0$ for $i\in I$,
the closed subscheme $Z_{\mf}$ of $X$ is equal to 
the support of 
\begin{equation}
R_{\mf}^{D}=\sum_{i\in I}\sw(\chi|_{K_{i}})D_{i}. \notag
\end{equation}
Since $\sw(\chi|_{K_{i}})=0$ if and only if $\dt(\chi|_{K_{i}})=1$ for $i\in I$, especially for $i\in I-I'$, 
by Remark \ref{remtameloc} (2), the support of 
\begin{equation}
R_{\mf}^{D'}+D'-D=\sum_{i\in I'}\sw(\chi|_{K_{i}})D_{i}+\sum_{i\in I-I'}(\dt(\chi|_{K_{i}})-1)D_{i} \notag
\end{equation}
is equal to that of $R^{D}_{\mf}$.
\end{proof}

The conductor $\cform^{D'}(\mf)$ is defined to be a global section of
$\Omega_{X}^{1}(\log D')(R_{\mf}^{D'})|_{Z_{\mf}^{1/p}}$
as follows:

\begin{defn}[{cf.\ \cite[5.2]{ma}}]
\label{defcform}
Let $I'\subset I$ be a subset such that $D'=\sum_{i\in I'}D_{i}$
has simple normal crossings.
Suppose that the closed subscheme $Z_{\mf}=\Supp (R_{\mf}^{D'}+D'-D)$ 
of $X$ (Definition \ref{defconddiv}, Lemma \ref{lemzf}) is non-empty. 
Let $s\ge 0$ be the integer such that the order of the $p$-part of 
$\chi$ is $p^{s}$.
We define the $\log$-$D'$-{\it characteristic form} $\cform^{D'}(\mf)$ of $\mf$ to be the image of the $p$-part of $\chi$ 
by the composition
\begin{align}
\label{compdefcf}
&\Gamma(X,\fillog_{R_{\mf}^{D'}}^{D'}R^{1}(\varepsilon \circ j)_{*}\mathbf{Z}/p^{s}\mathbf{Z})
\rightarrow \Gamma(X,\grlog_{R_{\mf}^{D'}}^{D'}
R^{1}(\varepsilon\circ j)_{*}\BZ/p^{s}\BZ)  \\
&\xrightarrow{\phi_{s}^{(D'\subset D,R_{\mf}^{D'})}(X)}
\Gamma(X,\grlog^{D'}_{R_{\mf}^{D'}}j_{*}\Omega_{U}^{1}\otimes_{\dvr_{Z_{\mf}}}\dvr_{Z_{\mf}^{1/p}}) 
=\Gamma(Z_{\mf}^{1/p},\Omega_{X}^{1}(\log D')(R_{\mf}^{D'})|_{Z_{\mf}^{1/p}}) \notag
\end{align}
of the canonical morphism
and the morphism $\phi_{s}^{(D'\subset D,R)}(X)$ constructed in Proposition \ref{propcf}.
\end{defn}

\begin{rem}
\label{remcfnormal}
Let the notation and the assumption be as in Definition \ref{defcform}.
\begin{enumerate}
\item The $\log$-$D'$-characteristic form $\cform^{D'}(\mf)$ is independent of the choice of $s\in \mathbf{Z}_{\ge 0}$ such that the order of the
$p$-part of $\chi$ is $\le p^{s}$.
Actually, the following diagram is commutative for any $t\ge t'\ge 0$:
\begin{equation}
\begin{xy}
(0,0)="A"*{\Gamma(X,\grlog^{D'}_{R_{\mf}^{D'}}j_{*}\Omega_{U}^{1}\otimes_{\dvr_{Z_{\mf}}}\dvr_{Z_{\mf}^{1/p}}),},
"A"+<3.6cm,1.5cm>="C"*{\Gamma(X, \grlog^{D'}_{R_{\mf}^{D'}}R^{1}(\varepsilon \circ j)_{*}\mathbf{Z}/p^{t}\mathbf{Z})},
"C"+/l7.2cm/="B"*{\Gamma(X, \grlog^{D'}_{R_{\mf}^{D'}}R^{1}(\varepsilon \circ j)_{*}\mathbf{Z}/p^{t'}\mathbf{Z})},
"C"+/u1.5cm/="D"*{\Gamma(X, \fillog^{D'}_{R_{\mf}^{D'}}R^{1}(\varepsilon \circ j)_{*}\mathbf{Z}/p^{t}\mathbf{Z})},
"B"+/u1.5cm/="E"*{\Gamma(X, \fillog^{D'}_{R_{\mf}^{D'}}R^{1}(\varepsilon \circ j)_{*}\mathbf{Z}/p^{t'}\mathbf{Z})}.
\ar "B"+/r2.6cm/;"C"+/l2.6cm/
\ar_(.1){\phi_{t'}^{(D'\subset D,R_{\mf}^{D'})}(X)} "B"+/d.4cm/;"A"+/u.4cm/+/l1.5cm/
\ar^(.1){\phi_{t}^{(D'\subset D,R_{\mf}^{D'})}(X)} "C"+/d.4cm/;"A"+<1.5cm,0.4cm>
\ar "D"+/d0.4cm/;"C"+/u0.4cm/
\ar "E"+/d0.4cm/;"B"+/u0.4cm/
\ar "E"+/r2.6cm/;"D"+/l2.6cm/
\end{xy} \notag
\end{equation}
where the morphisms in the square are canonical morphisms.

\item The $\log$-$D'$-characteristic form $\cform^{D'}(\mf)$ is dependent only on the $p$-part of $\chi$ 
by Remark \ref{remconddiv} and the definition of $\cform^{D'}(\mf)$
as the image of the $p$-part of $\chi$ by the morphism (\ref{compdefcf}).
Let $a$ be a section of $\fillog_{R_{\mf}^{D'}}^{D'}j_{*}W_{s}(\dvr_{U})$ 
whose image by $\delta_{s,j}$ (\ref{deltassh}) is locally the $p$-part of $\chi$.
By the commutativity of the diagram (\ref{comdiagcfshf}) in Proposition \ref{propcf},
the $\log$-$D'$-characteristic form $\cform^{D'}(\mf)$
is locally equal to the image of the section $a$ of $\fillog_{R_{\mf}^{D'}}^{D'}j_{*}W_{s}(\dvr_{U})$ by the composition
\begin{align}
\label{compcffmwitt}
&\fillog_{R_{\mf}^{D'}}^{D'}j_{*}W_{s}(\dvr_{U})\rightarrow 
\grlog_{R_{\mf}^{D'}}^{D'}j_{*}W_{s}(\dvr_{U})  \\
&\qquad \qquad \xrightarrow{\varphi_{s}^{(D'\subset D,R_{\mf}^{D'})}}
\grlog_{R_{\mf}^{D'}}^{D'}j_{*}\Omega_{U}^{1}\otimes_{\dvr_{Z}}\dvr_{Z^{1/p}}=\Omega_{X}^{1}(\log D')(R_{\mf}^{D'})|_{Z_{\mf}^{1/p}}, \notag
\end{align}
where the first arrow is the canonical surjection
and $\varphi_{s}^{(D'\subset D,R_{\mf}^{D'})}$ is the morphism
constructed in Lemma \ref{lemcfdef}. 

\item We can characterize the $\log$-$D'$-characteristic form $\cform^{D'}(\mf)$ as follows:
We identify $\rsw(\chi|_{K_{i}})$ with its image in $\grlog_{\sw(\chi|_{K_{i}})}\Omega^{1}_{K_{i}}\otimes_{F_{K_{i}}}F_{K_{i}}^{1/p}$
by the canonical injection
\begin{equation}
\grlog_{\sw(\chi|_{K_{i}})}\Omega^{1}_{K_{i}}\rightarrow \grlog_{\sw(\chi|_{K_{i}})}\Omega^{1}_{K_{i}}\otimes_{F_{K_{i}}}F_{K_{i}}^{1/p} \notag
\end{equation}
for $i\in I_{\mW,\mf}$ (Definition \ref{defindsub} (1)). 
Let $\mathfrak{p}_{i}'$ denote the generic point of $D_{i}^{1/p}$ for $i\in I_{\mW,\mf}$. 
Then the $\log$-$D'$-characteristic form $\cform^{D'}(\mf)$ is
a unique global section of $\Omega_{X}^{1}(\log D')(R_{\mf}^{D'})|_{Z_{\mf}^{1/p}}$ whose stalk at the generic point $\mathfrak{p}_{i}'$ of $D_{i}^{1/p}$ is
$\rsw(\chi|_{K_{i}})$ for $i\in I'\cap I_{\mW,\mf}$ and
$\cform(\chi|_{K_{i}})$ for $i\in (I-I')\cap I_{\mW,\mf}$
by (2) and the injectivity of the right vertical arrow in the 
commutative diagram (\ref{cdchargl}) in the proof of Proposition \ref{propcf}.

\item Suppose that $D$ has simple normal crossings. 
Then the $\log$-$D$-characteristic form $\cform^{D}(\mf)$ is equal to the image 
of the refined Swan conductor $\rsw(\chi)$ (\cite[(3.4.2)]{kalog})
by the canonical injection $\Omega_{X}^{1}(\log D)(R_{\mf}^{D})|_{Z_{\mf}}\rightarrow \Omega_{X}^{1}(\log D)(R_{\mf}^{D})|_{Z_{\mf}^{1/p}}$
by (3) and \cite[Remark 3.3.12]{ma}.
The $\log$-$\emptyset$-characteristic form $\cform^{\emptyset}(\mf)$ is
equal to the characteristic form $\cform(\chi)$ (\cite[Definition 1.42]{yafil}), since the construction of $\cform(\chi)$ is same as that of $\cform^{\emptyset}(\mf)$. 
If $p\neq 2$ and if we put $D_{\mI,\mf}=\bigcup_{i\in I_{\mI,\mf}}D_{i}$, then the $\log$-$D_{\mI,\mf}$-characteristic form $\cform^{D_{\mI,\mf}}(\mf)$ is equal to the mixed refined Swan conductor $\mathrm{mrsw}(\chi)$ (\cite[5.2]{ma}) by (3).

\item If $p\neq 2$ or if $\dt(\chi|_{K_{i}})\neq 2$ for every $i\in I-I'$,
then the image
of the morphism $\varphi_{s}^{(D'\subset D,R)}$
is contained in $\grlog_{R}^{D'}j_{*}\Omega^{1}_{U}$ by Lemma \ref{lemcfdef}.
Since the morphism $\delta_{s}^{(D'\subset D,R_{\mf}^{D'})}$ (\ref{deltasdpdr}) is surjective,
we may replace $\dvr_{Z_{\mf}^{1/p}}$ by $\dvr_{Z_{\mf}}$ in Proposition \ref{propcf} and
we can regard the $\log$-$D'$-characteristic form $\cform^{D'}(\mf)$
as a global section of $\Omega^{1}_{X}(\log D')(R_{\mf}^{D'})|_{Z_{\mf}}$, if $p\neq 2$ or if $\dt(\chi|_{K_{i}})\neq 2$ for every $i\in I-I'$.
Especially the $\log$-$D$-characteristic form $\cform^{D}(\mf)$ can be regarded as a global section of
$\Omega_{X}^{1}(\log D)(R_{\mf}^{D})|_{Z_{\mf}}$, 
if $D$ has simple normal crossings.
\end{enumerate}
\end{rem}

\subsection{Comparison of characteristic forms}
\label{sscompcf}
We compare two characteristic forms in several settings.
Let $x$ be a point on $X$.
We put 
\begin{equation}
\label{defix}
I_{x}=\{i\in I\; |\; x\in D_{i}\}, 
\end{equation}
and 
\begin{equation}
\label{eachindatx}
I_{*,\mf,x}=\{i\in I_{*,\mf}\; |\; x\in D_{i}\}=I_{*,\mf}\cap I_{x} 
\end{equation} 
for $*=\mT,\mW, \mI,\mII$ (Definition \ref{defindsub}).
If there is no risk of confusion, then
we use $I_{x}$ and $I_{*,\mf,x}$ for $*=\mT,\mW, \mI,\mII$ 
independently on the symbol denoting 
the index set of irreducible components of a divisor.

\begin{lem}[{cf.\ \cite[Lemma 2.24]{yacc}}]
\label{lemcfatx}
Suppose that $D$ has simple normal crossings. 
Let $I''\subset I'\subset I$ be subsets and
let $D'=\bigcup_{i\in I'}D_{i}$ and $D''=\bigcup_{i\in I''}D_{i}$.
Let $x$ be a closed point of $Z_{\mf}$ (Definition \ref{defconddiv} (2))
and let $d=\dim_{x}X$ be the dimension of $X$ at $x$.
We put $I_{x}=\{1,2,\ldots,r\}$ (\ref{defix}), $I'_{x}=I'\cap I_{x}=\{1,2,\ldots,r'\}$ and $I''\cap I_{x}=\{1,2,\ldots,r''\}$, where $r''\le r'\le r\le d$.
Let $(t_{1},t_{2},\ldots,t_{d})$ be a 
local coordinate system at $x$
such that $t_{i}$ is a local equation of $D_{i}$ for $i\in I_{x}$.
For the unique closed point $x'\in Z_{\mf}^{1/p}$ lying above $x$, we put
\begin{align}
\label{cformatxdp}
\cform^{D'}(\mf)_{x'}&=\left(\sum_{i=1}^{r'}\alpha_{i}\dlog t_{i}+\sum_{i=r'+1}^{d}\beta_{i}dt_{i}\right)/
\prod_{i=1}^{r}t_{i}^{n_{i}},  \\
\cform^{D''}(\mf)_{x'}&=\left(\sum_{i=1}^{r''}\alpha_{i}'d\log t_{i}+\sum_{i=r''+1}^{d}\beta_{i}'dt_{i}\right)/
\prod_{i=1}^{r}t_{i}^{m_{i}},  \notag
\end{align}
where $R_{\mf}^{D'}=\sum_{i=1}^{r}n_{i}D_{i}$ (Definition \ref{defconddiv} (1)), $R_{\mf}^{D''}=\sum_{i=1}^{r}m_{i}D_{i}$, and $\alpha_{i},\beta_{i}, \alpha_{i}', \beta_{i}'$ are elements
of the local ring $\dvr_{Z_{\mf}^{1/p},x'}$ of $Z_{\mf}^{1/p}$ at $x'$.
\begin{enumerate}
\item We put $J_{x}=I'_{x}-(I_{\mII,\mf,x}\cup I_{x}'')$ (\ref{eachindatx}).
Then we have the following for $i\in I_{x}$:
\begin{equation}
\begin{cases}
\alpha_{i}'=\alpha_{i}\prod_{i'\in J_{x}}t_{i'}
& (i\in I''), \\
\beta_{i}'=\alpha_{i}\prod_{i'\in J_{x}-\{i\}}t_{i'} &
(i\in I'\cap I_{\mI,\mf}-I''), \\
\beta_{i}'\in \prod_{i'\in J_{x}}t_{i'}
\cdot\dvr_{Z_{\mf}^{1/p},x'}
& (i\in I'\cap (I_{\mII,\mf}\cup I_{\mT,\mf})-I''), \\ 
\beta_{i}'=\beta_{i}\prod_{i'\in J_{x}}t_{i'}
& (i\in I-I'\text{ or }i=r+1,r+2,\ldots,d). \notag \\
\end{cases} \notag
\end{equation}
\item Let $i\in I'_{x}$.
Then we have $\alpha_{i}\in t_{i}\cdot \dvr_{Z_{\mf}^{1/p},x'}$
if and only if $i\in I'\cap (I_{\mII,\mf}\cup I_{\mT,\mf})$.
\end{enumerate}
\end{lem}

\begin{proof}
(1) The assertion follows similarly as the proof of \cite[Lemma 2.24]{yacc}
with the following relations between $n_{i}$ and $m_{i}$ for $i\in I_{x}$:
\begin{equation}
m_{i}=\begin{cases} n_{i}+1 & (i\in (I'-I'')\cap (I_{\mI,\mf}\cup I_{\mT,\mf})), \\ 
n_{i} & (i\in I_{\mII,\mf}\text{ or } i\in (I-I')\cup I'').
\end{cases} \notag
\end{equation}

(2) By (1) applied to the case where $(I',I'')$ is $(I,I')$,
the assertion in the case where $I=I'$ deduces that in the general case.
Hence we may assume that $I'=I$.
If $i\in I'\cap (I_{\mII,\mf}\cup I_{\mT,\mf})=I_{\mII,\mf}\cup I_{\mT,\mf}$,
then we have
$\alpha_{i}\in t_{i}\cdot \dvr_{Z_{\mf}^{1/p},x'}$ 
by \cite[Lemma 2.24]{yacc}.
The converse also holds by \cite[Lemma 2.23 (ii)]{yacc}.
\end{proof}

\begin{lem}
\label{lemclrelfirf}
Let $I''\subset I'\subset I$ be two subsets and
let $D'=\bigcup_{i\in I'}D_{i}$ and $D''=\bigcup_{i\in I''}D_{i}$.
Suppose that $D'$ has simple normal crossings and that $I'-I''$ is contained in $I_{\mII,\mf}$ (Definition \ref{defindsub} (2)).
Then we have $R_{\mf}^{D'}=R_{\mf}^{D''}$ (Definition \ref{defconddiv} (1))
and the $\log$-$D'$-characteristic form $\cform^{D'}(\mf)$
is the image of the $\log$-$D''$-characteristic form $\cform^{D''}(\mf)$ by the canonical morphism
\begin{equation}
\Gamma(Z_{\mf}^{1/p},\Omega_{X}^{1}(\log D'')(R_{\mf}^{D''})|_{{Z_{\mf}^{1/p}}}) \rightarrow
\Gamma(Z_{\mf}^{1/p},\Omega_{X}^{1}(\log D')(R_{\mf}^{D'})|_{{Z_{\mf}^{1/p}}}), \notag
\end{equation}
where $Z_{\mf}$ is as in Definition \ref{defconddiv} (2).
\end{lem}

\begin{proof}
Since we have $\sw(\chi|_{K_{i}})=\dt(\chi|_{K_{i}})$
for $i\in I'-I''\subset I_{\mII,\mf}$,
we have
\begin{align}
R_{\mf}^{D'}&=\sum_{i\in I'}\sw(\chi|_{K_{i}})D_{i}+\sum_{i\in I-I'}\dt(\chi|_{K_{i}})D_{i} \notag \\
&=\sum_{i\in I''}\sw(\chi|_{K_{i}})D_{i}+\sum_{i\in I-I''}\dt(\chi|_{K_{i}})D_{i}
=R_{\mf}^{D''}. \notag
\end{align}

By Remark \ref{remcfnormal} (2), we may assume that $\chi$ is of order $p^{s}$ for $s\in\mathbf{Z}_{\ge 0}$.
By (\ref{incjwsdppdp}), we have $\fillog^{D''}_{R_{\mf}^{D''}}j_{*}W_{s}(\dvr_{U})\subset \fillog_{R_{\mf}^{D'}}^{D'}j_{*}W_{s}(\dvr_{U})$.
Let $a$ be a section of $\fillog^{D''}_{R_{\mf}^{D''}}j_{*}W_{s}(\dvr_{U})\subset \fillog_{R_{\mf}^{D'}}^{D'}j_{*}W_{s}(\dvr_{U})$ whose image by $\delta_{s,j}$ (\ref{deltassh}) is locally $\chi$
and let $\bar{a}$ denote both the images of $a$ in $\grlog_{R_{\mf}^{D''}}^{D''}j_{*}W_{s}(\dvr_{U})$ and $\grlog_{R_{\mf}^{D'}}^{D'}j_{*}W_{s}(\dvr_{U})$.
Since the assertion is local by Remark \ref{remcfnormal} (3)
and since the characteristic forms $\cform^{D''}(\mf)$ and $\cform^{D'}(\mf)$ are locally the images of $\bar{a}$ by
the morphisms $\varphi_{s}^{(D''\subset D,R_{\mf}^{D''})}$ 
and $\varphi_{s}^{(D'\subset D,R_{\mf}^{D'})}$ constructed in Lemma \ref{lemcfdef},
respectively, by Remark \ref{remcfnormal} (2),
it is sufficient to prove the commutativity of
the diagram
\begin{equation}
\label{diagvarphidpdpp}
\xymatrix{
\grlog^{D''}_{R_{\mf}^{D''}}j_{*}W_{s}(\dvr_{U}) \ar[r]^-{(\ref{morgrjwsdppdp})}
\ar[d]_{\varphi_{s}^{(D''\subset D,R_{\mf}^{D''})}} &
\grlog^{D'}_{R_{\mf}^{D'}}j_{*}W_{s}(\dvr_{U})
\ar[d]^{\varphi_{s}^{(D'\subset D,R_{\mf}^{D'})}} \\
\Omega_{X}^{1}(\log D'')(R_{\mf}^{D''})|_{{Z_{\mf}^{1/p}}} \ar[r] &
\Omega_{X}^{1}(\log D')(R_{\mf}^{D'})|_{{Z_{\mf}^{1/p}}},
}
\end{equation}
where the lower horizontal arrow is the canonical morphism. 
We put $R_{\mf}^{D''}=\sum_{i\in I}n_{i}D_{i}$.
Since $n_{i'}$ for $i'\in I_{\mW,\mf}$ (Definition \ref{defindsub} (1)) are positive,
the image of $dt_{i}/t_{i}^{2}$ by the lower horizontal arrow is
\begin{equation}
dt_{i}/t_{i}^{2}=\left.\prod_{i'\in I_{\mW,\mf}-\{i\}}t_{i'}^{n_{i'}}\cdot t_{i}d\log t_{i}\right/\prod_{i'\in I_{\mW,\mf}}t_{i'}^{n_{i'}}=0  \notag
\end{equation}
for $i\in I'-I''$ such that $n_{i}=2$ 
and for a local equation $t_{i}$ of $D_{i}$.
Hence the diagram (\ref{diagvarphidpdpp}) is commutative
by the constructions of $\varphi_{s}^{(D''\subset D,R_{\mf}^{D''})}$ 
and $\varphi_{s}^{(D'\subset D,R_{\mf}^{D'})}$ in Lemma \ref{lemcfdef}.
\end{proof}

\begin{lem}
\label{lemclrelsecf}
Let $I'\subset I$ be a subset such that
$D'=\bigcup_{i\in I'}D_{i}$ has simple normal crossings.
Let $I''\subset I'\cap I_{\mT,\mf}$ (Definition \ref{defindsub} (1)) be a subset and let
$E=\bigcup_{i\in I-I''}D_{i}$ and $E'=\bigcup_{i\in I'-I''}D_{i}$. 
Let $\mf'$ be a smooth sheaf of $\Lambda$-modules of rank 1 on $V=X-E$
whose associated character $\chi'\colon \pi_{1}^{\ab}(V)\rightarrow \Lambda^{\times}$ has the $p$-part inducing that of $\chi$.
Then we have $R_{\mf}^{D'}=R_{\mf'}^{E'}$ (Definition \ref{defconddiv} (1)) and $Z_{\mf}=Z_{\mf'}$ (Definition \ref{defconddiv} (2)).
The $\log$-$D'$-characteristic form $\cform^{D'}(\mf)$ of $\mf$ is the image of the $\log$-$E'$-characteristic form $\cform^{E'}(\mf')$
of $\mf'$ by the canonical morphism
\begin{equation}
\label{canmorcfetocfd}
\Gamma(Z_{\mf'}^{1/p},\Omega_{X}^{1}(\log E')(R_{\mf'}^{E'})|_{Z_{\mf'}^{1/p}}) \rightarrow
\Gamma(Z_{\mf}^{1/p},\Omega_{X}^{1}(\log D')(R_{\mf}^{D'})|_{Z_{\mf}^{1/p}}).
\end{equation}
\end{lem}

\begin{proof}
We prove the assertions similarly as the proof of Lemma \ref{lemclrelfirf}.
By Remark \ref{remtameloc} (1), we have $\sw(\chi|_{K_{i}})=\sw(\chi'|_{K_{i}})$ and $\dt(\chi|_{K_{i}})=\dt(\chi'|_{K_{i}})$ for $i\in I-I''$.
Since we have $\sw(\chi|_{K_{i}})=0$ for $i\in I''\subset I_{\mT,\mf}$
and since $I-I'=(I-I'')-(I'-I'')$,
we have
\begin{align}
R_{\mf}^{D'}&=\sum_{i\in I'}\sw(\chi|_{K_{i}})D_{i}+\sum_{i\in I-I'}\dt(\chi|_{K_{i}})D_{i} \notag \\
&=\sum_{i\in I'-I''}\sw(\chi'|_{K_{i}})D_{i}+\sum_{i\in (I-I'')-(I'-I'')}\dt(\chi'|_{K_{i}})D_{i}=R_{\mf'}^{E'}.\notag
\end{align} 
By the equalities $D-D'=\sum_{i\in I-I'}D_{i}=E-E'$ of sums of divisors
and by Lemma \ref{lemzf}, 
we have 
\begin{align}
Z_{\mf}&=\Supp (R_{\mf}^{D'}+D'-D) \notag \\
&=\Supp (R_{\mf'}^{E'}+E'-E)=Z_{\mf'}. \notag
\end{align}

By Remark \ref{remcfnormal} (2), we may assume that both $\chi$ and $\chi'$ are of orders powers of $p$.
Let $s\ge 0$ be an integer such that the order of $\chi'$ is $p^{s}$.
Let $j'\colon V\rightarrow X$ denote the canonical open immersion. 
We regard $\fillog_{R_{\mf'}^{E'}}^{E'}j'_{*}W_{s}(\dvr_{V})$ as a subsheaf
of $\fillog_{R_{\mf}^{D'}}^{D'}j_{*}W_{s}(\dvr_{U})$
by the injection 
\begin{equation}
\label{morfilwsshcl}
\fillog_{R_{\mf'}^{E'}}^{E'}j'_{*}W_{s}(\dvr_{V})\rightarrow \fillog_{R_{\mf}^{D'}}^{D'}j_{*}W_{s}(\dvr_{U}) 
\end{equation}
induced by the canonical injection $j'_{*}W_{s}(\dvr_{V})\rightarrow j_{*}W_{s}(\dvr_{U})$.
Let $a$ be a section of $\fillog_{R_{\mf'}^{E'}}^{E'}j'_{*}W_{s}(\dvr_{V})
\subset \fillog_{R_{\mf}^{D'}}^{D'}j_{*}W_{s}(\dvr_{U})$
whose image by the morphism $\delta_{s,j'}\colon j'_{*}W_{s}(\dvr_{U})\rightarrow R^{1}(\varepsilon\circ j')_{*}\mathbf{Z}/p^{s}\mathbf{Z}$ 
(\ref{deltassh}) is locally $\chi'$.
Then $\chi$ is locally the image of $a$ by $\delta_{s,j}$.
Let $\bar{a}$ denote the images of $a$ in $\grlog_{R_{\mf'}^{E'}}^{E'}j'_{*}W_{s}(\dvr_{V})$ and $\grlog_{R_{\mf}^{D'}}^{D'}j_{*}W_{s}(\dvr_{U})$.
Since $R_{\mf'}^{E'}=R_{\mf}^{D'}$ and $Z_{\mf'}^{1/p}=Z_{\mf}^{1/p}$,
the morphism (\ref{morfilwsshcl}) induces the morphism
\begin{equation}
\label{eqgrepdpws}
\grlog^{E'}_{R_{\mf'}^{E'}} j'_{*}W_{s}(\dvr_{V}) \rightarrow
\grlog^{D'}_{R_{\mf}^{D'}} j_{*}W_{s}(\dvr_{U}) 
\end{equation}
and we can consider the diagram
\begin{equation}
\label{diagcfepdp}
\xymatrix{
\grlog^{E'}_{R_{\mf'}^{E'}} j'_{*}W_{s}(\dvr_{V}) \ar[r]^-{(\ref{eqgrepdpws})}
\ar[d]_{\varphi_{s}^{(E'\subset E,R_{\mf'}^{E'})}} &
\grlog^{D'}_{R_{\mf}^{D'}} j_{*}W_{s}(\dvr_{U})
\ar[d]^{\varphi_{s}^{(D'\subset D,R_{\mf}^{D'})}} \\
\Omega_{X}^{1}(\log E')(R_{\mf'}^{E'})|_{Z_{\mf'}^{1/p}}\ar[r] &
\Omega_{X}^{1}(\log D')(R_{\mf}^{D'})|_{Z_{\mf}^{1/p}},
}
\end{equation}
where the lower horizontal arrow is the canonical morphism
defining the morphism (\ref{canmorcfetocfd}).
Then the diagram (\ref{diagcfepdp}) is commutative
by the constructions of the morphisms $\varphi_{s}^{(E'\subset E,R_{\mf'}^{E'})}$
and $\varphi_{s}^{(D'\subset D,R_{\mf}^{D'})}$ in 
Lemma \ref{lemcfdef},
and the assertion follows,
since the assertion is local by Remark \ref{remcfnormal} (3)
and since the characteristic forms 
$\cform^{E'}(\mf')$ and $\cform^{D'}(\mf)$ are locally the images of $\bar{a}$ by
the morphisms $\varphi_{s}^{(E'\subset E,R_{\mf'}^{E'})}$ 
and $\varphi_{s}^{(D'\subset D,R_{\mf}^{D'})}$, respectively, by Remark \ref{remcfnormal} (2).
\end{proof}

\begin{lem}
\label{lemcformgl}
Let $I'\subset I$ be a subset such that $D'=\bigcup_{i\in I'}D_{i}$ has simple normal crossings 
and that $I_{\mT,\mf}$ (Definition \ref{defindsub} (1)) is contained in $I'$.
Let $h\colon W\rightarrow X$ be a morphism of smooth schemes over $k$.
Suppose that the pull-backs $h^{*}D_{i}=D_{i}\times_{X}W$ for all $i\in I'$ and $(h^{*}D')_{\red}=(D'\times_{X}W)_{\red}$ are divisors on $W$ with simple normal crossings
and that the pull-backs $h^{*}D_{i}$ for $i\in I-I'$ are smooth divisors on $W$. 
Let $dh^{D'}_{(h^{*}Z_{\mf}^{1/p})_{\red}}(h^{*}\cform^{D'}(\mf))$ 
denote the image by the morphism
\begin{align}
\Gamma((h^{*}Z_{\mf}^{1/p}&)_{\red},
h^{*}\Omega_{X}^{1}(\log D')(R_{\mf}^{D'})|_{(h^{*}Z_{\mf}^{1/p})_{\red}}) \notag \\
&\quad \longrightarrow
\Gamma((h^{*}Z_{\mf}^{1/p})_{\red}, \Omega_{W}^{1}(\log (h^{*}D')_{\red})(h^{*}R_{\mf}^{D'})|_{(h^{*}Z_{\mf}^{1/p})_{\red}}) \notag
\end{align}
induced by $h$ 
of the pull-back 
\begin{equation}
h^{*}\cform^{D'}(\mf)\in \Gamma((h^{*}Z_{\mf}^{1/p})_{\red},h^{*}\Omega_{X}^{1}(\log D')(R_{\mf}^{D'})|_{(h^{*}Z_{\mf}^{1/p})_{\red}})  \notag
\end{equation}
of the $\log$-$D'$-characteristic form $\cform^{D'}(\mf)$ 
by the morphism $(h^{*}Z_{\mf}^{1/p})_{\red}\rightarrow Z_{\mf}^{1/p}$
induced by $h$.
\begin{enumerate}
\item Let $\{E_{\theta}\}_{\theta\in \Theta}$ be the irreducible components
of $(h^{*}Z_{\mf})_{\red}$.
Let $\Theta'\subset \Theta$ be the index set 
of the irreducible components of $(h^{*}Z_{\mf})_{\red}$ contained in $h^{*}D'$ and
let $L_{\theta}=\Frac \hat{\dvr}_{W,\mathfrak{q}_{\theta}}$ denote the local field at the generic point $\mathfrak{q}_{\theta}$ of $E_{\theta}$ for $\theta\in \Theta$.
Let $\mathfrak{q}_{\theta}'$ be the unique point on $E_{\theta}^{1/p}$ lying above $\mathfrak{q}_{\theta}$
for $\theta\in\Theta$.
Then the following three conditions are equivalent:
\begin{enumerate}
\item $(dh^{D'}_{(h^{*}Z_{\mf}^{1/p})_{\red}}(h^{*}\cform^{D'}(\mf)))|_{E_{\theta}^{1/p}}\neq 0$ for every $\theta\in \Theta$.
\item $R_{h^{*}\mf}^{(h^{*}D')_{\red}}=h^{*}R_{\mf}^{D'}$. 
\item $(dh^{D'}_{(h^{*}Z_{\mf}^{1/p})_{\red}}(h^{*}\cform^{D'}(\mf)))_{\mathfrak{q}'_{\theta}}=\begin{cases}
\rsw(h^{*}\chi|_{L_{\theta}}) & (\theta \in \Theta'), \\
\cform(h^{*}\chi|_{L_{\theta}}) & (\theta\in \Theta-\Theta').
\end{cases}$
\end{enumerate}
\item 
If the equivalent conditions (a), (b), and (c) in (1) hold, then 
we have two equalities
\begin{align}
Z_{h^{*}\mf}=(h^{*}Z_{\mf})_{\red}, \ \ 
\cform^{(h^{*}D')_{\red}}(h^{*}\mf)=dh^{D'}_{(h^{*}Z_{\mf}^{1/p})_{\red}}(h^{*}\cform^{D'}(\mf)).
\notag 
\end{align} 
\end{enumerate}
\end{lem}

\begin{proof}
By Remark \ref{remconddiv} and Remark \ref{remcfnormal} (2), we may assume that $\chi$ is of order a power of $p$.
Then $\mf$ is unramified along $D_{i}$ for every $i\in I_{\mT,\mf}$ by Remark \ref{remtameloc} (3) and
so is $h^{*}\mf$ along every irreducible component of $h^{*}D$ not contained in $(h^{*}Z_{\mf})_{\red}$.
Therefore we have $Z_{h^{*}\mf}\subset (h^{*}Z_{\mf})_{\red}$.

(1) Let $s\ge 0$ be the integer such that the order of $\chi$ is
$p^{s}$.
Since the assertion is local, we may assume that $\chi$ is 
the image of a global section $a=(a_{s-1},\ldots,a_{1},a_{0})$ of $\fillog_{R_{\mf}^{D'}}^{D'}j_{*}W_{s}(\dvr_{U})$ by $\delta_{s,j}(X)\colon \Gamma(X, j_{*}W_{s}(\dvr_{U})) \rightarrow \Gamma(X,R^{1}(\varepsilon \circ j)_{*}\mathbf{Z}/p^{s}\mathbf{Z})$ (\ref{deltassh}).
Let $j'\colon h^{*}U=U\times_{X}W\rightarrow W$ be
the base change of $j\colon U\rightarrow X$ by $h$
and let $\varepsilon'\colon W_{\et}\rightarrow W_{\mathrm{Zar}}$ be
the canonical mapping from the \'{e}tale site of $W$ to the Zariski site of $W$.
Then the pull-back $h^{*}\chi\in \Gamma(W,R^{1}(\varepsilon' \circ j')_{*}\mathbf{Z}/p^{s}\mathbf{Z})$ is the image of 
$h^{*}a=(h^{*}a_{s-1},\ldots,h^{*}a_{1},h^{*}a_{0})\in \Gamma(W,h^{*}\fillog_{R_{\mf}^{D'}}^{D'}j_{*}W_{s}(\dvr_{U}))$
by the morphism
$\delta_{s,j'}(W)\colon \Gamma(W,j'_{*}W_{s}(\dvr_{h^{*}U}))\rightarrow \Gamma(W,R^{1}(\varepsilon' \circ j')_{*}\mathbf{Z}/p^{s}\mathbf{Z})$.
We regard $h^{*}a$ as a global section of $\fillog_{h^{*}R_{\mf}^{D'}}^{(h^{*}D')_{\red}}j'_{*}W_{s}(\dvr_{h^{*}U})$ by the inclusion (\ref{inchfilws}) in Remark \ref{remsncdfil} (3),
and denote the image of $h^{*}a$ in $\Gamma(W,\grlog_{h^{*}R_{\mf}^{D'}}^{(h^{*}D')_{\red}}j'_{*}W_{s}(\dvr_{h^{*}U}))$ by $\overline{h^{*}a}$.

We prove that $dh^{D'}_{(h^{*}Z_{\mf}^{1/p})_{\red}}(h^{*}\cform^{D'}(\mf))$ is the image of $\overline{h^{*}a}$ by the morphism
$\varphi_{s}^{((h^{*}D')_{\red}\subset (h^{*}D)_{\red},h^{*}R_{\mf}^{D'})}(W)$
constructed in Lemma \ref{lemcfdef}. 
We put $R_{\mf}^{D'}=\sum_{i\in I}n_{i}D_{i}$.
Then we have $R_{\mf}^{D'}=\sum_{i\in I_{\mW,\mf}}n_{i}D_{i}$,
since $n_{i}=0$ for $i\in I_{\mT,\mf}\subset I'$.
Hence we can put $h^{*}R_{\mf}^{D'}=\sum_{\theta\in \Theta}m_{\theta}E_{\theta}$.
Let $\Theta_{i'}=\{\theta\in \Theta\; |\; E_{\theta}\subset h^{*}D_{i'}\}$
for $i'\in I_{\mW,\mf}$.
If $p=2$ and if $i$ is an element of $I-I'$ such that $n_{i}=2$, 
then we locally have 
\begin{align}
&h^{*}(\sqrt{\overline{a_{0}t_{i}^{2}}}dt_{i}/t_{i}^{2}) \notag \\
&=\sqrt{\overline{(h^{*}a_{0})
\prod_{\theta'\in \Theta_{i}}s_{\theta'}^{2}}}\, 
d\prod_{\theta'\in \Theta_{i}}s_{\theta'}/\prod_{\theta'\in \Theta_{i}}s_{\theta'}^{2}
\notag \\
&=
\sum_{\substack{\theta\in \Theta_{i}}} 
\sqrt{\overline{(h^{*}a_{0})
\prod_{\substack{\theta'\in \Theta_{i}}}s_{\theta'}^{2m_{\theta'}-2}
\prod_{\theta'' \in \Theta-\Theta_{i}}s_{\theta''}^{2m_{\theta''}}}}
\prod_{\theta'''\in \Theta_{i}-\{\theta\}}s_{\theta'''}ds_{\theta}/
\prod_{\theta''''\in \Theta}s_{\theta''''}^{m_{\theta''''}}
\notag \\
&=
\sum_{\substack{\theta\in \Theta_{i} \\ m_{\theta}=2}} 
\sqrt{\overline{(h^{*}a_{0})s_{\theta}^{2}}}ds_{\theta}/s_{\theta}^{2}
\notag
\end{align}
in $\Omega_{W}^{1}(\log (h^{*}D')_{\red})(h^{*}R_{\mf}^{D'})|_{(h^{*}Z_{\mf}^{1/p})_{\red}}$,
where $t_{i'}$ is a local equation of $D_{i'}$ for $i'\in I_{\mW,\mf}$ 
and $s_{\theta}$ is a local equation of $E_{\theta}$ for $\theta\in \Theta$
such that $h^{*}t_{i}=\prod_{\theta\in \Theta_{i}}s_{\theta}$.
Hence $dh^{D'}_{(h^{*}Z_{\mf}^{1/p})_{\red}}(h^{*}\cform^{D'}(\mf))$ is the image of $\overline{h^{*}a}$ by 
$\varphi_{s}^{((h^{*}D')_{\red}\subset (h^{*}D)_{\red},h^{*}R_{\mf}^{D'})}(W)$
by the construction of $\varphi_{s}^{((h^{*}D')_{\red}\subset (h^{*}D)_{\red},h^{*}R_{\mf}^{D'})}$
in Lemma \ref{lemcfdef}.  

We prove the equivalence of the conditions (a), (b), and (c).
Let $j'_{\theta}\colon \Spec L_{\theta}\rightarrow X$ denote the canonical morphism
for $\theta\in \Theta$.
We consider the commutative diagrams
\begin{equation}
\label{diagrfhwshf}
\xymatrix{
\Gamma(W,\grlog_{h^{*}R_{\mf}^{D'}}^{(h^{*}D')_{\red}}j'_{*}W_{s}(\dvr_{h^{*}U})) \ar[rr]
\ar[d]&&
\Gamma(W, \Omega_{W}^{1}(\log(h^{*}D)_{\red})(h^{*}R_{\mf}^{D'})|_{(h^{*}Z_{\mf}^{1/p})_{\red}})
\ar[d] \\
\grlog_{m_{\theta}}W_{s}(L_{\theta}) \ar[rr]_-{\varphi_{s}^{(m_{\theta})}} &&
\grlog_{m_{\theta}}\Omega_{L_{\theta}}^{1}\otimes_{F_{L_{\theta}}} F_{L_{\theta}}^{1/p}
} 
\end{equation}
for $\theta\in \Theta'$ and 
\begin{equation}
\label{diagrfhwshs}
\xymatrix{
\Gamma(W,\grlog_{h^{*}R_{\mf}^{D'}}^{(h^{*}D')_{\red}}j'_{*}W_{s}(\dvr_{h^{*}U})) \ar[rr]
\ar[d]&&
\Gamma(W, \Omega_{W}^{1}(\log(h^{*}D)_{\red})(h^{*}R_{\mf}^{D'})|_{(h^{*}Z_{\mf}^{1/p})_{\red}})
\ar[d] \\
\gr_{m_{\theta}}W_{s}(L_{\theta}) \ar[rr]_-{\varphi_{s}'^{(m_{\theta})}} &&
\gr_{m_{\theta}}\Omega_{L_{\theta}}^{1}\otimes_{F_{L_{\theta}}} F_{L_{\theta}}^{1/p}
}
\end{equation}
for $\theta\in \Theta-\Theta'$,
where the upper horizontal arrows are $\varphi_{s}^{((h^{*}D')_{\red}\subset (h^{*}D)_{\red},h^{*}R_{\mf}^{D'})}(W)$ and
the vertical arrows are the canonical morphisms.
Then the condition (a) is equivalent to
that the image of $\overline{h^{*}a}$ by the composition of the upper horizontal arrow
and the right vertical arrow in (\ref{diagrfhwshf}) is not $0$
for every $\theta\in \Theta'$ 
and neither is that in (\ref{diagrfhwshs}) for every $\theta\in \Theta-\Theta'$.
Since the diagrams (\ref{diagrfhwshf}) and (\ref{diagrfhwshs}) are commutative,
the last condition is equivalent to both the conditions (b) and (c)
by Lemma \ref{lemrsw}.

(2) 
Suppose that the equivalent conditions (a), (b), and (c) in (1) hold.
Then we have 
\begin{equation}
Z_{h^{*}\mf}=\Supp (R_{h^{*}\mf}^{(h^{*}D')_{\red}}+(h^{*}D')_{\red}-(h^{*}D)_{\red})
\supset\Supp h^{*}(R_{\mf}^{D'}+D'-D)
=(h^{*}Z_{\mf})_{\red} \notag
\end{equation}
by the condition (b).
Since $Z_{h^{*}\mf}\subset (h^{*}Z_{\mf})_{\red}$, we obtain the first equality.
The second equality holds
by the first equality, by the condition (c), and by Remark \ref{remcfnormal} (3).
\end{proof}

\begin{rem}
\label{remcformgl}
Let the notation be as in Lemma \ref{lemcformgl}.
\begin{enumerate}
\item If $\cform^{D'}(\mf)$
can be regarded as a global section of $\Omega_{X}^{1}(\log D')(R_{\mf}^{D'})|_{Z_{\mf}}$ (see Remark \ref{remcfnormal} (5)), 
we can remove the index $1/p$ everywhere in Lemma \ref{lemcformgl} and its proof.
\item The assumptions on $h\colon W\rightarrow X$ in Lemma \ref{lemcformgl} are satisfied in the following cases:
\begin{enumerate}
\item $h\colon W\rightarrow X$ is the blow-up of $X$ along an intersection $Y$ of irreducible components of $D'$, where $Y$ is regarded as a closed subscheme of $X$ with the reduced subscheme structure.
\item $h\colon W\rightarrow X$ is $C_{D'\subset D}$-transversal (see Subsection \ref{sslogtrans}).
\end{enumerate}
\end{enumerate}
\end{rem}

\subsection{$\mathrm{Log}$-$D'$-cleanliness}
\label{sslogdpcl}

We introduce the notion that the ramification of $\mf$ is $\log$-$D'$-clean along $D$ 
for a divisor $D'$ on $X$ with simple normal crossings contained in $D$.

\begin{defn}[{cf.\ \cite[(3.4.3)]{kalog}, \cite[Definition 2.17]{sacot}}]
\label{deflogdcl}
Let $I'\subset I$ be a subset such that $D'=\bigcup_{i\in I'}D_{i}$ has simple normal crossings.
\begin{enumerate}
\item Let $x\in X$.
We say that the ramification of $\mf$ is $\log$-$D'$-{\it clean along} $D$ at $x$ if 
one of the following two conditions is satisfied:
\begin{enumerate}
\item $x\notin Z_{\mf}$ (Definition \ref{defconddiv} (2)).
\item $x\in Z_{\mf}$ and 
the germ $\cform^{D'}(\mf)_{x'}$ of the $\log$-$D'$-characteristic form 
$\cform^{D'}(\mf)$ of $\mf$ at the unique point $x'$ on $Z_{\mf}^{1/p}$ lying above $x$ is a part of a basis of 
the free $\dvr_{Z_{\mf}^{1/p},x'}$-module
$\Omega^{1}_{X}(\log D')(R_{\mf}^{D'})|_{Z_{\mf}^{1/p},x'}$.
\end{enumerate}
\item We say that the ramification of $\mf$ is $\log$-$D'$-{\it clean along} $D$ if the ramification of $\mf$ is $\log$-$D'$-clean along $D$ at every $x\in X$.
\end{enumerate}
\end{defn}

\begin{rem}
\label{remlogdpcl}
Let the notation be as in Definition \ref{deflogdcl}.
\begin{enumerate}
\item The $\log$-$D'$-cleanliness of the ramification of $\mf$ along $D$ is equivalent to
the $\log$-$D'$-cleanliness of the ramification of $\mf$ along $D$ at every point on $Z_{\mf}$.
If $I=I_{\mT,\mf}$ (Definition \ref{defindsub} (1)), then the ramification of $\mf$ is $\log$-$D'$-clean along $D$, since $Z_{\mf}=\emptyset$.
\item The $\log$-$D'$-cleanliness of the ramification of $\mf$ along $D$
is an open condition on $X$, since the condition (b) in Definition \ref{deflogdcl} (1)
is an open condition on $Z_{\mf}$. 
Since the germ $\rsw(\chi|_{K_{i}})$ or $\cform(\chi|_{K_{i}})$ of $\cform^{D'}(\mf)$ at the generic point of $D_{i}^{1/p}\subset Z_{\mf}^{1/p}$ is not $0$ 
by Remark \ref{remrsw}, the ramification of $\mf$ is $\log$-$D'$-clean along $D$
outside a closed subset of $X$ of codimension $\ge 2$.
\item Suppose that $D$ has simple normal crossings.
Then the $\log$-$D$-cleanliness 
is the same as the cleanliness in the sense of \cite[(3.4.3)]{kalog}. 
The $\log$-$\emptyset$-cleanliness
is the same as the non-degeneration in the sense of \cite[Definition 4.2]{sacot}.
\end{enumerate}
\end{rem}

\begin{lem}
\label{lemeqtologdpcl}
Let $I'\subset I$ be a subset such that $D'=\bigcup_{i\in I'}D_{i}$ has simple normal crossings.
\begin{enumerate}
\item Let $x$ be a closed point of $Z_{\mf}$ (Definition \ref{defconddiv} (2)) and
let $x'$ be the unique point on $Z_{\mf}^{1/p}$ lying above $x$.
Then the following two conditions are equivalent:
\begin{enumerate}
\item The ramification of $\mf$ is $\log$-$D'$-clean along $D$ at $x$.
\item $\cform^{D'}(\mf)(x')\neq 0$.
\end{enumerate}
\item The following two conditions are equivalent:
\begin{enumerate}
\item The ramification of $\mf$ is $\log$-$D'$-clean along $D$.
\item $\cform^{D'}(\mf)(x')\neq 0$ for all closed point $x'$ of $Z_{\mf}^{1/p}$.
\end{enumerate}
\end{enumerate}
\end{lem}

\begin{proof}
(1) By Nakayama's lemma, the condition (b) is equivalent to 
the condition (b) in Definition \ref{deflogdcl} (1),
and is equivalent to the condition (a).

(2) By (1), the condition (b) is equivalent to that
the ramification of $\mf$ is $\log$-$D'$-clean along $D$ at
every closed point of $Z_{\mf}$.
By Remark \ref{remlogdpcl} (2), the last condition is equivalent to that
the ramification of $\mf$ is $\log$-$D'$-clean along $D$ at every point on $Z_{\mf}$.
By Remark \ref{remlogdpcl} (1), the last condition is equivalent to the condition (a).
\end{proof}

\begin{defn}
\label{defximf}
Let $I'\subset I$ be a subset such that $D'=\bigcup_{i\in I'}D_{i}$ has simple normal crossings.
Let $i\in I'\cap I_{\mW,\mf}$ (Definition \ref{defindsub} (1)).
\begin{enumerate}
\item We define a morphism
\begin{equation}
\label{defmidpmf}
m_{i}^{D'}(\mf)\colon \dvr_{X}(-R_{\mf}^{D'})|_{D_{i}^{1/p}}\xrightarrow{\times \cform^{D'}(\mf)|_{D_{i}^{1/p}}} \Omega^{1}_{X}(\log D')|_{D_{i}^{1/p}} 
\end{equation}
of locally free sheaves on $D_{i}^{1/p}$
to be the multiplication by the restriction $\cform^{D'}(\mf)|_{D_{i}^{1/p}}$ of $\log$-$D'$-characteristic form $\cform^{D'}(\mf)$ to $D_{i}^{1/p}$.
Here $R_{\mf}^{D'}$ is as in Definition \ref{defconddiv} (1).
\item We define a morphism 
\begin{equation}
\label{xidpmf}
\xi_{i}^{D'}(\mf)\colon \dvr_{X}(-R_{\mf}^{D})|_{D_{i}^{1/p}}\rightarrow \dvr_{D_{i}^{1/p}} 
\end{equation} 
of invertible sheaves on $D_{i}^{1/p}$
to be the composition 
\begin{equation}
\label{xidpmfcomp}
\xi_{i}^{D'}(\mf)\colon 
\dvr_{X}(-R_{\mf}^{D'})|_{D_{i}^{1/p}}\xrightarrow{m_{i}^{D'}(\mf)} \Omega^{1}_{X}(\log D')|_{D_{i}^{1/p}}\rightarrow\dvr_{D_{i}^{1/p}}
\end{equation}
of $m_{i}^{D'}(\mf)$ 
and the base change of the residue mapping $\res_{i}\colon \Omega_{X}^{1}(\log D')|_{D_{i}}\rightarrow \dvr_{D_{i}}$ by the canonical morphism
$D_{i}^{1/p}\rightarrow D_{i}$.
\end{enumerate}
\end{defn}

\begin{rem}
\label{remximf}
Let the assumption and the notation be as in Definition \ref{defximf}.
\begin{enumerate}
\item If the ramification of $\mf$ is $\log$-$D'$-clean along $D$, then
the morphism $m_{i}^{D'}(\mf)$ for $i\in I'\cap I_{\mW,\mf}$ is injective,
since the germ $\cform^{D'}(\mf)_{x'}$ at $x'\in D_{i}^{1/p}$ is a part of a basis
of $\Omega_{X}^{1}(\log D')(R_{\mf}^{D'})|_{D_{i}^{1/p},x'}$
for every $x'\in D_{i}^{1/p}$.

\item Suppose that $D$ has simple normal crossings.
Since the $\log$-$D'$-characteristic form $\cform^{D'}(\mf)$ is locally the 
image of a section by the morphism $\varphi_{s}^{(D'\subset D,R_{\mf}^{D'})}$
constructed in
Lemma \ref{lemcfdef} by Remark \ref{remcfnormal} (2), 
we can regard $\xi_{i}^{D'}(\mf)$ as a morphism from $\dvr_{X}(-R_{\mf}^{D'})|_{D_{i}}$
to $\dvr_{D_{i}}$ for $i\in I'\cap I_{\mW,\mf}$ by the construction of 
$\varphi_{s}^{(D'\subset D,R_{\mf}^{D'})}$ in Lemma \ref{lemcfdef}. 

\item Suppose that $D$ has simple normal crossings and that $D'=D$. 
Then we may replace $D_{i}^{1/p}$ by $D_{i}$ in Definition \ref{defximf} 
by Remark \ref{remcfnormal} (5).
We denote by $\xi_{i}(\mf)$ the morphism $\xi_{i}^{D}(\mf)$ with $D_{i}^{1/p}$
replaced by $D_{i}$ for $i\in I_{\mW,\mf}$:
\begin{equation}
\label{ximf}
\xi_{i}(\mf)\colon 
\dvr_{X}(-R_{\mf}^{D})|_{D_{i}}\xrightarrow{\times\cform^{D}(\mf)|_{D_{i}}} \Omega^{1}_{X}(\log D)|_{D_{i}}\xrightarrow{\res_{i}}\dvr_{D_{i}}.
\end{equation}
If the image of $\xi_{i}(\mf)$ in the residue field $k(x)$ at some closed point $x$ of $Z_{\mf}$ is not $0$,
then the ramification of $\mf$ is $\log$-$D$-clean along $D$ at $x$
by Lemma \ref{lemeqtologdpcl} (1).
\end{enumerate}
\end{rem}

\begin{lem}
\label{lemcfclatx}
Suppose that $D$ has simple normal crossings.
Let $I'\subset I$ be a subset and let $D'=\bigcup_{i\in I'}D_{i}$.
Let $x$ be a closed point of $Z_{\mf}$ (Definition \ref{defconddiv}). 
\begin{enumerate}
\item If the ramification of $\mf$ is $\log$-$D'$-clean along $D$ at $x$,
then $I_{\mT,\mf,x}$ (\ref{eachindatx}) is contained in $I'$
and
the cardinality of $I_{\mI,\mf,x}-I'$ (\ref{eachindatx}) is $\le 1$. 
\item Suppose that $I_{\mT,\mf,x}\subset I'$ and that
the cardinality of $I_{\mI,\mf,x}-I'$ is $1$.
Then the ramification of $\mf$ is $\log$-$D'$-clean along $D$ at $x$ 
if and only if the image of $\xi_{i}(\mf)$ (\ref{ximf}) for the unique $i\in I_{\mI,\mf,x}-I'$ in the residue field $k(x)$ at $x$ is not $0$.
\end{enumerate}
\end{lem}

\begin{proof}
Let $x'$ be the unique closed point of $Z_{\mf}^{1/p}$ lying above $x$.

(1) We prove the contrapositive of the statement.
We put $I'_{x}=I'\cap I_{x}$ (\ref{defix}).
If the cardinality of $J_{x}=I_{x}-( I_{\mII,\mf,x}\cup I'_{x})$ is $\ge 2$
or if $J_{x}\cap I_{\mT,\mf,x}$ is non-empty,
then we have $\cform^{D'}(\mf)(x')=0$ by Lemma \ref{lemcfatx} (1) applied to the case where $(I',I'')$ is $(I,I')$.
Hence the assertion holds by Lemma \ref{lemeqtologdpcl} (1).

(2) Let $i$ be the unique element of $I_{\mI,\mf,x}-I'$. 
By Lemma \ref{lemcfatx} (1) applied to the case where $(I',I'')$ is $(I,I')$,
the image of $\xi_{i}(\mf)$ in $k(x)$
is not $0$ if and only if $\cform^{D'}(\mf)(x')\neq 0$.
By Lemma \ref{lemeqtologdpcl} (1), the last condition is equivalent to that the ramification of $\mf$ is $\log$-$D'$-clean along $D$ at $x$.
\end{proof}

In the rest of this article, 
we put 
\begin{equation}
\label{defd*mf}
D_{*,\mf}=\bigcup_{i\in I_{*,\mf}}D_{i} 
\end{equation} 
for $*=\mT,\mW, \mI,\mII$ (Definition \ref{defindsub}).
If $I_{*,\mf}=\emptyset$, then we have $D_{*,\mf}=\emptyset$ 
for $*=\mT,\mW, \mI,\mII$, by convention.
For $*=\mW$,
we have $D_{\mW,\mf}=Z_{\mf}$ (Definition \ref{defconddiv} (2)).
For a subset $I''\subset I$, we put
\begin{equation}
\label{defdipp}
D_{I''}=\bigcap_{i\in I''}D_{i},
\end{equation}
where $D_{I''}=X$ if $I''=\emptyset$.
We should distinguish the three similar notations 
$D_{\mI,\mf}$, $D_{I_{\mI,\mf}}$, and $D_{I_{\mI,\mf,x}}$, where $x$ is a point on $D$;
the first $D_{\mI,\mf}$ is the union of $D_{i}$ for $i\in I_{\mI,\mf}$
as in (\ref{defd*mf}) and the second $D_{I_{\mI,\mf}}$ 
(resp.\ the third $D_{I_{\mI,\mf,x}}$) is the intersection
of $D_{i}$ for $i\in I_{\mI,\mf}$ (resp.\ for $i\in I_{\mI,\mf,x}$)
as in (\ref{defdipp}).
If there is no risk of confusion, then
we use $D_{*,\mf}$, $D_{I_{*,\mf}}$, and $D_{I_{*,\mf,x}}$ for $*=\mT,\mW, \mI,\mII$ and for a point $x$ on a divisor with smooth irreducible components
independently on the symbol denoting the divisor.

We compare the logarithmic cleanliness among several settings.
As a consequence of the comparisons, we prove that the $\log$-$D_{\mI,\mf}\cup D_{\mT,\mf}$-cleanliness
is the weakest condition among the $\log$-$D'$-cleanliness 
for divisors $D'\subset D$ on $X$ when
$D$ has simple normal crossings.

\begin{lem}[{cf.\ \cite[Lemma 2.28 (i)]{yacc}}]
\label{lemclrelfirs}
Let $I''\subset I'\subset I$ be subsets such that $I'-I''$ is contained in $I_{\mII,\mf}$ (Definition \ref{defindsub} (2))
and that $D'=\bigcup_{i\in I'}D_{i}$ has simple normal crossings.
Let $x$ be a closed point of $Z_{\mf}$ (Definition \ref{defconddiv} (2)).
Suppose that the ramification of $\mf$ is $\log$-$D'$-clean along $D$ at $x$.
Then the ramification of $\mf$ is $\log$-$D''$-clean along $D$ at $x$
for $D''=\bigcup_{i\in I''}D_{i}$.
Consequently, the $\log$-$D'$-cleanliness of the ramification of $\mf$ along $D$
implies the $\log$-$D''$-cleanliness of the ramification of $\mf$ along $D$.
\end{lem}

\begin{proof}
By Lemma \ref{lemeqtologdpcl} (1), we have $\cform^{D'}(\mf)(x')\neq 0$ for the closed point $x'$ of $Z_{\mf}^{1/p}$ lying above $x$.
Thus we have $\cform^{D''}(\mf)(x')\neq 0$ 
by Lemma \ref{lemclrelfirf},
and first the assertion holds by Lemma \ref{lemeqtologdpcl} (1).
The last assertion follows from the first assertion and Remarks \ref{remlogdpcl} (1) and (2).
\end{proof}

\begin{lem}[{cf.\ \cite[Lemma 2.26 (ii)]{yacc}}]
\label{lemcldegeq}
Suppose that $D$ has simple normal crossings. 
Let $I'\subset I$ be a subset and let $D'=\bigcup_{i\in I'}D_{i}$.
Let $x$ be a closed point of $Z_{\mf}$ (Definition \ref{defconddiv} (2)).
Assume that the ramification of $\mf$ is $\log$-$D'$-clean along $D$ at $x$.
Then the ramification of $\mf$ 
is not $\log$-$\emptyset$-clean along $D$ at $x$ if and only if one of the following three conditions holds:
\begin{enumerate}
\item $I_{\mT,\mf,x}\neq \emptyset$ (\ref{eachindatx}).
\item $I_{\mT,\mf,x}= \emptyset$, $I_{\mI,\mf,x}\subset I'$,
the cardinality of $I_{\mI,\mf,x}$ (\ref{eachindatx}) is $1$, and 
the image of $\xi_{i}^{D'}(\mf)$ (\ref{ximf}) in $k(x)$ is $0$ for the unique $i\in I_{\mI,\mf,x}$.
\item $I_{\mT,\mf,x}= \emptyset$ 
and the cardinality of $I_{\mI,\mf,x}$ is $\ge 2$.
\end{enumerate}
\end{lem}

\begin{proof}
If one of the conditions (1) and (3) holds,
then the ramification of $\mf$ is not $\log$-$\emptyset$-clean along $D$ at $x$ by Lemma \ref{lemcfclatx} (1) applied to the case where $I'=\emptyset$.
Hence we may assume that neither the condition (1) nor the condition (3) holds.
Then we have $I_{\mT,\mf,x}=\emptyset$ and
the cardinality of $I_{\mI,\mf,x}$ is $\le 1$.
If $I'\cap I_{\mI,\mf,x}=\emptyset$, then
we have $I'\cap I_{x}\subset I_{\mII,\mf,x}$ ((\ref{defix}), (\ref{eachindatx})), and
the ramification of $\mf$ is $\log$-$\emptyset$-clean along $D$ 
at $x$ by Lemma \ref{lemclrelfirs}
applied to the case where $(I',I'')$ is $(I',\emptyset)$.
Thus we may assume that the cardinality of $I_{\mI,\mf,x}$ is $1$ and that $I_{\mI,\mf,x}\subset I'$.
Then the image of $\xi_{i}^{D'}(\mf)$ in $k(x)$ for the unique $i\in I_{\mI,\mf,x}$ is equal to
that of $\xi_{i}(\mf)$ (\ref{ximf}) in $k(x)$ by Lemma \ref{lemcfatx} (1) applied
to the case where $(I',I'')$ is $(I,I')$.
Therefore the desired equivalence follows from Lemma \ref{lemcfclatx} (2)
with $I'=\emptyset$.
\end{proof}

\begin{lem}[{cf.\ \cite[Lemma 2.28 (i)]{yacc}}]
\label{lemclrelsecs}
Let $I'\subset I$ be a subset such that $D'=\bigcup_{i\in I'}D_{i}$ has simple normal crossings.
Let $x$ be a closed point of $Z_{\mf}$ (Definition \ref{defconddiv} (2)).
Suppose that the ramification of $\mf$ is $\log$-$D'$-clean along $D$ at $x$.
Let $I''\subset I'\cap I_{\mT,\mf}$
(Definition \ref{defindsub} (1)) be a subset and
let $E=\bigcup_{i\in I-I''}D_{i}$ and $E'=\bigcup_{i\in I'-I''}D_{i}$.  
Let $\mf'$ be a smooth sheaf of $\Lambda$-modules of rank 1 on $V=X-E$
whose associated character $\chi'\colon \pi_{1}^{\ab}(V)\rightarrow \Lambda^{\times}$ has the $p$-part inducing the $p$-part of $\chi$.
Then the ramification of $\mf'$ is $\log$-$E'$-clean along $E$ at $x$.
Consequently, the $\log$-$D'$-cleanliness of the ramification of $\mf$ along $D$ implies the $\log$-$E'$-cleanliness of the ramification of $\mf'$ along $E$.
\end{lem}

\begin{proof}
The first assertion follows 
similarly as the proof of Lemma \ref{lemclrelfirs} using Lemma \ref{lemclrelsecf} instead of Lemma \ref{lemclrelfirf}. 
Since we have $Z_{\mf}=Z_{\mf'}$ by Lemma \ref{lemclrelsecf},
the last assertion follows from the first assertion and Remarks \ref{remlogdpcl} (1) and (2).
\end{proof}

\begin{prop}[{cf.\ \cite[Remark 4.13]{kalog}}]
\label{propblupcl}
Suppose that $D$ has simple normal crossings.
Let $I'\subset I$ be a subset and let $D'=\bigcup_{i\in I'}D_{i}$. 
Assume that the ramification of $\mf$ is $\log$-$D'$-clean along $D$
and that $I_{\mT,\mf}$ (Definition \ref{defindsub} (1)) is contained in $I'$.
We put $I=\{1,2,\ldots,r\}$, where $r\ge 1$.
Let $f\colon X'\rightarrow X$ be the blow-up along
an irreducible component $Y$ of
$D_{I''}$ (\ref{defdipp}) for a non-empty subset $I''\subset I'$,
where $Y$ is regarded as a closed subscheme of $X$ with the reduced subscheme structure.
Let $D_{i}'$ denote the proper transform of $D_{i}$ for $i\in I$
and let $D_{0}'=f^{-1}(Y)$ be the exceptional divisor.
We identify the index set of the irreducible components
of $f^{*}D$ with $I\cup\{0\}$.
Then the following hold:
\begin{enumerate}
\item $R_{f^{*}\mf}^{(f^{*}D')_{\red}}=f^{*}R_{\mf}^{D'}$ (Definition \ref{defconddiv} (1)), 
$Z_{f^{*}\mf}=(f^{*}Z_{\mf})_{\red}$ (Definition \ref{defconddiv} (2)),
and $\cform^{(f^{*}D')_{\red}}(f^{*}\mf)=df^{D'}_{(f^{*}Z_{\mf}^{1/p})_{\red}}(f^{*}\cform^{D'}(f^{*}\mf))$ (Definition \ref{defcform}).
Here $df^{D'}_{(f^{*}Z_{\mf}^{1/p})_{\red}}(f^{*}\cform^{D'}(f^{*}\mf))$ denotes the image
by the canonical morphism
\begin{align}
\Gamma((f^{*}Z_{\mf}^{1/p})_{\red},
&f^{*}\Omega_{X}^{1}(\log D')(R_{\mf}^{D'})|_{(f^{*}Z_{\mf}^{1/p})_{\red}})
\notag \\ 
&\rightarrow
\Gamma((f^{*}Z_{\mf}^{1/p})_{\red}, \Omega_{X'}^{1}(\log (f^{*}D')_{\red})(f^{*}R_{\mf}^{D'})|_{(f^{*}Z_{\mf}^{1/p})_{\red}}) \notag
\end{align}
induced by $f$ of the pull-back
\begin{equation}
f^{*}\cform^{D'}(\mf)\in \Gamma((f^{*}Z_{\mf}^{1/p})_{\red},
f^{*}\Omega_{X}^{1}(\log D')(R_{\mf}^{D'})|_{(f^{*}Z_{\mf}^{1/p})_{\red}}) \notag
\end{equation}
of the $\log$-$D'$-characteristic form $\cform^{D'}(\mf)$ by the morphism $(f^{*}Z_{\mf}^{1/p})_{\red}\rightarrow Z_{\mf}^{1/p}$ induced by $f$.

\item The ramification of $f^{*}\mf$ is $\log$-$(f^{*}D')_{\red}$-clean along $(f^{*}D)_{\red}$.
\item The image of $\xi_{i}^{(f^{*}D')_{\red}}(f^{*}\mf)$ 
(Definition \ref{defximf} (2)) for $i\in (I'\cup \{0\})\cap I_{\mW,f^{*}\mf}$
(Definiton \ref{defindsub} (1))
is equal to that of $f^{*}\xi_{i}^{D'}(\mf)$
in $\dvr_{D_{i}'^{1/p}}$ if $i\in I'$
and to the sum of the images of $f^{*}\xi_{i'}^{D'}(\mf)$ in $\dvr_{D_{i}'^{1/p}}$
for all $i'\in I''\cap I_{\mW,\mf}$ if $i=0$.
\end{enumerate}
\end{prop}

\begin{proof}
We may assume that $X$ is purely of dimension $d$.
Let $x$ be a closed point of $Z_{\mf}$ (Definition \ref{defconddiv} (2)).
Since the assertions are local, we may shrink $X$ to a neighborhood of $x$ and we may assume that $I=I_{x}$ (\ref{defix}).
We put $I'=\{1,2,\ldots,r'\}$
for $1\le r'\le r\le d$.
Let $(t_{1},t_{2},\ldots,t_{d})$ be a local coordinate system 
at $x$ such that $t_{i}$ is a local equation of $D_{i}$
for $i\in I$.
By shrinking $X$ if necessary,
we may assume that $X=\Spec A$ for a $k$-algebra $A$, 
that $t_{i}\in A$ for $i=1,2,\ldots,d$, 
that $D_{i}=(t_{i}=0)$ for $i\in I$,
and that $\Omega_{X}^{1}(\log D')$ is a free $\dvr_{X}$-module with the basis $(\dlog t_{1},\dlog t_{2},\ldots,\dlog t_{r'},dt_{r'+1},dt_{r'+2},\ldots,dt_{d})$. 
We put $R_{\mf}^{D'}=\sum_{i\in I}n_{i}D_{i}$ and
\begin{equation}
\label{eqcfdpfblup}
\cform^{D'}(\mf)=\frac{\sum_{i=1}^{r'}\alpha_{i}\dlog t_{i}+\sum_{i=r'+1}^{d}\beta_{i}dt_{i}}{
\prod_{i=1}^{r}t_{i}^{n_{i}}},
\end{equation}
where $\alpha_{i}$ for $i\in I'$ 
and $\beta_{i}$ for $i=r'+1,r'+2,\ldots,d$ are global sections
of $Z_{\mf}^{1/p}$.
Since the ramification of $\mf$ is $\log$-$D'$-clean along $D$,
we may assume that $\alpha_{i}$ for some $i\in I'$ or $\beta_{i}$ for some $i=r'+1,r'+2,\ldots,d$ 
is invertible in $Z_{\mf}^{1/p}$ by Lemma \ref{lemeqtologdpcl} (1)
and by shrinking $X$ if necessary.

Let $i\in I''=\{1,2,\ldots,r''\}$, where $1\le r''\le r'$.
We put $U_{i}'=X'-D_{i}'$.
We denote $f^{*}t_{i}$ simply by $t_{i}$ in $U_{i}'$ 
and put $f^{*}t_{i'}=t_{i}t_{i'}$ for $i'\in I''-\{i\}$ and $f^{*}t_{i'}=t_{i'}$ for $i'=r''+1,r''+2,\ldots,d$ in $U_{i}'$ by abuse of notation.
Then $D_{0}'=(t_{i}=0)$ in $U_{i}'$ and
we have
\begin{equation}
\label{eqcfblupcl}
df^{D'}_{(f^{*}Z_{\mf}^{1/p})_{\red}}(f^{*}\cform^{D'}(f^{*}\mf))=\frac{\sum_{i'=1}^{r'}\alpha_{i'}'\dlog t_{i'}+\sum_{i'=r'+1}^{d}f^{*}\beta_{i'}dt_{i'}}{
\prod_{i'=1}^{r}t_{i'}^{n_{i'}'}},
\end{equation}
where 
\begin{equation}
\label{alphap}
\alpha_{i'}'=\begin{cases}
f^{*}\alpha_{i'} &  (i'\neq i), \\
\sum_{i''=1}^{r''}f^{*}\alpha_{i''} & (i'=i) 
\end{cases}
\end{equation}
for $i'=1,2,\ldots,r'$ and 
\begin{equation}
n_{i'}'=\begin{cases}
n_{i'} &  (i'\neq i), \\
\sum_{i''=1}^{r''}n_{i''} & (i'=i)
\end{cases}
\notag
\end{equation}
for $i'=1,2,\ldots,r$, in $U_{i}'$.
Since $\alpha_{i'}$ for some $i'\in I'$ or 
$\beta_{i'}$ for some $i'=r'+1,r'+2,\ldots,d$ in (\ref{eqcfdpfblup}) 
is invertible in $Z_{\mf}^{1/p}$,
there exists $\alpha_{i'}'$ for some $i'=1,2,\ldots,r'$
or $f^{*}\beta_{i'}$ for some $i'=r'+1,r'+2,\ldots,d$ 
in (\ref{eqcfblupcl}) that is invertible in $k(y')$ for each closed point $y'$ of $(D_{0}'\cap U_{i}')^{1/p}$.
Therefore we have $df^{D'}_{(f^{*}Z_{\mf}^{1/p})_{\red}}(f^{*}\cform^{D'}(f^{*}\mf))(y')\neq 0$
for every closed point $y'$ of $D_{0}'^{1/p}$.
Since $X'$ is isomorphic to $X$ outside $D_{0}'$, 
we have $df^{D'}_{(f^{*}Z_{\mf}^{1/p})_{\red}}(f^{*}\cform^{D'}(f^{*}\mf))(y')\neq 0$ for every closed point $y'$ of $(f^{*}Z_{\mf}^{1/p})_{\red}$
by Lemma \ref{lemeqtologdpcl} (2) and the $\log$-$D'$-cleanliness of the ramification of $\mf$ along $D$.
Thus the assertion (1) holds by Lemma \ref{lemcformgl}.
The assertion (2) holds by the third equality in (1) and Lemma \ref{lemeqtologdpcl} (2).
The assertion (3) holds by the third equality in (1) and (\ref{alphap}),
since $f^{*}\alpha_{i''}|_{D_{0}'}=0$ for $i''\in I_{\mT,\mf}$ by Lemma \ref{lemcfatx} (2).
\end{proof}

\begin{prop}
\label{propdpdtdi}
Suppose that $D$ has simple normal crossings
and that the ramification of $\mf$ is $\log$-$D'$-clean along $D$
for some union $D'$ of irreducible components of $D$.
Then the ramification of $\mf$ is $\log$-$D_{\mI,\mf}\cup D_{\mT,\mf}$-clean along $D$,
where $D_{\mI,\mf}$ and $D_{\mT,\mf}$ are as in (\ref{defd*mf}).
\end{prop}

\begin{proof}
By Remarks \ref{remlogdpcl} (1) and (2),
it is sufficient to prove that the ramification of $\mf$ is $\log$-$D_{\mI,\mf}\cup D_{\mT,\mf}$-clean along $D$ at every closed point $x$ of $Z_{\mf}$ (Definition \ref{defconddiv} (2)).
Let $x$ be a closed point of $Z_{\mf}$.
Then $I_{\mT,\mf,x}$ (\ref{eachindatx}) is contained in $I'$ by Lemma \ref{lemcfclatx} (1).
If $I_{\mI,\mf,x}$ (\ref{eachindatx}) is contained in $I'$, then we have $D_{\mI,\mf}\cup D_{\mT,\mf}\subset D'$ in a neighborhood of $x$, and the assertion follows from
Lemma \ref{lemclrelfirs} applied to the case where $(I', I'')$ is $(I',I_{\mI,\mf}\cup I_{\mT,\mf})$.
If $I_{\mI,\mf,x}$ is not contained in $I'$, then the cardinality of $I_{\mI,\mf,x}-I'$ is $1$ by Lemma \ref{lemcfclatx} (1), and the image of $\xi_{i}(\mf)$ (\ref{ximf}) for the unique $i\in I_{\mI,\mf,x}-I'$ in $k(x)$ is not $0$ by Lemma \ref{lemcfclatx} (2).
Then the ramification of $\mf$ is $\log$-$D$-clean along $D$ at $x$
by Remark \ref{remximf} (3), and
the assertion follows from Lemma \ref{lemclrelfirs} applied to the case where $(I',I'')$ is $(I,I_{\mI,\mf}\cup I_{\mT,\mf})$.
\end{proof}

\subsection{Preparation for the main purpose}
\label{sspresecfive}
We devote this subsection to preparation for 
Subsections \ref{sscandss}, \ref{sshicc},
and Section \ref{ssrk1},
where we consider computations of the singular support $SS(j_{!}\mf)$ and
the characteristic cycle $CC(j_{!}\mf)$ of the zero extension $j_{!}\mf$ of $\mf$ by
the canonical open immersion $j\colon U\rightarrow X$ in terms of ramification theory.
We do not use anything given in this subsection before Subsection \ref{sscandss}.

\begin{defn}[{cf.\ \cite[Lemma 4.2]{kalog}, \cite[Lemma 2.17 (ii)]{yacc}}]
\label{deforder}
Let $I'\subset I$ be a subset such that $D'=\bigcup_{i\in I'}D_{i}$ has simple normal crossings.
For $i\in I_{\mW,\mf}$ (Definition \ref{defindsub} (1)) and a closed point $x$ of $D_{i}$,
we define an integer $\ord^{D'}(\mf;x,D_{i})$ 
to be the largest non-negative integer $n$ such that
\begin{align}
\cform^{D'}(\mf)|_{D_{i}^{1/p},x'}\in \mathfrak{m}_{x'}^{pn}\Omega^{1}_{X}(\log D')(R_{\mf}^{D'})|_{D_{i}^{1/p},x'}, \notag
\end{align}
where $x'$ denotes the unique point on $D_{i}^{1/p}$ lying above $x$
and $\mathfrak{m}_{x'}\subset \dvr_{D_{i}^{1/p},x'}$ denotes the maximal ideal at $x'$.
\end{defn}

\begin{lem}[{cf.\ \cite[Lemma 4.3]{kalog}, \cite[Lemma 2.18 (i)]{yacc}}]
\label{lemorder}
Let $I'\subset I$ be a subset such that $D'=\bigcup_{i\in I'}D_{i}$ has simple normal crossings.
For a closed point $x$ of $Z_{\mf}$ (Definition \ref{defconddiv} (2)), 
the following three conditions are equivalent:
\begin{enumerate}
\item The ramification of $\mf$ is $\log$-$D'$-clean along $D$ at $x$.
\item $\ord^{D'}(\mf;x, D_{i})=0$ for every $i\in I_{\mW,\mf,x}$ (\ref{eachindatx}).
\item $\ord^{D'}(\mf;x, D_{i})=0$ for some $i\in I_{\mW,\mf,x}$.
\end{enumerate}
\end{lem}

\begin{proof}
Both the conditions (2) and (3) are equivalent to
the condition (b) in Lemma \ref{lemeqtologdpcl} (1),
which is equivalent to the condition (1) by loc.\ cit.. 
\end{proof}

\begin{defn}
\label{defbcf}
Assume that $D$ has simple normal crossings.
Let $I'\subset I$ be a subset containing $I_{\mT,\mf}$ (Definition \ref{defindsub} (1))
and let $D'=\bigcup_{i\in I'}D_{i}$.
Assume that the ramification of $\mf$ is $\log$-$D'$-clean along $D$.
\begin{enumerate}
\item We regard the image $\Image \xi^{D'}_{i}(\mf)$ of $\xi_{i}^{D'}(\mf)$ (Definition \ref{defximf} (2)) for $i\in I'\cap I_{\mW,\mf}$ 
(Definition \ref{defindsub} (1))
as an ideal sheaf of $\dvr_{D_{i}}$ by Remark \ref{remximf} (2)
and denote by $V(\Image \xi^{D'}_{i}(\mf))$
the closed subset of $D_{i}$ defined by $\Image \xi^{D'}_{i}(\mf)$.
We define a closed subset $B_{I'',\mf}^{D'}\subset D_{I''}\cap Z_{\mf}$ (\ref{defdipp}) for $I''\subset I'$ (possibly $I'=\emptyset$)
to be $D_{I''}\cap \bigcup_{i\in (I-I')\cap I_{\mW,\mf}}D_{i}$
if $I'\subset I_{\mT,\mf}$
and to be the closure of 
$D_{I''}\cap \bigcap_{i'\in I''\cap I_{\mW,\mf}}V(\Image \xi_{i'}^{D'}(\mf))-
\bigcup_{i\in I'-I''}D_{i}$ in $D_{I''}\cap Z_{\mf}$ if $I''\cap I_{\mW,\mf}\neq \emptyset$.

We define a closed subset $B_{\mf}^{D'}\subset Z_{\mf}$
to be the union of $B_{I'',\mf}^{D'}$ for all non-empty subsets $I''\subset I'$. 

\item We define a closed subset $E_{I'',\mf}^{D'}\subset D_{I''}\cap Z_{\mf}$ for $I''\subset I'$ (possibly $I'=\emptyset$)
to be the union of the irreducible components of $B_{I'',\mf}^{D'}$
whose codimensions in $D_{I''}=\bigcap_{i\in I''}D_{i}$ are $0$.

We define a closed subset $E_{\mf}^{D'}\subset Z_{\mf}$
to be the union of $E_{I'',\mf}^{D'}$ for all (non-empty) subsets $I''\subset I'$.
\end{enumerate}
\end{defn}

\noindent
The closed subsets $B_{I'',\mf}^{D'}$, $B_{\mf}^{D'}$, $E_{I'',\mf}^{D'}$, and $E_{\mf}^{D'}$ of $Z_{\mf}$ defined in Definition \ref{defbcf} play key roles in Subsections \ref{sscandss}, \ref{sshicc},
and Section \ref{ssrk1}.

\begin{rem}
\label{remdefbcf}
Let the assumptions and the notation be as in Definition \ref{defbcf}.
We regard $V(\Image \xi^{D'}_{i}(\mf))$
as the closed subscheme of $D_{i}$ defined by 
the ideal sheaf $\Image \xi^{D'}_{i}(\mf)\subset \dvr_{D_{i}}$
for $i\in I'\cap I_{\mW,\mf}$.
\begin{enumerate}
\item We regard $\dvr_{X}(R_{\mf}^{D'})|_{Z_{\mf}^{1/p}}\cdot \cform^{D'}(\mf)$ (Definitions \ref{defconddiv}, \ref{defcform}) locally as a direct summand of
$\Omega_{X}^{1}(\log D')|_{Z_{\mf}^{1/p}}$ of rank $1$
by the $\log$-$D'$-cleanliness of the ramification of $\mf$ along $D$.
Then the restriction of $D_{I''}\cap \bigcap_{i'\in I''\cap I_{\mW,\mf}}V(\Image \xi_{i'}^{D'}(\mf))$ to $D_{I''}-\bigcup_{i\in I'-I''}D_{i}$ is the 
maximum subscheme $Y$ of $D_{I''}\cap Z_{\mf}-\bigcup_{i\in I'-I''}D_{i}$ 
such that
the kernel of the composition 
\begin{equation}
\label{compinvsh}
\Omega_{X}^{1}|_{Y^{1/p}}\rightarrow \Omega_{X}^{1}(\log D')|_{Y^{1/p}}
\rightarrow \Omega_{X}^{1}(\log D')|_{Y^{1/p}}/(\dvr_{X}(R_{\mf}^{D'})|_{Y^{1/p}}\cdot \cform^{D'}(\mf)|_{Y^{1/p}}) 
\end{equation}
of canonical morphisms is locally free of rank $1+\sharp I''$.
The closed subset $E_{I'',\mf}^{D'}$ of $D_{I''}\cap Z_{\mf}$
for $I''\subset I'$ is empty if $I''\subset I_{\mT,\mf}$ and is
the union of the irreducible components 
$Y'$ of $D_{I''}$ such that 
the image of $\xi_{i}^{D'}(\mf)$ is $0$ in $\dvr_{Y'}$ for every $i\in I''\cap I_{\mW,\mf}$ if $I''\cap I_{\mW,\mf}\neq \emptyset$.
\item By Lemma \ref{lemcfatx} (2), the codimension of $V(\Image\xi_{i}^{D'}(\mf))$ in $X$
for $i\in I'\cap I_{\mW,\mf}$ such that $V(\Image \xi_{i}^{D'}(\mf))\neq \emptyset$ is $2$ 
if and only if $i\in I_{\mI,\mf}$ (Definition \ref{defindsub} (2)).
By loc.\ cit., we have $V(\Image\xi_{i}^{D'}(\mf))=D_{i}$ for $i\in I'\cap I_{\mII,\mf}$
(Definition \ref{defindsub} (2)).
Hence a divisor $D_{i}$ for $i\in I'\cap I_{\mW,\mf}$ is contained in $B_{\{i\},\mf}^{D'}$ if and only if $i\in I'\cap I_{\mII,\mf}$, namely
we have $E_{\{i\},\mf}^{D'}\neq \emptyset$ for $i\in I'\cap I_{\mW,\mf}$ if and only if 
$i\in I'\cap I_{\mII,\mf}$.
If $i\in I'\cap I_{\mII,\mf}$, then we have $E_{\{i\},\mf}^{D'}=D_{i}$.
If $I'\cap I_{\mII,\mf}=\emptyset$ and if $I''$ is a non-empty subset of $I'$, then 
the codimension of $B_{I'',\mf}^{D'}$ in $X$ is $\ge 2$,
and so is that of $E_{I'',\mf}^{D'}$.

\item If $E_{\mf}^{D'}=\emptyset$, then we have $I'\cap I_{\mII,\mf}=\emptyset$
by (2).
If $X$ is a surface, namely is purely of dimension $2$,
and if the ramification of $\mf$ is $\log$-$D'$-clean along $D$,
then the converse is also true.
Actually, suppose that $X$ is a surface and 
that $I'\cap I_{\mII,\mf}=\emptyset$.
Then we locally have $E_{\mf}^{D'}=E_{I'',\mf}^{D'}$ for a subset $I''\subset I'$ of cardinality $2$ by (2).
If $x\in E_{\mf}^{D'}$ is a closed point, then we have $I_{x}=I''$ (\ref{defix}), 
and the image of $\xi_{i}^{D'}(\mf)$ in $k(x)$ is $0$ for every $i\in I''\cap I_{\mW,\mf}$
by (1).
Hence the ramification of $\mf$ is not $\log$-$D'$-clean along $D$ at $x$
by Lemma \ref{lemcfatx} (2) and Lemma \ref{lemeqtologdpcl} (1),
which contradicts the $\log$-$D'$-cleanliness of the ramification of $\mf$ along $D$ at $x$.

\item Suppose that $I'\cap I_{\mII,\mf}=\emptyset$.
Then each closed point $x$ of $Z_{\mf}$ satisfying the condition (1) in Lemma \ref{lemcldegeq} 
is a point on $B_{\{i\},\mf}^{D'}$ for some $i\in I_{\mT,\mf,x}\subset I'$ (\ref{eachindatx}).
Each closed point $x$ of $Z_{\mf}$ satisfying the condition (2) in Lemma \ref{lemcldegeq} is also a point on $B_{I'\cap I_{x},\mf}^{D'}$, 
since the image of $\xi_{i}^{D'}(\mf)$ in $k(x)$ is $0$ for every $i\in I'\cap I_{\mII,\mf}$ (\ref{eachindatx}) by Lemma \ref{lemcfatx} (2).
Let $x$ be a closed point of $Z_{\mf}$ satisfying the condition (3) in Lemma \ref{lemcldegeq}.
If $I_{\mI,\mf,x}-I'\neq \emptyset$ (\ref{eachindatx}), then the image of $\xi_{i}^{D'}(\mf)$ in $k(x)$ is $0$
for every $i\in I'\cap I_{\mW,\mf,x}$ (\ref{eachindatx}) by Lemma \ref{lemcfatx} (1) applied to the case where $(I',I'')$ is $(I,I')$,
and hence we have $x\in B_{I'\cap I_{x},\mf}^{D'}$,
where $I'\cap I_{x}\supset I'\cap I_{\mI,\mf,x}\neq \emptyset$ by Lemma \ref{lemcfclatx} (1).
If $I_{\mI,\mf,x}\subset I'$, then we have $x\in D_{i}\cap D_{i'}$ for any pair $(i,i')$ 
of distinct elements of $I_{\mI,\mf,x}\subset I'$.
Thus the set of points on $X$
where the ramification of $\mf$ is not $\log$-$\emptyset$-clean along $D$
is contained in 
$B_{\mf}^{D'}\cup \bigcup_{\substack{i,i'\in I'\cap I_{\mI,\mf}\\ i\neq i'}}(D_{i}\cap D_{i'})$
by Remarks \ref{remlogdpcl} (1) and (2) and by Lemma \ref{lemcldegeq}.
\end{enumerate}
\end{rem}

We prove that we can eliminate $E_{\mf}^{D_{\mI,\mf}\cup D_{\mT,\mf}}$
by taking blow-ups of $X$ along closed subschemes of $D_{\mI,\mf}\cup D_{\mT,\mf}$, where $D_{\mI,\mf}$ and $D_{\mT,\mf}$ are as in (\ref{defd*mf}).

\begin{lem}
\label{lemblupbth}
Suppose that $D$ has simple normal crossings. 
Assume that the ramification of $\mf$ is $\log$-$D_{\mI,\mf}\cup D_{\mT,\mf}$-clean (\ref{defd*mf}) along $D$.
Let $Y$ be an irreducible component of 
$E_{I'',\mf}^{D_{\mI,\mf}\cup D_{\mT,\mf}}$ (Definition \ref{defbcf} (2))
for a non-empty subset $I''\subset I_{\mI,\mf}\cup I_{\mT,\mf}$ (Definition \ref{defindsub}). 
We regard $Y$ as a closed subscheme of $X$ with the reduced subscheme structure.
Let $f\colon X'\rightarrow X$ be the blow-up of $X$ along $Y$.
Let $D_{i}'$ be the proper transform of $D_{i}$ 
for $i\in I=\{1,2,\ldots,r\}$ 
and let $D_{0}'=f^{-1}(Y)$ be the exceptional divisor.
We identify the index set of the irreducible components
of $f^{*}D$ with $I\cup\{0\}$.
Then we have the following:
\begin{enumerate}
\item $0\in I_{\mII,f^{*}\mf}$ (Definition \ref{defindsub} (2)), $I_{\mI,f^{*}\mf}=I_{\mI,\mf}$,
and $I_{\mT,f^{*}\mf}=I_{\mT,\mf}$.
\item The ramification of $f^{*}\mf$ is $\log$-$
D_{\mI,f^{*}\mf}\cup D_{\mT,f^{*}\mf}$-clean along $(f^{*}D)_{\red}$. 
\item Let $y\in D_{0}'\cap \bigcup_{i\in I_{\mI,\mf}\cup I_{\mT,\mf}}D_{i}'$ be a closed point and 
let $z\in Y\subset X$ denote the image of $y$ by $f$.
Then the cardinality $\sharp (I_{\mI,f^{*}\mf,y}\cup I_{\mT,f^{*}\mf,y})$
of $I_{\mI,f^{*}\mf,y}\cup I_{\mT,f^{*}\mf,y}$ (\ref{eachindatx}) is 
equal to $\sharp(I_{\mI,\mf,z}\cup I_{\mT,\mf,z})-1$.
If $y\in E_{I_{\mI,f^{*}\mf,y}\cup I_{\mT,f^{*}\mf,y},f^{*}\mf}^{D_{\mI,f^{*}\mf}\cup D_{\mT,f^{*}\mf}}$, 
where $I_{\mI,f^{*}\mf,y}$ and $I_{\mT,f^{*}\mf,y}$ are as in (\ref{eachindatx}),
then we have $z\in E_{I_{\mI,\mf,z}\cup I_{\mT,\mf,z},\mf}^{D_{\mI,\mf}\cup D_{\mT,\mf}}$.
If the cardinality of $I_{z}$ (\ref{defix}) is $2$, then the cardinality of 
$I_{\mI,f^{*}\mf,y}\cup I_{\mT,f^{*}\mf,y}$ is $1$
and we have $y\notin E^{D_{\mI,f^{*}\mf}\cup D_{\mT,f^{*}\mf}}_{f^{*}\mf}$.
\end{enumerate}
\end{lem}

\begin{proof}
(1) Since the image of $\xi_{i}^{D_{\mI,\mf}\cup D_{\mT,\mf}}(\mf)$ in $\dvr_{Y^{1/p}}$ is $0$ for every $i\in I''\cap I_{\mW,\mf}$
(Definition \ref{defindsub} (1)), 
the image of $\xi_{0}^{(f^{*}(D_{\mI,\mf}\cup D_{\mT,\mf}))_{\red}}(f^{*}\mf)$ (in $\dvr_{D_{0}'^{1/p}}$) is $0$ 
by Proposition \ref{propblupcl} (3) applied to the case where $(I',I'')$
is $(I_{\mI,\mf}\cup I_{\mT,\mf},I'')$. 
Therefore the first assertion follows from Lemma \ref{lemcfatx} (2).

Since $X'$ is isomorphic to $X$ outside $Y$ and since
the image of $\xi_{i}^{D_{\mI,\mf}\cup D_{\mT,\mf}}(\mf)$ (in $\dvr_{D_{i}^{1/p}}$) is not $0$ for $i\in I_{\mI,\mf}$ by Lemma
\ref{lemcfatx} (2),
the image of $\xi_{i}^{(f^{*}(D_{\mI,\mf}\cup D_{\mT,\mf}))_{\red}}(f^{*}\mf)$ (in $\dvr_{D_{i}'^{1/p}}$) is not $0$ for $i\in I_{\mI,\mf}$
by Proposition \ref{propblupcl} (3).
Hence we have $I_{\mI,f^{*}\mf}=I_{\mI,\mf}$ by the first assertion and Lemma \ref{lemcfatx} (2).
By the second equality in Proposition \ref{propblupcl} (1), we have $I_{\mT,f^{*}\mf}=I_{\mT,\mf}$.

(2) Since the ramification of $f^{*}\mf$ is $\log$-$(f^{*}(D_{\mI,\mf}\cup D_{\mT,\mf}))_{\red}$-clean along $(f^{*}D)_{\red}$ by Proposition \ref{propblupcl} (2),
the assertion follows from Proposition \ref{propdpdtdi}.

(3) The first assertion holds by (1).
Suppose that $y\in E_{I_{\mI,f^{*}\mf,y}\cup I_{\mT,f^{*}\mf,y},f^{*}\mf}^{D_{\mI,f^{*}\mf}\cup D_{\mT,f^{*}\mf}}$.
Let $Z'$ be the irreducible component of $D_{I_{\mI,f^{*}\mf,y}\cup I_{\mT,f^{*}\mf,y}}$
(\ref{defdipp}) such that $y\in Z'$.
We regard $Z'$ as a closed subscheme of $X'$ with the reduced subscheme structure.
Since we have $D_{\mI,f^{*}\mf}\cup D_{\mT,f^{*}\mf}\subset (f^{*}(D_{\mI,\mf}\cup D_{\mT,\mf}))_{\red}$,
the image of $\xi_{i}^{(f^{*}(D_{\mI,\mf}\cup D_{\mT,\mf}))_{\red}}(f^{*}\mf)$ in $\dvr_{Z'^{1/p}}$ is $0$
for every $i\in I_{\mI,f^{*}\mf,y}$ by Lemma \ref{lemclrelfirf} and Remark \ref{remdefbcf} (1).
Let $Z$ be the irreducible component of $D_{I_{\mI,\mf,z}\cup I_{\mT,\mf,z}}$
such that $z\in Z$.
Then we have $Z\subset Y$ and we regard $Z$ as a closed subscheme of $X$ with the reduced subscheme structure.
Then the image of 
$\xi_{i}^{D_{\mI,\mf}\cup D_{\mT,\mf}}(\mf)$ in $\dvr_{Z^{1/p}}$ is $0$ for every $i\in I_{\mI,f^{*}\mf,y}$ 
by Proposition \ref{propblupcl} (3).
Since the image of $\xi_{0}^{(f^{*}(D_{\mI,\mf}\cup D_{\mT,\mf}))_{\red}}(f^{*}\mf)$ 
(in $\dvr_{D_{0}'^{1/p}}$) is $0$ as is seen in the proof of (1),
the image of $\xi_{i}^{D_{\mI,\mf}\cup D_{\mT,\mf}}(\mf)$
in $\dvr_{Z^{1/p}}$ is also $0$ if $i\in I_{\mI,\mf,z}-I_{\mI,f^{*}\mf,y}$
by Proposition \ref{propblupcl} (3).
Thus the second assertion holds.

Suppose that the cardinality of $I_{z}$ is $2$.
Then we have $I_{z}=I''$,
since $I''$ is a subset of $I_{z}$ of cardinality $\ge 2$ by Remark \ref{remdefbcf} (2).
Hence the cardinality of $I_{\mI,f^{*}\mf,y}\cup I_{\mT,f^{*}\mf,y}$ is $1$ by the first assertion,
and we have $y\notin E_{f^{*}\mf}^{D_{\mI,f^{*}\mf}\cup D_{\mT,f^{*}\mf}}$
by Remark \ref{remdefbcf} (2).
\end{proof}

\begin{prop}
\label{propblupcth}
Suppose that $D$ has simple normal crossings
and that the ramification of $\mf$ is $\log$-$D'$-clean along $D$
for some union $D'$ of irreducible components of $D$.
If $E_{\mf}^{D'}\neq \emptyset$ (Definition \ref{defbcf} (2)), then there exist successive blow-ups 
\begin{equation}
f: X^{\prime}=X_{s}\xrightarrow{f_{s}} X_{s-1}\xrightarrow{f_{s-1}} \cdots \xrightarrow{f_{1}} X_{0}=X \notag
\end{equation}
satisfying the following three conditions:
\begin{enumerate}
\item $f_{1}$ (resp.\ $f_{i}$ for $i=2,3,\ldots,s$) is the blow-up along an irreducible component $Y_{0}$ of
an intersection of irreducible components of $D_{\mI,\mf}\cup D_{\mT,\mf}$ (\ref{defd*mf})
(resp.\ an irreducible component $Y_{i-1}$ of 
an intersection of irreducible components of $D_{\mI,(f_{1}\circ f_{2}\circ \cdots \circ f_{i-1})^{*}\mf}\cup D_{\mT,(f_{1}\circ f_{2}\circ \cdots \circ f_{i-1})^{*}\mf}$),
where $Y_{0}$ (resp.\ $Y_{i-1}$) is regarded as a closed subscheme of $X_{0}$
(resp.\ $X_{i-1}$) with the reduced subscheme structure.
\item The ramification of $f^{*}\mf$ is $\log$-$D_{\mI,f^{*}\mf}\cup D_{\mT,f^{*}\mf}$-clean along $(f^{*}D)_{\red}$. 
\item $E_{f^{*}\mf}^{D_{\mI,f^{*}\mf}\cup D_{\mT,f^{*}\mf}}=\emptyset$. 
\end{enumerate}
\end{prop}

\begin{proof}
By Proposition \ref{propdpdtdi}, the ramification of $\mf$ is $\log$-$D_{\mI,\mf}\cup D_{\mT,\mf}$-clean along $D$.
Let $\mathcal{S}_{0}$ be the set of subsets $I''$ of $I_{\mI,\mf}\cup I_{\mT,\mf}$ such that $E_{I'',\mf}^{D_{\mI,\mf}\cup D_{\mT,\mf}}\neq \emptyset$.
We take an element $I''_{0}$ of $\mathcal{S}_{0}$ whose cardinality $r_{0}$
is the largest among the elements of $\mathcal{S}_{0}$.
Let $Y_{0}$ be an irreducible component of $E_{I''_{0},\mf}^{D_{\mI,\mf}\cup D_{\mT,\mf}}$.
We regard $Y_{0}$ as a closed subscheme of $X$ with the reduced subscheme structure.
Let $f_{1}\colon X_{1}\rightarrow X$ be a blow-up of $X$ along
$Y_{0}$. 
Let $D_{i}'$ denote the proper transform of $D_{i}$ for $i\in I''_{0}=\{1,2,\ldots,r_{0}\}$
and let $D_{0}'=f_{1}^{-1}(Y_{0})$ be the exceptional divisor.
By Lemma \ref{lemblupbth} (1), 
every closed point of $D_{0}'-\bigcup_{i\in I_{\mI,\mf}\cup I_{\mT,\mf}}D_{i}'$
is not a point on $E_{f^{*}_{1}\mf}^{D_{\mI,f_{1}^{*}\mf}\cup D_{\mT,f_{1}^{*}\mf}}$.
If $y$ is a closed point of $D_{0}'\cap \bigcup_{i\in I_{\mI,\mf}\cup I_{\mT,\mf}}D_{i}'$
such that $y\in E_{I_{\mI,f_{1}^{*}\mf,y}\cup I_{\mT,f_{1}^{*}\mf,y},f_{1}^{*}\mf}^{D_{\mI,f_{1}^{*}\mf}\cup D_{\mT,f_{1}^{*}\mf}}$,
where $I_{\mI,f_{1}^{*}\mf,y}$ and $I_{\mT,f_{1}^{*}\mf,y}$ are as in (\ref{eachindatx}),
then we have $f_{1}(y)\in E_{I_{\mI,\mf,f_{1}(y)}\cup I_{\mT,\mf,f(y)},\mf}^{D_{\mI,\mf}\cup D_{\mT,\mf}}$
and the cardinality of $I_{\mI,f_{1}^{*}\mf,y}\cup I_{\mT,f_{1}^{*}\mf,y}$ is strictly
less than $r_{0}$ by Lemma \ref{lemblupbth} (3).

We inductively define $f_{i}\colon X_{i}\rightarrow X_{i-1}$ for $i>1$
to be the blow-up of $X_{i-1}$ as follows, if $E_{f^{*}_{1}\mf}^{D_{\mI,f_{1}^{*}\mf}\cup D_{\mT,f_{1}^{*}\mf}}\neq \emptyset$.
Let $\mf_{i-1}=(f_{1}\circ f_{2}\circ \cdots \circ f_{i-1})^{*}\mf$ and
let $\mathcal{S}_{i-1}$ be the set of subsets $I'''$ of 
$I_{\mI,\mf_{i-1}}\cup I_{\mT,\mf_{i-1}}$
such that 
$E_{I''',\mf_{i-1}}^{D_{\mI,\mf_{i-1}}\cup D_{\mT,\mf_{i-1}}}\neq \emptyset$.
We take an element $I''_{i-1}$ of $\mathcal{S}_{i-1}$ whose cardinality $r_{i-1}$ is the largest among the elements of $\mathcal{S}_{i-1}$.
Let $Y_{i-1}$ be an irreducible component of $E_{I''_{i-1},\mf_{i-1}}^{D_{\mI,\mf_{i-1}}\cup D_{\mT,\mf_{i-1}}}$.
We regard $Y_{i-1}$ as a closed subscheme of $X_{i-1}$ with the reduced subscheme structure
and define $f_{i}\colon X_{i}\rightarrow X_{i-1}$ to be the blow-up
of $X_{i-1}$ along $Y_{i-1}$.
By Lemma \ref{lemblupbth} (3),
for $i\ge 1$ such that $\mathcal{S}_{i}\neq \emptyset$,
there exists an integer $h_{i} > 0$ satisfying one of the following two conditions:
\begin{itemize}
\item[(a)] $\mathcal{S}_{i+h_{i}}\neq \emptyset$ and $r_{i}>r_{i+h_{i}}$.
\item[(b)] $\mathcal{S}_{i+h}\neq \emptyset$ for $h=1,2,\ldots, h_{i}-1$ and $\mathcal{S}_{i+h_{i}}=\emptyset$.
\end{itemize}
Hence there exists an integer $s>0$ 
such that $\mathcal{S}_{i}\neq \emptyset$ for $i=0,1,\ldots,s-1$ and that
$\mathcal{S}_{s}=\emptyset$.
For such $s$, we put $f=f_{1}\circ f_{2}\circ \cdots \circ f_{s}$.
Then $f$ satisfies the conditions (1) and (3).
By Lemma \ref{lemblupbth} (2), 
the composition $f$ also satisfies the condition (2). 
\end{proof}

At last, we prepare a lemma for a computation of the characteristic cycle $CC(j_{!}\mf)$
in the case where $X$ is a surface in Subsection \ref{sscompccvar}.

\begin{lem}
\label{lemcleanatx}
Suppose that $X$ is a surface, namely is purely of dimension $2$,
and that $D$ has simple normal crossings.
Let $I=\{1,2\}$ and let $I'=\{1\}\subset I$.
Let $D'=D_{1}$ and let $x\in D_{1}\cap D_{2}$.
Assume that the ramification of $\mf$ is $\log$-$D'$-clean along $D$
at $x$, that $I'\subset I_{\mI,\mf}$ (Definition \ref{defindsub} (2)), 
and that $x\in B_{\mf}^{D'}$ (Definition \ref{defbcf} (1)).
We regard the image of $\xi_{1}^{D'}(\mf)$ (Definition \ref{defximf} (2)) as an ideal sheaf of $\dvr_{D_{1}}$ by Remark \ref{remximf} (2).
Then the image of $\xi^{D'}_{1}(\mf)$ defines a divisor
$\ord^{\emptyset}(\mf;x,D_{1})(D_{2}\cap D_{1})$ on $D_{1}$ in a neighborhood of $x$
and we have $\ord^{\emptyset}(\mf;x,D_{2})=1$.
With the notation in (\ref{cformatxdp}) in Lemma \ref{lemcfatx},
we have $\beta_{2}\neq 0$ in $k(x)$.
If $2\in I_{\mI,\mf}$, then the ramification of $\mf$ is $\log$-$D$-clean along $D$ at $x$.
\end{lem}

\begin{proof}
By Lemma \ref{lemcfclatx} (1), we have $I_{\mT,\mf,x}\subset I'$ (\ref{eachindatx}).
Since $I'=\{1\}\subset I_{\mI,\mf}$, the first assertion holds by Lemma \ref{lemcfatx} (1) applied to the case
where $(I',I'')$ is $(I',\emptyset)$.
Let the notation be as in (\ref{cformatxdp}) in Lemma \ref{lemcfatx}.
Since $x\in B_{\mf}^{D'}=B_{\{1\},\mf}^{D'}$, we have $\alpha_{1}=0$ in $k(x)$. 
Therefore we have $\beta_{2}\neq 0$ in $k(x)$ by the $\log$-$D'$-cleanliness of the ramification of $\mf$ along $D$ at $x$ and
by Lemma \ref{lemeqtologdpcl} (1).
Since $\alpha_{1}=0$ and $\beta_{2}\neq 0$ in $k(x)$, 
we have $\ord^{\emptyset}(\mf;x,D_{2})=1$
by applying Lemma \ref{lemcfatx} (1) to the case where $(I',I'')$ is $(I',\emptyset)$.
If $2\in I_{\mI,\mf}$, then the ramification of $\mf$ is $\log$-$D$-clean along $D$ at $x$ by Lemma \ref{lemcfatx} applied to the case where $(I',I'')$ is $(I,I')$ and by Lemma \ref{lemeqtologdpcl} (1),
since $\beta_{2}\neq 0$ in $k(x)$.
\end{proof}

\section{Cleanliness and the direct image}
\label{secdil}
Let $D'$ be a divisor on $X$ with simple normal crossings
contained in $D$.
We prove in Proposition \ref{mainpropsec} (2) that the $\log$-$D'$-cleanliness of the ramification of $\mf$ along $D$ is a sufficient condition for
the canonical morphism $j_{!}\mf \rightarrow Rj_{*}\mf$ to be 
an isomorphism when $I_{\mT,\mf}=\emptyset$ (Definition \ref{defindsub} (1)).
Here $j\colon U\rightarrow X$ denotes the canonical open immersion
as is explained in Conventions.

\subsection{Dilatations}
\label{dil}
This subsection is devoted to preparation for the next subsection.
Through this subsection, we assume that $X$ is separated over $k$.
Let $I'\subset I$ be a subset such that $D'=\bigcup_{i\in I'}D_{i}$ has simple normal crossings.
Let $TX(\log D')=\Spec S^{\bullet}\Omega^{1}_{X}(\log D')$ be the logarithmic tangent bundle of $X$ with logarithmic poles along $D'$.
Let $R=\sum_{i\in I}n_{i}D_{i}$ with $n_{i}\in \mathbf{Z}_{\ge 0}$ for $i\in I'$
and with $n_{i}\in \mathbf{Z}_{\ge 1}$ for $i\in I-I'$.
Let $Z=\Supp (R+D'-D)$ denote the support of $R+D'-D$.
We construct a smooth scheme $P^{(D'\subset D,R)}$ over $k$
such that $Z^{(D'\subset D,R)}=P^{(D'\subset D,R)}\times_{X}Z$
is canonically isomorphic to $TX(\log D')(-R)\times_{X}Z$.

\begin{defn}[{\cite[1.3]{ascl}}]
\label{deflogprod}
Assume that $X$ is separated over $k$.
Let $I'\subset I$ be a subset such that $D'=\bigcup_{i\in I'}D_{i}$ has simple normal crossings.
We define the {\it log product} $X*_{k,D'}X$ {\it with respect to} $D'$
to be the complement of the proper transform of $(D'\times_{k} X)\cup (X\times_{k}D')$ in the blow-up $(X\times_{k}X)'$ of $X\times_{k}X$ along the closed subscheme of $X\times_{k}X$ defined by the product of ideal sheaves $\mathcal{I}_{D_{i}\times_{k}D_{i}}$ defining $D_{i}\times_{k}D_{i}\subset X\times_{k}X$ for all $i\in I'$.
\end{defn}

\begin{rem}
\label{remlogprod}
Let the assumption and the notation be as in Definition \ref{deflogprod}.
\begin{enumerate}
\item If $D'=\emptyset$, 
then we have $X*_{k,D'}X=X\times_{k}X$.
\item Let $p_{i}\colon X*_{k,D'}X\rightarrow X$ for $i=1,2$ be the composition
\begin{equation}
\label{defpiproj}
p_{i}\colon X*_{k,D'}X\rightarrow X\times_{k}X\xrightarrow{\pr_{i}} X
\end{equation}
of the canonical morphism and the $i$-th projection $\pr_{i}$.
By \cite[Remark 5.23 (iii)]{ascl}, the composition $p_{i}$ is a smooth morphism for $i=1,2$.
Consequently, the log product $X*_{k,D'}X$ is a smooth
scheme over $k$.

\item By \cite[Remark 5.23 (i)]{ascl}, the diagonal $\delta\colon X\rightarrow X\times_{k}X$ is uniquely lifted to a morphism 
\begin{equation}
\tilde{\delta}_{D'}\colon X\rightarrow X*_{k,D'}X. \notag
\end{equation}
We call the lift $\tilde{\delta}_{D'}$ of $\delta$ the {\it log diagonal}.
By \cite[Remark 5.23 (iii)]{ascl}, the log diagonal $\tilde{\delta}_{D'}\colon X\rightarrow X*_{k,D'}X$ 
is a regular closed immersion whose normal bundle $T_{X}(X*_{k,D'}X)$ is canonically isomorphic to $TX(\log D')$.
\end{enumerate}
\end{rem}

Let $I'\subset I$ be a subset such that $D'=\bigcup_{i\in I'}D_{i}$ has simple normal crossings.
If $X=\Spec A$ for a $k$-algebra $A$ and if $D_{i}$ is defined by $t_{i}\in A$
for $i\in I'=\{1,2,\ldots, r'\}$, then we have $X*_{k,D'}X=\Spec \tilde{A}_{D'}$ for
\begin{equation}
\tilde{A}_{D'}= (A\otimes_{k}A)\left[\left(\frac{1\otimes t_{1}}{t_{1}\otimes 1}\right)^{\pm 1}, \left(\frac{1\otimes t_{2}}{t_{2}\otimes 1}\right)^{\pm 1}, \ldots,
\left(\frac{1\otimes t_{r'}}{t_{r'}\otimes 1}\right)^{\pm 1}\right]. \notag
\end{equation}
If $I_{\delta}\subset A\otimes A$ is the defining ideal of the diagonal $\delta\colon X\rightarrow X\times_{k}X$, then
the log diagonal $\tilde{\delta}_{D'}\colon X\rightarrow X*_{k,D'}X$ (Remark \ref{remlogprod} (3)) is defined by the ideal
\begin{equation} 
\left(I_{\delta},\frac{1\otimes t_{1}}{t_{1}\otimes 1}-1,\frac{1\otimes t_{2}}{t_{2}\otimes 1}-1,\ldots,\frac{1\otimes t_{r'}}{t_{r'}\otimes 1}-1\right)\subset \tilde{A}_{D}. \notag
\end{equation}

\begin{defn}[{\cite[Definition 1.12]{sacot}}]
\label{defdilt}
Let 
\begin{equation}
\label{diltdefreg}
E\rightarrow Q\leftarrow Y
\end{equation} 
be closed immersions of schemes. 
Assume that $E$ is a Cartier divisor on $Q$.
Let $Q'\rightarrow Q$ be the blow-up of $Q$ along the intersection
$E\cap Y=E\times_{Q}Y$.
We define the {\it dilatation} $Q^{(E\cdot Y)}$ {\it of $Q$ with respect to} $(E,Y)$, or the {\it dilatation $Q^{(E\cdot Y)}$ of $Q$ with respect to} (\ref{diltdefreg}), to be
the complement of the proper transform of $E$ in $Q'$.
\end{defn}

\begin{rem}
\label{remdilt}
Let the notation be as in Definition \ref{defdilt}.
\begin{enumerate}
\item If $Y=Q$ in (\ref{diltdefreg}), then the canonical morphism 
$Q^{(E\cdot Y)}\rightarrow Q$ is an isomorphism
(\cite[a remark after Definition 1.12]{sacot}).
\item If the intersection $E_{Y}=E\cap Y$ in $Q$ is a
Cartier divisor on $Y$, then the closed immersion $Y\rightarrow Q$
is uniquely lifted to an immersion $Y\rightarrow Q^{(E\cdot Y)}$
(\cite[a remark after Example 1.13]{sacot}).
If $E_{Y}$ is a Cartier divisor on $Y$ and if  $Y\rightarrow Q$ is a regular closed immersion,
then the immersion $Y\rightarrow Q^{(E\cdot Y)}$ is a regular closed immersion with normal bundle canonically isomorphic to $T_{Y}Q(-E_{Y})$ (\cite[Lemma 1.14.2]{sacot}). 
\end{enumerate}
\end{rem}

Suppose that $X$ is separated over $k$.
Let $I'\subset I$ be a subset such that $D'=\bigcup_{i\in I'}D_{i}$ has simple normal crossings.
We construct a scheme $P^{(D'\subset D)}$ by applying the construction of dilatation to $X*_{k,D'}X$ and prove several properties of $P^{(D'\subset D)}$.
We put $D''=\bigcup_{i\in I-I'}D_{i}$.
By Remark \ref{remlogprod} (2), 
the pull-back $p_{i}^{*}D''$ of $D''$ by $p_{i}$ (\ref{defpiproj}) is a Cartier divisor on $X*_{k,D'}X$ for $i=1,2$. 
We have regular closed immersions
\begin{equation}
\label{regimmspdpd}
p_{i}^{*}D''\rightarrow X*_{k,D'}X\xleftarrow{\tilde{\delta}_{D'}} X,
\end{equation}
where the first arrow is the canonical closed immersion.
We define a scheme $P^{(D'\subset D)}$ to be the intersection of 
the dilatations $(X*_{k,D'}X)^{(p_{i}^{*}D''\cdot X)}$ of $X*_{k,D'}X$
with respect to (\ref{regimmspdpd}) for $i=1,2$.
The dilatation $(X*_{k,D'}X)^{(p_{i}^{*}D''\cdot X)}$ for $i=1,2$ 
is an open subscheme of the blow-up $(X*_{k,D'}X)'$ of $X*_{k,D'}X$ along $D''=p_{i}^{*}D''\cap X\subset X \subset X*_{k,D'}X$, and
the scheme $P^{(D'\subset D)}$ is an open subscheme of
$(X*_{k,D'}X)^{(p_{i}^{*}D''\cdot X)}$ for $i=1,2$.

Suppose that $X=\Spec A$ and that $D_{i}$ is defined by $t_{i}\in A$ for $i\in I=\{1,2,\ldots,r\}$.
If $I'=\{1,2,\ldots,r'\}$ with $r'\le r$ and if $I_{\delta}$ denotes the ideal of $A\otimes_{k}A$ defining the diagonal
$\delta\colon X\rightarrow X\times_{k}X$,
then we have $P^{(D'\subset D)}=\Spec A^{(D'\subset D)}$ for 
\begin{align}
A^{(D'\subset D)} =\tilde{A}_{D'}&\left[ \frac{I_{\delta}}{\prod_{i\in I-I'}t_{i}\otimes 1},\left(\frac{1\otimes t_{r'+1}}{t_{r'+1}\otimes 1}\right)^{\pm 1},
\ldots, \left(\frac{1\otimes t_{r}}{t_{r}\otimes 1}\right)^{\pm 1}, \right. \notag\\
&\qquad \left. \frac{(1\otimes t_{1})-(t_{1}\otimes 1)}{(t_{1}\otimes 1)(\prod_{i\in I-I'}t_{i}\otimes 1)},
\ldots,
\frac{(1\otimes t_{r'})-(t_{r'}\otimes 1)}{(t_{r'}\otimes 1)(\prod_{i\in I-I'}t_{i}\otimes 1)}\right]. \notag
\end{align}

Let $p'_{i}\colon P^{(D'\subset D)}\rightarrow X$ for $i=1,2$ denote the composition
\begin{equation}
\label{pdptox}
p'_{i}\colon P^{(D'\subset D)} 
\rightarrow X*_{k,D'}X \xrightarrow{p_{i}} X  
\end{equation}
of the canonical morphism and the projection $p_{i}$ (\ref{defpiproj}).

\begin{lem}
\label{lempdptox}
Suppose that $X$ is separated over $k$.
Let $I'\subset I$ be a subset such that $D'=\bigcup_{i\in I'}D_{i}$ has simple normal crossings. 
Then we have the following for the scheme $P^{(D'\subset D)}$ defined above:
\begin{enumerate}
\item The projection $p_{i}'\colon P^{(D'\subset D)}\rightarrow X$ (\ref{pdptox}) is a smooth morphism for $i=1,2$.
\item The inverse image $D_{i'}^{(D'\subset D)}$ of $D_{i'}$ for $i'\in I$
by $p'_{i}$ for $i=1,2$ is a smooth divisor on $P^{(D'\subset D)}$ and
is independent of the choice of $i=1,2$.
The complement $P^{(D'\subset D)}-D^{(D'\subset D)}$ of
$D^{(D'\subset D)}=\bigcup_{i'\in I}D_{i'}^{(D'\subset D)}$ is $U\times_{k}U$.
If $D$ has simple normal crossings, then $D^{(D'\subset D)}$ is a divisor on $P^{(D'\subset D)}$ with simple normal crossings.
\item The log diagonal $\tilde{\delta}_{D'}\colon X\rightarrow X*_{k,D'}X$ 
(Remark \ref{remlogprod} (3)) is uniquely lifted to a regular closed immersion $\delta^{(D'\subset D)}\colon X\rightarrow P^{(D'\subset D)}$. The normal bundle of $\delta^{(D'\subset D)}$ is canonically isomorphic to $T_{X}(X*_{k,D'}X)(-D'')$, where $D''=\bigcup_{i'\in I-I'}D_{i'}$.
\end{enumerate}
\end{lem}

\begin{proof}
(1) Let $\tilde{\delta}_{D'}\colon X\rightarrow X*_{k,D'}X$ denote the log diagonal 
(Remark \ref{remlogprod} (3)).
We consider the commutative diagram 
\begin{equation}
\label{diagdiltlem}
\xymatrix{
p_{i}^{*}D'' \ar[r] \ar[d] \ar@{}[dr] | {\square} & X*_{k,D'}X \ar[d]^-{p_{i}} & X \ar[l]_-{\tilde{\delta}_{D'}} \ar[d]^-{\id_{X}} \\
D'' \ar[r] & X & X, \ar[l]^-{\id_{X}}
} 
\end{equation}
where the first lower horizontal arrow is the canonical regular closed immersion
and the left square is cartesian, for $i=1,2$.
Then the upper line in (\ref{diagdiltlem}) is (\ref{regimmspdpd}).
By Remark \ref{remdilt} (1), the canonical morphism
\begin{equation}
\label{morxkdpxproj}
X^{(D''\cdot X)}\rightarrow X 
\end{equation}
from the dilatation of $X$ with respect to the lower line in (\ref{diagdiltlem}) is an isomorphism.
Since the left square in (\ref{diagdiltlem}) is cartesian,
the projection $p_{i}\colon X*_{k,D'}X\rightarrow X$ is canonically lifted
to a morphism
\begin{equation}
\label{liftpdpdlem}
(X*_{k,D'}X)^{(p_{i}^{*}D''\cdot X)}\rightarrow X^{(D''\cdot X)}
\end{equation}
by the universality of blow-ups.
Since the intersection $p_{i}^{*}D''\cap X$ in $X*_{k,D'}X$ is the Cartier divisor $D''$ on $X$
and since the log diagonal $\tilde{\delta}_{D'}$ is a regular closed immersion by Remark \ref{remlogprod} (3),
the lift (\ref{liftpdpdlem}) of $p_{i}$ is smooth by \cite[Corollary 1.17.1]{sacot}.
Since the projection $p_{i}'\colon P^{(D'\subset D)}\rightarrow X$ is the composition of
the canonical morphisms (\ref{morxkdpxproj}) and (\ref{liftpdpdlem})
and the canonical open immersion $P^{(D'\subset D)}\rightarrow (X*_{k,D'}X)^{(p_{i}^{*}D''\cdot X)}$, the projection $p_{i}'$ is a smooth morphism for $i=1,2$. 

(2) Since the log product $X*_{k,D'}X$ is the complement of the proper transform of $(D'\times_{k}X)\cup(X\times_{k}D')$ in the blow-up of $X\times_{k}X$
along the union of $D_{i}\times_{k}D_{i}$ for all $i\in I'$,
the inverse image $D_{i'}^{(D'\subset D)}$ of $D_{i'}$ for $i'\in I'$ by $p_{i}'$ is independent of
the choice of $i=1,2$.
Since the scheme $P^{(D'\subset D)}$ is the complement of the proper transform of $p_{1}^{*}D''\cup p_{2}^{*}D''$ in the blow-up $(X*_{k,D'}X)'$ of $X*_{k,D'}X$ along $D''(\subset X)\subset X*_{k,D'}X$,
the independence on the choice of $p_{i}'$ for the inverse image $D_{i'}^{(D'\subset D)}$ of $D_{i'}$ by $p_{i}'$ also follows for $i'\in I-I'$.
The regularity of $D_{i'}^{(D'\subset D)}$
for $i'\in I$ follows from (1). 
By the constructions of $X*_{k,D'}X$ and $P^{(D'\subset D)}$ as the complements of the proper transforms of $(D'\times_{k}X)\cup (X\times_{k}D')$
and $p_{1}^{*}D''\cup p_{2}^{*}D''$ in the blow-ups of $X\times_{k}X$ and $X*_{k,D'}X$
along the union of $D_{i}\times_{k}D_{i}$ for all $i\in I'$ and $D''\subset X*_{k,D'}X$, respectively, 
we have $P^{(D'\subset D)}-D^{(D'\subset D)}=U\times_{k}U$.
If $D$ has simple normal crossings, then so does $D^{(D'\subset D)}$ by (1).

(3) Since the intersection $p_{i}^{*}D''\cap X$ in $X*_{k,D'}X$ is the Cartier divisor $D''$ on $X$
and since $\tilde{\delta}_{D'}$ is a regular closed immersion by Remark \ref{remlogprod} (3),
the log diagonal $\tilde{\delta}_{D'}\colon X\rightarrow X*_{k,D'}X$ 
is uniquely lifted to a regular closed immersion
\begin{equation}
\label{reglifttoxxdila}
X\rightarrow (X*_{k,D'}X)^{(p_{i}^{*}D''\cdot X)}
\end{equation} 
for $i=1,2$ by Remark \ref{remdilt} (2).
By loc.\ cit., the normal bundle of (\ref{reglifttoxxdila}) 
is canonically isomorphic to $T_{X}(X*_{k,D'}X)(-D'')$ for $i=1,2$.
By the universality of blow-ups, the composition of 
the lift (\ref{reglifttoxxdila}) of $\tilde{\delta}_{D'}$ and the canonical open immersion
$(X*_{k,D'}X)^{(p_{i}^{*}D''\cdot X)}\rightarrow (X*_{k,D'}X)'$
to the blow-up of $X*_{k,D'}X$ along $D''\subset X*_{k,D'}X$ is independent of the choice of $i=1,2$. 
Since the scheme $P^{(D'\subset D)}$ is the intersection of the open subschemes $(X*_{k,D'}X)^{(p_{i}^{*}D''\cdot X)}$ of $(X*_{k,D'}X)'$ for $i=1,2$, the assertions hold.
\end{proof}

We construct the scheme $P^{(D'\subset D,R)}$ for $R=\sum_{i'\in I}n_{i'}D_{i'}$ where $n_{i'}\in \mathbf{Z}_{\ge 0}$ for $i'\in I'$ and $n_{i'}\in \mathbf{Z}_{\ge 1}$ for $i'\in I-I'$.
We put
\begin{equation}
R^{(D'\subset D)}=\sum_{i'\in I'}n_{i'}D_{i'}^{(D'\subset D)}+\sum_{i'\in I-I'}(n_{i'}-1)D_{i'}^{(D'\subset D)}. \notag
\end{equation}
Here $D_{i'}^{(D'\subset D)}$ denotes the inverse image of $D_{i'}$ by $p_{i}'$ (\ref{pdptox}), which is a smooth divisor on $P^{(D'\subset D)}$
and is independent of the choice of $i=1,2$ by Lemma \ref{lempdptox} (2).
We consider the regular closed immersions
\begin{equation}
\label{regimmspdpdr}
R^{(D'\subset D)}
\rightarrow P^{(D'\subset D)}\xleftarrow{\delta^{(D'\subset D)}} X, 
\end{equation}
where the first arrow is the canonical closed immersion and $\delta^{(D'\subset D)}$ is the lift of the log diagonal $\tilde{\delta}_{D'}$ (Remark \ref{remlogprod} (3))
constructed in Lemma \ref{lempdptox} (3).
We define a scheme $P^{(D'\subset D,R)}$ to be the dilatation of $P^{(D'\subset D)}$ with respect to (\ref{regimmspdpdr}).
If $D$ has simple normal crossings and if $I'=\emptyset$,
then $P^{(D'\subset D)}$ and $P^{(D'\subset D,R)}$ are none other than 
$P_{2}^{(D)}$ and $P_{2}^{(R)}$ in \cite[Subsection 2.1]{sacot}, respectively.

Suppose that $X=\Spec A$ for a $k$-algebra $A$ and that $D_{i}$ is defined by $t_{i}\in A$
for $i\in I=\{1,2,\ldots,r\}$.
We put $I'=\{1,2,\ldots, r'\}$, where $r'\le r$.
If $I_{\delta}$ denotes the ideal of $A$ defining the diagonal $\delta \colon X\rightarrow X\times_{k}X$,
then we have
$P^{(D'\subset D,R)} =\Spec A^{(D'\subset D, R)}$ for
\begin{align}
A&^{(D'\subset D, R)}\notag  \\
&=A^{(D'\subset D)}\left[
\frac{I_{\delta}}{\prod_{i\in I}t_{i}^{n_{i}}\otimes 1}, \frac{(1\otimes t_{1})-(t_{1}\otimes 1)}{(t_{1}\otimes 1)(\prod_{i\in I}t_{i}^{n_{i}}\otimes 1)}, \ldots, \frac{(1\otimes t_{r'})-(t_{r'}\otimes 1)}{(t_{r'}\otimes 1)(\prod_{i\in I}t_{i}^{n_{i}}\otimes 1)}
\right]. \notag
\end{align}

We prove that $P^{(D'\subset D,R)}$ has the desired property,
namely that the fiber product $Z^{(D'\subset D,R)}=P^{(D'\subset D,R)}\times_{X}Z$
is canonically isomorphic to $TX(\log D')\times_{X}Z$.
Let \begin{equation}
\label{canpdpdrtopdpd}
P^{(D\subset D,R)}\rightarrow P^{(D'\subset D)} 
\end{equation}
be the canonical morphism and
let $p''_{i}\colon P^{(D'\subset D,R)}\rightarrow X$ for $i=1,2$ denote the composition
\begin{equation}
\label{compqi}
p''_{i}\colon P^{(D'\subset D,R)}\xrightarrow{(\ref{canpdpdrtopdpd})} P^{(D'\subset D)} 
\xrightarrow{p_{i}'}X
\end{equation}
of the morphisms (\ref{canpdpdrtopdpd}) and $p'_{i}$ (\ref{pdptox}).

\begin{lem}
\label{lempdpdr}
Suppose that $X$ is separated over $k$.
Let $I'\subset I$ be a subset such that $D'=\bigcup_{i'\in I}D_{i'}$ has simple normal crossings. 
Let $R=\sum_{i\in I}n_{i}D_{i}$ with $n_{i}\in \mathbf{Z}_{\ge 0}$ for $i\in I'$ and $n_{i}\in \mathbf{Z}_{\ge 1}$ for $i\in I-I'$.
We put $Z=\Supp(R+D'-D)$.
Then we have the following for the scheme $P^{(D'\subset D,R)}$ defined above:
\begin{enumerate}
\item The projection $p_{i}''\colon P^{(D'\subset D,R)}\rightarrow X$ (\ref{compqi}) is a smooth morphism for $i=1,2$.
\item The inverse image $D_{i'}^{(D'\subset D,R)}$
of $D_{i'}$ for $i'\in I$ by the projection $p_{i}''$ for $i=1,2$ is a smooth divisor on $P^{(D'\subset D,R)}$ and is independent of the choice of $i=1,2$.
The complement $P^{(D'\subset D, R)}-D^{(D'\subset D,R)}$
of $D^{(D'\subset D,R)}=\bigcup_{i'\in I}D_{i'}^{(D'\subset D,R)}$ is
$U\times_{k}U$.
If $D$ has simple normal crossings, then 
$D^{(D'\subset D,R)}$ is a divisor on $P^{(D'\subset D,R)}$
with simple normal crossings.
\item Let $Z^{(D'\subset D,R)}$ be the inverse image of $Z$ by $p_{i}''$,
which is independent of the choice of $i=1,2$ by (2).
Then there is a canonical isomorphism
\begin{equation}
Z^{(D'\subset D,R)} \xrightarrow{\sim} TX(\log D')(-R)\times_{X}Z. \notag
\end{equation} 
\end{enumerate}
\end{lem}

\begin{proof}
(1) We prove the assertions similarly as the proof of Lemma \ref{lempdptox} (1).
Let $\delta^{(D'\subset D)}\colon X\rightarrow P^{(D'\subset D)}$ be the regular closed immersion 
constructed in Lemma \ref{lempdptox} (3).
We put 
\begin{equation}
R'=R+D'-D=\sum_{i'\in I'} n_{i'}D_{i'}+\sum_{i'\in I-I'}(n_{i'}-1)D_{i'}. \notag
\end{equation} 
For $i=1,2$, we consider the commutative diagram
\begin{equation}
\label{diagpdpdrpf}
\xymatrix{
R^{(D'\subset D)} \ar[r] \ar[d] \ar@{}[dr] | {\square} & P^{(D'\subset D)} \ar[d]^-{p_{i}'} & X \ar[l]_-{\delta^{(D'\subset D)}}  \ar[d]^-{\id_{X}} \\
R' \ar[r] & X & X \ar[l]^-{\id_{X}},
}
\end{equation}
where the first lower horizontal arrow is the canonical regular closed immersion
and the left square is cartesian.
Then the upper line in (\ref{diagpdpdrpf}) is (\ref{regimmspdpdr}).
By Remark \ref{remdilt} (1), the canonical morphism
\begin{equation}
\label{morxrpxx}
X^{(R'\cdot X)}\rightarrow X
\end{equation}
from the dilatation of $X$ with respect to the lower line in (\ref{diagpdpdrpf}) is an isomorphism.
Since the left square in (\ref{diagpdpdrpf}) is cartesian,
the projection $p_{i}'\colon P^{(D'\subset D)}\rightarrow X$ is 
canonically lifted to a morphism
\begin{equation}
\label{morpdpdrxrpx}
P^{(D'\subset D,R)}\rightarrow X^{(R'\cdot X)}
\end{equation}
by the universality of blow-ups.
Since the intersection of $R^{(D'\subset D)}$ and $X$ in $P^{(D'\subset D)}$ is
identified with the Cartier divisor $R'$ on $X$
and since the morphism $\delta^{(D'\subset D)}$ is a regular closed immersion by Lemma \ref{lempdptox} (3),
the lift (\ref{morpdpdrxrpx}) of $p_{i}'$ is 
smooth by \cite[Corollary 1.17.1]{sacot}.
Since the projection $p_{i}''\colon P^{(D'\subset D)}\rightarrow X$ is the composition of the isomorphism (\ref{morxrpxx}) and
the morphism (\ref{morpdpdrxrpx}),
the projection $p_{i}''$ is a smooth morphism for $i=1,2$.

(2) By (1), the inverse image $D_{i'}^{(D'\subset D,R)}$ of $D_{i'}$ for $i'\in I$ by $p_{i}''$ for $i=1,2$ is a smooth divisor on $P^{(D'\subset D,R)}$.
If $D$ has simple normal crossings, then so does $D^{(D'\subset D,R)}$ by loc.\ cit..

We prove the independence of the choice of $i=1,2$ for the inverse image $D_{i'}^{(D'\subset D,R)}$ for $i'\in I$ by $p_{i}''$
and the equality $P^{(D'\subset D,R)}-D^{(D'\subset D,R)}=U\times_{k}U$.
By the construction of $p_{i}''$ as the composition of
the canonical morphism (\ref{canpdpdrtopdpd}) and $p_{i}'$, 
the inverse image $D_{i'}^{(D'\subset D,R)}$ of $D_{i'}$ by $p_{i}''$ is the inverse image of the divisor $D_{i'}^{(D'\subset D)}$ on $P^{(D'\subset D)}$ constructed in Lemma \ref{lempdptox} (2)
by the morphism (\ref{canpdpdrtopdpd}).
Hence the independence follows and
the complement $P^{(D'\subset D,R)}-D^{(D'\subset D,R)}$ is
equal to the inverse image of $P^{(D'\subset D)}-D^{(D'\subset D)}$
by the morphism (\ref{canpdpdrtopdpd}).
Since $P^{(D'\subset D,R)}$ is an open subscheme of the blow-up of
$P^{(D'\subset D)}$ along $R'=R^{(D'\subset D)}\cap X\subset P^{(D'\subset D)}$
and since the support of $R'\subset P^{(D'\subset D)}$ is contained in $D^{(D'\subset D)}$,
the morphism (\ref{canpdpdrtopdpd}) induces an isomorphism
$P^{(D'\subset D,R)}-D^{(D'\subset D,R)}\rightarrow P^{(D'\subset D)}-D^{(D'\subset D)}=U\times_{k}U$, where the equality
holds by Lemma \ref{lempdptox} (2).

(3) By Remark \ref{remlogprod} (3), the normal bundle $T_{X}(X*_{k,D'}X)$ of the
log diagonal $\tilde{\delta}_{D'}$ (Remark \ref{remlogprod} (3)) is canonically isomorphic to $TX(\log D')$.
By Lemma \ref{lempdptox} (3), 
the normal bundle $T_{X}P^{(D'\subset D)}$ of the lift $\delta^{(D'\subset D)}$ of $\tilde{\delta}_{D'}$ constructed in loc.\ cit.\ 
is canonically isomorphic
to $T_{X}(X*_{k,D'}X)(-D'')$, where $D''=\bigcup_{i'\in I-I'}D_{i'}$.
Therefore the normal bundle $T_{X}P^{(D'\subset D)}$ of $\delta^{(D'\subset D)}$ is canonically isomorphic to $TX(\log D')(-D'')$.
Since the intersection of $R^{(D'\subset D)}$ and $X$ in $P^{(D'\subset D)}$ is
the Cartier divisor $R'$ on $X$
and since the morphism $\delta^{(D'\subset D)}$ is a regular closed immersion by Lemma \ref{lempdptox} (3),
there is a canonical isomorphism
\begin{equation}
\label{isomdnb}
P^{(D'\subset D,R)}\times_{P^{(D'\subset D)}}R^{(D'\subset D)}\xrightarrow{\sim}
TX(\log D')(-(R'+D''))\times_{X}R'
\end{equation}
by \cite[Lemma 1.14.1]{sacot}.
Since $R=R'+D''$ 
and since $R^{(D'\subset D)}$ is the fiber product of $R'$ and $P^{(D'\subset D)}$ over $X$, 
we obtain the desired isomorphism by taking the base change of (\ref{isomdnb}) by the canonical morphism $Z=\Supp R'\rightarrow R'$.
\end{proof}

\subsection{Cleanliness and the direct image}
\label{seccldim}

Let $I'\subset I$ be a subset such that $D'=\bigcup_{i\in I'}D_{i}$ has simple normal crossings.
Suppose that $X$ is separated over $k$.
Let $R_{\mf}^{D'}$ be as in Definition \ref{defconddiv} (1).
We consider the smooth scheme $P^{(D'\subset D,R^{D'}_{\mf})}$ over $k$ 
constructed in the previous subsection.
By Lemma \ref{lempdpdr} (2), the inverse image 
$D^{(D'\subset D,R_{\mf}^{D'})}$ of $D$ 
by the projection $p_{i}''\colon P^{(D'\subset D,R_{\mf}^{D'})}\rightarrow X$ (\ref{compqi}) 
is independent of the choice of  $i=1,2$ 
and is a divisor on $P^{(D'\subset D,R_{\mf}^{D'})}$ with smooth irreducible components.
Since the complement $P^{(D'\subset D,R_{\mf}^{D'})}-D^{(D'\subset D,R_{\mf}^{D'})}$ is $U\times_{k}U$ by Lemma \ref{lempdpdr} (2), 
the canonical open immersion $j\times j\colon U\times_{k}U\rightarrow X\times_{k}X$
is uniquely lifted to the canonical open immersion 
\begin{equation}
\label{jdr}
j_{D'}^{(R_{\mf}^{D'})}\colon U\times_{k}U=P^{(D'\subset D,R_{\mf}^{D'})}-D^{(D'\subset D,R_{\mf}^{D'})}\rightarrow P^{(D'\subset D,R_{\mf}^{D'})}. 
\end{equation}

Let $L_{i}$ for $i\in I$ be the local field
at a generic point of $D_{i}^{(D'\subset D,R_{\mf}^{D'})}$.
Then the projections $p_{1}'',p_{2}''\colon P^{(D'\subset D,R_{\mf}^{D'})}\rightarrow X$ induces the morphisms 
\begin{equation}
u_{s,i},v_{s,i}\colon W_{s}(K_{i})\rightarrow W_{s}(L_{i}), \notag
\end{equation} 
respectively,
for $s\ge 0$ and for $i\in I$.
Let $\pr_{i}\colon U\times_{k}U\rightarrow U$ denote the $i$-th projection
for $i=1,2$
and let 
\begin{equation}
\label{defsheafh}
\mh=\mh om(\pr_{2}^{*}\mf,\pr_{1}^{*}\mf).
\end{equation}
Let $\varphi\colon \pi_{1}^{\ab}(U\times_{k}U)\rightarrow \Lambda^{\times}$ denote the character corresponding 
to the smooth sheaf $\mh$
of $\Lambda$-modules of rank 1 on $U\times_{k}U$
and let $\varphi|_{L_{i}}\colon G_{L_{i}}\rightarrow \mathbf{Q}/\mathbf{Z}$ 
for $i\in I$ denote the composition
\begin{equation}
\varphi|_{L_{i}}\colon G_{L_{i}}\rightarrow G_{L_{i}}^{\ab}=\pi_{1}^{\ab}(\Spec L_{i})\rightarrow \pi^{ab}(U\times_{k}U)\xrightarrow{\varphi} \notag
\Lambda^{\times} 
\end{equation}
of the canonical morphisms and $\varphi$.
If the order $n$ of $\chi|_{K_{i}}$ for $i\in I$ is prime to $p$,
then $\chi|_{K_{i}}$ is defined by a Kummer equation $t^{n}=a$ for the image $a$ in $K_{i}$ of a section of $\dvr_{U}$, 
and $\varphi|_{L_{i}}$ is defined by $t^{n}=v_{1,i}(a)/u_{1,i}(a)$.
If the order of $\chi|_{K_{i}}$ for $i\in I$ is $p^{s}$ for $s\ge 0$,
then $\chi|_{K_{i}}$ is defined by an Artin-Schreier-Witt equation $F(t)-t=a$ 
for the image $a$ in $W_{s}(K_{i})$ of a section of $W_{s}(\dvr_{U})$,
and $\varphi|_{L_{i}}$ is defined by $F(t)-t=v_{s,i}(a)-u_{s,i}(a)$.
Here $F$ denotes the Frobenius.

Let $F_{L_{i}}$ denote the residue field of $L_{i}$ and
$\mathfrak{p}_{i}$ the generic point of $D_{i}$ for $i\in I$.
We put $R_{\mf}^{D'}=\sum_{i\in I}n_{i}D_{i}$.
Then we have $n_{i}>0$ for $i\in I'\cap I_{\mW,\mf}$ (Definition \ref{defindsub} (1))
and $n_{i}>1$ for $i\in (I-I')\cap I_{\mW,\mf}$.
By Lemma \ref{lempdpdr} (3), there are canonical injections
\begin{equation}
t_{i,n_{i}}\colon \grlog_{n_{i}}\Omega^{1}_{K_{i}}
=\Omega_{X}^{1}(\log D')(R_{\mf}^{D'})|_{D_{i},\mathfrak{p}_{i}}\rightarrow F_{L_{i}} \notag
\end{equation}
for $i\in I'\cap I_{\mW,\mf}$ 
and 
\begin{equation}
t_{i,n_{i}}'\colon \gr_{n_{i}}\Omega^{1}_{K_{i}}\otimes_{F_{K_{i}}}F_{K_{i}}^{1/p}
=\Omega^{1}_{X}(\log D')(R_{\mf}^{D'})|_{D_{i},\mathfrak{p}_{i}}\otimes_{F_{K_{i}}}F_{K_{i}}^{1/p}\rightarrow F_{L_{i}}
\otimes_{F_{K_{i}}}F_{K_{i}}^{1/p} \notag
\end{equation}
for $i\in (I-I')\cap I_{\mW,\mf}$.
Recall that the refined Swan conductor $\rsw(\chi|_{K_{i}})$ (Definition \ref{defrefcond} (1)) 
and the characteristic form $\cform(\chi|_{K_{i}})$ (Definition \ref{defrefcond} (2)) 
for $i\in I_{\mW,\mf}$ are elements of 
$\grlog_{n_{i}}\Omega^{1}_{K_{i}}$ and $\gr_{n_{i}}\Omega^{1}_{K_{i}}\otimes_{F_{K_{i}}}F_{K_{i}}^{1/p}$, respectively.

\begin{lem}[cf.\ {\cite[Lemma 4.2.3]{aschar}}]
\label{mainlem}
Suppose that $X$ is separated over $k$.
Let the notation be as above.
\begin{enumerate}
\item The character $\varphi|_{L_{i}}$  
for $i\in I_{\mW,\mf}$ is unramified and is of order $p$.
\item Let $i\in I'\cap I_{\mW,\mf}$.
Then the character $\varphi|_{L_{i}}$ is 
defined by the Artin-Schreier equation $t^{p}-t=-t_{i,n_{i}}(\rsw(\chi|_{K_{i}}))$, 
that is, $\varphi|_{L_{i}}$ is the image of $-t_{i,n_{i}}(\rsw(\chi|_{K_{i}}))$ by 
the canonical morphism
\begin{equation}
\label{flih1}
F_{L_{i}}\rightarrow 
H^{1}(F_{L_{i}}, \BZ/p\BZ)\subset H^{1}(L_{i},\mathbf{Z}/p\mathbf{Z}) 
\end{equation}
of Artin-Schreier theory.
\item Let $i\in (I-I')\cap I_{\mW,\mf}$.
Then the character $\varphi|_{L_{i}}$ is  
defined by $t^{p}-t=-t'_{i,n_{i}}(\cform(\chi|_{K_{i}}))$, 
namely, $\varphi|_{L_{i}}$ is the image of an element of $F_{L_{i}}$  
that is equal to $-t'_{i,n_{i}}(\cform(\chi|_{K_{i}}))-(b^{p}-b)$ in $F_{L_{i}}\otimes_{F_{K_{i}}}F_{K_{i}}^{1/p}$
for some $b\in F_{L_{i}}\otimes_{F_{K_{i}}}F_{K_{i}}^{1/p}$ by the morphism (\ref{flih1}).
\end{enumerate}
\end{lem}

\begin{proof}
Since the assertions are local on the generic point of $D_{i}$ for $i\in I_{\mW,\mf}$, we may assume that $D$ is irreducible
and that $I=I_{\mW,\mf}$.
Then the assertion (1) for $i\in I'\cap I_{\mW,\mf}$ and
the assertion (2) are nothing but \cite[Lemma 4.2.3]{aschar}.
Since $\chi$ is the sum of the $p$-part and prime to $p$-part of $\chi$,
we may assume that the order of $\chi$ is a power of $p$  to prove the remaining assertions
by \cite[Lemmas 2.7 (ii), (iii)]{yafil}.
Then the assertion (3) follows from \cite[Proposition 2.12 (ii)]{yafil},
and the assertion (1) for $i\in (I-I')\cap I_{\mW,\mf}$ follows.
\end{proof}

\begin{lem}
\label{mainlemsec}
Suppose that $X$ is separated over $k$.
Let $I'\subset I$ be a subset such that $D'=\bigcup_{i\in I'}D_{i}$ has simple normal crossings.
Then the following hold:
\begin{enumerate}
\item The direct image $j_{D' *}^{(R_{\mf}^{D'})}\mh$ of $\mh$ (\ref{defsheafh}) by $j_{D'}^{(R_{\mf}^{D'})}\colon U\times_{k}U\rightarrow P^{(D'\subset D,R_{\mf}^{D'})}$ (\ref{jdr}) is a smooth sheaf 
on $P^{(D'\subset D,R_{\mf}^{D'})}$.
\item We regard the $\log$-$D'$-characteristic form $\cform^{D'}(\mf)$ as a global section of $\dvr_{Z_{\mf}^{(D'\subset D,R_{\mf}^{D'})}}\otimes _{\dvr_{Z_{\mf}}}\dvr_{Z_{\mf}^{1/p}}$ by Lemma \ref{lempdpdr} (3), 
where $Z_{\mf}$ (Definition \ref{defconddiv} (2)) is equal to the support
of $R_{\mf}^{D'}+D'-D$ by Lemma \ref{lemzf}
and we put $Z_{\mf}^{(D'\subset D,R_{\mf}^{D'})}=P^{(D'\subset D,R_{\mf}^{D'})}\times_{X}Z_{\mf}$.
If $Z_{\mf}$ is non-empty, then the restriction
$j_{D' *}^{(R_{\mf}^{D'})}\mh|_{Z_{\mf}^{(D'\subset D,R_{\mf}^{D'})}}$
of $j_{D' *}^{(R_{\mf}^{D'})}\mh$
to $Z_{\mf}^{(D'\subset D,R_{\mf}^{D'})}$ is defined by the Artin-Schreier
equation $t^{p}-t=-\cform^{D'}(\mf)$.

\item Let $\bar{x}$ be a geometric point on $Z_{\mf}$.
Let $P^{(D'\subset D,R_{\mf}^{D'})}_{\bar{x}}$ denote the fiber product
of $P^{(D'\subset D,R_{\mf}^{D'})}$ and $\bar{x}$ over $X\times_{k}X$,
where $\bar{x}$ is regarded as a geometric point on $X\times_{k}X$ by the diagonal $X\rightarrow X\times_{k}X$.
If the ramification of $\mf$ is $\log$-$D'$-clean along $D$ at the image $x\in X$ of $\bar{x}$,
then $j_{D' *}^{(R_{\mf}^{D'})}\mh|_{P_{\bar{x}}^{(D'\subset D,R_{\mf}^{D'})}}$ is not constant.
\end{enumerate}
\end{lem}

\begin{proof}
(1) By the purity of Zariski-Nagata, it is sufficient to prove that
$\varphi|_{L_{i}}$ is unramified for $i\in I$.
Since the assertion is local on the generic point of $D_{i}$ for $i\in I$, we may assume that $D$ is irreducible.
If $I=I_{\mW,\mf}$, then the assertion follows from Lemma \ref{mainlem} (1).
If $I=I_{\mT,\mf}=I'$, then the assertion is nothing but the last assertion in
\cite[Lemma 2.32 (i)]{yacc}.

Suppose that $I=I_{\mT,\mf}$ and that $I'=\emptyset$.
Then, since $\chi$ is decomposed to the sum of its $p$-part and its prime to $p$-part, we may assume that $\chi$ is of order a power of $p$ or of order prime to $p$.
If the order of $\chi$ is prime to $p$, then the assertion follows from
\cite[Lemma 2.7 (i), (ii)]{yafil}.
If the order of $\chi$ is a power of $p$,
then the assertion follows from 
\cite[Lemma 2.9 (i)]{yafil}.

(2) By Remark \ref{remcfnormal} (3), the $\log$-$D'$-characteristic form $\cform^{D'}(\mf)$ is the unique global section of $\dvr_{Z_{\mf}^{(D'\subset D,R_{\mf}^{D'})}}\otimes _{\dvr_{Z_{\mf}}}\dvr_{Z_{\mf}^{1/p}}$ whose 
germ at the generic point of $D_{i}^{1/p}$ for $i\in I_{\mW,\mf}$ is
$\rsw(\chi|_{K_{i}})$ if $i\in I'\cap I_{\mW,\mf}$ and is $\cform(\chi|_{K_{i}})$ if $i\in (I- I')\cap I_{\mW,\mf}$.
Thus the assertion follows from Lemma \ref{mainlem} (2) and (3).

(3) By Lemma \ref{lempdpdr} (2),  we can identify $P_{\bar{x}}^{(D'\subset D,R_{\mf}^{D'})}$ with the fiber product of 
$Z_{\mf}^{(D'\subset D,R^{D'}_{\mf})}$ and $\bar{x}$ over $X\times_{k}X$.
Then the restriction $j_{D' *}^{(R_{\mf}^{D'})}\mh|_{P_{\bar{x}}^{(D'\subset D,R_{\mf}^{D'})}}$ is defined by $t^{p}-t=-\cform^{D'}(\mf)$ by (2), and is not constant by the $\log$-$D'$-cleanliness of the ramification of $\mf$ along $D$ at $x$.
\end{proof}

\begin{prop}
\label{mainpropsec}
Let $I'\subset I$ be a subset such that $D'=\bigcup_{i\in I'}D_{i}$ has simple normal crossings. 
\begin{enumerate}
\item Suppose that the ramification of $\mf$ is $\log$-$D'$-clean along $D$ at a point $x$ on $Z_{\mf}$ (Definition \ref{defconddiv} (2)).
Let $\bar{x}$ be a geometric point lying above $x$.
Then we have $(Rj_{*}\mf)_{\bar{x}}=0$.
\item Suppose that $I_{\mT,\mf}=\emptyset$ (Definition \ref{defindsub} (1)) 
(or equivalently $Z_{\mf}=D$)
and that the ramification of $\mf$ is $\log$-$D'$-clean along $D$.
Then the canonical morphism $j_{!}\mf\rightarrow Rj_{*}\mf$ is
an isomorphism.
\end{enumerate}
\end{prop}

\begin{proof}
Since the assertion is local, we may assume that $X$ is separated over $k$.
Then the assertion (1) follows similarly as the proof of \cite[Proposition 2.34]{yacc}
by replacing $(X*_{k}X)^{(R_{\chi})}$, $P^{(R_{\chi})}_{\bar{x}}$, and \cite[Lemma 2.33]{yacc} by $P^{(D'\subset D,R_{\mf}^{D'})}$, $P_{\bar{x}}^{(D'\subset D,R_{\mf}^{D'})}$,
and Lemma \ref{mainlemsec} (3), respectively. 
The assertion (2) is a consequence of (1),
since $Z_{\mf}=X-U$ when $I_{\mT,\mf}=\emptyset$.
\end{proof}

\section{Singular support and characteristic cycle}
\label{ss}

Let $\mg$ be a constructible complex of $\Lambda$-modules on $X$,
namely $\mg$ has the constructible cohomology sheaves $\mathcal{H}^{i}(\mg)$ that are $0$ except for finitely many $i$.
We recall the singular support $SS(\mg)$ of $\mg$ introduced  
in \cite{be} in Subsection \ref{ssss} and 
the characteristic cycle $CC(\mg)$ of $\mg$ introduced 
in \cite{sacc} in Subsection \ref{sscc}.

\subsection{Singular support}
\label{ssss}
For a vector bundle $E$ on $X$, we say that a closed subset $C\subset E$ is a closed {\it conical} subset if $C$ is stable under the $\mathbf{G}_{m}$-action.
The {\it base} of a closed conical subset $C\subset E$ is the intersection of $C$ and the zero section of $E$.
For a morphism $h\colon W\rightarrow X$ of smooth schemes over $k$,
let 
\begin{equation}
\label{defdh}
dh\colon T^{*}X\times_{X}W\rightarrow T^{*}W 
\end{equation}
denote the morphism defined by $h$.

\begin{defn}
\label{defctr}
Let $C\subset T^{*}X$ be a closed conical subset.
\begin{enumerate}
\item Let $h\colon W\rightarrow X$ be a morphism of smooth schemes over $k$.
We say that $h$ is $C$-{\it transversal} at $w\in W$ if the intersection 
\begin{equation}
(h^{*}C\cap dh^{-1}(T^{*}_{W}W))\times_{W}w \subset T^{*}X\times_{X}w,\notag
\end{equation} 
where $dh$ is as in (\ref{defdh}),
is contained in the zero section $T^{*}_{X}X\times_{X}w$.

We say that $h$ is $C$-{\it transversal}
if $h$ is $C$-transversal at every $w\in W$, namely if the intersection
\begin{equation}
h^{*}C\cap dh^{-1}(T^{*}_{W}W) \subset T^{*}X \notag
\end{equation} 
is contained in the zero section $T^{*}_{X}X\times_{X}W$.

If $h$ is $C$-transversal, then we define a closed conical subset 
\begin{equation}
h^{\circ}C\subset T^{*}W \notag
\end{equation}
to be the image of $h^{*}C=C\times_{X}W$ by $dh$.

\item Let $f\colon X\rightarrow Y$ be a morphism of smooth schemes over $k$.
We say that $f$ is $C$-{\it transversal} at $x\in X$ if the inverse image 
\begin{equation}
df^{-1}(C)\times_{X}x \subset T^{*}Y\times_{Y}x, \notag
\end{equation} 
where $df\colon T^{*}Y\times_{Y}X\rightarrow T^{*}X$ is as in (\ref{defdh}), 
is contained in the zero section $T^{*}_{Y}Y\times_{Y}x$.

We say that $f$ is $C$-{\it transversal}
if $f$ is $C$-transversal at every $x\in X$, namely the inverse image
\begin{equation}
df^{-1}(C) \subset T^{*}Y \notag
\end{equation} 
is contained in the zero section $T^{*}_{Y}Y\times_{Y}X$.

\item Let $(h,f)$ be a pair of morphisms $h\colon W\rightarrow X$ and
$f\colon W\rightarrow Y$ of smooth schemes over $k$.
We say that
$(h,f)$ is $C$-transversal if $h$ is $C$-transversal
and $f$ is $h^{\circ}C$-transversal.
\end{enumerate}
\end{defn}

\begin{rem}
\label{remctr}
\begin{enumerate}
\item In Definition \ref{defctr} (1), if the morphism $h\colon W\rightarrow X$ is $C$-transversal, then
the restriction $dh|_{h^{*}C}\colon h^{*}C\rightarrow T^{*}W$ of $dh$ to $h^{*}C=C\times_{X}W\subset T^{*}X\times_{X}W$ is finite by \cite[Lemma 1.2 (ii)]{be}.
Hence the image $h^{\circ}C\subset T^{*}W$ of $h^{*}C$ by $dh$
is a closed conical subset of $T^{*}W$
and,
if further $h^{*}C$ is purely of dimension $e$,
then so is $h^{\circ}C$.

If the morphism $h\colon W\rightarrow X$ is a smooth morphism over $k$,
then the morphism $dh$ 
is injective.
Thus the smooth morphism $h$ is $C$-transversal and
the closed conical subset $h^{\circ}C\subset T^{*}W$ 
is isomorphic to $h^{*}C\subset T^{*}X\times_{X}W$.
\item 
 For a closed conical subset $C\subset T^{*}X$,
the $C$-transversality of a morphism $h\colon W\rightarrow X$ of smooth schemes over $k$ is an open condition on $W$ by \cite[Lemma 1.2 (i)]{be}.
\item 
Let $C, C'\subset T^{*}X$ be two closed conical subsets.
If $C'\subset C$, then every $C$-transversal morphism $h\colon W\rightarrow X$ is $C'$-transversal and so is every $C$-transversal morphism $f\colon X\rightarrow Y$.
Therefore so is every $C$-transversal pair of morphisms. 
In general, a morphism $h\colon W\rightarrow X$ of smooth schemes over $k$ is
$C\cup C'$-transversal if and only if $h$ is $C$-transversal and is $C'$-transversal.
It also holds for a morphism $f\colon X\rightarrow Y$ of smooth schemes
over $k$ and a pair $(h,f')$ of morphisms $h\colon W\rightarrow X$ and $f'\colon W\rightarrow Y$ of smooth schemes over $k$, respectively.
\end{enumerate}
\end{rem}

\begin{defn}[{\cite[1.2]{be}}]
\label{defpfc}
Let $f\colon X\rightarrow Y$ be a proper morphism to a smooth scheme $Y$ over $k$ and let $C\subset T^{*}X$ be a closed conical subset.
Then we define a closed conical subset 
\begin{equation}
f_{\circ}C\subset T^{*}Y \notag
\end{equation}
to be the image by the projection $\pr_{1}\colon T^{*}Y\times_{Y}X\rightarrow T^{*}Y$
of the inverse image $df^{-1}(C)$ of $C$ by $df\colon T^{*}Y\times_{Y}X\rightarrow T^{*}X$ (\ref{defdh}).
\end{defn}

There is a compatibility of 
the closed conical subset $f_{\circ}C$ of $T^{*}Y$ 
for a proper morphism $f\colon X\rightarrow Y$
and a closed conical subset $C\subset T^{*}X$
with the composition of proper morphisms:

\begin{lem}
\label{lemcompps}
Let $f\colon X\rightarrow Y$ and $g\colon Y\rightarrow Z$ be proper morphisms
of smooth schemes over $k$.
Let $C\subset T^{*}X$ be a closed conical subset.
Then we have $(g\circ f)_{\circ}C=g_{\circ}(f_{\circ}C)$.
\end{lem}

\begin{proof}
We consider the diagram
\begin{equation}
\xymatrix{T^{*}Z\times_{Z}X \ar[r]^-{dg_{X}} \ar[d]_-{\pr} & T^{*}Y\times_{Y}X
\ar[r]^-{df} \ar[d]^-{\pr_{1}} & T^{*}X \\
T^{*}Z\times_{Z}Y \ar[r]_-{dg} \ar[d]_-{\pr_{1}} & T^{*}Y \\
T^{*}Z,} \notag
\end{equation}
where $\pr\colon T^{*}Z\times_{Z}X\rightarrow T^{*}Z\times_{Z}Y$
denotes the projection
and $dg_{X}\colon T^{*}Z\times_{Z}X\rightarrow T^{*}Y\times_{Y}X$
is the base change of $dg$ (\ref{defdh}) by $f$.
Since the square is cartesian, we obtain the desired equality.
\end{proof}

The singular support $SS(\mg)$ of $\mg$ is defined as follows:

\begin{defn}[{\cite[1.3]{be}}]
\label{defss}
Let $\mg$ be a constructible complex of $\Lambda$-modules on $X$
and let $C\subset T^{*}X$ be a closed conical subset.
\begin{enumerate}
\item We say that $\mg$ is {\it micro-supported} on $C$ if $f$ is locally acyclic relative to $h^{*}\mg$ for every $C$-transversal pair
$(h,f)$ of morphisms $h\colon W\rightarrow X$ and $f\colon W\rightarrow Y$
of smooth schemes over $k$.
\item We define the {\it singular support} $SS(\mg)$ of $\mg$ to be the smallest closed conical subset where $\mg$ is micro-supported.
\end{enumerate}
\end{defn}

The existence of the singular support follows from
\cite[Theorem 1.3]{be}. 

\begin{rem}
\label{remmicsup}
Let $\mg$ be a constructible complex of $\Lambda$-modules on $X$.
\begin{enumerate}
\item Since the local acyclicity is an \'{e}tale local condition,
the singular support $SS(\mg)$ of $\mg$ is defined \'{e}tale locally (\cite[Theorem 1.4 (i)]{be}).

\item If $X$ is purely of dimension $d$, then so is
the singular support $SS(\mg)$ of $\mg$ (\cite[Theorem 1.3 (ii)]{be}).

\item If $\mg$ is micro-supported on a closed conical subset $C\subset T^{*}X$, 
then $\mg$ is micro-supported on $C'$ for every closed conical subset $C'\subset T^{*}X$
containing $C$.
Indeed, if $C\subset C'$, 
then every $C'$-transversal pair $(h,f)$ is $C$-transversal by Remark \ref{remctr} (3), and the assertion follows.
Consequently, the complex $\mg$ is micro-supported on a closed conical subset $C\subset T^{*}X$ if and only if 
the singular support $SS(\mg)$ of $\mg$ is contained in $C$.
\end{enumerate}
\end{rem}

Let $D_{c}^{b}(X,\Lambda)$ denote the derived category of constructible complexes of $\Lambda$-modules on $X$.
We recall properties of singular supports.

\begin{lem}[{cf.\ \cite[Lemma 2.1 (iv)]{be}}]
\label{lemdistss}
Let $\mg$, $\mg'$, and $\mg''$ be constructible complexes of $\Lambda$-modules on $X$ and
let $\mg' \rightarrow \mg \rightarrow \mg'' \rightarrow$ be a distinguished triangle in $D_{c}^{b}(X,\Lambda)$.
Then we have 
\begin{equation}
SS(\mg')\subset SS(\mg)\cup SS(\mg''). \notag
\end{equation}
\end{lem}

\begin{proof}
We put $C=SS(\mg)\cup SS(\mg'')$
so that both $\mg$ and $\mg''$ are micro-supported on $C$
by Remark \ref{remmicsup} (3).
Let $(h,f)$ be a $C$-transversal pair
of morphisms $h\colon W\rightarrow X$ and $f\colon W\rightarrow Y$.
We prove that $f$ is locally acyclic relative to $h^{*}\mg'$,
which deduces the assertion.
By \cite[Theorem 1.5]{be}, we may assume that $Y$ is a smooth curve over $k$.
Let $w$ be a point on $W$ and $\bar{w}$ an algebraic geometric point lying above $w$.
We consider the morphism
\begin{equation}
\label{diaglocacyc}
\xymatrix{
h^{\ast}\mg'_{\bar{w}}\ar[d] \ar[r] &h^{\ast}\mg_{\bar{w}} \ar[r] \ar[d] & h^{\ast}\mg^{\prime \prime}_{\bar{w}} \ar[d] \ar[r] & \\
\varphi_{\bar{w}}(h^{\ast}\mg',f) \ar[r] &\varphi_{\bar{w}}(h^{\ast}\mg,f) \ar[r] & \varphi_{\bar{w}} (h^{\ast}\mg^{\prime \prime},f) \ar[r] & 
} 
\end{equation}
of distinguished triangles of complexes on $\bar{w}$.
Here the symbol $\varphi_{\bar{w}}$ denotes the stalk of nearby cycle complex at $\bar{w}$.
Since $f$ is locally acyclic relative to both $h^{*}\mg$ and $h^{*}\mg''$, 
the middle and right vertical arrows in (\ref{diaglocacyc}) are isomorphisms.
Hence the left vertical arrow is an isomorphism,
and $f$ is locally acyclic relative to $h^{*}\mg$.
\end{proof}

\begin{thm}[{\cite[Theorems 1.4 (i), 1.5, Lemmas 2.2 (ii),  2.5 (i)]{be}}]
\label{thmss}
We have the following for a constructible complex $\mg$ of $\Lambda$-modules on $X$:
\begin{enumerate}
\item Let $h\colon W\rightarrow X$ be a smooth morphism.
Then we have $SS(h^{*}\mg)=h^{\circ}SS(\mg)$.
\item Let $i\colon X\rightarrow Y$ be a closed immersion to a smooth scheme $Y$ over $k$.
Then we have $SS(i_{*}\mg)=i_{\circ}SS(\mg)$.
\item Let $f\colon X\rightarrow Y$ be a proper morphism to a smooth scheme $Y$ over $k$.
Then we have $SS(Rf_{*}\mg)\subset f_{\circ}SS(\mg)$.
\end{enumerate}
\end{thm}

Let $h\colon W\rightarrow X$ be a  separated morphism of schemes of finite type over $k$.
We define a morphism
\begin{equation}
\label{morchg}
c_{h,\mg}\colon h^{*}\mg\otimes_{\Lambda}^{L}Rh^{!}\Lambda
\rightarrow Rh^{!}\mg  
\end{equation} 
to be the adjunction of the composition
\begin{equation}
Rh_{!}(h^{*}\mg\otimes_{\Lambda}^{L}Rh^{!}\Lambda)
\rightarrow \mg\otimes_{\Lambda}^{L}Rh_{!}Rh^{!}\Lambda
\xrightarrow{\id\otimes\varepsilon} \mg\otimes_{\Lambda}^{L}\Lambda=\mg, \notag
\end{equation}
where the first arrow is the inverse of the isomorphism of the projection formula
and $\varepsilon\colon Rh_{!}Rh^{!}\Lambda\rightarrow \Lambda$
denotes the adjunction mapping.

\begin{defn}[{\cite[Definition 8.5]{sacc}}]
Let $\mg$ be a constructible complex of $\Lambda$-modules on $X$ and
let $h\colon W\rightarrow X$ be a separated morphism of smooth schemes over $k$.
We say that the morphism $h\colon W\rightarrow X$ is $\mg$-{\it transversal} if the morphism $c_{h,\mg}$ (\ref{morchg}) is an isomorphism.
\end{defn}

\begin{prop}[{\cite[Lemma 1.12]{yacc}, cf.\ \cite[Proposition 8.8.1]{sacc}}]
\label{propjftr}
Let 
\begin{equation}
\label{diagbcjftr}
\xymatrix{
h^{*}V=V\times_{X}W\ar[r]^-{j'} \ar[d] & W \ar[d]^-{h} \\ V \ar[r]_-{j} & X } 
\end{equation}
be a cartesian diagram of smooth schemes over $k$.
Assume that the horizontal arrows in (\ref{diagbcjftr}) are open immersions
and that the morphism $h\colon W\rightarrow X$ is separated over $k$.
Let $\mh$ be a smooth sheaf of $\Lambda$-modules on $V$
such that both of the canonical morphisms
\begin{equation}
j_{!}\mh\rightarrow Rj_{*}\mh \notag
\end{equation}
and
\begin{equation}
j'_{!}h^{*}\mh\rightarrow Rj'_{*}h^{*}\mh \notag
\end{equation}
are isomorphisms. 
Then the separated morphism $h\colon W\rightarrow X$ is $j_{!}\mh$-transversal.
\end{prop}

\begin{prop}[{\cite[Proposition 1.13]{yacc}, cf.\ \cite[Proposition 8.13]{sacc}}]
\label{propgtrmicsup}
Let $\mg$ be a constructible complex of $\Lambda$-modules on $X$ and
let $C\subset T^{*}X$ be a closed conical subset.
Then the following are equivalent:
\begin{enumerate}
\item The complex $\mg$ is micro-supported on $C$.
\item The support $\Supp \mg\subset X$ of $\mg$ is contained in the base $C\cap T^{*}_{X}X\subset X$ of $C$ and
every separated $C$-transversal morphism is $\mg$-transversal.
\end{enumerate}
\end{prop}

\subsection{Characteristic cycle}
\label{sscc}

We first recall the notion that a point is at most a characteristic point
to define characteristic cycles.

\begin{defn}[{\cite[Definition 5.3.1]{sacc}}]
Let $C\subset T^{*}X$ be a closed conical subset.
Let $h\colon W\rightarrow X$ be an \'{e}tale morphism and let $f\colon W\rightarrow Y$ be a morphism to a smooth curve $Y$ over $k$.
Let $w\in W$ be a closed point.
If the restriction $X\xleftarrow{h}W-\{w\}\xrightarrow{f}Y$
of the pair $(h,f)$ to $W-\{w\}$ is
$C$-transversal, then
we say that $w\in W$ is {\it at most an isolated} $C$-{\it characteristic point} of $f$.
\end{defn}

Let $\mg$ be a constructible complex of $\Lambda$-modules on $X$.
Then the characteristic cycle $CC(\mg)$ of $\mg$ is defined as follows:

\begin{defn}[{\cite[Definition 5.10, Theorem 5.18]{sacc}}]
\label{defcc}
Assume that $X$ is purely of dimension $d$.
Let $\mg$ be a constructible complex of $\Lambda$-modules on $X$ and
let $C\subset T^{*}X$ be a closed conical subset of pure dimension $d$
containing the singular support $SS(\mg)$ of $\mg$.
Let $\{C_{a}\}_{a}$ denote the irreducible components of $C$.
We define the {\it characteristic cycle} $CC(\mg)$ 
to be a unique $\mathbf{Z}$-linear combination $A=\sum_{a}m_{a}[C_{a}]
\in Z_{d}(T^{*}X)$, which is independent of the choice of $C$,
satisfying the following Milnor formula (\cite[Theorem 5.9]{sacc}):
If $h\colon W\rightarrow X$ is an \'{e}tale morphism and if $f\colon W\rightarrow Y$ is a morphism to a smooth curve $Y$ over $k$, then we have
\begin{equation}
\label{milnoreq}
-\dimtot \phi_{w}(h^{*}\mg,f)=(h^{*}A,df)_{T^{*}W,w} 
\end{equation}
for every at most isolated characteristic point $w\in W$ of $f$.
In (\ref{milnoreq}), the left-hand side denotes $-\sum_{i}(-1)^{i}\dimtot \phi^{i}_{w}(h^{*}\mg,f)$, where $\dimtot \phi^{i}_{w}(h^{*}\mg,f)$ is the total dimension of the stalk at $w$ of the $i$-th cohomology sheaf of the vanishing cycle complex $\phi(h^{*}\mg ,f)$.
The right-hand side denotes the intersection number supported on the fiber of $T^{*}W$ at $w$ of $h^{*}A$ with the section $df\subset T^{*}W$ defined by the pull-back of a basis of $T^{*}Y$.
\end{defn}

\begin{rem}
\label{remsupcc}
Let $\mg$ be a constructible complex of $\Lambda$-modules on $X$.
\begin{enumerate}
\item[(1)] Since the singular support $SS(\mg)$ of $\mg$ is defined \'{e}tale locally by Remark \ref{remmicsup} (1) and since
the statement of Milnor formula (\cite[Theorem 5.9]{sacc})
is \'{e}tale local, the characteristic cycle $CC(\mg)$ of $\mg$ is defined \'{e}tale locally
(\cite[Lemma 5.11.2]{sacc}). 

\item[(2)] By the definition of $CC(\mg)$ as a linear combination of the prime cycles defined by the irreducible components of $SS(\mg)$,
the support of the characteristic cycle $CC(\mg)$ of $\mg$ is contained in
the singular support $SS(\mg)$ of $\mg$.
\end{enumerate}
\end{rem}

We recall properties of characteristic cycles. 

\begin{lem}[{\cite[Lemma 5.13.1]{sacc}}]
\label{lemsumcc}
Suppose that $X$ is of pure dimension.
Let $\mg$, $\mg'$, and $\mg''$ be constructible complexes of $\Lambda$-modules on $X$ and
let $\mg'\rightarrow \mg\rightarrow \mg''\rightarrow $ be a distinguished triangle in $D_{b}^{c}(X,\Lambda)$.
Then we have
\begin{equation}
CC(\mg)=CC(\mg')+CC(\mg''). \notag
\end{equation}
\end{lem}

\begin{prop}[{\cite[Proposition 5.14]{sacc}}]
\label{propsupcc}
Let $\mg$ be a constructible complex of $\Lambda$-modules on $X$.
Suppose that $X$ is of pure dimension and that
$\mg$ is a perverse sheaf on $X$.
Then we have the following for the characteristic cycle $CC(\mg)$ of $\mg$:
\begin{enumerate}
\item $CC(\mg)\ge 0$.
\item The support of $CC(\mg)$ is equal to the singular support $SS(\mg)$ of $\mg$.
\end{enumerate}
\end{prop}

\begin{cor}[{cf.\ \cite[Lemma 1.7]{yacc}}]
\label{corssccsm}
Suppose that $X$ is purely of dimension $d$.
Let $j\colon V\rightarrow X$ be an affine open immersion and
let $\mh$ be a smooth sheaf of $\Lambda$-modules on $V$.
Then we have the following for the characteristic cycle $CC(j_{!}\mh)$ of $j_{!}\mh$:
\begin{enumerate}
\item $(-1)^{d}CC(j_{!}\mh)\ge 0$.
\item The support of $CC(j_{!}\mh)$ is equal to the singular support $SS(j_{!}\mh)$ of $j_{!}\mh$.
\end{enumerate}
\end{cor}

\begin{proof}
By \cite[Examples 4.0, Corollaire 4.1.10 (i)]{bbd}, the complex $j_{!}\mh[d]$ is a perverse sheaf on $X$.
By Lemma \ref{lemsumcc}, we have 
$CC(j_{!}\mh)=(-1)^{d}CC(j_{!}\mh[d])$.

(1) By Proposition \ref{propsupcc} (1),
we have 
$(-1)^{d}CC(j_{!}\mh)=CC(j_{!}\mh[d])\ge 0$.

(2) By Proposition \ref{propsupcc} (2), the support of $CC(j_{!}\mh)=(-1)^{d}CC(j_{!}\mh[d])$ is equal to $SS(j_{!}\mh[d])$.
Since the support of $CC(j_{!}\mh)$ is contained in $SS(j_{!}\mh)$ by Remark \ref{remsupcc} (2),
we have $SS(j_{!}\mh[d])\subset SS(j_{!}\mh)$.
Conversely, we have $SS(j_{!}\mh)\subset SS(j_{!}\mh[d])$ by Lemma \ref{lemdistss}.
\end{proof}

\begin{prop}[{cf.\ \cite[Theorem 0.1]{sy}}]
\label{propcorssccsm}
Suppose that $X$ is of pure dimension.
Let $j\colon V\rightarrow X$ be an affine open immersion
and let $\mh$ and $\mh'$ be two smooth sheaves of
$\Lambda$-modules of rank 1 on $V$ such that the $p$-parts of the characters $\pi_{1}^{\ab}(V)
\rightarrow \Lambda^{\times}$ associated to $\mh$ and $\mh'$
are the same.
Then we have the following:
\begin{enumerate}
\item $SS(j_{!}\mh)=SS(j_{!}\mh')$.
\item $CC(j_{!}\mh)=CC(j_{!}\mh')$.
\end{enumerate}
\end{prop} 

\begin{proof}
The assertion (2) is a special case
of \cite[Theorem 0.1]{sy}. 
The assertion (1) 
is a consequence of (2) and Corollary \ref{corssccsm} (2).
\end{proof}

\begin{defn}[{cf.\ \cite[Definition 7.1]{sacc}}]
\label{defproptrans}
Suppose that $X$ is purely of dimension $d$.
Let $C\subset T^{*}X$ be a closed conical subset whose irreducible components $\{C_{a}\}_{a}$ are of dimension $d$.
Let $h\colon W\rightarrow X$ be a morphism from a smooth scheme $W$ of pure dimension $e$ over $k$.
\begin{enumerate}
\item We say that the morphism $h\colon W\rightarrow X$ 
is {\it properly} $C$-transversal if $h$ is $C$-transversal and if
$h^{*}C=C\times_{X}W$ is purely of dimension $e$.
\item Assume that the morphism $h\colon W\rightarrow X$ is properly $C$-transversal.
We define $h^{!}A\in Z_{e}(T^{*}W)$ for $A=\sum_{a}m_{a}[C_{a}]\in Z_{d}(T^{*}X)$
to be the push-forward of $(-1)^{d-e}\sum_{a}m_{a}[h^{*}C_{a}]\in Z_{e}(T^{*}X\times_{X}W)$
by $dh\colon T^{*}X\times_{X}W\rightarrow T^{*}W$ (\ref{defdh}).
\end{enumerate}
\end{defn}

\begin{thm}[{\cite[Theorem 7.6]{sacc}}]
\label{thmccpb}
Assume that $X$ is of pure dimension.
Let $\mg$ be a constructible complex of $\Lambda$-modules on $X$ and
let $h \colon W\rightarrow X$ be a properly $C$-transversal morphism
for a closed conical subset $C\subset T^{*}X$. 
If $\mg$ is micro-supported on $C$, then we have
\begin{equation}
CC(h^{*}\mg)=h^{!}CC(\mg). \notag
\end{equation}
\end{thm}

Suppose that $X$ is purely of dimension $d$.
Let $C\subset T^{*}X$ be a closed conical subset of pure dimension $d$.
For a proper morphism $f\colon X\rightarrow Y$ to a smooth scheme $Y$ 
of pure dimension $e$ over $k$, 
we define a morphism
\begin{equation}
\label{fexp}
f_{!}\colon Z_{d}(C)=CH_{d}(C)\xrightarrow{\pr_{1*}\circ\, df^{!}} CH_{e}(f_{\circ}C) 
\end{equation}
to be the composition of the Gysin homomorphism
$df^{!}\colon CH_{d}(C)\rightarrow CH_{e}(df^{-1}(C))$ (\cite[6.6]{ful}) defined by the l.c.i.\ morphism $df\colon T^{*}Y\times_{Y}X\rightarrow T^{*}X$ (\ref{defdh})
and the push-forward 
$\pr_{1*}\colon CH_{e}(df^{-1}(C))\rightarrow CH_{e}(f_{\circ}C)$ defined by the first projection $\pr_{1}\colon T^{*}Y\times_{Y}X\rightarrow T^{*}Y$.
If every irreducible component of $f_{\circ}C$ is of dimension $\le e$,
then we have $CH_{e}(f_{\circ}C)=Z_{e}(f_{\circ}C)$ and the morphism $f_{!}$ (\ref{fexp}) defines the morphism
\begin{equation}
\label{fexpzd}
f_{!}\colon Z_{d}(C)\rightarrow Z_{e}(f_{\circ}C) 
\end{equation}
of abelian groups (\cite[(2.3)]{sacond}).

\begin{lem}[{\cite[Lemma 5.13.2]{sacc}}]
\label{lemclimpush}
Suppose that $X$ is of pure dimension.
Let $\mg$ be a constructible complex of $\Lambda$-modules on $X$
and let $i\colon X\rightarrow Y$ be a closed immersion to a smooth scheme $Y$ of pure dimension over $k$.
Then we have
\begin{equation}
CC(i_{*}\mg)=i_{!}CC(\mg). \notag
\end{equation}
Consequently, if $i'\colon Y\rightarrow Z$ is another closed immersion to a smooth scheme $Z$ of pure dimensions over $k$,
then we have
\begin{equation}
\label{eqlblb}
(i'\circ i)_{!}CC(\mg)=i'_{!}i_{!}CC(\mg). 
\end{equation}
\end{lem}

\begin{proof}
The first assertion is none other than \cite[Lemma 5.13.2]{sacc}.
We prove the second assertion.
By the first assertion, both sides of (\ref{eqlblb}) are equal to
$CC((i'\circ i)_{*}\mg)=CC(i'_{*}i_{*}\mg)$.
\end{proof}

\begin{lem}
\label{lemsscclim}
Suppose that $X$ is of pure dimension.
Let $j\colon V\rightarrow X$ be an affine open immersion and
let $\mh$ be a smooth sheaf of $\Lambda$-modules on $V$.
If $i\colon X\rightarrow Y$ is a closed immersion to a smooth scheme $Y$ of pure dimension over $k$,
then the support of $i_{!}CC(j_{!}\mh)$ is $i_{\circ}SS(j_{!}\mh)$. 
\end{lem}

\begin{proof}
By Lemma \ref{lemclimpush}, we have $i_{!}CC(j_{!}\mh)=CC(i_{*}j_{!}\mh)$.
By Theorem \ref{thmss} (2),  we have $i_{\circ}SS(j_{!}\mh)=SS(i_{*}j_{!}\mh)$.
Therefore it is sufficient to prove that the support of $CC(i_{*}j_{!}\mh)$ is $SS(i_{*}j_{!}\mh)$.
By \cite[Examples 4.0, Corollaire 4.1.3]{bbd}, the complex $i_{*}j_{!}\mh[d]$ is a perverse sheaf on $Y$.
Then the assertion follows similarly as the proof of Corollary \ref{corssccsm} (2)
with $j_{!}\mh$ replaced by $i_{*}j_{!}\mh$.
\end{proof}

\section{Log-$D'$-characteristic cycle and a candidate of singular support}
\label{seclogtrans}

Let $j\colon U\rightarrow X$ denote the canonical open immersion
as is explained in Conventions.
For a divisor $D'$ on $X$ with simple normal crossings contained in $D$
and containing $D_{\mT,\mf}$ (\ref{defd*mf}),
we define  
the $\log$-$D'$-characteristic cycle $CC_{D'}^{\log}(j_{!}\mf)$ 
of $j_{!}\mf$
in Definition \ref{deflifdp} (4) 
when the ramification of $\mf$ is $\log$-$D'$-clean along $D$
(Definition \ref{deflogdcl} (2)).
We construct a closed conical subset $S_{D'}(j_{!}\mf)\subset T^{*}X$
as a union of conormal bundles and the inverse image $\tau_{D'}^{-1}(S_{D'}^{\log}(j_{!}\mf))$ of the support $S_{D'}^{\log}(j_{!}\mf)$ of $CC_{D'}^{\log}(j_{!}\mf)$
by the canonical morphism $\tau_{D'}\colon T^{*}X\rightarrow T^{*}X(\log D')$
in Definition \ref{defsdpf}, when $D$ has simple normal crossings.
We prove that 
the singular support $SS(j_{!}\mf)$ (Definition \ref{defss} (2))
is contained in $S_{D'}(j_{!}\mf)$ in Theorem \ref{thmmains}.
If further the dimension of the inverse image $\tau_{D'}^{-1}(S_{D'}^{\log}(j_{!}\mf))$ is the same as that of $X$, 
then $SS(j_{!}\mf)$ is proved to be contained in $\tau_{D'}^{-1}(S_{D'}^{\log}(j_{!}\mf))\subset S_{D'}(j_{!}\mf)$ in Corollary \ref{corsdpdim}
and we prove that the characteristic cycle $CC(j_{!}\mf)$ (Definition \ref{defcc}) is determined by 
the pull-back $\tau_{D'}^{!}CC_{D'}^{\log}(j_{!}\mf)$ in some sense in Proposition \ref{prophicc}.

\subsection{Log-$D'$-characteristic cycle}
\label{sslogtrans}
We define the $\log$-$D'$-characteristic cycle $CC_{D'}^{\log}(j_{!}\mf)$ of $j_{!}\mf$ 
when the ramification of $\mf$ is $\log$-$D'$-clean along $D$,
where $D'$ is a divisor 
on $X$ with simple normal crossings contained in $D$.

\begin{defn}[{cf.\ \cite[(3.4.4)]{kalog}, \cite[Definition 3.5.1]{sacot}}]
\label{deflifdp}
Let $I'\subset I$ be a subset such that
$D'=\bigcup_{i\in I'}D_{i}$ has simple normal crossings.
Assume that the ramification of $\mf$ is $\log$-$D'$-clean along $D$.
Let $R_{\mf}^{D'}$ and $Z_{\mf}$ be as in Definition \ref{defconddiv} (1) and (2), respectively,
and let $\cform^{D'}(\mf)$ be the $\log$-$D'$-characteristic form of $\mf$
(Definition \ref{defcform}).

\begin{enumerate}
\item We define a closed subscheme
\begin{equation}
L'^{D'}_{\mf}\subset T^{*}X(\log D')\times_{X}Z_{\mf}^{1/p} \notag
\end{equation} 
to be the sub line bundle of $T^{*}X(\log D')\times_{X}Z_{\mf}^{1/p}$
defined by the invertible sheaf $\dvr_{X}(-R_{\mf}^{D'})|_{Z_{\mf}^{1/p}}\cdot \cform^{D'}(\mf)$,
which is locally a direct summand of $\Omega_{X}^{1}(\log D')|_{Z_{\mf}^{1/p}}$ 
by the $\log$-$D'$-cleanliness of the ramification of $\mf$ along $D$.

For $i\in I_{\mW,\mf}$ (Definition \ref{defindsub} (1)),
we define a closed subscheme
\begin{equation}
L'^{D'}_{i,\mf}\subset T^{*}X(\log D')\times_{X}D_{i}^{1/p} \notag
\end{equation} 
to be the the sub line bundle of $T^{*}X(\log D')\times_{X}D_{i}^{1/p}$
defined by the invertible sheaf $\dvr_{X}(-R_{\mf}^{D'})|_{D_{i}^{1/p}}\cdot \cform^{D'}(\mf)|_{D_{i}^{1/p}}$.

\item Let $f\colon Z_{\mf}^{1/p}\rightarrow Z_{\mf}$ and $f_{i}\colon D_{i}^{1/p}\rightarrow D_{i}$ for $i\in I_{\mW,\mf}$ be the structure morphisms
(\ref{radicialcvmap}).
We define a closed subscheme
\begin{equation}
L^{D'}_{\mf}\subset T^{*}X(\log D') \notag
\end{equation} 
to be the image of $L_{\mf}'^{D'}$ by the canonical morphism 
\begin{equation}
T^{*}X(\log D')\times_{X}Z_{\mf}^{1/p}\rightarrow T^{*}X(\log D')\times_{X}Z_{\mf} \subset T^{*}X(\log D') \notag
\end{equation}
with the closed subscheme structure satisfying the equality $[L_{\mf}^{D'}]=f_{*}[L_{\mf}'^{D'}]$
as algebraic cycles on $T^{*}X(\log D')$.

For $i\in I_{\mW,\mf}$,
we define a closed subscheme
\begin{equation}
L^{D'}_{i,\mf}\subset T^{*}X(\log D') \notag
\end{equation}
to be the image of $L_{i,\mf}'^{D'}$ by the canonical morphism
\begin{equation}
T^{*}X(\log D')\times_{X}D_{i}^{1/p}\rightarrow T^{*}X(\log D')\times_{X}D_{i} \subset T^{*}X(\log D')
\notag
\end{equation}
with the closed subscheme structure satisfying the equality $[L_{i,\mf}^{D'}]=f_{i*}[L_{i,\mf}'^{D'}]$
as algebraic cycles on $T^{*}X(\log D')$.

\item Assume that $I_{\mT,\mf}$ (Definition \ref{defindsub} (1)) is contained in $I'$.
Then we define the $\log$-$D'$-{\it singular support} $S_{D'}^{\log}(j_{!}\mf)$ of
$j_{!}\mf$ by
\begin{equation}
S_{D'}^{\log}(j_{!}\mf)=T^{*}_{X}X(\log D')\cup L^{D'}_{\mf}, \notag
\end{equation} 
which is a closed conical subset of $T^{*}X(\log D')$.
\item Assume that $X$ is purely of dimension $d$ and that $I_{\mT,\mf}$ is contained in $I'$.
Then we define the $\log$-$D'$-{\it characteristic cycle} $CC^{\log}_{D'}(j_{!}\mf)$
by
\begin{equation}
CC^{\log}_{D'}(j_{!}\mf)=(-1)^{d}([T^{*}_{X}X(\log D')]
+\sum_{i\in I_{\mW,\mf}}
\sw^{D'}(\chi|_{K_{i}})[L_{i,\mf}^{D'}])\in Z_{d}(T^{*}X(\log D')), \notag
\end{equation}
where $\sw^{D'}(\chi|_{K_{i}})$ is as in (\ref{defswdp}).
\end{enumerate}
\end{defn}

\begin{rem}
\label{remlifdp}
Let the notation be as in Definition \ref{deflifdp}.
Suppose that the ramification of $\mf$ is $\log$-$D'$-clean along $D$.
\begin{enumerate}
\item In Definition \ref{deflifdp} (1) (resp.\ (2)), 
the closed conical subsets $L_{i,\mf}'^{D'}\subset T^{*}X(\log D')\times_{X}Z_{\mf}^{1/p}$ (resp.\ $L_{i,\mf}^{D'}\subset T^{*}X(\log D')$) 
for $i\in I_{\mW,\mf}$ are the irreducible components of $L_{\mf}'^{D'}$
(resp.\ of $L_{\mf}^{D'}$),
and we have $L_{\mf}'^{D'}=\bigcup_{i\in I_{\mW,\mf}}L_{i,\mf}'^{D'}$
(resp.\ $L_{\mf}^{D'}=\bigcup_{i\in I_{\mW,\mf}}L_{i,\mf}^{D'}$).

\item By Remark \ref{remcfnormal} (2), both the sub line bundles
$L_{\mf}'^{D'}\subset T^{*}X(\log D')\times_{X}Z_{\mf}^{1/p}$
and $L_{\mf}^{D'}\subset T^{*}X(\log D')\times_{X}Z_{\mf}$ 
are stable under the replacement of $\chi$ by the $p$-part of $\chi$. 
Therefore so is $S_{D'}^{\log}(j_{!}\mf)$,
and so is $CC_{D'}^{\log}(j_{!}\mf)$ 
by (1) and Remark \ref{remtameloc} (1).

\item Suppose that $D$ has simple normal crossings.
Then the closed subscheme $L_{i,\mf}^{D}\subset T^{*}X(\log D)$ for $i\in I_{\mW,\mf}$
is the same as $\Image \varphi_{\mathfrak{p}_{i}}$ in \cite[(3.4.4)]{kalog}, where $\mathfrak{p}_{i}$ is the generic point of $D_{i}$, and $L_{i,\chi}$ in \cite[the remark before Definition 3.1]{yacc} by Remark \ref{remcfnormal} (4).
The closed subscheme $L_{i,\mf}'^{\, \emptyset}\subset T^{*}X\times_{X}Z_{\mf}^{1/p}$ 
is the same as $L_{i,\chi}''$ in \cite[the remark after Lemma 2.18]{yacc} by Remark \ref{remcfnormal} (4), and
is equal to $L_{\chi|_{K_{i}}}$ in \cite[the remark before Proposition 4.13]{sacc}
by \cite[Corollary 2.13 (i)]{yafil} and Remark \ref{remlogdpcl} (3). 
The cycle $[L_{i,\mf}^{\emptyset}]$ on $T^{*}X$ defined by $L_{i,\mf}^{\emptyset}$ is
the same as 
$[L_{i,\chi}']$ in \cite[the remark after Lemma 2.18]{yacc}
and $\pi_{\chi|_{K_{i}}*}[L_{\chi|_{K_{i}}}]$ in \cite[the remark before Theorem 7.14]{sacc}.

\item Let the assumptions be as in Definition \ref{deflifdp} (4).
Since $\sw^{D'}(\chi|_{K_{i}})>0$ for $i\in I_{\mW,\mf}$,
we have $(-1)^{d}CC_{D'}^{\log}(j_{!}\mf)>0$ and
the support of $CC^{\log}_{D'}(j_{!}\mf)$ is the $\log$-$D'$-singular support $S_{D'}^{\log}(j_{!}\mf)$ of $j_{!}\mf$ by (1).

\item Let the assumptions be as in Definition \ref{deflifdp} (4).
Suppose further that $D$ has simple normal crossings.
Then the $\log$-$D$-characteristic cycle $CC^{\log}_{D}(j_{!}\mf)$ is the same as 
$\mathrm{Char}^{\log}(X,U,\chi)$ in \cite[Definition 3.1]{yacc},
which is a modification of \cite[(3.4.4)]{kalog}, by (3) and Remark \ref{remlogdpcl} (3),
and is equal to $CC(\mf)$ in \cite[Definition 3.6]{sawild}
by \cite[Corollaire 9.12]{asan} and Remark \ref{remlogdpcl} (3).
By (3), by \cite[Theorem 3.1]{yafil}, and by Remark \ref{remlogdpcl} (3), the $\log$-$\emptyset$-characteristic cycle $CC^{\log}_{\emptyset}(j_{!}\mf)$ is 
equal to $C(j_{!}\mf)$ in 
\cite[the remark before Theorem 7.14]{sacc},
which is equal to the characteristic cycle $CC(j_{!}\mf)$ of $j_{!}\mf$ (Definition \ref{defcc})
by \cite[Theorem 7.14]{sacc}.
\end{enumerate}
\end{rem}

In order to discuss compatibilities with morphisms of schemes,
we recall the theory of logarithmic transversality introduced in \cite[Subsection 3.2]{yacc}.
Let $I'\subset I$ be a subset such that
$D'=\bigcup_{i\in I'}D_{i}$ has simple normal crossings. 
Let 
\begin{equation}
\label{defcdp}
C_{D'}=\bigcup_{I''\subset I'}T^{*}_{D_{I''}}X \subset T^{*}X 
\end{equation} 
be the union of the conormal bundles $T^{*}_{D_{I''}}X$ of 
$D_{I''}=\bigcap_{i\in I''}D_{i}\subset X$ for all subsets $I''\subset I'$ including $I''=\emptyset$.
If $I''=\emptyset$, then $T^{*}_{D_{I''}}X$ is the zero section $T^{*}_{X}X$ of $T^{*}X$ by convention.

By \cite[Lemma 3.4.5]{sacc}, 
a morphism $h\colon W\rightarrow X$ of smooth schemes over $k$ is $C_{D'}$-transversal (Definition \ref{defctr} (1)) if and only if
the pull-back $h^{*}D_{i}=D_{i}\times_{k}W$ for $i\in I'$ is a smooth divisor on $W$ and if $h^{*}D'=D'\times_{k}W$ is a divisor on $W$ with simple normal crossings.
Therefore, if a morphism $h\colon W\rightarrow X$
of smooth schemes over $k$ is $C_{D'}$-transversal, 
then the morphism $h$ induces the morphism 
\begin{equation}
\label{defdhdp}
dh^{D'}\colon T^{*}X(\log D')\times_{X}W\rightarrow T^{*}W(\log h^{*}D'). 
\end{equation}
 
\begin{defn}[{\cite[Definition 3.6]{yacc}}]
\label{deflogdpctr}
Let $I'\subset I$ be a subset such that
$D'=\bigcup_{i\in I'}D_{i}$ has simple normal crossings and
let $C_{D'}$ be as in (\ref{defcdp}).
Let $h\colon W\rightarrow X$ be a $C_{D'}$-transversal morphism of smooth schemes over $k$
and let $C$ be a closed conical subset of $T^{*}X(\log D')$.
Then we say that $h$ is {\it $\log$-$D'$-$C$-transversal} at $w\in W$ 
if the subset
\begin{equation}
(h^{*}C\cap {dh^{D'}}^{-1}(T^{*}_{W}W(\log h^{*}D')))\times_{W}w\subset T^{*}X(\log D')\times_{X}w, \notag
\end{equation}
where $dh^{D'}$ is as in (\ref{defdhdp}),
is contained in the zero section $T^{*}_{X}X(\log D')\times_{X}w$.

We say that $h$ is {\it $\log$-$D'$-$C$-transversal} 
if $h$ is $\log$-$D'$-$C$-transversal at every $w\in W$, 
namely if the intersection
\begin{equation}
h^{*}C\cap {dh^{D'}}^{-1}(T^{*}_{W}W(\log h^{*}D'))\subset T^{*}X(\log D')\times_{X}W \notag
\end{equation} 
is contained in the zero section $T^{*}_{X}X(\log D')\times_{X}W$.

If $h$ is $\log$-$D'$-$C$-transversal, then we define a closed conical subset
\begin{equation} 
h^{\circ}C\subset T^{*}W(\log h^{*}D') \notag
\end{equation} 
to be the image of $h^{*}C$ by $dh^{D'}$.
\end{defn}

\begin{rem}
\label{remlogdrcdptr}
Let the notation be as in Definition \ref{deflogdpctr}.
\begin{enumerate}
\item In the last part of Definition \ref{deflogdpctr}, 
if the morphism $h\colon W\rightarrow X$ is $\log$-$D'$-transversal,
then the image $h^{\circ}C$ of $h^{*}C=C\times_{X}W$ by $dh^{D'}$ is a closed conical subset of $T^{*}W(\log h^{*}D')$,
since the restriction $dh^{D'}|_{h^{*}C}\colon h^{*}C\rightarrow T^{*}W(\log h^{*}D')$ 
of $dh^{D'}$ to $h^{*}C\subset T^{*}X(\log D')\times_{X}W$ is finite
by \cite[Lemma 3.1]{sacc}.
\item Let $C, C'\subset T^{*}X(\log D')$ be two closed conical subsets.
If $C'\subset C$, then every $\log$-$D'$-$C$-transversal morphism $h\colon W\rightarrow X$ is $\log$-$D'$-$C'$-transversal.
Generally, a $C_{D'}$-transversal morphism
$h\colon W\rightarrow X$ of smooth schemes over $k$ 
is $\log$-$D'$-$C\cup C'$-transversal
if and only if $h$ is $\log$-$D'$-$C$-transversal and is $\log$-$D'$-$C'$-transversal.
\item Let $h\colon W\rightarrow X$ be a $C_{D'}$-transversal morphism 
of smooth schemes over $k$.
Since we have $h^{*}T^{*}_{X}X(\log D')=T^{*}_{X}X(\log D')\times_{X}W$, 
the morphism $h$ is $\log$-$D'$-$T^{*}_{X}X(\log D')$-transversal and we have $h^{\circ}T^{*}_{X}X(\log D')=T^{*}_{W}W(\log h^{*}D')$.
Therefore $h$ is $\log$-$D'$-$T^{*}_{X}X(\log D')\cup C$-transversal for a closed conical subset 
$C\subset T^{*}X(\log D')$ if and only if $h$ is $\log$-$D'$-$C$-transversal by (2). 
\end{enumerate}
\end{rem}

Similarly as the $C$-transversality of a morphism 
$h\colon W\rightarrow X$ of smooth schemes over $k$ for 
a closed conical subset $C\subset T^{*}X$, 
the $\log$-$D'$-$C'$-transversality of 
a $C_{D'}$-transversal morphism $h'\colon W'\rightarrow X$ 
of smooth schemes over $k$ for a closed conical subset $C'\subset T^{*}X(\log D')$ is an open condition on $W$ (cf.\ Remark \ref{remctr} (2)):

\begin{lem}[cf.\ {\cite[Lemma 1.2 (i)]{be}}]
Let $I'\subset I$ be a subset such that
$D'=\bigcup_{i\in I'}D_{i}$ has simple normal crossings and
let $C\subset T^{*}X(\log D')$ be a closed conical subset.
For a morphism $h\colon W\rightarrow X$ of smooth schemes over $k$,
the set of points on $W$ at which $h$ is $\log$-$D'$-$C$-transversal
is an open subset of $W$.
\end{lem}

\begin{proof}
The assertion follows, since the considered subset of $W$ is the complement of the image of
the closed subset 
\begin{equation}
\mathbf{P}(h^{*}C\cap {dh^{D'}}^{-1}(T^{*}_{W}W(\log h^{*}D')))\subset \mathbf{P}(T^{*}X(\log D')\times_{X}W) \notag
\end{equation}
of the projective space bundle on $W$ associated to the vector bundle $T^{*}X(\log D')\times_{X}W$ on $W$
by the canonical projection
$\mathbf{P}(T^{*}X(\log D')\times_{X}W)\rightarrow W$. 
\end{proof}

For two divisors $E'\subset E$ on $X$ with simple normal crossings, let
\begin{equation}
\label{deftauepse}
\tau_{E'/E}\colon T^{*}X(\log E')\rightarrow T^{*}X(\log E) 
\end{equation}
denote the canonical morphism of vector bundles on $X$.
If $E'=\emptyset$, then we simply denote $\tau_{\emptyset/E}\colon T^{*}X\rightarrow T^{*}X(\log E)$  
by 
\begin{equation}
\label{deftaue}
\tau_{E}\colon T^{*}X\rightarrow T^{*}X(\log E).
\end{equation}
The closed conical subset $C_{E}\subset T^{*}X$ (\ref{defcdp}) is
equal to the inverse image by $\tau_{E}\colon T^{*}X\rightarrow T^{*}X(\log E)$ of the zero section
$T^{*}_{X}X(\log E)\subset T^{*}X(\log E)$ as sets, namely, we have
\begin{equation}
\label{eqinvzero}
C_{E}=\tau_{E}^{-1}(T^{*}_{X}X(\log E)).
\end{equation}
If $\{E_{\theta}\}_{\theta\in \Theta}$ denotes the irreducible components of
$E$ and if $E_{\Theta'}$ denotes the intersection $\bigcap_{\theta\in \Theta'}E_{\theta}$
for each subset $\Theta'\subset \Theta$,
then we have
\begin{equation}
\label{eqinvzerocyc}
[\tau_{E}^{-1}(T^{*}_{X}X(\log E))]=\sum_{\Theta'\subset \Theta}[T^{*}_{E_{\Theta'}}X]
\end{equation}
as algebraic cycles on $T^{*}X$.

The transversality and the $\log$-$D'$-transversality for
a $C_{D'}$-transversal morphism of smooth schemes over $k$
are equivalent in the following sense:

\begin{prop}[{\cite[Proposition 3.9]{yacc}}]
\label{proplogtr}
Let $I'\subset I$ be a subset such that
$D'=\bigcup_{i\in I'}D_{i}$ has simple normal crossings. 
Let $C_{D'}$ be as in (\ref{defcdp}) and
let $\tau_{D'}\colon T^{*}X\rightarrow T^{*}X(\log D')$ be as in (\ref{deftaue}).
Let $h\colon W\rightarrow X$ be a $C_{D'}$-transversal morphism 
of smooth schemes over $k$
and let $C\subset T^{*}X(\log D')$ be a closed conical subset. 
Then the following
two conditions are equivalent:
\begin{enumerate}
\item $h$ is $\log$-$D'$-$C$-transversal.
\item $h$ is $\tau_{D'}^{-1}(C)$-transversal. 
\end{enumerate}
\end{prop}

Let $I'\subset I$ be a subset such that
$D'=\bigcup_{i\in I'}D_{i}$ has simple normal crossings.
Let 
\begin{equation}
\label{defcdpd}
C_{D'\subset D}=C_{D'}\cup\bigcup_{i\in I-I'}T^{*}_{D_{i}}X\subset T^{*}X
\end{equation} 
be the union of $C_{D'}$ (\ref{defcdp}) and the conormal bundles $T^{*}_{D_{i}}X$
of $D_{i}\subset X$ for all $i\in I-I'$.
By Remark \ref{remctr} (3),
a $C_{D'\subset D}$-transversal morphism $h\colon W\rightarrow X$ 
of smooth schemes over $k$ is $C_{D'}$-transversal.
By \cite[Lemma 3.4.5]{sacc},
the pull-back $h^{*}D_{i}=D_{i}\times_{X}W$ of $D_{i}$ 
by a $C_{D'\subset D}$-transversal morphism $h\colon W\rightarrow X$
for $i\in I$ is a smooth divisor
on $W$ and
the pull-back $(h^{*}D)_{\red}=(D\times_{X}W)_{\red}$ is a divisor on $W$ with smooth irreducible components.

We study compatibilities of $\log$-$D'$-singular support $S_{D'}^{\log}(j_{!}\mf)$
and $\log$-$D'$-characteristic cycle $CC_{D'}^{\log}(j_{!}\mf)$
with a $\log$-$D'$-$S_{D'}^{\log}(j_{!}\mf)$-transversal morphism.

\begin{prop}[cf.\ {\cite[Proposition 3.11]{yacc}}]
\label{propeqcl} 
Let $I'\subset I$ be a subset such that
$D'=\bigcup_{i\in I'}D_{i}$ has simple normal crossings.
Suppose that the ramification of $\mf$ is $\log$-$D'$-clean along $D$
and that $I_{\mT,\mf}$ (Definition \ref{defindsub} (1)) is contained in $I'$.
Let $C_{D'\subset D}$ be as in (\ref{defcdpd}) and
let $h\colon W\rightarrow X$ be a $C_{D'\subset D}$-transversal morphism
of smooth schemes over $k$.
\begin{enumerate}
\item The following three conditions are equivalent:
\begin{enumerate}
\item $h$ is $\log$-$D'$-$S_{D'}^{\log}(j_{!}\mf)$-transversal.
\item $R_{h^{*}\mf}^{h^{*}D'}=h^{*}R_{\mf}^{D'}$
(Definition \ref{defconddiv} (1)), $Z_{h^{*}\mf}=(h^{*}Z_{\mf})_{\red}$ (Definition \ref{defconddiv} (2)), and $\cform^{h^{*}D'}(h^{*}\mf)
=dh^{D'}_{(h^{*}Z_{\mf}^{1/p})_{\red}}(h^{*}\cform^{D'}(\mf))$,
where 
\begin{equation}
dh^{D'}_{(h^{*}Z_{\mf}^{1/p})_{\red}}\colon h^{*}\Omega^{1}_{X}(\log D')(R_{\mf}^{D'})|_{(h^{*}Z_{\mf}^{1/p})_{\red}}\rightarrow \Omega^{1}_{W}(\log h^{*}D')(h^{*}R_{\mf}^{D'})|_{(h^{*}Z_{\mf}^{1/p})_{\red}}  \notag
\end{equation}
is the morphism yielded by $h$.  
The ramification of $h^{*}\mf$ is
$\log$-$h^{*}D'$-clean along $(h^{*}D)_{\red}$.
\item The image of $L_{\mf}'^{D'}\times_{Z_{\mf}^{1/p}}(h^{*}Z_{\mf}^{1/p})_{\red}$
(Definition \ref{deflifdp} (1)) by the base change
\begin{equation}
dh^{D'}_{(h^{*}Z_{\mf}^{1/p})_{\red}}\colon T^{*}X(\log D')\times_{X}(h^{*}Z_{\mf}^{1/p})_{\red}
\rightarrow T^{*}W(\log h^{*}D')\times_{W}(h^{*}Z_{\mf}^{1/p})_{\red} \notag
\end{equation}
of $dh^{D'}$ (\ref{defdhdp}) by the canonical morphism $(h^{*}Z_{\mf}^{1/p})_{\red}\rightarrow W$ is a sub line bundle of $T^{*}W(\log h^{*}W)\times _{W}(h^{*}Z_{\mf}^{1/p})_{\red}$.
\end{enumerate}
\item 
Suppose that the equivalent conditions (a), (b), and (c) in (1) holds
and that $h^{*}Z_{\mf}=Z_{\mf}\times_{X}W$ is reduced.
Then we have $D_{\mT,h^{*}\mf}\subset h^{*}D'$ (\ref{defd*mf}) and
\begin{equation}
\label{eqsdlog}
L^{h^{*}D'}_{h^{*}\mf}=h^{\circ}L_{\mf}^{D'}
\end{equation}
for closed conical subsets defined in Definition \ref{deflifdp} (2).
If $j'\colon h^{*}U=U\times_{k}W\rightarrow W$ denotes
the base change of $j\colon U\rightarrow X$ by $h$,
then we have
\begin{equation}
\label{eqlogsspb}
S_{h^{*}D'}^{\log}(j'_{!}h^{*}\mf)=h^{\circ}S_{D'}^{\log}(j_{!}\mf).
\end{equation}
\end{enumerate}
\end{prop}

\begin{proof}
(1) By Remark \ref{remlogdrcdptr} (3), the condition (a) is equivalent to the condition that $h$ is $\log$-$D'$-$L_{\mf}^{D'}$-transversal.
Since $L_{\mf}^{D'}$ is defined to be the image of 
$L_{\mf}'^{D'}$ by the canonical morphism
$T^{*}X(\log D')\times_{X}Z_{\mf}^{1/p}\rightarrow T^{*}X(\log D')\times_{X}Z_{\mf}$, 
each of the conditions (a) and (c) is equivalent to the condition 
\begin{itemize}
\item[(d)] $dh^{D'}_{(h^{*}Z_{\mf}^{1/p})_{\red}}(h^{*}\cform^{D'}(\mf))(w')\neq 0$
for every closed point $w'\in (h^{*}Z_{\mf}^{1/p})_{\red}$.
\end{itemize}
Thus it is sufficient to prove the equivalence of the conditions (b) and (d). 

In order to prove the equivalence,
we may assume that all of the three equalities in the condition (b) hold.
In fact, the condition (d) implies the first equality in (b) by Lemma \ref{lemcformgl} (1) 
and the first equality implies the second and third equalities in (b) by
Lemma \ref{lemcformgl} (2).
Then the equivalence of the (last) condition (in) (b) and the condition (d) follows from Lemma \ref{lemeqtologdpcl} (2).

(2) By the second equality in the condition (b) in (1), 
we have $D_{\mT,h^{*}\mf}\subset h^{*}D'$.
We consider the cartesian diagram
\begin{equation}
\label{cartdxwlog}
\xymatrix{
T^{*}X(\log D')\times_{X}h^{*}Z_{\mf}^{1/p} \ar[d]_-{dh^{D'}_{h^{*}Z_{\mf}^{1/p}}} \ar[r] \ar@{}[dr] | {\square} & 
T^{*}X(\log D')\times_{X}h^{*}Z_{\mf} \ar[d]^-{dh^{D'}_{h^{*}Z_{\mf}}} \\ 
T^{*}W(\log h^{*}D')\times_{W}h^{*}Z_{\mf}^{1/p} \ar[r] & T^{*}W(\log h^{*}D')\times_{W}h^{*}Z_{\mf},
}
\end{equation}
where the horizontal arrows are canonical morphisms and 
the left and right vertical arrows are the base changes of $dh^{D'}$ 
by the canonical morphisms $h^{*}Z_{\mf}^{1/p}\rightarrow W$
and $h^{*}Z_{\mf}\rightarrow W$, respectively.
By the condition (b) (resp.\ the condition (a)) in (1),
the left-hand side (resp.\ the right-hand side) of (\ref{eqsdlog}) 
is the image of $h^{*}L_{\mf}'^{D'}$ by the composition of the left vertical arrow and the lower horizontal arrow (resp.\ by the composition of the upper horizontal arrow and the right vertical arrow)
in (\ref{cartdxwlog}).
Thus the equality (\ref{eqsdlog}) holds by the commutativity of (\ref{cartdxwlog}). 
Since we have $h^{\circ}T^{*}_{X}X(\log D')=T^{*}_{W}W(\log h^{*}D')$ 
by Remark \ref{remlogdrcdptr} (3),
the equality (\ref{eqlogsspb}) follows from the equality (\ref{eqsdlog}).
\end{proof}

\begin{defn}
\label{deflogproptr}
Assume that $X$ is purely of dimension $d$.
Let $I'\subset I$ be a subset such that $D'=\bigcup_{i\in I'}D_{i}$
has simple normal crossings.
Let $C\subset T^{*}X(\log D')$ be a closed conical subset
of pure dimension $d$ and $\{C_{a}\}_{a}$ the irreducible components of $C$.
Let $h\colon W\rightarrow X$ be a $C_{D'}$-transversal morphism from a smooth scheme $W$ over $k$
of pure dimension $e$, where $C_{D'}$ is as in (\ref{defcdp}).
\begin{enumerate}
\item We say that the morphism $h\colon W\rightarrow X$ is  
{\it properly} $\log$-$D'$-$C$-transversal if $h$ is $\log$-$D'$-$C$-transversal and if $h^{*}C=C\times_{X}W$ is purely of dimension $e$.
\item Assume that the morphism $h\colon W\rightarrow X$ is properly $\log$-$D'$-$C$-transversal.
Then we define 
\begin{equation}
h^{!}A\in Z_{e}(T^{*}W(\log h^{*}D')) \notag
\end{equation}
for $A=\sum_{a}m_{a}[C_{a}]\in Z_{d}(T^{*}X(\log D'))$ to be
the push-forward of
$(-1)^{d-e}\sum_{a}m_{a}[h^{*}C_{a}]\in Z_{e}(T^{*}X(\log D')\times_{X}W)$ by
$dh^{D'}\colon T^{*}X(\log D')\times_{X}W\rightarrow T^{*}W(\log h^{*}D')$ (\ref{defdhdp}),
whose restriction to $h^{*}C$ is finite by Remark \ref{remlogdrcdptr} (1).
\end{enumerate}
\end{defn}

\begin{prop}
\label{proppblogcc}
Suppose that $X$ is of pure dimension 
and that $D$ has simple normal crossings.
Let $I'\subset I$ be a subset and let $D'=\bigcup_{i\in I'}D_{i}$.
Assume that the ramification of $\mf$ is $\log$-$D'$-clean along $D$ 
and that $I_{\mT,\mf}$ (Definition \ref{defindsub} (1)) is contained in $I'$.
Let $h\colon W\rightarrow X$ be a $C_{D}$-transversal and $\log$-$D'$-$S_{D'}^{\log}(j_{!}\mf)$-transversal morphism from a smooth scheme $W$ of pure dimension 
$e$ over $k$,
where $C_{D}$ is as in (\ref{defcdp}).
Then $h$ is properly $\log$-$D'$-$S_{D'}^{\log}(j_{!}\mf)$-transversal,
the ramification of $h^{*}\mf$ is $\log$-$h^{*}D'$-clean along $h^{*}D$,
we have $D_{\mT,h^{*}\mf}\subset h^{*}D'$ (\ref{defd*mf}), and 
we have
\begin{equation}
CC^{\log}_{h^{*}D'}(j'_{!}h^{*}\mf)=h^{!}CC^{\log}_{D'}(j_{!}\mf) \notag 
\end{equation}
in $Z_{e}(T^{*}W(\log h^{*}D'))$.
Here $j'\colon h^{*}U\rightarrow W$ denotes the base change of $j$ by $h$.
\end{prop}

\begin{proof}
Since $h$ is $C_{D}$-transversal,
the pull-back $h^{*}D_{i}$ is a smooth divisor on $W$ for every $i\in I$
and $h^{*}D$ is a divisor on $W$ with simple normal crossings
by \cite[Lemma 3.4.5]{sacc}.
Since $L_{\mf}^{D'}$ (Definition \ref{deflifdp} (2)) is isomorphic to the sub line bundle 
$L_{\mf}'^{D'}$ (Definition \ref{deflifdp} (1)) of $T^{*}X(\log D')\times_{X}Z_{\mf}^{1/p}$
(Definition \ref{defconddiv} (2)) as topological spaces
and since $Z_{\mf}$ is a union of irreducible components of $D$, 
the dimension of $h^{*}L_{\mf}^{D'}$ is $e$. 
Since $h^{*}T^{*}_{X}X(\log D')$ is isomorphic to $T^{*}_{W}W(\log h^{*}D')$ by $dh^{D'}$ 
(\ref{defdhdp}) and is of dimension $e$,
the first assertion holds.

Let $\{E_{i}\}_{i\in I_{\mW,h^{*}\mf}}$ be the irreducible components of $Z_{h^{*}\mf}$ and let $L_{i}=\Frac \hat{\dvr}_{W,\mathfrak{q}_{i}}$ denote the local field at the generic point $\mathfrak{q}_{i}$ of $E_{i}$
for $i\in I_{\mW,h^{*}\mf}$.
Since the ramification of $\mf$ is $\log$-$D'$-clean along $D$ and
since the morphism $h$ is $C_{D'\subset D}$-transversal and $\log$-$D'$-$S_{D'}^{\log}(j_{!}\mf)$-transversal by Remark \ref{remctr} (3),
the second and third assertions hold by Propositions \ref{propeqcl} (1) and (2), respectively.
By  Remark \ref{remlifdp} (1) and  Proposition \ref{propeqcl},
we have
\begin{equation}
h^{!}\sum_{i\in I_{\mW,\mf}}\sw^{D'}(\chi|_{K_{i}})[L_{i,\mf}^{D'}]
=(-1)^{d-e}\sum_{i\in I_{\mW,h^{*}\mf}}\sw^{D'}(h^{*}\chi|_{L_{i}})[L_{i,h^{*}\mf}^{h^{*}D'}],\notag
\end{equation}
where $d$ is the dimension of $X$.
Since we have $h^{!}[T^{*}_{X}X(\log D')]=(-1)^{d-e}[T^{*}_{W}W(\log h^{*}D')]$,
the last assertion holds.
\end{proof}

Let $I'\subset I$ be a subset such that $D'=\bigcup_{i\in I'}D_{i}$ has simple normal crossings.
Suppose that $I_{\mT,\mf}$ (Definition \ref{defindsub} (1)) is contained in $I'$.
If $X$ is purely of dimension $d$ and if the ramification of $\mf$ is simple normal crossings, then let
\begin{equation}
\label{gysin}
\tau_{D'}^{!}\colon Z_{d}(S_{D'}^{\log}(j_{!}\mf))=CH_{d}(S_{D'}^{\log}(j_{!}\mf))\rightarrow CH_{d}(\tau_{D'}^{-1}(S_{D'}^{\log}(j_{!}\mf))) 
\end{equation}
denote the Gysin homomorphism (\cite[6.6]{ful}) defined by the l.c.i.\ morphism $\tau_{D'}\colon T^{*}X\rightarrow T^{*}X(\log D')$ (\ref{deftaue}).
If further the inverse image $\tau_{D'}^{-1}(S_{D'}^{\log}(j_{!}\mf)$ of the $\log$-$D'$-singular support
$S_{D'}^{\log}(j_{!}\mf)\subset T^{*}X(\log D')$ is of dimension $d$, then 
we have $CH_{d}(\tau_{D'}^{-1}(S_{D'}^{\log}(j_{!}\mf))=Z_{d}(\tau_{D'}^{-1}(S_{D'}^{\log}(j_{!}\mf))$ and
the morphism $\tau_{D'}^{!}$ (\ref{gysin}) defines the morphism
\begin{equation}
\label{taudpgysin}
\tau_{D'}^{!}\colon Z_{d}(S_{D'}^{\log}(j_{!}\mf))\rightarrow Z_{d}(\tau_{D'}^{-1}(S_{D'}^{\log}(j_{!}\mf))
\end{equation}
of groups of $d$-cycles.

\begin{prop}
\label{propinvcc}
Suppose that $X$ is purely of dimension $d$ and that $D$ has simple normal crossings.
Let $I'\subset I$ be a subset containing $I_{\mT,\mf}$ (Definition \ref{defindsub} (1))
and let $D'=\bigcup_{i\in I'}D_{i}$.
Assume that the ramification of $\mf$ is $\log$-$D'$-clean along $D$
and that the inverse image $\tau_{D'}^{-1}(S_{D'}^{\log}(j_{!}\mf))\subset T^{*}X$
of $\log$-$D'$-singular support $S_{D'}^{\log}(j_{!}\mf)\subset T^{*}X(\log D')$ by $\tau_{D'}$ (\ref{deftaue}) 
is of dimension $d$.
Let $h\colon W\rightarrow X$ be a $C_{D}$-transversal and properly $\tau_{D'}^{-1}(S_{D'}^{\log}(j_{!}\mf))$-transversal morphism (Definition \ref{defproptrans} (1)) 
from a smooth scheme $W$ over $k$ of pure dimension $e$,
where $C_{D}$ is as in (\ref{defcdp}).
Let $j'\colon h^{*}U\rightarrow W$ denote the base change of $j$ by $h$.
Then we have the following:
\begin{enumerate}
\item The ramification of $h^{*}\mf$ is $\log$-$h^{*}D'$-clean along $h^{*}D$
and we have $D_{\mT,h^{*}\mf}\subset h^{*}D'$ (\ref{defd*mf}).
The inverse image $\tau_{h^{*}D'}^{-1}(S_{h^{*}D'}^{\log}(j'_{!}h^{*}\mf))\subset T^{*}W$ of  
is equal to $h^{\circ}\tau_{D'}^{-1}(S_{D'}^{\log}(j_{!}\mf))$, where $h^{\circ}$ is as in Definition \ref{defctr} (1), and
is of dimension $e$.

\item Let $\tau_{D'}^{!}$ and $\tau_{h^{*}D'}^{!}$ be as in (\ref{taudpgysin}).
Then we have
\begin{equation}
\tau_{h^{*}D'}^{!}CC_{h^{*}D'}^{\log}(j'_{!}h^{*}\mf)
=h^{!}\tau_{D'}^{!}CC_{D'}^{\log}(j_{!}\mf) \notag
\end{equation}
in $Z_{e}(\tau_{h^{*}D'}^{-1}(S_{h^{*}D'}^{\log}(j'_{!}h^{*}\mf)))$, where $h^{!}$ 
is as in Definition \ref{defproptrans} (2).
\end{enumerate}
\end{prop}

\begin{proof}
Since $h$ is $C_{D'}$-transversal by Remark \ref{remctr} (3),
the morphism $h$ is $\log$-$D'$-$S_{D'}^{\log}(j_{!}\mf)$-transversal
by Proposition \ref{proplogtr}.
Since $h$ is $C_{D}$-transversal, the morphism $h$ is properly $\log$-$D'$-$S_{D'}^{\log}(j_{!}\mf)$-transversal,
the ramification of $h^{*}\mf$ is $\log$-$h^{*}D'$-clean along $h^{*}D$,
we have $D_{\mT,h^{*}\mf}\subset h^{*}D'$, 
and we have
\begin{equation}
\label{eqcomgtaud}
CC_{h^{*}D'}^{\log}(j'_{!}h^{*}\mf)=h^{!}CC_{D'}^{\log}(j_{!}\mf),
\end{equation}
where $h^{!}$ is as in Definition \ref{deflogproptr} (2),
by Proposition \ref{proppblogcc}.

We consider the commutative diagram
\begin{equation}
\label{diaggysin}
\xymatrix{
T^{*}W  
\ar[d]_-{\tau_{h^{*}D'}} & T^{*}X\times_{X}W \ar[l]_-{dh} \ar[r]^-{\pr_{1}} \ar[d]^-{\tau_{D',W}}  \ar@{}[dr] | {\square} & T^{*}X \ar[d]^-{\tau_{D'}} \\
T^{*}W(\log h^{*}D') & T^{*}X(\log D')\times_{X}W \ar[l]^-{dh^{D'}} \ar[r]_-{\pr_{1}} & T^{*}X(\log D'),
}
\end{equation}
where the right square is cartesian and the morphisms 
$dh$ and $dh^{D'}$ are as in (\ref{defdh}) and (\ref{defdhdp}),
respectively.
Since $h$ is $C_{D'}$-transversal,
the left square in (\ref{diaggysin}) is cartesian 
by \cite[Lemma 3.8 (ii)]{yacc}.
Therefore we have
\begin{equation}
\label{eqtaugdpcar}
\tau_{h^{*}D'}^{-1}(h^{\circ}S_{D'}^{\log}(j_{!}\mf))
=h^{\circ}\tau_{D'}^{-1}(S_{D'}^{\log}(j_{!}\mf)).
\end{equation}
Since $h$ is $C_{D}$-transversal, the pull-back
$h^{*}D$ is a divisor on $W$ with simple normal crossings by \cite[Lemma 3.4.5]{sacc},
and the pull-back $h^{*}Z_{\mf}$ is reduced.
Thus we have $h^{\circ}S_{D'}^{\log}(j_{!}\mf)=S_{h^{*}D'}^{\log}(j'_{!}h^{*}\mf)$
by Proposition \ref{propeqcl} (2),
and the left-hand side of (\ref{eqtaugdpcar}) is equal to
$\tau_{h^{*}D'}^{-1}(S_{h^{*}D'}^{\log}(j'_{!}h^{*}\mf))$.
Since the right-hand side of (\ref{eqtaugdpcar}) is of dimension $e$
by the proper $\tau_{D'}^{-1}(S_{D'}^{\log}(j_{!}\mf))$-transversality of $h$
and Remark \ref{remctr} (1),
the assertion (1) holds.
Since the left square in (\ref{diaggysin}) is cartesian, 
the assertion (2) holds by (\ref{eqcomgtaud}) and the commutativity of (\ref{diaggysin}).
\end{proof}

\subsection{Reduction to the totally wildly ramified case}
\label{ssredtw}

As a preparation for following subsections,
we prove the following proposition:

\begin{prop}
\label{proplogtrclpp}
Suppose that $D$ has simple normal crossings.
Let $I'\subset I$ be a subset containing $I_{\mT,\mf}$ (Definition \ref{defindsub} (1))
and let $D'=\bigcup_{i\in I'}D_{i}$.
Assume that the ramification of $\mf$ is $\log$-$D'$-clean along $D$.
Let $I'''\subset I''\subset I_{\mT,\mf}$ be subsets.
We put $E=\bigcup_{i\in I-I''}D_{i}$, $E'=\bigcup_{i\in I'-I''}D_{i}$, and
$V=X-E$.
Let $\mf'$ be a smooth sheaf of $\Lambda$-modules of rank 1 on $V$ whose 
associated character $\chi'\colon \pi_{1}^{\ab}(V)\rightarrow \Lambda^{\times}$ has the $p$-part inducing
the $p$-part of $\chi$ and let $i'\colon D_{I'''}=\bigcap_{i\in I'''}D_{i} \rightarrow X$
denote the canonical closed immersion.  
Then we have the following:
\begin{enumerate} 
\item $I_{\mT,\mf'}$ is contained in $I'-I''$ and the closed immersion $i'$ is $C_{E}$-transversal,
where $C_{E}$ is as in (\ref{defcdp}).
\item $R_{i'^{*}\mf'}^{i'^{*}E'}=i'^{*}R_{\mf'}^{E'}$, $Z_{i'^{*}\mf'}=i'^{*}Z_{\mf'}$, and $\cform^{i'^{*}E'}(i'^{*}\mf')
=di'^{E'}_{i'^{*}Z_{\mf'}^{1/p}}(i'^{*}\cform^{E'}(\mf'))$,
where $di'^{E'}_{i'^{*}Z_{\mf'}^{1/p}} \colon i'^{*}\Omega^{1}_{X}(\log E')|_{i'^{*}Z_{\mf'}^{1/p}}\rightarrow \Omega_{D_{I'''}}^{1}(\log i'^{*}E')
|_{i'^{*}Z_{\mf'}^{1/p}}$ is the morphism induced by $i'$.
The ramification of $i'^{*}\mf'$ is
$\log$-$i'^{*}E'$-clean along $i'^{*}E$. 
\end{enumerate}
\end{prop}

In order to prove Proposition \ref{proplogtrclpp}, we prove several properties of
the canonical closed immersion $i_{r}\colon D_{r}\rightarrow X$ for $r\in I_{\mT,\mf}$
related to its logarithmic transversality.

\begin{defn}
\label{deffcircclog}
Let $I'\subset I$ be a subset such that
$D'=\bigcup_{i\in I'}D_{i}$ has simple normal crossings.
Let $C_{D'}$ be as in (\ref{defcdp}).
Let $f\colon X'\rightarrow X$ be a $C_{D'}$-transversal proper morphism of smooth schemes over $k$. 
For a closed conical subset $C'\subset T^{*}X'(\log f^{*}D')$,
we define a closed conical subset 
\begin{equation}
f_{\circ}C'\subset T^{*}X(\log D') \notag
\end{equation}
to be the image by the first projection $\pr_{1}\colon T^{*}X(\log D')\times_{X}X'\rightarrow T^{*}X(\log D')$ of the inverse image of $C'$ by the morphism
$df^{D'}\colon T^{*}X(\log D')\times_{X}X'\rightarrow T^{*}X'(\log f^{*}D')$
(\ref{defdhdp}).
\end{defn}

Let $I'\subset I$ be a subset such that $D'=\bigcup_{i\in I'}D_{i}$ has simple normal crossings
and let $I'''\subset I''\subset I'$ be subsets.
We put $E'=\bigcup_{i\in I'-I''}D_{i}$ and we denote the canonical closed immersion by $i'\colon D_{I'''}=\bigcap_{i\in I'''}D_{i}\rightarrow X$.
Then we construct a canonical injection
\begin{equation}
\label{iotaed}
i'_{E'/D'}\colon \Omega^{1}_{D_{I'''}}(\log i'^{*}E')\rightarrow \Omega_{X}^{1}(\log D')|_{D_{I'''}} \end{equation}
of locally free $\dvr_{D_{I'''}}$-modules as follows:
Let $\mathcal{I}\subset \dvr_{X}$ denote the defining ideal sheaf of $D_{I'''}\subset X$.
We consider the sequence 
\begin{equation}
\label{exseqlogsch}
\mathcal{I}/\mathcal{I}^{2}\xrightarrow{\bar{d}}
\Omega^{1}_{X}(\log E')|_{D_{I'''}}\xrightarrow{f}
\Omega^{1}_{D_{I'''}}(\log i'^{*}E')\rightarrow 0,
\end{equation}
where the morphism $\bar{d}$ sends the class of a section $a$ of $\mathcal{I}$ to
the image of the section $da$ of $\Omega_{X}^{1}|_{D_{I'''}}$ in $\Omega_{X}^{1}(\log E')|_{D_{I'''}}$ and 
the morphism $f$
is the morphism yielded by $i'$. 
Since $i'$ is a strict closed immersion for the log structures
on $D_{I'''}$ and $X$ defined by $i'^{*}E'$ and $E'$, respectively,
the sequence (\ref{exseqlogsch}) is exact  
by \cite[IV, Proposition 2.3.2]{og}.
We consider the exact sequence
\begin{equation}
\mathcal{I}/\mathcal{I}^{2}\xrightarrow{\bar{d}} \Omega_{X}^{1}(\log E')|_{D_{I'''}}
\xrightarrow{\tau_{E'/D',D_{I'''}}} \Omega_{X}^{1}(\log D')|_{D_{I'''}}, \notag
\end{equation}
where the morphism $\tau_{E'/D',D_{I'''}}$ is the canonical morphism.
Then the morphism $\tau_{E'/D',D_{I'''}}$ induces an injection $i'_{E'/D'}$: 
\begin{equation}
\label{diagcancl}
\xymatrix{
\Omega_{X}^{1}(\log E')|_{D_{I'''}}\ar[rr]^-{\tau_{E'/D',D_{I'''}}} \ar[dr]_-{f}& 
& \Omega_{X}^{1}(\log D')|_{D_{I'''}} \\
& \Omega_{D_{I'''}}^{1}(\log i'^{*}E'). \ar[ur]_-{i'_{E'/D'}}
}
\end{equation}
If $I'''=I''$, then the image of $\tau_{E'/D',D_{I'''}}$ is locally a direct summand of rank $\dim X-\sharp I'''$, and the image of $i'_{E'/D',D_{I'''}}$ is 
locally a direct summand of $\Omega_{X}^{1}(\log D')|_{D_{I'''}}$ 
of rank $\dim X -\sharp I'''$.

\begin{lem}[cf.\ {\cite[Lemma 3.12]{yacc}}]
\label{lemlogtrcl}
Let $I'\subset I$ be a subset such that
$D'=\bigcup_{i\in I'}D_{i}$ has simple normal crossings
and let $r\in I'$. 
We put $E=\bigcup_{i\in I-\{r\}}D_{i}$, $E'=\bigcup_{i\in I'-\{r\}}D_{i}$, and $V=X-E$.
Let $i_{r}\colon D_{r}\rightarrow X$ denote the canonical closed immersion and
let $j_{v}\colon V\rightarrow X$ denote the canonical open immersion.
\begin{enumerate}
\item Let $C\subset T^{*}X(\log E')$ be a closed conical subset.
If the closed immersion $i_{r}$ is $\log$-$E'$-$C$-transversal,
then we have 
\begin{equation}
\tau_{E'/D'}^{-1}(\tau_{E'/D'}(C))=C\cup {i_{r}}_{\circ}i_{r}^{\circ}C, \notag
\end{equation}
where $\tau_{E'/D'}$ is an in (\ref{deftauepse}). 
\item Suppose that $D$ has simple normal crossings,
that $I_{\mT,\mf}$ (Definition \ref{defindsub} (1)) is contained in $I'$,
and that $r\in I_{\mT,\mf}$.
Let $\mf'$ be a smooth sheaf of $\Lambda$-modules of rank 1 on $V=X-E$
whose associated character $\chi'\colon \pi_{1}^{\ab}(V)\rightarrow \Lambda^{\times}$ has the $p$-part induing the $p$-part of $\chi$.
Then the following are equivalent:
\begin{enumerate}
\item The ramification of $\mf$ is $\log$-$D'$-clean along $D$.
\item The ramification of $\mf'$ is $\log$-$E'$-clean along $E$ and the closed immersion $i_{r}$ is $\log$-$E'$-$S_{E'}^{\log}(j_{v!}\mf')$-transversal (Definition \ref{deflifdp} (3)).
\end{enumerate}
\item Let the assumptions and the notation be as in (2).  
Assume that the equivalent conditions (a) and (b) in (2) hold. 
Let $j_{v}'\colon D'\cap V\rightarrow D'$ denote the canonical open immersion.
Then we have two equalities
\begin{equation}
\label{issfp}
L_{i_{r}^{*}\mf'}^{i_{r}^{*}E'}=i_{r}^{\circ}L_{\mf'}^{E'}
\end{equation}
and
\begin{equation}
\label{tauinvcont}
\tau_{E'/D'}^{-1}(L_{\mf}^{D'})= L_{\mf'}^{E'}\cup 
{i_{r}}_{\circ}L^{i_{r}^{*}E'}_{i_{r}^{*}\mf'},
\end{equation}
where $L_{i_{r}^{*}\mf'}^{i_{r}^{*}E'}$, $L_{\mf'}^{E'}$, and $L_{\mf}^{D'}$
are as in Definition \ref{deflifdp} (2).
\end{enumerate}
\end{lem}

\begin{proof}
(1) The assertion in nothing but \cite[Lemma 3.12 (i)]{yacc}.

(2) By Lemma \ref{lemclrelsecs}, the condition (a) implies the first condition in (b).
Thus we may assume that the first condition in (b) holds.
By Lemma \ref{lemclrelsecf}, 
the $\log$-$D'$-characteristic form $\cform^{D'}(\mf)$ is the image of $\cform^{E'}(\mf')$ by the canonical morphism 
\begin{equation}
\label{canmffpom}
\Omega_{X}^{1}(\log E')(R_{\mf'}^{E'})|_{Z_{\mf'}^{1/p}}
\rightarrow \Omega_{X}^{1}(\log D')(R_{\mf}^{D'})|_{Z_{\mf}^{1/p}}, 
\end{equation}
where $Z_{\mf'}=Z_{\mf}$ and $R_{\mf'}^{E'}=R_{\mf}^{D'}$.
Since the morphism (\ref{canmffpom}) of sheaves on $Z_{\mf}^{1/p}$ is an isomorphism outside $i_{r}^{*}Z_{\mf}^{1/p}=Z_{\mf}^{1/p}\cap D_{r}^{1/p}$,
the first condition in (b) implies 
the $\log$-$D'$-cleanliness of the ramification of $\mf$ along $D$
at every point on the complement $X-i_{r}^{*}Z_{\mf}$. 
Therefore it is sufficient to prove the equivalence of the $\log$-$D'$-cleanliness of the ramification of $\mf$ along $D$ at every closed point of
$i_{r}^{*}Z_{\mf}$ and the second condition in (b) by Remark \ref{remlogdpcl} (2).

We consider the commutative diagram
\begin{equation}
\label{modlocdsepfp}
\xymatrix{
\Omega_{X}^{1}(\log E')(R_{\mf}^{E'})|_{i_{r}^{*}Z_{\mf'}^{1/p}}\ar[rr]^-{\tau_{E'/D',i_{r}^{*}Z_{\mf}^{1/p}}} \ar[dr]_-{f_{i_{r}^{*}Z_{\mf'}^{1/p}}}& 
& \Omega_{X}^{1}(\log D')(R_{\mf}^{D'})|_{i_{r}^{*}Z_{\mf}^{1/p}} \\
& \Omega_{D_{r}}^{1}(\log i_{r}^{*}E')(i_{r}^{*}R_{\mf'}^{E'})|_{i_{r}^{*}Z_{\mf'}^{1/p}}. \ar[ur]_-{\ \ \ \ \ \ \ i_{r, E'/D',i_{r}^{*}Z_{\mf}^{1/p}}} 
}
\end{equation}
induced by (\ref{diagcancl}) with $I'''=I''=\{r\}$ and $i'=i_{r}$.
Since the image of $\Omega_{D_{r}}^{1}(\log i_{r}^{*}E')$ by
the injection $i_{r, E'/D'}$ (\ref{iotaed}) is locally a direct summand 
of $\Omega_{X}^{1}(\log D')(R_{\mf}^{D'})|_{D_{r}}$,
the morphism $i_{r, E'/D',i_{r}^{*}Z_{\mf}^{1/p}}$ in (\ref{modlocdsepfp}) is an injection.
By Lemma \ref{lemeqtologdpcl} (1), the ramification of $\mf$ is $\log$-$D'$-clean
along $D$ at every closed point of $i_{r}^{*}Z_{\mf}$ if and only if
the image by $\tau_{E'/D',i_{r}^{*}Z_{\mf}^{1/p}}$ of 
\begin{equation}
\dvr_{X}(R_{\mf'}^{E'})|_{i_{r}^{*}Z_{\mf'}^{1/p}}\cdot \cform^{E'}(\mf')|_{i_{r}^{*}Z_{\mf}^{1/p}}\subset \Omega_{X}^{1}(\log E')(R_{\mf'}^{E'})|_{i_{r}^{*}Z_{\mf'}^{1/p}} \notag
\end{equation}
is locally a direct summand of $\Omega_{X}^{1}(\log D')(R_{\mf}^{D'})|_{i_{r}^{*}Z_{\mf}^{1/p}}$ of rank $1$.
By Proposition \ref{propeqcl} (1),
the last condition is equivalent to the second condition in (b).

(3) By Proposition \ref{propeqcl} (2), we obtain the equality (\ref{issfp}). 
By Lemma \ref{lemclrelsecf}, we have 
\begin{equation}
\label{eqsdfsefplog}
L^{D'}_{\mf}=\tau_{E'/D'}(L^{E'}_{\mf'}).
\end{equation}
By taking the inverse images by $\tau_{E'/D'}$ of (\ref{eqsdfsefplog}),
we have
\begin{equation}
\label{containsdfsefp}
\tau_{E'/D'}^{-1}(L^{D'}_{\mf})=
\tau_{E'/D'}^{-1}(\tau_{E'/D'}(L^{E'}_{\mf'})). 
\end{equation}
Then we obtain the equality (\ref{tauinvcont})
by applying (1) to the right-hand side of (\ref{containsdfsefp})
and then applying the equality (\ref{issfp}).
\end{proof}

We prove Proposition \ref{proplogtrclpp}.

\begin{proof}[Proof of Proposition \ref{proplogtrclpp}]
(1) By Lemma \ref{lemclrelsecf}, we have $I_{\mT,\mf'}=I_{\mT,\mf}-I''\subset I'-I''$.
Since $D$ has simple normal crossings, the closed immersion $i'$ is $C_{E}$-transversal 
by \cite[Lemma 3.4.5]{sacc}.

(2) If $I'''=\emptyset$,
then the closed immersion $i'$ is the identity mapping of $X$
and there is nothing to prove except the last assertion.
The last assertion follows from  
Lemma \ref{lemclrelsecs}. 
Hence we may assume that $I'''\neq \emptyset$.
Then we have $I''\neq \emptyset$.

We prove the assertions by the induction on the cardinality of $I''$.
Suppose that the cardinality of $I''$ is $1$. 
Then we have $I'''=I''$, and
the closed immersion $i'$ is $\log$-$E'$-$S_{E'}^{\log}(j_{v!}\mf')$-transversal by Lemma \ref{lemlogtrcl} (2).
Thus the assertions follow from Proposition \ref{propeqcl} (1).

Suppose that the cardinality of $I''$ is $\ge 2$.  
We take $r\in I'''$. 
We put $F=\bigcup_{i\in I-\{r\}}D_{i}$, $F'=\bigcup_{i\in I'-\{r\}}D_{i}$, $W=X-F$, and $\mf''=\mf'|_{W}$.
Then $\mf''$ is a smooth sheaf of $\Lambda$-modules of rank $1$
on $W$ whose  associated character $\chi''\colon \pi_{1}^{\ab}(W)\rightarrow \Lambda^{\times}$ 
is induced by $\chi'$ and has the $p$-part inducing the $p$-part of $\chi$. 
By the case where $I'''=\emptyset$,
the ramification of $\mf'$ is $\log$-$E'$-clean along $E$
and the 
ramification of $\mf''$ is $\log$-$F'$-clean along $F$.
By Lemma \ref{lemclrelsecf},
we have $R_{\mf'}^{E'}=R_{\mf}^{D'}=R_{\mf''}^{F'}$, $Z_{\mf'}=Z_{\mf}=Z_{\mf''}$,
$I_{\mT,\mf}=I_{\mT,\mf''}-(I''-\{r\})$, and $I_{\mT,\mf''}=I_{\mT,\mf}-\{r\}\subset I'-\{r\}$.

Let $i_{r}\colon D_{r}\rightarrow X$ and $i'_{r}\colon D_{I'''}\rightarrow D_{r}$ denote the canonical closed immersions.
Then we have $i'=i_{r}\circ i_{r}'$.
By the case where the cardinality of $I''$ is $1$,
we have $R_{i_{r}^{*}\mf''}^{i_{r}^{*}F'}=i_{r}^{*}R_{\mf''}^{F'}$ and
$Z_{i_{r}^{*}\mf''}=i_{r}^{*}Z_{\mf''}$
and the ramification of $i_{r}^{*}\mf''$
is $\log$-$i_{r}^{*}F'$-clean along $i_{r}^{*}F$.
Then the last equality enables us to identify
$I_{\mT,i_{r}^{*}\mf''}$ with $I_{\mT,\mf''}=I_{\mT,\mf}-\{r\}$ locally.
Since the character $i_{r}^{*}\chi''$ corresponding to $i_{r}^{*}\mf''$
is induced by the character $i_{r}^{*}\chi'$ corresponding to $i_{r}^{*}\mf'$,
we have $R_{i_{r}^{*}\mf'}^{i_{r}^{*}E'}=R_{i_{r}^{*}\mf''}^{i_{r}^{*}F'}$ and
$Z_{i_{r}^{*}\mf'}=Z_{i_{r}^{*}\mf''}$ by Lemma \ref{lemclrelsecf}.
Therefore we have $R_{i_{r}^{*}\mf'}^{i_{r}^{*}E'}=i_{r}^{*}R_{\mf''}^{F'}=i_{r}^{*}R_{\mf'}^{E'}$
and $Z_{i_{r}^{*}\mf'}=i_{r}^{*}Z_{\mf''}=i_{r}^{*}Z_{\mf'}$, and
we can locally identify $I_{\mT,i_{r}^{*}\mf'}$ with $I_{\mT,\mf'}=I_{\mT,\mf''}-(I''-\{r\})
=I_{\mT,i_{r}^{*}\mf''}-(I''-\{r\})$.
Since we can locally identify the index set of irreducible components of $i_{r}^{*}F'$
with $I'-\{r\}$,
we can apply
the induction hypothesis to $i_{r}^{*}\mf''$, $i_{r}^{*}\mf'$, and 
$i_{r}'$.
Then we have $R_{i_{r}'^{*}i_{r}^{*}\mf'}^{i_{r}'^{*}i_{r}^{*}E'}=i_{r}'^{*}R_{i_{r}^{*}\mf'}^{i_{r}^{*}E'}
=i_{r}'^{*}i_{r}^{*}R_{\mf'}^{E'}$
and the ramification of $i_{r}'^{*}i_{r}^{*}\mf'$ is $\log$-$i_{r}'^{*}i_{r}^{*}E'$-clean along $i_{r}'^{*}i_{r}^{*}E$.
Thus the first and last assertions hold, since $i'=i_{r}\circ i_{r}'$.
The other two assertions hold by Lemma \ref{lemcformgl} (2).
\end{proof}

\begin{cor}
\label{corlogtrans}
Let the notation and the assumptions be as in Proposition \ref{proplogtrclpp}.
Then the closed immersion $i'\colon D_{I'''}\rightarrow X$ is
$\log$-$E'$-$S_{E'}^{\log}(j'_{!}\mf')$-transversal. 
\end{cor}

\begin{proof}
Since the closed immersion $i'$ is $C_{E'\subset E}$-transversal by
Remark \ref{remctr} (3) and Proposition \ref{proplogtrclpp} (1),
where $C_{E'\subset E}$ is as in (\ref{defcdpd}),
the assertion follows from Propositions \ref{propeqcl} (1) and \ref{proplogtrclpp} (2).
\end{proof}

\subsection{Singular support and $\log$-$D'$-singular support}
\label{ssinvimlogss}

We construct a closed conical subset $S_{D'}(j_{!}\mf)\subset T^{*}X$
by using the inverse image $\tau_{D'}^{-1}(S_{D'}^{\log}(j_{!}\mf))\subset T^{*}X$ of 
the $\log$-$D'$-singular support $S_{D'}^{\log}(j_{!}\mf)\subset T^{*}X(\log D')$ of $j_{!}\mf$ (Definition \ref{deflifdp} (3)) by $\tau_{D'}$ (\ref{deftaue})
and prove that the singular support $SS(j_{!}\mf)$ is contained in the closed conical subset 
$S_{D'}(j_{!}\mf)\subset T^{*}X$.

\begin{defn}
\label{defsdpf}
Assume that $D$ has simple normal crossings.
Let $I'\subset I$ be a subset containing $I_{\mT,\mf}$ (Definition \ref{defindsub} (1))
and let $D'=\bigcup_{i\in I'}D_{i}$.
Assume that the ramification of $\mf$ is $\log$-$D'$-clean along $D$.
Let $i_{I''}\colon D_{I''}=\bigcap_{i\in I''}D_{i}\rightarrow X$ denote the canonical closed immersion for each $I''\subset I_{\mT,\mf}$.
We put $D_{\mW,\mf}'=\bigcup_{i\in I'\cap I_{\mW,\mf}}D_{i}$.
Then we define a closed conical subset
$S_{D'}(j_{!}\mf)\subset T^{*}X$ by
\begin{equation}
S_{D'}(j_{!}\mf)=\bigcup_{I''\subset I_{\mT,\mf}}i_{I''\circ}C_{i_{I''}^{*}D'_{\mW,\mf}\subset i_{I''}^{*}D_{\mW,\mf}}\cup \tau_{D'}^{-1}(S^{\log}_{D'}(j_{!}\mf)),  \notag
\end{equation}
where $i_{I''\circ}$ and $C_{i_{I''}^{*}D'_{\mW,\mf}\subset i_{I''}^{*}D_{\mW,\mf}}$ are as in Definition \ref{defpfc}
and (\ref{defcdpd}), respectively, and 
$\tau_{D'}^{-1}(S_{D'}^{\log}(j_{!}\mf))$ is the inverse image
of the $\log$-$D'$-singular support $S_{D'}^{\log}(j_{!}\mf)$ of $j_{!}\mf$ (Definition \ref{deflifdp} (3)) by
$\tau_{D'}$ (\ref{deftaue}).
\end{defn}

\begin{rem}
\label{remdefsdpf}
Let the notation and the assumptions be as in Definition \ref{defsdpf}.
\begin{enumerate}
\item Since $i_{I''\circ}C_{i_{I''}^{*}D'_{\mW,\mf}\subset i_{I''}^{*}D_{\mW,\mf}}$
for $I''\subset I_{\mT,\mf}$ is a union of irreducible components of $C_{D}$ (\ref{defcdp}),
the closed conical subset $S_{D'}(j_{!}\mf)\subset T^{*}X$ is contained in
the union $C_{D}\cup \tau_{D'}^{-1}(S_{D'}^{\log}(j_{!}\mf))$.
If $I_{\mT,\mf}=\emptyset$, then 
the closed conical subset $S_{D'}(j_{!}\mf)\subset T^{*}X$ is equal to
the union $C_{D}\cup \tau_{D'}^{-1}(S_{D'}^{\log}(j_{!}\mf))$.

\item The closed conical subset $S_{D'}(j_{!}\mf)\subset T^{*}X$ is stable under the replacement of $\chi$ by the $p$-part of $\chi$ 
by Remarks \ref{remdefitw} (2) and \ref{remlifdp} (2).
\end{enumerate}
\end{rem}

Let 
\begin{equation}
j_{w}\colon U_{\mW,\mf}=X-D_{\mW,\mf}\rightarrow X \notag
\end{equation}
be the canonical open immersion,
where $D_{\mW,\mf}$ is as in (\ref{defd*mf}).
Let $\mf_{\mW}$ be a smooth sheaf of $\Lambda$-modules of rank 1 on $U_{\mW,\mf}$ 
such that the character
$\chi_{\mW}\colon \pi_{1}^{\ab}(U_{\mW,\mf})\rightarrow \Lambda^{\times}$ corresponding to $\mf_{\mW}$ has the $p$-part
inducing the $p$-part of $\chi$.
For the proof of the inclusion $SS(j_{!}\mf)\subset S_{D'}(j_{!}\mf)$,
we use the following {\bf (Induction Step)} on the cardinality of $I_{\mT,\mf}$
(Definition \ref{defindsub} (1))
under the assumptions in Definition \ref{defsdpf}:

\bigskip
\noindent
{\bf (Induction Step)}:
Suppose that $\mf$ satisfies a condition (P) when $I_{\mT,\mf}=\emptyset$.
In the case where the cardinality of $I_{\mT,\mf}$ is $\ge 1$, then
we take $r\in I_{\mT,\mf}$.
We put $E=\bigcup_{i\in I-\{r\}}D_{i}$, $E'=\bigcup_{i\in I'-\{r\}}D_{i}$,  $V=X-E$, and $\mf'=\mf_{\mW}|_{V}$.
By Lemma \ref{lemclrelsecf}, we have $I_{\mT,\mf'}=I_{\mT,\mf}-\{r\}\subset I'-\{r\}$.
By Lemma \ref{lemclrelsecs}, the ramification of $\mf'$ is $\log$-$E'$-clean along $E$.
Let $i_{r}\colon D_{r}\rightarrow X$ denote the canonical closed immersion and let $j_{v}\colon V\rightarrow X$ denote the canonical open immersion.
We locally identify the index set $I_{r}'$ of $i_{r}^{*}E'$ with $I'-\{r\}$.
By Proposition \ref{proplogtrclpp} (2), we can locally identify $I_{\mT,i_{r}^{*}\mf'}$ with $I_{\mT,\mf'}=I_{\mT,\mf}-\{r\}$ and
the ramification of $i_{r}^{*}\mf'$ is $\log$-$i_{r}^{*}E'$-clean along $i_{r}^{*}E$.
Then we can apply the induction hypothesis to each of $\mf'$ and $i_{r}^{*}\mf'$, and 
it follows that each of them satisfies the condition (P).

\begin{lem}
\label{lemsss}
Suppose that $D$ has simple normal crossings.
Let $I'\subset I$ be a subset containing $I_{\mT,\mf}$ (Definition \ref{defindsub} (1))
and let $D'=\bigcup_{i\in I'}D_{i}$.
Assume that the ramification of $\mf$ is $\log$-$D'$-clean along $D$.
Let $\mf_{\mW}$ be a smooth sheaf of $\Lambda$-modules of rank 1 on $U_{\mW,\mf}=X-D_{\mW,\mf}$ (\ref{defd*mf}) such that the $p$-part of the character $\chi_{\mW}\colon \pi_{1}^{\ab}(U_{\mW,\mf})\rightarrow \Lambda^{\times}$ corresponding to $\mf_{\mW}$ induces the $p$-part of $\chi$.
Let $i_{I''}\colon D_{I''}=\bigcap_{i\in I''}D_{i}\rightarrow X$ denote the canonical closed immersion for $I''\subset I_{\mT,\mf}$
and let $j_{w,I''}\colon i_{I''}^{*}U_{\mW,\mf}\rightarrow D_{I''}$ denote the canonical open immersion.
We put $D_{\mW,\mf}'=\bigcup_{i\in I'\cap I_{\mW,\mf}}D_{i}$.
Then we have
\begin{equation}
\label{defsdjfgen}
S_{D'}(j_{!}\mf)=\bigcup_{I''\subset I_{\mT,\mf}}
{i_{I''}}_{\circ}S_{i_{I''}^{*}D'_{\mW,\mf}}(j_{w,I''!}i_{I''}^{*}\mf_{\mW}),
\end{equation}
where the ramification of $i_{I''}^{*}\mf_{\mW}$ is $\log$-$i_{I''}^{*}D_{\mW,\mf}'$-clean 
along $i_{I''}^{*}D_{\mW,\mf}$ for every $I''\subset I_{\mT,\mf}$ by Proposition \ref{proplogtrclpp} (2).
\end{lem}

\begin{proof}
We use the induction on the cardinality of $I_{\mT,\mf}$. 
If $I_{\mT,\mf}=\emptyset$,
then $D'=D'_{\mW,\mf}$ and the assertion follows from Remark \ref{remdefsdpf} (2).

Suppose that $I_{\mT,\mf}\neq \emptyset$.
We take $r\in I_{\mT,\mf}\subset I'$ and
we put $E=\bigcup_{i\in I-\{r\}}D_{i}$, $E'=\bigcup_{i\in I'-\{r\}}D_{i}$, $V=X-E$, and $\mf'=\mf_{\mW}|_{V}$. 
Then we have $I_{\mW,\mf}=I_{\mW,\mf'}$ by Lemma \ref{lemclrelsecf},
since the $p$-part of $\chi$ is induced by that of the character $\chi'\colon \pi_{1}^{\ab}(V)\rightarrow \Lambda^{\times}$ corresponding to $\mf'$.
Hence we have $I'\cap I_{\mW,\mf}=(I'-\{r\})\cap I_{\mW,\mf'}$.
Let $i_{r}\colon D_{r}\rightarrow X$ be the canonical closed immersion.
Then we can locally identify the index set $I_{r}'$ of the irreducible components
of $i_{r}^{*}E'$ with $I'-\{r\}$.
By Proposition \ref{proplogtrclpp} (2), we can locally identify $I_{\mW,i_{r}^{*}\mf'}$ with $I_{\mW,\mf'}=I_{\mW,\mf}$, and $I'_{r}\cap I_{\mW,i_{r}^{*}\mf'}$
with $(I'-\{r\})\cap I_{\mW,\mf'}=I'\cap I_{\mW,\mf}$.
Let $j_{v}\colon V\rightarrow X$ and $j_{v}'\colon i_{r}^{*}V\rightarrow D_{r}$ be the canonical open immersions.
By applying {\bf (Induction Step)}, we have
\begin{equation}
\label{sepjvfppf}
S_{E'}(j_{v!}\mf')
=\bigcup_{I''\subset I_{\mT,\mf}-\{r\}}i_{I''\circ}S_{i_{I''}^{*}D_{\mW,\mf}'}
(j_{w,I''!}i_{I''}^{*}\mf_{\mW})
\end{equation}
and 
\begin{equation}
\label{sireirjvfp}
S_{i_{r}^{*}E'}(j_{v!}'i_{r}^{*}\mf')
= \bigcup_{I''\subset I_{\mT,\mf}-\{r\}}i_{I''\circ}'S_{i_{I''}'^{*}i_{r}^{*}D_{\mW,\mf}'}(j_{w,I''\cup\{r\}!}i_{I''}'^{*}i_{r}^{*}\mf_{\mW}),  
\end{equation}
where $i'_{I''}\colon i_{r}^{*}D_{I''}\rightarrow D_{r}$ denotes the base change of $i_{I''}$ by $i_{r}$.
By applying $i_{r\circ}$ to (\ref{sireirjvfp}) and then applying 
Lemma \ref{lemcompps} to $i_{I''\cup\{r\}}=i_{r}\circ i_{I''}'$ for $I''\subset I_{\mT,\mf}-\{r\}$, we have
\begin{align}
\label{irsireirjvfp}
i_{r\circ}S_{i_{r}^{*}E'}(j_{v!}'i_{r}^{*}\mf')
&=\bigcup_{\substack{I''\subset I_{\mT,\mf}\\ r\in I''}}i_{I''\circ}S_{i_{I''}^{*}D'_{\mW,\mf}}(j_{w,I''!}i_{I''}^{*}\mf_{\mW}). 
\end{align}
Since the union of the right-hand sides of (\ref{sepjvfppf}) and (\ref{irsireirjvfp}) is the right-hand side of (\ref{defsdjfgen}),
it is sufficient to prove the equality
\begin{equation}
\label{lemsssteirtecont}
S_{D'}(j_{!}\mf)=S_{E'}(j_{v!}\mf')\cup i_{r\circ}S_{i_{r}^{*}E'}(j_{v!}'i_{r}^{*}\mf'). 
\end{equation}

By applying Lemma \ref{lemcompps} to $i_{I''\cup\{r\}}=i_{r}\circ i_{I''}'$ for $I''\subset I_{\mT,\mf}-\{r\}$,
we have
\begin{align}
\bigcup_{I''\subset I_{\mT,\mf}}
&i_{I''\circ}C_{i_{I''}^{*}D'_{\mW,\mf}\subset i_{I''}^{*}D_{\mW,\mf}} \notag \\
&=\bigcup_{I''\subset I_{\mT,\mf}-\{r\}}i_{I''\circ}C_{i_{I''}^{*}D'_{\mW,\mf}\subset i_{I''}^{*}D_{\mW,\mf}}
\cup \bigcup_{I''\subset I_{\mT,\mf}-\{r\}}i_{r\circ}i'_{I''\circ}C_{i_{I''}'^{*}i_{r}^{*}D'_{\mW,\mf}\subset i_{I''}'^{*}i_{r}^{*}D_{\mW,\mf}}. \notag
\end{align}
By (\ref{eqinvzero}),
we have
\begin{align}
\tau_{D'}^{-1}(T^{*}_{X}X(\log D'))&=
\bigcup_{I''\subset I'-\{r\}}T^{*}_{D_{I''}}X\cup\bigcup_{\substack{I''\subset I' \\ r\in I''}}T^{*}_{D_{I''}}X \notag \\
&=\tau_{E'}^{-1}(T^{*}_{X}X(\log E'))
\cup i_{r\circ}\tau_{i_{r}^{*}E'}^{-1}(T^{*}_{D_{r}}D_{r}(\log i_{r}^{*}E')). \notag
\end{align}
By applying $\tau_{E'}^{-1}$ to the equation (\ref{tauinvcont}) in Lemma \ref{lemlogtrcl} (3), 
we have
\begin{equation}
\tau_{D'}^{-1}(L^{D'}_{\mf})=\tau_{E'}^{-1}(L^{E'}_{\mf'})\cup \tau_{E'}^{-1}
(i_{r\circ}L^{i_{r}^{*}E'}_{i_{r}^{*}\mf'}). \notag
\end{equation}
We consider the commutative diagram
\begin{equation}
\xymatrix{ T^{*}X \ar[d]_-{\tau_{E'}} \ar@{}[dr] | {\square} & T^{*}X\times_{X}D_{r} \ar[l]_-{\pr_{1}} \ar[r]^-{di_{r}} \ar[d]^-{\tau_{E', D_{r}}}  & T^{*}D_{r} \ar[d]^-{\tau_{i_{r}^{*}E'}} \\
T^{*}X(\log E') & T^{*}X(\log E')\times_{X}D_{r} \ar[l]^-{\pr_{1}} \ar[r]_-{di_{r}^{E'}} & T^{*}D_{r}(\log i_{r}^{*}E'), } \notag
\end{equation}
where the left square is cartesian.
Then we have
\begin{equation}
\tau_{E'}^{-1}(i_{r\circ}L^{i_{r}^{*}E'}_{i_{r}^{*}\mf'})
=i_{r\circ}\tau_{i_{r}^{*}E'}^{-1}(L^{i_{r}^{*}E'}_{i_{r}^{*}\mf'}), \notag
\end{equation}
and we obtain the desired equality (\ref{lemsssteirtecont}).
\end{proof}

\begin{rem}
\label{remlemsdpd}
Suppose that $D$ has simple normal crossings.
Let $I'\subset I$ be a subset containing $I_{\mT,\mf}$ (Definition \ref{defindsub} (1))
and let $D'=\bigcup_{i\in I'}D_{i}$.
Assume that the ramification of $\mf$ is $\log$-$D'$-clean along $D$.
As is seen in the proof of Lemma \ref{lemsss}, 
the closed conical subset $S_{D'}(j_{!}\mf)\subset T^{*}X$ (Definition \ref{defsdpf})
satisfies the following relation:
Let $r\in I_{\mT,\mf}$.
We put $E=\bigcup_{i\in I-\{r\}}D_{i}$, $E'=\bigcup_{i\in I'-\{r\}}D_{i}$, and $V=X-E$.
Let $\mf'$ be a smooth sheaf of $\Lambda$-modules of rank 1 on $V$ whose 
associated character $\chi'\colon \pi_{1}^{\ab}(V)\rightarrow \Lambda^{\times}$ has the $p$-part inducing
the $p$-part of $\chi$.
Let $i_{r}\colon D_{r}\rightarrow X$ denote the canonical closed immersion
and $j_{v}\colon V\rightarrow X$ the canonical open immersion.
Then we have 
\begin{equation}
S_{D'}(j_{!}\mf)=S_{E'}(j_{v!}\mf')\cup i_{r\circ}S_{i_{r}^{*}E'}(j_{v!}'i_{r}^{*}\mf'), \notag
\end{equation}
where  $j_{v}'\colon i_{r}^{*}V\rightarrow D_{r}$ denotes the base change of $j_{v}$ by $i_{r}$.
\end{rem}

\begin{prop}
\label{proptowild}
Suppose that $X$ is purely of dimension $d$ and that $D$ has simple normal crossings.
Let $I'\subset I$ be a subset containing $I_{\mT,\mf}$ (Definition \ref{defindsub} (1))
and let $D'=\bigcup_{i\in I'}D_{i}$.
Assume that the ramification of $\mf$ is $\log$-$D'$-clean along $D$.
Let $j_{w}\colon U_{\mW,\mf}=X-D_{\mW,\mf}\rightarrow X$ (\ref{defd*mf}) denote the canonical open 
immersion
and let $\mf_{\mW}$ be a smooth sheaf of $\Lambda$-modules 
of rank 1 on $U_{\mW,\mf}$ whose associated 
character $\chi_{\mW}\colon \pi_{1}^{\ab}(U_{\mW,\mf})\rightarrow \Lambda^{\times}$ 
has the $p$-part inducing the $p$-part of $\chi$.
Let $i_{I''}\colon D_{I''}=\bigcap_{i\in I''}D_{i}\rightarrow X$ denote the canonical closed immersion for $I''\subset I_{\mT,\mf}$
and let $j_{w,I''}\colon i_{I''}^{*}U_{\mW,\mf}\rightarrow D_{I''}$ denote the canonical open immersion.
Let $r_{I''}$ denote the cardinality of $I''$ for $I''\subset I_{\mT,\mf}$.
Then we have the following:
\begin{enumerate}
\item For the singular support $SS(j_{!}\mf)$ (Definition \ref{defss} (2)), we have
\begin{equation}
SS(j_{!}\mf)=\bigcup_{I''\subset I_{\mT,\mf}} {i_{I''}}_{\circ}SS(j_{w,I''!}i_{I''}^{*}\mf_{\mW}). \notag
\end{equation}
\item For the characteristic cycle $CC(j_{!}\mf)$ (Definition \ref{defcc}), we have
\begin{equation}
CC(j_{!}\mf)=\sum_{I''\subset I_{\mT,\mf}}(-1)^{r_{I''}} {i_{I''}}_{!}CC(j_{w,I''!}i_{I''}^{*}\mf_{\mW}) \notag
\end{equation}
in $Z_{d}(T^{*}X)$.
\end{enumerate}
\end{prop}

\begin{proof}
By Proposition \ref{propcorssccsm}, we may assume that both $\chi$
and $\chi_{\mW}$ are of orders powers of $p$.
Then $\mf$ is unramified along $D_{\mT,\mf}$ (\ref{defd*mf}),
since $\mf$ is tamely ramified along $D_{\mT,\mf}$ by Remark \ref{remtameloc}
(3),
and we have $\mf=\mf_{\mW}|_{U}$.

We prove the implication (2) $\Rightarrow$ (1).
Suppose that the equality in (2) holds.
By Corollary \ref{corssccsm} (2),
it is sufficient to prove that the support of the right-hand side of
the equality in (2) is equal to the right-hand side of the equality in (1).
By Corollary \ref{corssccsm} (1), we have $(-1)^{d-r_{I''}}CC(j_{w,I''!}i_{I''}^{*}\mf_{\mW})\ge 0$  
for all $I''\subset I_{\mT,\mf}$.
Therefore the support of the right-hand side of the equality in (2)
is equal to the union of the support of $i_{I''!}CC(j_{w,I''!}i_{I''}^{*}\mf_{\mW})$
for all $I''\subset I_{\mT,\mf}$.
Thus the assertion holds by Lemma \ref{lemsscclim}.

We prove (2) by the induction on the cardinality of $I_{\mT,\mf}$.
If $I_{\mT,\mf}=\emptyset$, then we have $\mf=\mf_{\mW}$, and
there is nothing to show.
Suppose that $I_{\mT,\mf}\neq \emptyset$. 
We take $r\in I_{\mT,\mf}$. 
We put $E=\bigcup_{i\in I-\{r\}}D_{i}$,  
$V=X-E$, and $\mf'=\mf_{\mW}|_{V}$.
Let $i_{r}\colon D_{r}\rightarrow X$ be the canonical closed immersion
and $j_{v}\colon V\rightarrow X$ 
the canonical open immersion.
Let $j_{v}'\colon i_{r}^{*}V\rightarrow D_{r}$ and $i_{I''}'\colon i_{r}^{*}D_{I''}\rightarrow D_{r}$ denote the base changes of $j_{v}$ and $i_{I''}$ by $i_{r}$, respectively.
Then we apply {\bf (Induction Step)} and we have 
\begin{equation}
\label{eqccirsum}
CC(j_{v!}\mf')=\sum_{I''\subset I_{\mT,\mf}-\{r\}}(-1)^{r_{I''}}{i_{I''}}_{!}CC(j_{w,I''!}i_{I''}^{*}\mf_{\mW})
\end{equation}
and
\begin{equation}
\label{eqccirfpsum}
CC(j'_{v!}i_{r}^{*}\mf')=\sum_{I''\subset I_{\mT,\mf}-\{r\}}(-1)^{r_{I''}}i'_{I''!}CC(j_{w,I''\cup\{r\}!}i'^{*}_{I''}i_{r}^{*}\mf_{\mW}).
\end{equation}
By applying $i_{r!}$ to (\ref{eqccirfpsum}) and then applying the last assertion of Lemma \ref{lemclimpush} to $i_{I''\cup\{r\}}=i_{r}\circ i'_{I''}$ for $I''\subset I_{\mT,\mf}-\{r\}$, we have
\begin{align}
\label{eqirccirfpsum}
i_{r!}CC(j_{v!}'i_{r}^{*}\mf') 
&=\sum_{\substack{I''\subset I_{\mT,\mf} \\ r\in I''}}(-1)^{r_{I''}-1}i_{I''!}CC(j_{w,I''!}i^{*}_{I''}\mf_{\mW}). 
\end{align}
Since the difference between the right-hand side of (\ref{eqccirsum}) 
and that of (\ref{eqirccirfpsum}) is the right-hand side of the desired equality,
it is sufficient to prove the equality
\begin{equation}
\label{eqsccred}
CC(j_{!}\mf)=CC(j_{v!}\mf')-i_{r!}CC(j_{v!}'i_{r}^{*}\mf').
\end{equation}

We consider the diagram
\begin{equation}
\label{cartduvdr}
\xymatrix{ 
U \ar[r] \ar[d] & V \ar[d]^-{j_{v}} \\ X-D_{r} \ar[r]_-{j_{r}} & X & D_{r} \ar[l]^-{i_{r}}, }
\end{equation}
where the arrows in the square are the canonical open immersions and
the square is cartesian.
Since $\mf'|_{U}=\mf$,
the sheaf $j_{!}\mf$ is canonically isomorphic to $j_{r!}j_{r}^{*}j_{v!}\mf'$, and we have 
\begin{equation}
CC(j_{!}\mf)=CC(j_{r!}j_{r}^{*}j_{v!}\mf'). \notag
\end{equation}
By applying Lemma \ref{lemsumcc} to the distinguished triangle 
\begin{equation}
j_{r!}j_{r}^{*}j_{v!}\mf' \rightarrow j_{v!}\mf' \rightarrow i_{r*}i_{r}^{*}j_{v!}\mf'\rightarrow, \notag
\end{equation}
we have
\begin{equation}
CC(j_{r!}j_{r}^{*}j_{v!}\mf')=CC(j_{v!}\mf')-CC(i_{r*}i_{r}^{*}j_{v!}\mf'). \notag
\end{equation}
Since $j_{v!}'i_{r}^{*}\mf'$ is canonically isomorphic to $i_{r}^{*}j_{v!}\mf'$,
we have 
\begin{equation}
CC(i_{r*}i_{r}^{*}j_{v!}\mf')=i_{r!}CC(j_{v!}'i_{r}^{*}\mf') \notag
\end{equation}
by Lemma \ref{lemclimpush}, and 
we obtain the desired equality (\ref{eqsccred}).
\end{proof}

\begin{rem}
\label{remlemtowild}
Suppose that $X$ is of pure dimension and that $D$ has simple normal crossings.
Let $I'\subset I$ be a subset containing $I_{\mT,\mf}$ (Definition \ref{defindsub} (1))
and let $D'=\bigcup_{i\in I'}D_{i}$.
Assume that the ramification of $\mf$ is $\log$-$D'$-clean along $D$.
As is seen in the proof of Proposition \ref{proptowild}, the singular support $SS(j_{!}\mf)$ and the characteristic cycle $CC(j_{!}\mf)$
satisfy the following relations:
Let $r\in I_{\mT,\mf}$. We put $E=\bigcup_{i\in I-\{r\}}D_{i}$ and $V=X-E$. 
Let $\mf'$ be a smooth sheaf of $\Lambda$-modules of rank 1 on $V$ whose 
associated character $\chi'\colon \pi_{1}^{\ab}(V)\rightarrow \Lambda^{\times}$ has the $p$-part inducing
the $p$-part of $\chi$.
Let $i_{r}\colon D_{r}\rightarrow X$ denote the canonical closed immersion
and $j_{v}\colon V\rightarrow X$ the canonical open immersion.
Let $j_{v}'\colon i_{r}^{*}V\rightarrow D_{r}$ be the base change of $j_{v}$ by $i_{r}$.
Then we have 
\begin{align}
SS(j_{!}\mf)=SS(j_{v!}\mf')\cup i_{r\circ}SS(j_{v!}'i_{r}^{*}\mf') \notag
\end{align}
and
\begin{align}
CC(j_{!}\mf)=CC(j_{v!}\mf')- i_{r!}CC(j_{v!}'i_{r}^{*}\mf').  \notag
\end{align}
\end{rem}

The following theorem is a refinement of \cite[Theorem 3.13]{yacc}.

\begin{thm}[cf.\ {\cite[Theorem 3.13]{yacc}}]
\label{thmmains}
Suppose that $D$ has simple normal crossings.
Let $I'\subset I$ be a subset containing $I_{\mT,\mf}$ (Definition \ref{defindsub} (1))
and let $D'=\bigcup_{i\in I'}D_{i}$.
Assume that the ramification of $\mf$ is $\log$-$D'$-clean along $D$.
Then $j_{!}\mf$ is micro-supported
on the closed conical subset $S_{D'}(j_{!}\mf)\subset T^{*}X$ (Definition \ref{defsdpf}),
or equivalently we have $SS(j_{!}\mf)\subset S_{D'}(j_{!}\mf)$ by Remark \ref{remmicsup} (3).
\end{thm}

\begin{proof}
Since $X$ is smooth over $k$, we may assume that $X$ is of pure dimension.
We prove the assertion
by the induction on the cardinality of $I_{\mT,\mf}$.

Suppose that $I_{\mT,\mf}=\emptyset$.
Then we have 
\begin{equation}
\label{eqtamesdp}
S_{D'}(j_{!}\mf)=C_{D'\subset D}\cup \tau_{D'}^{-1}(S_{D'}^{\log}(j_{!}\mf))
\end{equation}
by Remark \ref{remdefsdpf} (1).
Let $h\colon W\rightarrow X$ be a separated $S_{D'}(j_{!}\mf)$-transversal morphism of smooth schemes over $k$ and
let $j'\colon h^{*}U=U\times_{X}W\rightarrow W$ denote the bace change of $j$ by $h$.
By Remark \ref{remctr} (3) and (\ref{eqtamesdp}),
the morphism $h$ is $C_{D'}$-transversal and
$\tau_{D'}^{-1}(S_{D'}^{\log}(j_{!}\mf))$-transversal.
Thus the morphism $h$ is $\log$-$D'$-$S_{D'}^{\log}(j_{!}\mf)$-transversal by Proposition \ref{proplogtr}.
Since $h$ is $C_{D'\subset D}$-transversal by Remark \ref{remctr} (3) and (\ref{eqtamesdp}),
we have $Z_{h^{*}\mf}=(h^{*}Z_{\mf})_{\red}=(h^{*}D)_{\red}$
and the ramification of $h^{*}\mf$ is $\log$-$h^{*}D'$-clean along $(h^{*}D)_{\red}$
by Proposition \ref{propeqcl} (1).
Thus the canonical morphisms $j_{!}\mf\rightarrow Rj_{*}\mf$ 
and $j'_{!}h^{*}\mf\rightarrow Rj'_{*}h^{*}\mf$ are isomorphisms
by Proposition \ref{mainpropsec} (2).
By Proposition \ref{propjftr},
the morphism $h$ is $j_{!}\mf$-transversal.
Since $S_{D'}(j_{!}\mf)$ contains $T^{*}_{X}X$,
the base $S_{D'}(j_{!}\mf)\cap T^{*}_{X}X$ of $S_{D'}(j_{!}\mf)$
is $X=\Supp (j_{!}\mf)$.
Therefore the assertion holds by Proposition \ref{propgtrmicsup}.

Suppose that $I_{\mT,\mf}\neq \emptyset$.
We take $r\in I_{\mT,\mf}$ and
we put $E=\bigcup_{i\in I-\{r\}}D_{i}$, $E'=\bigcup_{i\in I'-\{r\}}D_{i}$, and
$V=X-E$. 
Let $\mf'$ be a smooth sheaf of $\Lambda$-modules of rank 1 on $V$ whose 
associated character $\chi'\colon \pi_{1}^{\ab}(V)\rightarrow \Lambda^{\times}$ has the $p$-part inducing
the $p$-part of $\chi$.
Let $i_{r}\colon D_{r}\rightarrow X$ denote the canonical closed immersion 
and let $j_{v}\colon V\rightarrow X$ denote the canonical open immersion.
Let $j_{v}'\colon i_{r}^{*}V\rightarrow D_{r}$ be the base change of $j_{v}$ by $i_{r}$.
We apply {\bf (Induction Step)} and we have 
\begin{align}
\label{ssjvfcontsepjvfp}
SS(j_{v!}\mf')\subset S_{E'}(j_{v!}\mf') 
\end{align}
and
\begin{align}
\label{ssirjvfcontsirjvepjvfp}
SS(j_{v!}'i_{r}^{*}\mf')\subset S_{i_{r}^{*}E'}(j_{v!}'i_{r}^{*}\mf'). 
\end{align}
By applying $i_{r\circ}$ to (\ref{ssirjvfcontsirjvepjvfp}), we have
\begin{equation}
\label{irssirjvfpcont}
i_{r\circ}SS(j_{v!}'i_{r}^{*}\mf')\subset i_{r\circ}S_{i_{r}^{*}E'}(j_{v!}'i_{r}^{*}\mf'). 
\end{equation}
By Remark \ref{remlemtowild},
the union of left-hand sides of (\ref{ssjvfcontsepjvfp})
and (\ref{irssirjvfpcont}) is equal to $SS(j_{!}\mf)$.
By Remark \ref{remlemsdpd}, the union of right-hand sides
of (\ref{ssjvfcontsepjvfp}) and (\ref{irssirjvfpcont}) is equal to $S_{D'}(j_{!}\mf)$.
\end{proof}

\subsection{Candidate of singular support}
\label{sscandss}

We improve Theorem \ref{thmmains} under the assumption that 
inverse image $\tau_{D'}^{-1}(S_{D'}^{\log}(j_{!}\mf))\subset T^{*}X$ 
of the $\log$-$D'$-singular support $S_{D'}^{\log}(j_{!}\mf)\subset T^{*}X(\log D')$
(Definition \ref{deflifdp} (3)) by $\tau_{D'}$ (\ref{deftaue}) has the same dimension with $X$.
We first recall the computation of the singular support $SS(j_{!}\mf)$ (Definition \ref{defss} (2)) in 
codimension $1$, namely outside a closed subset of $X$ of codimension $\ge 2$, given in \cite{sacc}.

\begin{prop}[{cf.\ \cite[Proposition 4.13]{sacc}}]
\label{propssndeg}
Suppose that $D$ has simple normal crossings and that
the ramification of $\mf$ is $\log$-$\emptyset$-clean along $D$.
\begin{enumerate}
\item If $I=I_{\mT,\mf}$ (Definition \ref{defindsub} (1)), then we have
\begin{align}
SS(j_{!}\mf)=\bigcup_{I''\subset I}T^{*}_{D_{I''}}X, \notag
\end{align}
where $T^{*}_{D_{I''}}X$ denotes the conormal bundle of $D_{I''}=\bigcap_{i\in I''}D_{i}\subset X$.
\item If $I=I_{\mW,\mf}$ (Definition \ref{defindsub} (1)), then we have 
\begin{align}
SS(j_{!}\mf)=T^{*}_{X}X\cup \bigcup_{i\in I}L_{i,\mf}^{\emptyset}, 
\notag
\end{align}
where $L_{i,\mf}^{\emptyset}$ is as in Definition \ref{deflifdp} (2).
\end{enumerate}
\end{prop}

\begin{proof}
By Remark \ref{remlogdpcl} (3), the assertions are special cases of \cite[Proposition 4.13]{sacc}.
\end{proof}

\begin{rem}
\label{remcompssnondeg}
Suppose that $D$ has simple normal crossings.
\begin{enumerate}
\item By Lemma \ref{lemcfclatx} (1), we locally have $I=I_{\mT,\mf}$ 
(Definition \ref{defindsub} (1)) or 
$I=I_{\mW,\mf}$ (Definition \ref{defindsub} (1)) when the ramification of $\mf$ is $\log$-$\emptyset$-clean along $D$.
That is, the $\log$-$\emptyset$-cleanliness is equivalent to the strong non-degeneration 
in the sense of \cite[the remark before Proposition 4.13]{sacc} for the ramification of $\mf$
along $D$ by Remark \ref{remlogdpcl} (3).
\item Proposition \ref{propssndeg} generally gives a computation of 
the singular support $SS(j_{!}\mf)$
outside a closed subscheme of $X$ of codimension $\ge 2$ 
by Remarks \ref{remlogdpcl} (2) and \ref{remmicsup} (1).
\end{enumerate}
\end{rem}

We next prove that the inverse image $\tau_{D'}^{-1}(S_{D'}^{\log}(j_{!}\mf))\subset T^{*}X$ of the $\log$-$D'$-singular support $S_{D'}^{\log}(j_{!}\mf)$ (Definition \ref{deflifdp} (3)) 
and the pull-back $\tau_{D'}^{!}CC_{D'}^{\log}(j_{!}\mf)$ of the $\log$-$D'$-characteristic cycle $CC_{D'}^{\log}(j_{!}\mf)$
(Definition \ref{deflifdp} (4)) by $\tau_{D'}$ (\ref{deftaue}) satisfy the similar equalities 
as in Proposition \ref{proptowild},
when the dimension of $\tau_{D'}^{-1}(S_{D'}^{\log}(j_{!}\mf))$ is the same as that of $X$.

\begin{lem}
\label{lemlcidim}
Let $f\colon Y\rightarrow Z$ be a morphism of smooth schemes $Y$ and $Z$ of relative dimension $c$ and $d$ over a scheme $S$, respectively.
Let $C\subset Z$ be an irreducible closed subset of dimension $e$.
Then every irreducible component of the inverse image $f^{-1}(C)$ is of dimension $\ge c-d+e$.
\end{lem}

\begin{proof}
The assertion follows,
since $f$ factors as the graph $\Gamma_{f}\colon Y\rightarrow Y\times_{S}Z$ followed by the second projection 
$\pr_{2}\colon Y\times_{S}Z\rightarrow Z$ and since the graph $\Gamma_{f}$ is a regular closed immersion of codimension $d$.
\end{proof}

\begin{lem}
\label{lemvbdim}
Let $f\colon E'\rightarrow E$ be the morphism of vector bundles on 
a connected scheme $S$ of the same rank.
Let $f_{x}\colon E'_{x}\rightarrow E_{x}$ denote the morphism
of fibers of $E$ and $E'$ at $x\in S$ induced by $f$.
Let $0_{E}\subset E$ be the zero section and 
let $C\subset E$ be a sub bundle.
Suppose that the inverse image $Z=f^{-1}(0_{E})$ of the zero section $0_{E}\subset E$ by $f$ is of dimension $d$ 
and that the intersection $C_{x}\cap \mathrm{Im}f_{x}$ of the fiber $C_{x}$ of $C$ at $x$ and the image $\Image f_{x}$ 
of $f_{x}$ is of dimension $r$ for every $x\in S$.
Then we have the following:
\begin{enumerate}
\item The inverse image $f^{-1}(C)$ of $C$ by $f$ is of dimension $r+d$.
\item There is a one-to-one correspondence between 
the set of irreducible components of $Z$ of dimension $d$
and that of irreducible components of $f^{-1}(C)$ of dimension $r+d$,
which associates an irreducible component $Z_{a}$ of $Z$ of dimension $d$ to
the irreducible component of $f^{-1}(C)$ that is the closure of the fiber of $f^{-1}(C)$ at the generic point of the base $Z_{a}\cap 0_{E'}$ of $Z_{a}$.
Here $0_{E'}\subset E'$ denotes the zero section.
\end{enumerate}
\end{lem}

\begin{proof}
Let $0_{E'}\subset E'$ be the zero section and
let $x\in Z\cap 0_{E'}\subset S$.
Since $Z$ is of dimension $d$, 
the dimension of the fiber $Z_{x}$ of $Z$ at $x$ is $\le d-\dim \overline{\{x\}}$, where $\overline{\{x\}}$ is the closure of $x$ in $S$, 
and the equality holds if and only if $x$ is the generic point of
the base of an irreducible component of $Z$ of dimension $d$.
Since the sequence
\begin{equation}
0\rightarrow Z_{x}\rightarrow f^{-1}(C)_{x}\xrightarrow{f_{x}} C_{x}
\cap \Image f_{x}\rightarrow 0 \notag
\end{equation}
of vector bundles on $\Spec k(x)$ is exact,
the dimension of the fiber $f^{-1}(C)_{x}$ of $f^{-1}(C)$ at $x$
is $\le r+d-\dim\overline{\{x\}}$,
and the equality holds if and only if $x$ is the generic point of the base of an irreducible component of $Z$ of dimension $d$.
Therefore the assertions hold.
\end{proof}

\begin{prop}
\label{propinvli}
Suppose that $X$ is purely of dimension $d$ and that $D$ has simple normal crossings.
Let $I'\subset I$ be a subset containing $I_{\mT,\mf}$ (Definition \ref{defindsub} (1))
and let $D'=\bigcup_{i\in I'}D_{i}$.
Assume that the ramification of $\mf$ is $\log$-$D'$-clean along $D$.
Let $(B_{I'',\mf}^{D'})^{0}$ for $I''\subset I'$ denote the set of generic points of 
$B_{I'',\mf}^{D'}$ (Definition \ref{defbcf} (1)) whose codimension in $D_{I''}$ (\ref{defdipp}) is $\le 1$.
Let $Y_{\mathfrak{p}}$ be the irreducible component of $D_{I''}\cap \bigcap_{i\in I''\cap I_{\mW,\mf}}V(\Image\xi_{i}^{D'}(\mf))-\bigcup_{i\in I'-I''}D_{i}$ 
whose generic point is $\mathfrak{p}$ for $\mathfrak{p}\in (B_{I'',\mf}^{D'})^{0}$,
where $I''\subset I'$ and where $V(\Image \xi^{D'}_{i}(\mf))$
denotes the closed subscheme of $D_{i}$ defined by the image $\Image \xi^{D'}_{i}(\mf)$
of $\xi_{i}^{D'}(\mf)$ (Definition \ref{defximf} (2)), which is regarded as an ideal sheaf of $\dvr_{D_{i}}$ by Remark \ref{remximf} (2).
We regard $Y_{\mathfrak{p}}$ for $\mathfrak{p}\in (B_{I'',\mf}^{D'})^{0}$ as the closed subscheme of $X-\bigcup_{i\in I-I''}D_{i}$
whose local ring at $\mathfrak{p}$ is of the same length 
with the local ring of $D_{I''}\cap \bigcap_{i\in I''\cap I_{\mW,\mf}}V(\Image\xi_{i}^{D'}(\mf))$ at $\mathfrak{p}$.
Let $\omega_{\mathfrak{p}}$ be a section of $\Omega_{X}^{1}|_{Y_{\mathfrak{p}}^{1/p}}$ whose image in $\Omega_{X}^{1}(\log D')|_{Y_{\mathfrak{p}}^{1/p}}$ generates
the image of $m_{i}^{D'}(\mf)$ (Definition \ref{defximf} (1)) 
in $\Omega_{X}^{1}(\log D')|_{Y_{\mathfrak{p}}^{1/p}}$
for some (or equivalently any) $i\in I''\cap I_{\mW,\mf,\mathfrak{p}}$ (\ref{eachindatx}).
Let $\langle \omega_{\mathfrak{p}}, dt_{i'}; i'\in I''/ Y_{\mathfrak{p}}^{1/p}\rangle$ 
denote the image by the canonical finite morphism
\begin{equation}
T^{*}X\times_{X}Z_{\mf}^{1/p}\rightarrow T^{*}X\times_{X}Z_{\mf} \notag
\end{equation}
of 
the closure in $T^{*}X\times_{X}Z_{\mf}^{1/p}$ of
the sub bundle $C_{\mathfrak{p}}$ of $T^{*}X\times_{X}Y_{\mathfrak{p}}^{1/p}$
with the basis $\omega_{\mathfrak{p}}$ and $dt_{i'}$ for $i'\in I''$.
We regard $\langle \omega_{\mathfrak{p}}, dt_{i'}; i'\in I''/ Y_{\mathfrak{p}}^{1/p}\rangle$ 
as the closed subscheme of $T^{*}X\times_{X}Z_{\mf}$ whose local ring at the generic point 
is of the same length with the closed subscheme $C_{\mathfrak{p}}'\subset T^{*}X\times_{X}Y_{\mathfrak{p}}$ 
such that the algebraic cycle $[C_{\mathfrak{p}}']$ defined by $C'_{\mathfrak{p}}$ is
equal to the push-forward of $[C_{\mathfrak{p}}]$ 
by the canonical finite morphism $T^{*}X\times_{X}Y_{\mathfrak{p}}^{1/p}\rightarrow T^{*}X\times_{X}Y_{\mathfrak{p}}$.
Then we have the following for the inverse images by $\tau_{D'}$ (\ref{deftaue}) 
of the closed subschemes
$L_{\mf}^{D'}\subset T^{*}X(\log D')\times_{X}Z_{\mf}$
and $L_{i,\mf}^{D'}\subset T^{*}X(\log D')\times_{X}D_{i}$ 
for $i\in I_{\mW,\mf}$ defined in Definition \ref{deflifdp} (2):
\begin{enumerate}
\item
The equalities
\begin{align}
\tau_{D'}^{-1}(L_{\mf}^{D'})=\bigcup_{\substack{I''\subset I' \\ I''\cap I_{\mW,\mf}\neq \emptyset}}T^{*}_{D_{I''}}X 
\cup \bigcup_{I''\subset I'}\bigcup_{\mathfrak{p}\in (B_{I'',\mf}^{D'})^{0}} \langle \omega_{\mathfrak{p}}, dt_{i'}; i'\in I''/Y_{\mathfrak{p}}^{1/p}\rangle
\notag
\end{align}
and 
\begin{equation}
\tau_{D'}^{-1}(L_{i,\mf}^{D'})=\begin{cases}
\bigcup_{\substack{I''\subset I' \\ i\in I''}}T^{*}_{D_{I''}}X \cup \bigcup_{\substack{I''\subset I'\\ i\in I''}}\bigcup_{\mathfrak{p}\in (B_{I'',\mf}^{D'})^{0}} \langle \omega_{\mathfrak{p}}, dt_{i'}; i'\in I''/Y_{\mathfrak{p}}^{1/p}\rangle & (i\in I'), \\
\bigcup_{I''\subset I'}\bigcup_{\substack{\mathfrak{p}\in (B_{I'',\mf}^{D'})^{0} \\ i\in I_{\mathfrak{p}}}} \langle \omega_{\mathfrak{p}}, dt_{i'}; i'\in I''/Y_{\mathfrak{p}}^{1/p}\rangle\times_{X}D_{i} & (i\in I-I')
\end{cases}
\notag
\end{equation}
for $i\in I_{\mW,\mf}$ (Definition \ref{defindsub} (1)) of closed conical subsets of $T^{*}X$ hold.

\item
The equalities
\begin{align}
[\tau_{D'}^{-1}(L_{\mf}^{D'})]=\sum_{\substack{I''\subset I' \\ I''\cap I_{\mW,\mf}\neq \emptyset}}[T^{*}_{D_{I''}}X] 
+ \sum_{I''\subset I'}\sum_{\mathfrak{p}\in (B_{I'',\mf}^{D'})^{0}} [\langle \omega_{\mathfrak{p}}, dt_{i'}; i'\in I''/Y_{\mathfrak{p}}^{1/p}\rangle]
\notag
\end{align}
and 
\begin{align}
&[\tau_{D'}^{-1}(L_{i,\mf}^{D'})] \notag \\
&\quad =\begin{cases}
\sum_{\substack{I''\subset I' \\ i\in I''}}[T^{*}_{D_{I''}}X] + \sum_{\substack{I''\subset I'\\ i\in I'}}\sum_{\mathfrak{p}\in (B_{I'',\mf}^{D'})^{0}} [\langle \omega_{\mathfrak{p}}, dt_{i'}; i'\in I''/Y_{\mathfrak{p}}^{1/p}\rangle] & (i\in I'), \\
\sum_{I''\subset I'}\sum_{\substack{\mathfrak{p}\in (B_{I'',\mf}^{D'})^{0} \\ i\in I_{\mathfrak{p}}}} [\langle \omega_{\mathfrak{p}}, dt_{i'}; i'\in I''/Y_{\mathfrak{p}}^{1/p}\rangle\times_{X}D_{i}] & (i\in I-I')
\end{cases}
\notag
\end{align}
for $i\in I_{\mW,\mf}$ (Definition \ref{defindsub} (1)) of algebraic cycles on $T^{*}X$
hold.

\item The inverse image $\tau_{D'}^{-1}(L_{\mf}^{D'})$ is of dimension $\le d+1$.
The dimension of $\tau_{D'}^{-1}(L_{\mf}^{D'})$ is $\le d$ 
if and only if $E_{\mf}^{D'}=\emptyset$ (Definition \ref{defbcf} (2)). 
\end{enumerate}
\end{prop}

\begin{proof}
We consider the cartesian diagram
\begin{equation}
\xymatrix{
T^{*}X\times_{X}Z_{\mf}^{1/p} \ar[r]^-{\tau_{D',Z_{\mf}^{1/p}}} \ar[d] \ar@{}[dr] | {\square} & T^{*}X(\log D')\times_{X}Z_{\mf}^{1/p} \ar[d] \\
T^{*}X\times_{X}Z_{\mf}\ar[r]_-{\tau_{D',Z_{\mf}}} & T^{*}X(\log D')\times_{X}Z_{\mf},} \notag
\end{equation}
where $\tau_{D',Z_{\mf}}$ and $\tau_{D',Z_{\mf}^{1/p}}$ are the base changes of $\tau_{D'}$ by the canonical closed immersion $Z_{\mf}\rightarrow X$ and the canonical morphism $Z_{\mf}^{1/p}\rightarrow X$, respectively,
and the vertical arrows are the canonical finite morphisms.
Since $L_{\mf}^{D'}$ (resp.\ $L_{i,\mf}^{D'}$) 
is the image of $L_{\mf}'^{D'}\subset T^{*}X(\log D')\times_{X}Z_{\mf}^{1/p}$ 
(resp.\ $L_{i,\mf}'^{D'}\subset T^{*}X(\log D')\times_{X}Z_{\mf}^{1/p}$) 
(Definition \ref{deflifdp} (1))
by the right vertical arrow, 
the inverse image $\tau_{D'}^{-1}(L_{\mf}^{D'})$ (resp.\ $\tau_{D'}^{-1}(L_{i,\mf}^{D'})$
for $i\in I_{\mW,\mf}$)
is equal to the image by the left vertical arrow of the inverse image $\tau_{D',Z_{\mf}^{1/p}}^{-1}(L_{\mf}'^{D'})$ 
(resp.\ $\tau_{D',Z_{\mf}^{1/p}}^{-1}(L_{i,\mf}'^{D'})$),
and is of the same dimension with $\tau_{D',Z_{\mf}^{1/p}}^{-1}(L_{\mf}'^{D'})$ 
(resp.\ $\tau_{D',Z_{\mf}^{1/p}}^{-1}(L_{i,\mf}'^{D'})$).

Let $Y$ be a closed subscheme of $Z_{\mf}$.
Then 
the inverse image of the zero section $T^{*}_{X}X(\log D')\times_{X}Y^{1/p}\subset T^{*}X(\log D')\times_{X}Y^{1/p}$ by $\tau_{D',Z_{\mf}^{1/p}}$ 
is equal to $\bigcup_{I''\subset I'}T^{*}_{D_{I''}}X\times_{X}Y^{1/p}$
by (\ref{eqinvzero}).
Since the dimension of $T^{*}_{D_{I''}}X\times_{X}Y'$ for $I''\subset I'$ is $\le d$,
the dimension of each irreducible component of $\tau_{D',Z_{\mf}^{1/p}}^{-1}(L_{\mf}'^{D'})$
(resp.\ $\tau_{D',Z_{\mf}^{1/p}}^{-1}(L_{i,\mf}'^{D'})$ for $i\in I_{\mW,\mf}$)
is $d$ or $d+1$ by Lemma \ref{lemlcidim} and Lemma \ref{lemvbdim} (1).
Then the assertion (2) holds by Remark \ref{remdefbcf} (1) and
Lemma \ref{lemvbdim} (2), and the assertion (1) follows from the assertion (2).
Since the dimension of $T^{*}_{D_{I''}}X\times_{X}Y$ for $I''\subset I'$ is $d$ if and only if 
the codimension of $Y\cap D_{I''}$ in $D_{I''}$ is $0$,
the dimension of $\tau_{D',Z_{\mf}^{1/p}}^{-1}(L_{\mf}'^{D'})$
(resp.\ $\tau_{D',Z_{\mf}^{1/p}}^{-1}(L_{i,\mf}'^{D'})$ for $i\in I_{\mW,\mf}$) is $d+1$ 
if and only if we have 
$E_{I'',\mf}^{D'}\neq \emptyset$ for some $I''\subset I'$
(resp.\ for some $I''\subset I'$ such that $i\in I''$) by Lemma \ref{lemvbdim} (1).
Hence the assertion (3) holds.
\end{proof}

\begin{rem}
\label{rempropinvli}
In Proposition \ref{propinvli}, we have
$L_{i,\mf}^{\emptyset}=\langle \omega_{\mathfrak{p}_{i}}/Y_{\mathfrak{p}_{i}}^{1/p}\rangle\times_{X}D_{i}$ 
for $i\in (I-I')\cap I_{\mW,\mf}$,
where $\mathfrak{p}_{i}$ denotes the generic point 
of $D_{i}\subset B_{\emptyset, \mf}^{D'}=\bigcup_{i'\in (I-I')\cap I_{\mW,\mf}}D_{i'}$.
For $i\in I'\cap I_{\mI,\mf}$ (Definition \ref{defindsub} (2)),
we have $L_{i,\mf}^{\emptyset}=T^{*}_{D_{i}}X$ by \cite[Lemma 2.23 (ii)]{yacc}.
\end{rem}

\begin{cor}
\label{corinvli}
Suppose that $X$ is purely of dimension $d$ and that $D$ has simple normal crossings.
Let $I'\subset I$ be a subset containing $I_{\mT,\mf}$ (Definition \ref{defindsub} (1))
and let $D'=\bigcup_{i\in I'}D_{i}$.
Assume that the ramification of $\mf$ is $\log$-$D'$-clean along $D$.
Then the following five conditions for the closed conical subset $S_{D'}(j_{!}\mf)\subset T^{*}X$ 
(Definition \ref{defsdpf}), 
the inverse image $\tau_{D'}^{-1}(S_{D'}^{\log}(j_{!}\mf))$ of the $\log$-$D'$-singular support $S_{D'}^{\log}(j_{!}\mf)\subset T^{*}X(\log D')$ 
(Definition \ref{deflifdp} (3)) by $\tau_{D'}$ (\ref{deftaue}), and 
the closed subset $E_{\mf}^{D'}\subset X$ (Definition \ref{defbcf} (2))
are equivalent:
\begin{enumerate}
\item $S_{D'}(j_{!}\mf)$ is of dimension $d$.
\item $S_{D'}(j_{!}\mf)$ is purely of dimension $d$.
\item $\tau_{D'}^{-1}(S_{D'}^{\log}(j_{!}\mf))$ is of dimension $d$.
\item $\tau_{D'}^{-1}(S_{D'}^{\log}(j_{!}\mf))$ 
is purely of dimension $d$.
\item $E_{\mf}^{D'}=\emptyset$. 
\end{enumerate}

If the equivalent conditions are satisfied, 
then the singular support $SS(j_{!}\mf)$
of $j_{!}\mf$ is a union of irreducible components of $S_{D'}(j_{!}\mf)$.
\end{cor}

\begin{proof}
Let the notation be as in Definition \ref{defsdpf}.
Since $i_{I''\circ}C_{i_{I''}^{*}D'_{\mW,\mf}\subset i_{I''}^{*}D_{\mW,\mf}}$ is purely of dimension $d$ for every subset $I''\subset I_{\mT,\mf}$,
the equivalence of (1) and (3) and that of (2) and (4) hold
by the definition of $S_{D'}(j_{!}\mf)$ as the union of $i_{I''\circ}C_{i_{I''}^{*}D'_{\mW,\mf}\subset i_{I''}^{*}D_{\mW,\mf}}$
for all $I''\subset I_{\mT,\mf}$ and $\tau_{D'}^{-1}(S_{D'}^{\log}(j_{!}\mf))$.
The equivalence of (3) and (4) holds by 
Lemma \ref{lemlcidim},
since $S_{D'}^{\log}(j_{!}\mf)$ is of dimension $d$.
The equivalence of (3) and (5) holds by Proposition \ref{propinvli} (3),
since the inverse image $\tau_{D'}^{-1}(T^{*}_{X}X(\log D'))$ is of dimension $d$ by (\ref{eqinvzero}).
Since $SS(j_{!}\mf)$ is contained in $S_{D'}(j_{!}\mf)$ by Theorem \ref{thmmains},
the last assertion holds by Remark \ref{remmicsup} (2).
\end{proof}

\begin{cor}
\label{corcompcc}
Let the notation and the assumptions be as in Proposition \ref{propinvli}.
Then we have the following: 
\begin{enumerate}
\item For the inverse image $\tau_{D'}^{-1}(S_{D'}^{\log}(j_{!}\mf))\subset T^{*}X$
of the $\log$-$D'$-singular support $S_{D'}^{\log}(j_{!}\mf)\subset T^{*}X(\log D')$ (Definition \ref{deflifdp} (3)) by $\tau_{D'}$ (\ref{deftaue}),
we have
\begin{align}
&\tau_{D'}^{-1}(S_{D'}^{\log}(j_{!}\mf)) \notag \\
&\qquad
=\bigcup_{I''\subset I'}T^{*}_{D_{I''}}X \cup \bigcup_{I''\subset I'}\bigcup_{\mathfrak{p}\in (B_{I'',\mf}^{D'})^{0}}
\langle \omega_{\mathfrak{p}},dt_{i'}; i'\in I''/Y_{\mathfrak{p}}^{1/p}\rangle \notag \\
&\quad 
\qquad \qquad \cup \bigcup_{I''\subset I'}\bigcup_{\mathfrak{p}\in (B_{I'',\mf}^{D'})^{0}}
\bigcup_{i\in (I-I')\cap I_{\mW,\mf,\mathfrak{p}}}
\langle \omega_{\mathfrak{p}},dt_{i'}; i'\in I''/Y_{\mathfrak{p}}^{1/p}\rangle\times_{X}D_{i}, \notag
\end{align}
where $I_{\mW,\mf,\mathfrak{p}}$ is as in (\ref{eachindatx}).

\item Suppose that the inverse image $\tau_{D'}^{-1}(S_{D'}^{\log}(j_{!}\mf))$ is of dimension $d$.
Then, for the image $\tau_{D'}^{!}CC_{D'}^{\log}(j_{!}\mf)$ of the $\log$-$D'$-characteristic cycle $CC_{D'}^{\log}(j_{!}\mf)$ 
(Definition \ref{deflifdp} (4)) by $\tau_{D'}^{!}$ (\ref{taudpgysin}),
we have
\begin{align}
&\tau_{D'}^{!}CC_{D'}^{\log}(j_{!}\mf) \notag \\
&=(-1)^{d}
(\sum_{I''\subset I'}(1+\sum_{i\in I''\cap I_{\mW,\mf}}\sw^{D'}(\chi|_{K_{i}}))[T^{*}_{D_{I''}}X]
\notag \\
&\qquad \qquad +\sum_{I''\subset I'}\sum_{\mathfrak{p}\in (B_{I'',\mf}^{D'})^{0}}
\sum_{i\in I''\cap I_{\mW,\mf}}\sw^{D'}(\chi|_{K_{i}})
[\langle \omega_{\mathfrak{p}},dt_{i'}; i'\in I''/Y_{\mathfrak{p}}^{1/p}\rangle] \notag \\
&\qquad \qquad + \sum_{I''\subset I'}\sum_{\mathfrak{p}\in (B_{I'',\mf}^{D'})^{0}}
\sum_{i\in (I-I')\cap I_{\mW,\mf,\mathfrak{p}}}\sw^{D'}(\chi|_{K_{i}})
[\langle \omega_{\mathfrak{p}},dt_{i'}; i'\in I''/Y_{\mathfrak{p}}^{1/p}\rangle\times_{X}D_{i}]) \notag
\end{align}
in $Z_{d}(\tau_{D'}^{-1}(S_{D'}^{\log}(j_{!}\mf)))$,
where $\sw^{D'}(\chi|_{K_{i}})$ is as in (\ref{defswdp}). 
\end{enumerate}
\end{cor}

\begin{proof}
The assertion (1) holds by (\ref{eqinvzero}) and Proposition \ref{propinvli} (1).
The assertion (2) holds by (\ref{eqinvzerocyc}) and Proposition \ref{propinvli} (2).
\end{proof}

\begin{prop}
\label{proptauinvtw}
Suppose that $X$ is purely of dimension $d$ and that $D$ has simple normal crossings.
Let $I'\subset I$ be a subset containing $I_{\mT,\mf}$ (Definition \ref{defindsub} (1))
and let $D'=\bigcup_{i\in I'}D_{i}$.
Assume that the ramification of $\mf$ is $\log$-$D'$-clean along $D$
and that the inverse image $\tau_{D'}^{-1}(S_{D'}^{\log}(j_{!}\mf))\subset T^{*}X$
of the $\log$-$D'$-singular support $S_{D'}^{\log}(j_{!}\mf)\subset T^{*}X(\log D')$ (Definition \ref{deflifdp} (3)) by $\tau_{D'}$ (\ref{deftaue})
is of dimension $d$.
Let $\mf_{\mW}$ be a smooth sheaf of $\Lambda$-modules 
of rank 1 on $U_{\mW,\mf}=X-D_{\mW,\mf}$ (\ref{defd*mf}) whose associated 
character $\chi_{\mW}\colon \pi_{1}^{\ab}(U_{\mW,\mf})\rightarrow \Lambda^{\times}$ 
has the $p$-part inducing the $p$-part of $\chi$.
We put $D_{\mW,\mf}'=\bigcup_{i\in I'\cap I_{\mW,\mf}}D_{i}$,
where $I_{\mW,\mf}$ is as in Definition \ref{defindsub} (1).
Let $i_{I''}\colon D_{I''}=\bigcap_{i\in I''}D_{i}\rightarrow X$
be the canonical closed immersion 
and $j_{w,I''}\colon i_{I''}^{*}U_{\mW,\mf}\rightarrow D_{I''}$ 
the canonical open immersion for $I''\subset I'$.
Then we have the following:
\begin{enumerate}
\item The dimension of the inverse image $\tau_{i_{I''}^{*}D'_{\mW,\mf}}^{-1}(S_{i_{I''}^{*}D'_{\mW,\mf}}^{\log}(j_{w,I''!}i_{I''}^{*}\mf_{\mW}))\subset T^{*}D_{I''}$
is the same as that of $D_{I''}$ for every $I''\subset I_{\mT,\mf}$,
where the ramification of $i_{I''}^{*}\mf_{\mW}$ is $\log$-$i_{I''}^{*}D_{\mW,\mf}'$-clean 
along $i_{I''}^{*}D_{\mW,\mf}$ by Proposition \ref{proplogtrclpp} (2).
\item For the inverse image $\tau_{D'}^{-1}(S_{D'}^{\log}(j_{!}\mf))$ of $S_{D'}^{\log}(j_{!}\mf)$, we have 
\begin{equation}
\tau_{D'}^{-1}(S_{D'}^{\log}(j_{!}\mf))=
\bigcup_{I''\subset I_{\mT,\mf}}i_{I''\circ}\tau_{i_{I''}^{*}D'}^{-1}(S_{i_{I''}^{*}D_{\mW,\mf}'}^{\log}(j_{w,I''!}i_{I''}^{*}\mf_{\mW})). \notag
\end{equation} 
\item For the image $\tau_{D'}^{!}CC_{D'}^{\log}(j_{!}\mf)$ of
the $\log$-$D'$-characteristic cycle $CC_{D'}^{\log}(j_{!}\mf)$ (Definition \ref{deflifdp} (4)) by $\tau_{D'}^{!}$ (\ref{taudpgysin}), we have
\begin{equation}
\tau_{D'}^{!}CC_{D'}^{\log}(j_{!}\mf)=
\sum_{I''\subset I_{\mT,\mf}}i_{I''!}\tau_{i_{I''}^{*}D'}^{!}CC_{i_{I''}^{*}D_{\mW,\mf}'}^{\log}(j_{w,I''!}i_{I''}^{*}\mf_{\mW}). \notag 
\end{equation}
\end{enumerate}
\end{prop}

\begin{proof}
By Proposition \ref{proplogtrclpp} (2), we have $Z_{i_{I''}^{*}\mf_{\mW}}=i_{I''}^{*}Z_{\mf_{\mW}}=i_{I''}^{*}D_{\mW,\mf}$ and
the ramification of $i_{I''}^{*}\mf_{\mW}$ is $\log$-$i_{I''}^{*}D'_{\mW,\mf}$-clean along $i_{I''}^{*}D_{\mW,\mf}$ for $I''\subset I_{\mT,\mf}$.
Then we can locally identify $I_{\mW,i_{I''}^{*}\mf_{\mW}}$
with $I_{\mW,\mf_{\mW}}=I_{\mW,\mf}$ and 
the index set of irreducible components of $i_{I''}^{*}D'_{\mW,\mf}$ with $I'\cap I_{\mW,\mf_{\mW}}=I'\cap I_{\mW,\mf}$.

First, let $I''\subset I_{\mT,\mf}$ and $I'''\subset I'\cap I_{\mW,\mf}$ be subsets. 
Then we prove that 
\begin{equation}
\label{eqbf}
B_{I''\cup I''',\mf}^{D'}=B_{I''',i_{I''}^{*}\mf_{\mW}}^{i_{I''}^{*}D'_{\mW,\mf}},
\end{equation} 
where $B_{I''\cup I''',\mf}^{D'}$ and $B_{I''',i_{I''}^{*}\mf_{\mW}}^{i_{I''}^{*}D'_{\mW,\mf}}$
are as in Definition \ref{defbcf} (1).
If $I'''=\emptyset$, then we have
$B_{I''\cup I''',\mf}^{D'} 
=\bigcup_{i\in I-I'}i_{I''}^{*}D_{i}=B_{I''',i_{I'''}^{*}\mf_{\mW}}^{i_{I''}^{*}D_{\mW,\mf}'}$.
Suppose that $I'''\neq \emptyset$.
Let $V(\Image\xi_{i}^{D'}(\mf))$ (resp.\ $V(\Image\xi_{i}^{i_{I''}^{*}D'_{\mW,\mf}}(i_{I''}^{*}\mf_{\mW}))$) for $i\in I'\cap I_{\mW,\mf}$
be the closed subscheme of $D_{i}$ (resp.\ $i_{I''}^{*}D_{i}$) 
defined by the image $\Image\xi_{i}^{D'}(\mf)$ of $\xi_{i}^{D'}(\mf)$ (Definition \ref{defximf} (2)) 
(resp.\ $\Image \xi_{i}^{i_{I''}^{*}D_{\mW,\mf}'}(i_{I''}^{*}\mf_{\mW})$ of 
$\xi_{i}^{i_{I''}^{*}D_{\mW,\mf}'}(i_{I''}^{*}\mf_{\mW})$),
where $\Image\xi_{i}^{D'}(\mf)$ 
(resp.\ $\Image \xi_{i}^{i_{I''}^{*}D_{\mW,\mf}'}(i_{I''}^{*}\mf_{\mW})$) is regarded 
as an ideal sheaf of $\dvr_{D_{i}}$ (resp.\ $\dvr_{i_{I''}^{*}D_{i}}$) by Remark \ref{remximf} (2).
Since $\cform^{D'}(\mf)$ is the image of $\cform^{D'_{\mW,\mf}}(\mf_{\mW})$ by
the canonical morphism 
\begin{equation}
\Omega_{X}^{1}(\log D_{\mW,\mf}')(R_{\mf_{\mW}}^{D'_{\mW,\mf}})|_{Z_{\mf_{\mW}}^{1/p}}\rightarrow \Omega_{X}^{1}(\log D')(R_{\mf}^{D'})|_{Z_{\mf}^{1/p}}, \notag
\end{equation}
where we have $R_{\mf_{\mW}}^{D'_{\mW,\mf}}=R_{\mf}^{D'}$ and $Z_{\mf_{\mW}}^{1/p}=Z_{\mf}^{1/p}$,
by Lemma \ref{lemclrelsecf},
and since
we have
\begin{equation}
\cform^{i_{I''}^{*}D_{\mW,\mf}'}(i_{I''}^{*}\mf_{\mW})=di_{I'', i_{I''}^{*}Z_{\mf}^{1/p}}^{D'_{\mW,\mf}}(i_{I''}^{*}\cform^{D_{\mW,\mf}'}(\mf_{\mW})), \notag
\end{equation} 
where $di_{I'', i_{I''}^{*}Z_{\mf}^{1/p}}^{D'_{\mW,\mf}}\colon i_{I''}^{*}\Omega_{X}^{1}(\log D_{\mW,\mf}')|_{i_{I''}^{*}Z_{\mf}^{1/p}}
\rightarrow \Omega_{D_{I''}}^{1}(\log i_{I''}^{*}D'_{\mW,\mf})|_{i_{I''}^{*}Z_{\mf}^{1/p}}$ 
is the morphism induced by $i_{I''}$, by Proposition \ref{proplogtrclpp} (2),
we have
\begin{equation}
\label{eqinnerbf}
D_{I''\cup I'''}\cap \bigcap_{i\in (I''\cup I''')\cap I_{\mW,\mf}}V(\Image \xi_{i}^{D'}(\mf))
=i_{I''}^{*}D_{I'''}\cap \bigcap_{i\in I''\cap I_{\mW,i_{I''}^{*}\mf_{\mW}}}V(\Image \xi_{i}^{i_{I''}^{*}D_{\mW,\mf}'}(i_{I''}^{*}\mf_{\mW})).
\end{equation}
Since we have
\begin{equation}
D_{I''\cup I'''}-\bigcup_{i\in I-I'}D_{i}=i_{I''}^{*}D_{I'''}-\bigcup_{i\in I_{\mW,\mf}-I'\cap I_{\mW,\mf}}i_{I''}^{*}D_{i}, \notag
\end{equation}
we have
\begin{align}
&D_{I''\cup I'''}\cap \bigcap_{i\in (I''\cup I''')\cap I_{\mW,\mf}}V(\Image \xi_{i}^{D'}(\mf))
-\bigcup_{i\in I-I'}D_{i}  \notag \\
&\qquad =i_{I''}^{*}D_{I'''}\cap \bigcap_{i\in I''\cap I_{\mW,i_{I''}^{*}\mf_{\mW}}}V(\Image \xi_{i}^{i_{I''}^{*}D_{\mW,\mf}'}(i_{I''}^{*}\mf_{\mW}))-\bigcup_{i\in I_{\mW,\mf}-I'\cap I_{\mW,\mf}}i_{I''}^{*}D_{i} \notag
\end{align}
as closed subschemes of $D_{I''\cap I'''}=i_{I''}^{*}D_{I'''}$, and we obtain the equality (\ref{eqbf}). 

We then prove the assertion (1).
By Corollary \ref{corinvli}, we have $E_{\mf}^{D'}=\emptyset$ (Definition \ref{defbcf} (2)).
Hence we have $E_{I''',i_{I''}^{*}\mf_{\mW}}^{i_{I''}^{*}D_{\mW,\mf}'}=E_{I''\cup I''',\mf}^{D'}=\emptyset$ (Definition \ref{defbcf} (2)) for any $I''\subset I_{\mT,\mf}$ and $I'''\subset I'\cap I_{\mW,\mf}$,
and we have $E_{i_{I''}^{*}\mf_{\mW}}^{i_{I''}^{*}D_{\mW,\mf}'}=\emptyset$ for every $I''\subset I_{\mT,\mf}$, 
which deduces the assertion (1) by Corollary \ref{corinvli}.

We prove the assertions (2) and (3).
By (\ref{eqinvzero}) and (\ref{eqinvzerocyc}),
we have
\begin{align}
i_{I''\circ}\tau_{i_{I''}^{*}D_{\mW,\mf}'}^{-1}(T^{*}_{D_{I''}}D_{I''}(\log i_{I''}^{*}D_{\mW,\mf}'))
=\bigcup_{I'''\subset I'\cap I_{\mW,\mf}}T^{*}_{D_{I''\cup I'''}}X \notag
\end{align}
and \begin{align}
i_{I''!}[\tau_{i_{I''}^{*}D_{\mW,\mf}'}^{-1}(T^{*}_{D_{I''}}D_{I''}(\log i_{I''}^{*}D_{\mW,\mf}'))]
=\sum_{I'''\subset I'\cap I_{\mW,\mf}}[T^{*}_{D_{I''\cup I'''}}X], \notag
\end{align}
respectively,
for $I''\subset I_{\mT,\mf}$. 
Let $I''\subset I_{\mT,\mf}$ and $I'''\subset I'\cap I_{\mW,\mf}$ be subsets
and let the notation be as in Proposition \ref{propinvli}.
Then we have 
\begin{equation}
(B_{I''\cup I''',\mf}^{D'})^{0}=(B_{I''',i_{I''}^{*}\mf_{\mW}}^{i_{I''}^{*}D'_{\mW,\mf}})^{0} \notag
\end{equation}
by (\ref{eqbf}).
We consider the commutative diagram
\begin{equation}
\xymatrix{
\Omega_{X}^{1}|_{Y^{1/p}_{\mathfrak{p}}}\ar[rr]^-{\tau_{\emptyset/D'_{\mW,\mf},Y_{\mathfrak{p}}^{1/p}}} \ar[d]_-{f_{Y_{\mathfrak{p}}^{1/p}}} & &\Omega_{X}^{1}(\log D'_{\mW,\mf})|_{Y_{\mathfrak{p}}^{1/p}} \ar[d]
\ar[r]^-{\tau_{\emptyset/D',Y_{\mathfrak{p}}^{1/p}}} & \Omega_{X}^{1}(\log D')|_{Y_{\mathfrak{p}}^{1/p}}\\
\Omega_{D_{I''}}^{1}|_{Y_{\mathfrak{p}}^{1/p}}\ar[rr]_-{\tau_{\emptyset/i_{I''}^{*}E',Y_{\mathfrak{p}}^{1/p}}} & & \Omega_{D_{I''}}^{1}(\log i_{I''}^{*}D'_{\mW,\mf})|_{Y_{\mathfrak{p}}^{1/p}}, \ar[ur]_-{\ \ \ \ \ \ \ i_{I'', D'_{\mW,\mf}/D',Y_{\mathfrak{p}}^{1/p}}}
} \notag
\end{equation}
where the horizontal arrows are the canonical morphisms, the vertical arrows are the morphisms yielded by $i_{I''}$,
and $i_{I'', D'_{\mW,\mf}/D',Y_{\mathfrak{p}}^{1/p}}$ is
the morphism induced by the injection $i_{I'', D_{\mW,\mf}'/D'}$ (\ref{iotaed}).
By the commutativity of the square, we may assume that the section $\omega_{\mathfrak{p}}\in \Omega_{D_{I''}}^{1}|_{Y_{\mathfrak{p}}^{1/p}}$ 
for $i_{I''}^{*}\mf_{\mW}$ and $\mathfrak{p}\in (B_{I''',i_{I''}^{*}\mf_{\mW}}^{i_{I''}^{*}D'_{\mW,\mf}})^{0}=(B_{I''\cup I''',\mf}^{D'})^{0}$ 
is the image of the section $\omega_{\mathfrak{p}}\in \Omega_{X}^{1}|_{Y_{\mathfrak{p}}^{1/p}}$ for $\mf$ and $\mathfrak{p}$ by $f_{Y_{\mathfrak{p}}^{1/p}}$.
Then we have
\begin{align}
&i_{I''\circ}\tau_{i_{I''}^{*}D_{\mW,\mf}'}^{-1}(L_{i,i_{I''}^{*}\mf_{\mW}}^{i_{I''}^{*}D_{\mW,\mf}'})\notag \\
&=\begin{cases}
\bigcup_{\substack{I'''\subset I'\cap I_{\mW,\mf}\\ i\in I'''}}T^{*}_{D_{I''\cup I'''}}X \\
\quad \cup \bigcup_{\substack{I'''\subset I'\cap I_{\mW,\mf} \\ i\in I'''}}
\bigcup_{\mathfrak{p}\in (B_{I''\cap I''',\mf}^{D'})^{0}}\langle \omega_{\mathfrak{p}},dt_{i'}; i'\in I''\cup I'''/Y_{\mathfrak{p}}^{1/p}\rangle
& (i\in I'), \\
\bigcup_{I'''\subset I'\cap I_{\mW,\mf}}
\bigcup_{\substack{\mathfrak{p}\in (B_{I''\cap I''',\mf}^{D'})^{0} \\
i\in I_{\mathfrak{p}}}}\langle \omega_{\mathfrak{p}},dt_{i'}; i'\in I''\cup I'''/Y_{\mathfrak{p}}^{1/p}\rangle\times_{X}D_{i}
&(i\in I-I')
\end{cases}
\notag
\end{align}
and 
\begin{align}
&i_{I''!}[\tau_{i_{I''}^{*}D_{\mW,\mf}'}^{-1}(L_{i,i_{I''}^{*}\mf_{\mW}}^{i_{I''}^{*}D_{\mW,\mf}'})]\notag \\
&=\begin{cases}
\sum_{\substack{I'''\subset I'\cap I_{\mW,\mf}\\ i\in I'''}}[T^{*}_{D_{I''\cup I'''}}X] \\
\quad + \sum_{\substack{I'''\subset I'\cap I_{\mW,\mf} \\ i\in I'''}}
\sum_{\mathfrak{p}\in (B_{I''\cap I''',\mf}^{D'})^{0}}[\langle \omega_{\mathfrak{p}},dt_{i'}; i'\in I''\cup I'''/Y_{\mathfrak{p}}^{1/p}\rangle]
& (i\in I'), \\
\sum_{I'''\subset I'\cap I_{\mW,\mf}}
\sum_{\substack{\mathfrak{p}\in (B_{I''\cap I''',\mf}^{D'})^{0} \\
i\in I_{\mathfrak{p}}}}[\langle \omega_{\mathfrak{p}},dt_{i'}; i'\in I''\cup I'''/Y_{\mathfrak{p}}^{1/p}\rangle\times_{X}D_{i}]
&(i\in I-I')
\end{cases}
\notag
\end{align}
for each $i\in I_{\mW,\mf}$ by Propositions \ref{propinvli} (1) and (2), respectively.
Thus the assertion (2) holds by Corollary \ref{corcompcc} (1)
and the assertion (3) holds by Corollary \ref{corcompcc} (2).
\end{proof}

\begin{cor}
\label{corsdpdim}
Suppose that $X$ is purely of dimension $d$ and that $D$ has simple normal crossings.
Let $I'\subset I$ be a subset containing $I_{\mT,\mf}$ (Definition \ref{defindsub} (1)) and let $D'=\bigcup_{i\in I'}D_{i}$.
Assume that the ramification of $\mf$ is $\log$-$D'$-clean along $D$
and that the inverse image
$\tau_{D'}^{-1}(S_{D'}^{\log}(j_{!}\mf))$ of the $\log$-$D'$-singular support
$S_{D'}^{\log}(j_{!}\mf)\subset T^{*}X(\log D')$ (Definition \ref{deflifdp} (3))
by $\tau_{D'}$ (\ref{deftaue})
is of dimension $d$.
Then the singular support $SS(j_{!}\mf)$ is a union of irreducible components of $\tau_{D'}^{-1}(S_{D'}^{\log}(j_{!}\mf))$.
\end{cor}

\begin{proof}
By Corollary \ref{corinvli}, 
the closed conical subset $S_{D'}(j_{!}\mf)\subset T^{*}X$ (Definition \ref{defsdpf}) 
is of dimension $d$
and the singular support $SS(j_{!}\mf)$ is a union of irreducible components of $S_{D'}(j_{!}\mf)$.
By Propositions \ref{proptowild} (1) and \ref{proptauinvtw} (1) and (2),
we may assume that $I_{\mT,\mf}=\emptyset$. 
Then we have $S_{D'}(j_{!}\mf)=\bigcup_{i\in I_{\mII,\mf}}T^{*}_{D_{i}}X \cup \tau_{D'}^{-1}(S_{D'}^{\log}(j_{!}\mf))$,
since $C_{D'}$ (\ref{defcdp}) is contained in $\tau_{D'}^{-1}(T^{*}_{X}X(\log D'))$ by (\ref{eqinvzero})
and since $T^{*}_{D_{i}}X$ for $i\in I_{\mI,\mf}-I'$ (Definition \ref{defindsub} (2))
is equal to $L_{i,\mf}^{\emptyset}$ by \cite[Lemma 2.23 (ii)]{yacc} and is contained in
$\tau_{D'}^{-1}(L_{i,\mf}^{D'})$ by Remark \ref{rempropinvli}.
Since we have $L_{i,\mf}^{\emptyset}\neq T^{*}_{D_{i}}X$ for $i\in I_{\mII,\mf}$ (Definition \ref{defindsub} (2))
by \cite[Lemma 2.23 (ii)]{yacc},
we have $T^{*}_{D_{i}}X\not\subset SS(j_{!}\mf)$ for $i\in I_{\mII,\mf}$ 
by Proposition \ref{propssndeg} (2) and Remark \ref{remcompssnondeg} (2).
Therefore the singular support $SS(j_{!}\mf)$ is a union of irreducible components of $\tau_{D'}^{-1}(S_{D'}^{\log}(j_{!}\mf))$.
\end{proof}

\begin{cor}
\label{corsddim}
Suppose that $X$ is purely of dimension $d$ and that $D$ has simple normal crossings.
Assume that the ramification of $\mf$ is $\log$-$D'$-clean along $D$
for some union $D'$ of irreducible components of $D$.
Let \begin{equation}
f\colon X'=X_{s}\xrightarrow{f_{s}} X_{s-1}\xrightarrow{f_{s-1}} \cdots\xrightarrow{f_{1}} X_{0}=X \notag
\end{equation} 
be successive blow-ups satisfying the conditions (1)--(3) in Proposition \ref{propblupcth}
and let $j'\colon f^{*}U\rightarrow X'$ be the base change of $j$ by $f$.
Then the dimension of the inverse image
$\tau_{D_{\mI,f^{*}\mf}\cup D_{\mT,f^{*}\mf}}^{-1}
(S^{\log}_{D_{\mI,f^{*}\mf}\cup D_{\mT,f^{*}\mf}}(j'_{!}f^{*}\mf))$
of the $\log$-$D_{\mI,f^{*}\mf}\cup D_{\mT,f^{*}\mf}$-singular support
$S^{\log}_{D_{\mI,f^{*}\mf}\cup D_{\mT,f^{*}\mf}}(j'_{!}f^{*}\mf)\subset T^{*}X'(\log D_{\mI,f^{*}\mf}\cup D_{\mT,f^{*}\mf})$ 
(Definition \ref{deflifdp} (3)) by $\tau_{D_{\mI,f^{*}\mf}\cup D_{\mT,f^{*}\mf}}$ 
(\ref{deftaue}) is $d$, 
where $D_{\mI,f^{*}\mf}$ and $D_{\mT,f^{*}\mf}$ are as in (\ref{defd*mf}).
The singular support $SS(j'_{!}f^{*}\mf)$ 
is a union of irreducible components of the inverse image $\tau_{D_{\mI,f^{*}\mf}\cup D_{\mT,f^{*}\mf}}^{-1}
(S^{\log}_{D_{\mI,f^{*}\mf}\cup D_{\mT,f^{*}\mf}}(j'_{!}f^{*}\mf))$.
\end{cor}

\begin{proof}
By the condition (2) in Proposition \ref{propblupcth},
the ramification of $f^{*}\mf$ is $\log$-$
D_{\mI,f^{*}\mf}\cup D_{\mT,f^{*}\mf}$-clean along $(f^{*}D)_{\red}$.
Since $E_{f^{*}\mf}^{D_{\mI,f^{*}\mf}\cup D_{\mT,f^{*}\mf}}=\emptyset$ 
(Definition \ref{defbcf} (2)) by the condition (3) in Proposition \ref{propblupcth},
the first assertion follows from Corollary \ref{corinvli},
and the last assertion follows from Corollary \ref{corsdpdim}.
\end{proof}

\subsection{Homotopy invariance of characteristic cycles}
\label{sshicc}

Finally, we prove the following proposition saying that the pull-back
$\tau_{D'}^{!}CC_{D'}^{\log}(j_{!}\mf)$ (\ref{taudpgysin})
of the $\log$-$D'$-characteristic cycle $CC_{D'}^{\log}(j_{!}\mf)$ of $j_{!}\mf$ (Definition \ref{deflifdp} (4)) by $\tau_{D'}$ (\ref{deftaue})
determines the characteristic cycle $CC(j_{!}\mf)$ of $j_{!}\mf$ (Definition \ref{defcc})
when the ramification of $\mf$ is $\log$-$D'$-clean along $D$ and  
when the inverse image $\tau_{D'}^{-1}(S_{D'}^{\log}(j_{!}\mf))\subset T^{*}X$ 
of the $\log$-$D'$-singular support $S_{D'}^{\log}(j_{!}\mf)\subset T^{*}X(\log D')$ (Definition \ref{deflifdp} (3)) by $\tau_{D'}$ is of dimension $d$:

\begin{prop}
\label{prophicc}
Suppose that $X$ is purely of dimension $d$ and
that $D$ has simple normal crossings.
Let $I'\subset I$ be a subset and let $D'=\bigcup_{i\in I'}D_{i}$.
Let $\mf_{i}$ for $i=0,1$ be smooth sheaves of $\Lambda$-modules of
rank $1$ on $U$ whose ramifications are $\log$-$D'$-clean along $D$.
Assume that $I_{\mT,\mf_{i}}$ (Definition \ref{defindsub} (1)) for $i=0,1$ are contained in $I'$,
that the inverse image $\tau_{D'}^{-1}(S_{D'}^{\log}(j_{!}\mf_{i}))\subset T^{*}X$
of the $\log$-$D'$-singular support $S_{D'}^{\log}(j_{!}\mf_{i})\subset T^{*}X(\log D')$ (Definition \ref{deflifdp} (3)) by $\tau_{D'}$ (\ref{deftaue})
for $i=0,1$ are of dimension $d$, 
and that we have 
\begin{equation}
\tau_{D'}^{!}CC_{D'}^{\log}(j_{!}\mf_{0})=\tau_{D'}^{!}CC_{D'}^{\log}(j_{!}\mf_{1}) \notag
\end{equation}
in $Z_{d}(\tau_{D'}^{-1}(S_{D'}^{\log}(j_{!}\mf_{0})\cup S_{D'}^{\log}(j_{!}\mf_{1})))$,
where $\tau_{D'}^{!}$ is as in (\ref{taudpgysin}).
Then we have 
\begin{equation}
CC(j_{!}\mf_{0})=CC(j_{!}\mf_{1}). \notag
\end{equation}
\end{prop}

\begin{proof}
Since we have $E_{\mf_{i}}^{D'}=\emptyset$ by Corollary \ref{corinvli},
we have $I'\cap I_{\mII,\mf_{i}}=\emptyset$ for $i=0,1$ by Remark \ref{remdefbcf} (3).
By Corollary \ref{corcompcc} (2) and the comparison of the terms of $\tau_{D'}^{!}CC_{D'}^{\log}(j_{!}\mf_{0})$ and $\tau_{D'}^{!}CC_{D'}^{\log}(j_{!}\mf_{1})$ whose supports have 
bases of codimension $1$ in $X$,
we have $R_{\mf_{0}}^{D'}=R_{\mf_{1}}^{D'}$ (Definition \ref{defconddiv} (1))
and $I_{\mW,\mf_{0}}=I_{\mW,\mf_{1}}$ (Definition \ref{defindsub} (1)).
By the same comparison and Remark \ref{rempropinvli}, we also have
$L_{i',\mf_{0}}^{\emptyset}=L_{i',\mf_{1}}^{\emptyset}$
(Definition \ref{deflifdp} (2)) for every $i'\in I_{\mW,\mf_{0}}$.
Since $L_{i',\mf_{i}}^{\emptyset}=T^{*}_{D_{i'}}X$ if and only if $i'\in I_{\mI,\mf_{i}}$ 
(Definition \ref{defindsub} (2)) for
$i'\in I_{\mW,\mf_{i}}$ and $i=0,1$  
by \cite[Lemma 2.23 (ii)]{yacc},
we have $I_{\mI,\mf_{0}}=I_{\mI,\mf_{1}}$.
Therefore we have $\sw(\chi_{0}|_{K_{i'}})=\sw(\chi_{1}|_{K_{i'}})$ 
(Definition \ref{defwedt} (1)) for every $i'\in I$.
By Corollary \ref{corcompcc} (2), the subset
$(B_{I'',\mf_{i}}^{D'})^{0}\subset X$ for $I''\subset I'$ and the closed subscheme 
$Y_{\mathfrak{p}}\subset X$ for $\mathfrak{p}\in (B_{I'',\mf_{i}}^{D'})^{0}$
are independent of the choice of $i=1,2$, where the notation is as in Proposition \ref{propinvli}.

Let $\chi_{i}\colon \pi_{1}^{\ab}(U)\rightarrow \Lambda^{\times}$ be
the character corresponding to $\mf_{i}$ for $i=0,1$.
By Remark \ref{remlifdp} (2) and \cite[Theorem 0.1]{sy}, we may assume that $\chi_{i}$ for $i=0,1$ are of orders powers of $p$.
Let $s\ge 2$ be an integer such that both the orders of $\chi_{i}$ for $i=0,1$ are $\le p^{s}$.
Since the assertion is local, we may assume that $X=\Spec A$ 
for a $k$-algebra $A$.
By shrinking $X$ if necessary,
we can take global sections $a=(a_{s-1},\ldots,a_{1},a_{0})$ 
and $a'=(a_{s-1}',\ldots,a_{1}',a_{0}')$ of $\fillog_{R_{\mf_{0}}^{D'}}^{D'}j_{*}W_{s}(\dvr_{U})
=\fillog_{R_{\mf_{1}}^{D'}}^{D'}j_{*}W_{s}(\dvr_{U})$
whose images by $\delta_{s,j}\colon j_{*}W_{s}(\dvr_{U})\rightarrow R^{1}(\varepsilon\circ j)_{*}\mathbf{Z}/p^{s}\mathbf{Z}$ (\ref{deltassh}) are $\chi_{0}$ and $\chi_{1}$,
respectively.
Then the images of $a$ and $a'$ by the composition (\ref{compcffmwitt}) 
are equal to the $\log$-$D'$-characteristic forms
$\cform^{D'}(\mf_{0})$ and $\cform^{D'}(\mf_{1})$, respectively, by Remark \ref{remcfnormal} (2).
We put $\tilde{X}=X\times_{k}\mathbf{A}_{k}^{1}$, 
where $\mathbf{A}_{k}^{1}=\Spec k[T]$.
Let $\tilde{j}\colon \tilde{U}=U\times_{k}\mathbf{A}_{k}^{1}\rightarrow \tilde{X}$ denote the canonical open immersion and
let $b$ be the global section of $\tilde{j}_{*}W_{s}(\dvr_{\tilde{U}})$ defined by
\begin{equation}
\label{eqb}
b=(a_{s-1}(1-T)^{p},\ldots,a_{1}(1-T)^{p^{s-1}}, a_{0}(1-T)^{p^{s}})
+(a_{s-1}'T^{p},\ldots,a_{1}'T^{p^{s-1}}, a_{0}T^{p^{s}}). 
\end{equation}
Let $\varphi\colon \pi_{1}^{\ab}(\tilde{U})\rightarrow \mathbf{Z}/p^{s}\mathbf{Z}$ 
be the image of $b$ by $\delta_{s,\tilde{j}}\colon \tilde{j}_{*}W_{s}(\dvr_{\tilde{U}})\rightarrow R^{1}(\tilde{\varepsilon}\circ \tilde{j})_{*}\mathbf{Z}/p^{s}\mathbf{Z}$,
where $\tilde{\varepsilon}\colon \tilde{X}_{\et}\rightarrow \tilde{X}_{\mathrm{Zar}}$ is the canonical mapping
from the \'{e}tale site of $\tilde{X}$ to the Zariski site of $\tilde{X}$,
and let $\mg$ be a smooth sheaf of $\Lambda$-module of rank $1$ on $\tilde{U}$ whose associated character is $\varphi$.
We put $\tilde{D}_{i'}=D_{i'}\times_{k}\mathbf{A}_{k}^{1}\subset \tilde{X}$ for
$i'\in I$, $\tilde{D}=D\times_{k}\mathbf{A}_{k}^{1}=\bigcup_{i'\in I}\tilde{D}_{i'}$,
and $\tilde{D}'=D'\times_{k}\mathbf{A}_{k}^{1}=\bigcup_{i'\in I'}\tilde{D}_{i'}$.
Then we can identify the index set of irreducible components of $\tilde{D}$
and that of irreducible components of $\tilde{D}'$ with $I$ and $I'$, respectively.
Let $\tilde{K}_{i'}=\Frac \hat{\dvr}_{\tilde{X},\tilde{\mathfrak{p}}_{i'}}$ be the local field at the generic point $\tilde{\mathfrak{p}}_{i'}$ of $\tilde{D}_{i'}$
for $i'\in I$.
Then we have $\sw(\varphi|_{\tilde{K}_{i'}})=\sw(\chi_{0}|_{K_{i'}})$ for every $i'\in I$ by \cite[Lemma 5.4]{yacc},
and we have $I_{\mW,\mg}=I_{\mW,\mf_{0}}$.
Thus we have $Z_{\mg}=Z_{\mf_{0}}\times_{k}\mathbf{A}_{k}^{1}$ (Definition \ref{defconddiv} (2)).
By Remark \ref{remcfnormal} (4) and \cite[Lemma 5.5]{yacc},
we have
\begin{equation}
\cform^{\tilde{D}}(\mg)=\cform^{D}(\mf_{0})(1-T)^{p^{s}}+\cform^{D}(\mf_{1})T^{p^{s}} \notag
\end{equation}
in $\Gamma(Z_{\mg},\Omega_{\tilde{X}}^{1}(\log \tilde{D})(R_{\mg}^{\tilde{D}})|_{Z_{\mg}})$.
By Lemma \ref{lemcfatx} (2) applied to the case where $I'$ is $I$,
we have $I_{\mI,\mg}=I_{\mI,\mf_{0}}$,
and we have $\sw^{\tilde{D}'}(\varphi|_{\tilde{K}_{i'}})=\sw^{D'}(\chi_{0}|_{K_{i'}})$
(\ref{defswdp}) for every $i'\in I$.
Thus we have $R_{\mg}^{\tilde{D}'}=R_{\mf_{0}}^{D'}\times_{k}\mathbf{A}_{k}^{1}$.

Let $h_{i}\colon X=X\times_{k}\{i\}\rightarrow \tilde{X}$ be the canonical closed immersions for $i=0,1$.
Then the closed immersions $h_{i}$ for $i=0,1$ are $C_{\tilde{D}}$-transversal by \cite[Lemma 3.4.5]{sacc}.
By shrinking $X$ if necessary, we may assume that 
$D_{i'}=(t_{i'}=0)$ for $t_{i'}\in A$ for $i'\in I$.

\begin{lem}[{cf.\ \cite[Lemma 5.5]{yacc}}]
\label{lemtildecf}
Let the notation and the assumptions be as above.
Then the following hold:
\begin{enumerate}
\item Let $dh_{i,Z_{\mf_{i}}^{1/p}}^{\tilde{D}'}\colon (h_{i}^{*}\Omega_{\tilde{X}}^{1}(\log \tilde{D}')(R_{\mg}^{\tilde{D'}}))|_{Z_{\mf_{i}}^{1/p}}\rightarrow \Omega_{X}^{1}(\log D')(R_{\mf_{i}}^{D'})|_{Z_{\mf_{i}}^{1/p}}$
be the morphism yielded by $h_{i}$ for $i=0,1$.
Then we have
\begin{equation}
dh_{i,Z_{\mf_{i}}^{1/p}}^{\tilde{D}'}(h_{i}^{*}\cform^{\tilde{D'}}(\mg))
=\cform^{D'}(\mf_{i}) \notag
\end{equation}
for $i=0,1$
\item The morphism $\xi_{i'}^{\tilde{D}'}(\mg)$ (Definition \ref{defximf} (2))
for $i'\in I'\cap I_{\mW,\mg}$
is the sum of the compositions
\begin{equation}
\xi_{i'}^{D'}(\mf_{0})(1-T)^{p^{s}}\colon \dvr_{\tilde{X}}(R_{\mg}^{\tilde{D}'})|_{\tilde{D}_{i'}^{1/p}}
\xrightarrow{\times \cform^{D'}(\mf_{0})|_{D_{i'}^{1/p}}(1-T)^{p^{s}}}
\Omega_{\tilde{X}}^{1}(\log \tilde{D}')|_{\tilde{D}_{i'}^{1/p}}
\rightarrow
\dvr_{\tilde{D}_{i'}^{1/p}} \notag
\end{equation} 
and
\begin{equation}
\xi_{i'}^{D'}(\mf_{1})T^{p^{s}}\colon\dvr_{\tilde{X}}(R_{\mg}^{\tilde{D}'})|_{\tilde{D}_{i'}^{1/p}}
\xrightarrow{\times \cform^{D'}(\mf_{1})|_{D_{i'}^{1/p}}T^{p^{s}}}
\Omega_{\tilde{X}}^{1}(\log \tilde{D}')|_{\tilde{D}_{i'}^{1/p}}
\rightarrow
\dvr_{\tilde{D}_{i'}^{1/p}} \notag
\end{equation} 
of multiplications and the base change of
the residue mapping $\Omega_{\tilde{X}}^{1}(\log \tilde{D}')|_{\tilde{D}_{i'}}
\rightarrow \dvr_{\tilde{D}_{i'}}$ by the canonical morphism
$\tilde{D}_{i'}^{1/p}\rightarrow \tilde{D}_{i'}$.
\end{enumerate}
\end{lem}

\begin{proof}
We put $n_{i'}=\sw^{D'}(\mf_{0}|_{K_{i'}})$ for $i'\in I$ and
put $b=(b_{s-1},\ldots,b_{1},b_{0})$ for the global section $b$ (\ref{eqb})
of $\tilde{j}_{*}W_{s}(\dvr_{\tilde{U}})$.
Since we have $R_{\mg}^{\tilde{D}'}=R_{\mf_{i}}^{D'}\times_{k}\mathbf{A}_{k}^{1}$ for $i=0,1$,
the global sections 
$(a_{s-1}(1-T)^{p},\ldots,a_{1}(1-T)^{p^{s-1}}, a_{0}(1-T)^{p^{s}})$ and
$(a_{s-1}'T^{p},\ldots,a_{1}'T^{p^{s-1}}, a_{0}T^{p^{s}})$
of $\tilde{j}_{*}W_{s}(\dvr_{\tilde{U}})$ are global sections of
$\fillog_{R_{\mg}^{\tilde{D}'}}^{\tilde{D}'}\tilde{j}_{*}W_{s}(\dvr_{\tilde{U}})$,
and so is $b$.
Hence the image of $b$ by the composition (\ref{compcffmwitt}) is equal to
the $\log$-$\tilde{D}'$-characteristic form $\cform^{\tilde{D'}}(\mg)$ by
Remark \ref{remcfnormal} (2).
Since the image of $b$ by the composition (\ref{compcffmwitt}) is 
\begin{align}
-F^{s-1}db&=-\sum_{i=0}^{s-1}(a_{i}(1-T)^{p^{s-i}})^{p^{i}-1}d(a_{i}(1-T)^{p^{s-i}})
-\sum_{i=0}^{s-1}(a_{i}'T^{p^{s-i}})^{p^{i}-1}d(a_{i}'T^{p^{s-i}}) \notag \\ 
&=-(1-T)^{p^{s}}\sum_{i=0}^{s-1}a_{i}^{p^{i}-1}da_{i}-
T^{p^{s}}\sum_{i=0}^{s-1}a_{i}'^{p^{i}-1}da_{i}' \notag \\
&=(-F^{s-1}da)(1-T)^{p^{s}}+(-F^{s-1}da')T^{p^{s}} \notag
\end{align}
if $p\neq 2$ and is 
\begin{align}
&-F^{s-1}db+\sum_{\substack{i'\in I-I' \\ n_{i'}=2}}\sqrt{\overline{b_{0}t_{i'}^{2}}}dt_{i'}/t_{i'}^{2} \notag \\
&\qquad=-\sum_{i=0}^{s-1}(a_{i}(1-T)^{p^{s-i}})^{2^{i}-1}d(a_{i}(1-T)^{2^{s-i}})
+\sum_{\substack{i'\in I-I' \\ n_{i'}=2}}\sqrt{\overline{a_{0}(1-T)^{2^{s}}t_{i'}^{2}}}dt_{i'}/t_{i'}^{2} \notag \\
&\qquad\qquad\qquad -\sum_{i=0}^{s-1}(a_{i}'T^{2^{s-i}})^{2^{i}-1}d(a_{i}'T^{2^{s-i}})
+\sum_{\substack{i'\in I-I' \\ n_{i'}=2}}\sqrt{\overline{a_{0}'T^{2^{s}}t_{i'}^{2}}}dt_{i'}/t_{i'}^{2} \notag \\ 
&\qquad=-(1-T)^{2^{s}}\sum_{i=0}^{s-1}a_{i}^{2^{i}-1}da_{i}-
T^{2^{s}}\sum_{i=0}^{s-1}a_{i}'^{2^{i}-1}da_{i}'  \notag \\
&\qquad\qquad\qquad+\sum_{\substack{i'\in I-I' \\ n_{i'}=2}}(\sqrt{\overline{a_{0}t_{i'}^{2}}}(1-T)^{2^{s-1}}dt_{i'}/t_{i'}^{2}+\sqrt{\overline{a_{0}'t_{i'}^{2}}}T^{2^{s-1}}dt_{i'}/t_{i'}^{2})
\notag \\
&\qquad=(-F^{s-1}da)(1-T)^{2^{s}}+(1-T)^{2^{s-1}}\sum_{\substack{i'\in I-I' \\ n_{i'}=2}}\sqrt{\overline{a_{0}t_{i'}^{2}}}dt_{i'}/t_{i'}^{2} \notag \\
&\qquad\qquad\qquad+(-F^{s-1}da')T^{2^{s}} +T^{2^{s-1}}\sum_{\substack{i'\in I-I' \\ n_{i'}=2}}\sqrt{\overline{a_{0}'t_{i'}^{2}}}dt_{i'}/t_{i'}^{2}\notag
\end{align}
if $p=2$ by the construction of the morphism $\varphi_{s}^{(\tilde{D}'\subset \tilde{D},R_{\mg}^{\tilde{D}'})}$ given in Lemma \ref{lemcfdef}
and since the images of $a$ and $a'$ by the composition (\ref{compcffmwitt})
are $\cform^{D'}(\mf_{0})$ and $\cform^{D'}(\mf_{1})$, respectively,
the assertions hold by the constructions of $\varphi_{s}^{(D'\subset D,R_{\mf_{i}}^{D'})}$
for $i=1,2$ given in Lemma \ref{lemcfdef}.
\end{proof}

Let $x\in Z_{\mf_{0}}$ be a closed point
and let $x'$ be the unique point on $Z_{\mf_{0}}^{1/p}$
lying above $x$.
Let $x_{i}\in X\times_{k}\{i\}\subset \tilde{X}$ be the closed point corresponding to $x$  
via $h_{i}$ for $i=0,1$
and let $x_{i}'$ be the unique point on $Z_{\mg}^{1/p}$ lying above $x_{i}$.
Since the ramifications of $\mf_{i}$ are $\log$-$D'$-clean along $D$,
we have $\cform^{D'}(\mf_{i})(x')\neq 0$ for $i=0,1$ by Lemma \ref{lemeqtologdpcl} (1).
Hence we have $\cform^{\tilde{D}'}(\mg)(x_{i}')\neq 0$ for $i=0,1$ by Lemma \ref{lemtildecf} (1),
and the ramification of $\mg$ is $\log$-$\tilde{D}'$-clean along $\tilde{D}$ at $x_{i}$ for $i=0,1$ by Lemma \ref{lemeqtologdpcl} (1).
Thus the ramification of $\mg$ is $\log$-$\tilde{D}'$-clean along $\tilde{D}$ at any point on $X\times_{k}\{i\}$ for $i=0,1$ by Remarks \ref{remlogdpcl} (1) and (2).
Let $\tilde{V}$ be the subset of $\tilde{X}$ 
consisting of the points on $\tilde{X}$ where the ramification of $\mg$ is $\log$-$\tilde{D}'$-clean along $\tilde{D}$.
Then $\tilde{V}$ is an open subset of $\tilde{X}$ by Remark \ref{remlogdpcl} (2) 
and we have $X\times_{k}\{i\}\subset \tilde{V}$ for $i=0,1$.

We show that the canonical closed immersion $h_{i}'\colon X\times_{k}\{i\}\rightarrow \tilde{V}$ for $i=0,1$ are properly $SS((\tilde{j}_{!}\mg)|_{\tilde{V}})$-transversal 
and that we have $h_{0}'^{!}CC((\tilde{j}_{!}\mg)|_{\tilde{V}})=h_{1}'^{!}CC((\tilde{j}_{!}\mg)|_{\tilde{V}})$,
which deduce the assertion by Theorem \ref{thmccpb} for we have $h_{i}'^{*}(\tilde{j}_{!}\mg)|_{\tilde{V}}=j_{!}\mf_{i}$ for $i=0,1$.
By replacing $\tilde{X}$ by $\tilde{V}$,
we may assume that the ramification of $\mg$ is $\log$-$\tilde{D}'$-clean along $\tilde{D}$.
Then the closed immersions
$h_{i}$ for $i=0,1$ are $\log$-$\tilde{D}'$-$S_{\tilde{D}'}^{\log}(\tilde{j}_{!}\mg)$-transversal by Proposition \ref{propeqcl} (1) and Lemma \ref{lemtildecf},
since $h_{i}$ is $C_{\tilde{D}'\subset \tilde{D}}$-transversal by Remark \ref{remctr} (3).
Since $h_{i}$ is $C_{\tilde{D}'}$-transversal by Remark \ref{remctr} (3),
the closed immersions $h_{i}$ for $i=0,1$ are 
$\tau_{\tilde{D}'}^{-1}(S_{\tilde{D}'}^{\log}(\tilde{j}_{!}\mg))$-transversal by Proposition \ref{proplogtr}.
Let the notation be as in Proposition \ref{propinvli} and Lemma \ref{lemtildecf}.
Then we have $B_{I'',\mg}^{\tilde{D}'}=\tilde{D}'_{I''}\cap \bigcup_{i'\in I-I'}\tilde{D}_{i'}=B_{I'',\mf_{i}}^{D'}\times_{k}\mathbf{A}_{k}^{1}$
for $I''\subset I_{\mT,\mg}=I_{\mT,\mf_{i}}$ and $i=0,1$.
By Lemma \ref{lemtildecf} (2),
we have
\begin{equation}
\label{eqbfbg}
\tilde{D}_{I''}\cap \bigcap_{i'\in I''}V(\Image \xi_{i'}^{\tilde{D}'}(\mg))
=\tilde{D}_{I''}\cap \bigcap_{i'\in I''}V(\Image (\xi_{i'}^{D'}(\mf_{0})(1-T)^{p^{s}}+\xi_{i'}^{D'}(\mf_{1})T^{p^{s}})) 
\end{equation}
for $I''\subset I'$ such that $I''\cap I_{\mW,\mg}\neq \emptyset$.
Hence we have $E_{\mg}^{\tilde{D}'}=\bigcup_{I''\subset I'}E_{I'',\mg}^{\tilde{D}'}=\emptyset$ (Definition \ref{defbcf} (2))
for $E_{\mf_{i}}^{D'}=\bigcup_{I''\subset I'}E_{I'',\mf_{i}}^{D'}=\emptyset$ for $i=0,1$.
Then the inverse image $\tau_{\tilde{D}'}^{-1}(S_{\tilde{D}'}^{\log}(\tilde{j}_{!}\mg))\subset T^{*}\tilde{X}$
is purely of dimension $d+1$ by Corollary \ref{corinvli}.
We denote the closed subscheme $Y_{\tilde{\mathfrak{p}}}\subset \tilde{X}$ 
and the section $\omega_{\tilde{\mathfrak{p}}}\in \Omega_{\tilde{X}}^{1}|_{Y_{\tilde{\mathfrak{p}}}^{1/p}}$ 
for $\tilde{\mathfrak{p}}\in (B_{I'',\mg}^{\tilde{D}'})^{0}$, where $I''\subset I'$, 
by $\tilde{Y}_{\tilde{\mathfrak{p}}}$ and $\tilde{\omega}_{\tilde{\mathfrak{p}}}$, repsectively.
By (\ref{eqbfbg}), the pull-back $h_{i}^{*}\tilde{Y}_{\tilde{\mathfrak{p}}}$ for $\tilde{\mathfrak{p}}\in (B_{I'',\mg}^{\tilde{D}'})^{0}$
is equal to $Y_{\mathfrak{p}}$ if $\tilde{\mathfrak{p}}$ is the generic point of $Y_{\mathfrak{p}}\times \mathbf{A}_{k}^{1}$
for an element $\mathfrak{p}\in (B_{I'',\mf_{0}}^{D'})^{0}$,
and is empty if otherwise.
By Lemma \ref{lemtildecf} (1), we may assume that the image of 
$h_{i}^{*}\tilde{\omega}_{\tilde{\mathfrak{p}}}$ by the morphism
$dh_{i,h_{i}^{*}\tilde{Y}_{\tilde{\mathfrak{p}}}^{1/p}}\colon\Omega_{\tilde{X}}^{1}|_{h_{i}^{*}\tilde{Y}_{\tilde{\mathfrak{p}}}^{1/p}}
\rightarrow \Omega_{X}^{1}|_{Y_{\mathfrak{p}}^{1/p}}$ yielded by $h_{i}$ is equal to $\omega_{\mathfrak{p}}$
for each element
$\tilde{\mathfrak{p}}\in (B_{I'',\mf_{0}}^{D'})^{0}$ that is the generic point of $Y_{\mathfrak{p}}\times \mathbf{A}_{k}^{1}$
for an element $\mathfrak{p}\in (B_{I'',\mf_{0}}^{D'})^{0}$
and for $i=0,1$.
Therefore the dimension of $h_{i}^{*}\tau_{\tilde{D}'}^{-1}(S_{\tilde{D}'}^{\log}(\tilde{j}_{!}\mg))$ 
is purely of dimension $d$ 
and we have $h_{0}^{\circ}\tilde{C}_{a}=h_{1}^{\circ}\tilde{C}_{a}$ 
for each irreducible component $\tilde{C}_{a}$ of $\tau_{\tilde{D}'}^{-1}(S_{\tilde{D}'}^{\log}(\tilde{j}_{!}\mg))$ 
by Corollary \ref{corcompcc} (1).
Since the singular support $SS(\tilde{j}_{!}\mg)$ is a union
of irreducible components of $\tau_{\tilde{D}'}^{-1}(S_{\tilde{D}'}^{\log}(\tilde{j}_{!}\mg))$
by Corollary \ref{corsdpdim},
the closed immersion $h_{i}$ for $i=0,1$ are properly $SS(\tilde{j}\mg)$-transversal by Remark \ref{remctr} (3), and
we have $h_{0}^{!}CC(\tilde{j}_{!}\mg)=h_{1}^{!}CC(\tilde{j}_{!}\mg)$.
\end{proof}

We formulate the following conjecture inspired by Proposition \ref{prophicc}:

\begin{conj}[{cf.\ \cite[Conjecture 3.14]{yacc}}]
\label{conjcc}
Suppose that $X$ is purely of dimension $d$ and
that $D$ has simple normal crossings.
Let $I'\subset I$ be a subset containing $I_{\mT,\mf}$ (Definition \ref{defindsub} (1)) and let $D'=\bigcup_{i\in I'}D_{i}$.
Assume that the ramification of $\mf$ is $\log$-$D'$-clean along $D$
and that the inverse image $\tau_{D'}^{-1}(S_{D'}^{\log}(j_{!}\mf))\subset T^{*}X$
of the $\log$-$D'$-singular support $S_{D'}^{\log}(j_{!}\mf)\subset T^{*}X(\log D')$ 
(Definition \ref{deflifdp} (3)) by the canonical morphism $\tau_{D'}$ (\ref{deftaue})
is of dimension $d$.
Then we have
\begin{equation}
\label{eqdesired}
CC(j_{!}\mf)=\tau_{D'}^{!}CC_{D'}^{\log}(j_{!}\mf) 
\end{equation}
in $Z_{d}(\tau_{D'}^{-1}(S_{D'}^{\log}(j_{!}\mf)))$,
where $\tau_{D'}^{!}$ is as in (\ref{taudpgysin}). 
\end{conj}

\begin{rem}
Suppose $X$ is purely of dimension $d$, that $D$ has simple normal crossings,
and that the ramification of $\mf$ is $\log$-$D'$-clean along $D$ for
some union $D'$ of irreducible components of $D$.
If Conjecture \ref{conjcc} holds generally and if we admit blowing up $X$
along a closed subscheme of $D$, then the characteristic cycle $CC(j_{!}\mf)$ is 
computed in terms of ramification theory as follows:
Let $f\colon X'\rightarrow X$ be the composition of successive blow-ups satisfying the conditions (1)--(3) in Proposition \ref{propblupcth}
and let $j'\colon f^{*}U\rightarrow X'$ be the base change of $j$ by $f$.
Then the inverse image $\tau_{D_{\mI,f^{*}\mf}\cup D_{\mT,f^{*}\mf}}^{-1}(S_{D_{\mI,f^{*}\mf}\cup D_{\mT,f^{*}\mf}}^{\log}(j'_{!}f^{*}\mf))\subset T^{*}X$
of the $\log$-$D_{\mI,f^{*}\mf}\cup D_{\mT,f^{*}\mf}$-singular support 
$S_{D_{\mI,f^{*}\mf}\cup D_{\mT,f^{*}\mf}}^{\log}(j'_{!}f^{*}\mf)\subset T^{*}X'(\log D_{\mI,f^{*}\mf}\cup D_{\mT,f^{*}\mf})$ 
(Definition \ref{deflifdp} (3)) by 
$\tau_{D_{\mI,f^{*}\mf}\cup D_{\mT,f^{*}\mf}}$ (\ref{deftaue}),
where $D_{\mI,f^{*}\mf}$ and $D_{\mT,f^{*}\mf}$ are as in (\ref{defd*mf}),
is of dimension $d$ by Corollary \ref{corsddim}, and 
Conjecture \ref{conjcc} deduces the equality
\begin{equation}
CC(j'_{!}f^{*}\mf)=\tau_{D_{\mI,f^{*}\mf}\cup D_{\mT,f^{*}\mf}}^{!}CC_{D_{\mI,f^{*}\mf}\cup D_{\mT,f^{*}\mf}}^{\log}(j'_{!}f^{*}\mf) \notag
\end{equation}
in $Z_{d}(\tau_{D_{\mI,f^{*}\mf}\cup D_{\mT,f^{*}\mf}}^{-1}(S_{D_{\mI,f^{*}\mf}\cup D_{\mT,f^{*}\mf}}^{\log}(j'_{!}f^{*}\mf)))$.
\end{rem}

\section{Computation of characteristic cycle in codimension $2$}
\label{ssrk1}

In this section, 
we prove Conjecture \ref{conjcc} under the assumption
that the bases of irreducible components
of the inverse image $\tau_{D'}^{-1}(S_{D'}^{\log}(j_{!}\mf))\subset T^{*}X$
of the $\log$-$D'$-singular support $S_{D'}^{\log}(j_{!}\mf)\subset T^{*}X(\log D')$ (Definition \ref{deflifdp} (3)) by the canonical morphism $\tau_{D'}$ (\ref{deftaue})
are of codimension $\le 2$ in $X$.
The equality (\ref{eqdesired}) in Conjecture \ref{conjcc}
gives a computation of the characteristic cycle $CC(j_{!}\mf)$
(Definition \ref{defcc})
and further a computation of the singular support $SS(j_{!}\mf)$ 
(Definition \ref{defss} (2)) as the support of $CC(j_{!}\mf)$
in terms of ramification theory.

\subsection{Computations in known cases} 
\label{sscompndeg}

We first recall the computation of the characteristic cycle $CC(j_{!}\mf)$
in the case where the ramification of $\mf$ is $\log$-$\emptyset$-clean along $D$
given in \cite{sacc}.

\begin{thm}[{cf.\ \cite[Theorem 7.14]{sacc}}]
\label{thmccndeg}
Suppose that $X$ is purely of dimension $d$ and that
$D$ has simple normal crossings.
Assume that
the ramification of $\mf$ is $\log$-$\emptyset$-clean along $D$.
\begin{enumerate}
\item If $I=I_{\mT,\mf}$ (Definition \ref{defindsub} (1)), then we have 
\begin{align}
CC(j_{!}\mf)=(-1)^{d}\sum_{I''\subset I}[T^{*}_{D_{I''}}X], \notag
\end{align}
where $T^{*}_{D_{I''}}X$ denotes the conormal bundle of $D_{I''}=\bigcap_{i\in I''}D_{i}\subset X$.
\item If $I=I_{\mW,\mf}$ (Definition \ref{defindsub} (1)), then we have 
\begin{align}
CC(j_{!}\mf)=
(-1)^{d}([T^{*}_{X}X]+\sum_{i\in I}\dt(\chi|_{K_{i}})[L_{i,\mf}^{\emptyset}]),
\notag
\end{align}
where $L_{i,\mf}^{\emptyset}$ is as in Definition \ref{deflifdp} (2).
\end{enumerate}
\end{thm}

\begin{proof}
By Remark \ref{remlogdpcl} (3),
the assertions are special cases of \cite[Theorem 7.14]{sacc}.
\end{proof}

\begin{rem}
\label{remcompnondeg}
Suppose that $D$ has simple normal crossings.
\begin{enumerate}
\item As is seen in Remark \ref{remcompssnondeg} (1), 
the $\log$-$\emptyset$-cleanliness is equivalent to the strong non-degeneration
in the sense of \cite[the remark before Proposition 4.13]{sacc} for the ramification of $\mf$ along $D$.
\item Theorem \ref{thmccndeg} generally gives a computation of 
the characteristic cycle $CC(j_{!}\mf)$ outside a closed subscheme of $X$ of 
codimension $\ge 2$ by Remarks \ref{remlogdpcl} (2) and \ref{remsupcc} (1).
\end{enumerate}
\end{rem}

We then recall the computation of the characteristic cycle $CC(j_{!}\mf)$ 
in the case where $X$ is a surface, namely is purely of dimension $2$, given in \cite{yacc}.

Suppose that $X$ is a surface.
We first define 
two additional invariants $\lambda_{x}\in \mathbf{Z}$ and $s_{x}\in \mathbf{Z}$ 
for the ramification of $\mf$ at a closed point $x\in D$ 
as follows (\cite[Remark 5.8]{kalog}):
If $x\notin Z_{\mf}$ (Definition \ref{defconddiv} (2)), then we define both $\lambda_{x}$ and $s_{x}$ to be $0$.
If $x \in Z_{\mf}$, then,
by \cite[Theorem 4.1]{kalog}, we can take successive blow-ups 
\begin{equation}
\label{seqbl}
f: X^{\prime}=X_{s}\xrightarrow{f_{s}} X_{s-1}\xrightarrow{f_{s-1}} \cdots \xrightarrow{f_{1}} X_{0}=X, 
\end{equation}
where $f_{1}\colon X_{1}\rightarrow X_{0}=X$ is the blow-up at $x_{0}=x$
and $f_{i}\colon X_{i}\rightarrow X_{i-1}$ for $i=2,3,\ldots,s$ is 
the blow-up at a closed point $x_{i-1}$ of $X_{i-1}$ lying over $x_{0}$, such that
the ramification of $f^{*}\mf$ is $\log$-$f^{*}D$-clean along $f^{*}D$ at every point on $f^{-1}(x_{0})$.
Let $\{D_{i'}^{(i)}\}_{i'\in I_{i}}$ be the irreducible components of the pull-back of $D$ to $X_{i}$ 
and let $K_{i'}^{(i)}$ be the local field at the generic point of $D_{i'}^{(i)}$ for 
$i=0,1,\ldots,s$ and $i'\in I_{i}$,
where $I_{0}=I$ and $D_{i'}^{(0)}=D_{i'}$ for $i'\in I_{0}$.
Let $r_{i}$ for $i=0,1,\ldots,s-1$ be the cardinality of
\begin{align}
I_{x_{i}}=\{i'\in I_{i}\; |\; x_{i}\in D_{i'}^{(i)}\}. \notag
\end{align}
Then $r_{i}$ is $1$ or $2$ for each $i=0,1,\ldots,s-1$.
By renumbering $\{D_{i'}^{(i)}\}_{i'\in I_{i}}$ if necessary, we may assume that 
$0\in I_{i}$ and that $D_{0}^{(i)}$ is the exceptional divisor $f_{i}^{-1}(x_{i-1})$ of the blow-up $f_{i}$ for 
$i=1,2,\ldots,s$.
We put
\begin{align}
\label{defei}
e_{i}=\sum_{i'\in I_{x_{i-1}}}\sw(\chi|_{K_{i'}^{(i-1)}})-\sw(\chi|_{K_{0}^{(i)}})
\end{align}
and 
\begin{align}
\mu_{i}=\begin{cases}e_{i}(e_{i}-1) &(r_{i-1}=1), \\
e_{i}^{2} &(r_{i-1}=2)
\end{cases} \notag
\end{align}
for $i=1,2,\ldots,s$.
We define an integer $\lambda_{x}$ by
\begin{equation}
\label{deflambdax}
\lambda_{x}=\sum_{i=1}^{s}\mu_{i}, 
\end{equation}
and define $s_{x}$ by
\begin{equation}
s_{x}=\sum_{i\in I_{\mW,\mf,x}}\sw(\chi|_{K_{i}})\ord^{D}(\mf; x,D_{i})-\lambda_{x},
\notag
\end{equation}
where $I_{\mW,\mf,x}$ is as in (\ref{eachindatx}) and
$\ord^{D}(\mf;x,D_{i})$ is as in Definition \ref{deforder}.

\begin{rem}
\label{remsx}
Suppose that $X$ is a surface. Let the notation be as above.
\begin{enumerate}
\item If $x$ is a closed point of $Z_{\mf}$, then the integer $\lambda_{x}$ 
and hence $s_{x}$ 
are independent of the choice of a sequence (\ref{seqbl})
of blow-ups by \cite[Remark 5.7]{kalog}. 
Consequently,
if the ramification of $\mf$ is $\log$-$D$-clean along $D$ at $x$,
then we have $s_{x}=\lambda_{x}=0$.
\item[(2)] It is conjectured that 
$s_{x}\ge 0$ for any closed point $x$
of $D$ (\cite[Remark 5.8]{kalog}).
\end{enumerate}
\end{rem}

Then the computation of the characteristic cycle $CC(j_{!}\mf)$
in the case where $X$ is a surface given in \cite{yacc} is as follows:

\begin{thm}[{\cite[Theorem 6.1]{yacc}}]
\label{thmcalsf}
Suppose that $X$ is a surface and that $D$ has simple normal crossings.
Let $|D|$ be the set of closed points of $D$.
Let $\lambda_{x}$ and $s_{x}$ for $x\in |D|$ be as above
and let $\ord^{\emptyset}(\mf;x,D_{i})$ and $\ord^{D}(\mf;x,D_{i})$
for $x\in |D|$ and for $i\in I_{\mW,\mf}$ (Definition \ref{defindsub} (1))
be as in Definition \ref{deforder}.
Let $r_{x}$ be the cardinality of $I_{x}$ (\ref{defix}) for each $x\in |D|$
and let $\delta_{r_{x},2}$ denote the Kronecker delta.
Then we have
\begin{equation}
\label{eqccsurf}
CC(j_{!}\mf)=[T^{*}_{X}X]+\sum_{i\in I}\dt(\chi|_{K_{i}})[L_{i,\mf}^{\emptyset}]
+\sum_{x\in |D|}t_{x}[T^{*}_{x}X], 
\end{equation}
where $L_{i,\mf}^{\emptyset}$ for $i\in I_{\mW,\mf}$ (Definition \ref{defindsub} (1)) is as in Definition \ref{deflifdp} (2),
where we put $L_{i,\mf}^{\emptyset}=T^{*}_{D_{i}}X$ for $i\in I_{\mT,\mf}$
(Definition \ref{defindsub} (1)), and where we have 
\begin{align}
t_{x}&=s_{x}+\delta_{r_{x},2}+\sum_{i\in I_{\mW,\mf,x}}
\sw(\chi|_{K_{i}})(\ord^{\emptyset}(\mf;x,D_{i})-\ord^{D}(\mf;x,D_{i})) \notag \\
&\qquad \qquad \qquad \quad \; +\sum_{i\in I_{\mII,\mf,x}}(-\delta_{r_{x},2}+\ord^{D}(\mf;x,D_{i})) \notag \\
&=-\lambda_{x}+\delta_{r_{x},2}+\sum_{i\in I_{\mW,\mf,x}}\sw(\chi|_{K_{i}})\ord^{\emptyset}(\mf;x,D_{i})
+\sum_{i\in I_{\mII,\mf,x}}(-\delta_{r_{x},2}+\ord^{D}(\mf;x,D_{i})). \notag
\end{align}
\end{thm}

Finally in this subsection,
we give an example of computation of 
$\lambda_{x}$ (\ref{deflambdax}) for a closed point $x$ of $Z_{\mf}$, which
we use in the proof of the main theorem
Theorem \ref{mainthmosct} in the next subsection.

\begin{exa} 
\label{exalamdax}
Let $X=\Spec A$ be a smooth affine surface over $k$ and let $x\in X$ be a closed point corresponding to the maximal ideal $(t_{1},t_{2})$ of $A$
generated by two elements $t_{1},t_{2}$ of $A$.
Let $D_{i}=(t_{i}=0)$ for $i=1,2$ and let $D=D_{1}\cup D_{2}$.
Suppose that the $\log$-$D$-characteristic form $\cform^{D}(\mf)\in \Gamma(Z_{\mf},\Omega_{X}^{1}(\log D)(R_{\mf}^{D})|_{Z_{\mf}})$
(Definition \ref{defcform}) is of the form
\begin{equation}
\label{formcformexa}
\cform^{D}(\mf)=\frac{\alpha_{1}d\log t_{1}+\beta_{2}t_{2}d\log t_{2}}{t_{1}^{n_{1}}t_{2}^{n_{2}}}, 
\end{equation}
where $n_{1}=\sw(\chi|_{K_{1}})> 0$, $n_{2}=\sw(\chi|_{K_{2}})\ge p$, $\alpha_{1}\in A/t_{1}t_{2}A-t_{2}A/t_{1}t_{2}A$, and
$\beta_{2}\in (A/t_{1}t_{2}A)^{\times}$.
Here $Z_{\mf}=\Spec (A/t_{1}t_{2}A)$.
Then $\ord^{D}(\mf;x,D_{2})$ (Definition \ref{deforder}) is equal to the (normalized) valuation $n$ of the image of $\alpha_{1}$ in the local ring $\dvr_{D_{2},x}$ of $D_{2}$ at $x$.
We can prove that $\lambda_{x}=n$ by the induction on $n$ as follows. 

If $n=0$, then the ramification of $\mf$ is $\log$-$D$-clean along $D$ at $x$ by Lemma \ref{lemorder}, and we have $\lambda_{x}=0$
by Remark \ref{remsx} (1).

If $n>0$,
then the ramification of $\mf$ is not $\log$-$D$-clean along $D$ at $x$
by Lemma \ref{lemorder} and
we have $\alpha_{1}\in (t_{1},t_{2})A/t_{1}t_{2}A-t_{2}A/t_{1}t_{2}A$.
Since the assertion is local, we may assume that the ramification of $\mf$ is $\log$-$D$-clean along $D$ except at $x$.
Let $f_{1}\colon X_{1}\rightarrow X_{0}=X$ be the blow-up at $x$.
Let $D_{0}^{(1)}=f_{1}^{-1}(x)\subset X_{1}$ be the exceptional divisor
and let $D_{i}^{(1)}$ be the proper transform of $D_{i}$ for $i=1,2$.
We put $U_{i}^{(1)}=X_{1}-D_{i}^{(1)}$ for $i=1,2$.
Let $x^{(1)}\in U_{1}^{(1)}$ be the unique closed point of the intersection $D_{0}^{(1)}\cap 
D_{2}^{(1)}\cap U_{1}^{(1)}\subset U_{1}^{(1)}$ and
let $e_{1}$ be as in (\ref{defei}) for the blow-up $f_{1}$.
By induction, it is sufficient to prove that the following three conditions hold:
\begin{enumerate}
\item The ramification of $f_{1}^{*}\mf$ is $\log$-$(f_{1}^{*}D)_{\red}$-clean along $(f_{1}^{*}D)_{\red}$ at every point on $X_{1}=U_{1}^{(1)}\cup U_{2}^{(1)}$ except
at $x^{(1)}$.
\item The $\log$-$(f_{1}^{*}D)_{\red}$-characteristic form $\cform^{(f_{1}^{*}D)_{\red}}(f_{1}^{*}\mf)$ is of the form (\ref{formcformexa}) in 
a neighborhood of $x^{(1)}$.
\item $e_{1}=1$ and $\ord^{(f_{1}^{*}D)_{\red}}(f_{1}^{*}\mf; x^{(1)}, D_{2}^{(1)})=n-1$.
\end{enumerate}
Actually, if $n=1$, then the second equation in the condition (3) implies 
that the ramification of $f_{1}^{*}\mf$ is $\log$-$(f_{1}^{*}D)_{\red}$-clean along $(f_{1}^{*}D)_{\red}$ at $x^{(1)}$ by Lemma \ref{lemorder}.
Thus the ramification of $f_{1}^{*}\mf$ is $\log$-$(f_{1}^{*}D)_{\red}$-clean along $(f_{1}^{*}D)_{\red}$
by the condition (1), and we have $\lambda_{x}=e_{1}^{2}=1$ 
by the first equation in the condition (3).
If $n\ge 2$, then we can use the induction by the conditions (1), (2),
and the second equation in (3)
so that we have 
\begin{equation}
\lambda_{x}=\sum_{i=1}^{n}e_{i}^{2}=n \notag
\end{equation}
by the first equation in the condition (3)
and the induction hypothesis.

We first compute $\cform^{(f_{1}^{*}D)_{\red}}(f_{1}^{*}\mf)$
in order to show that the three conditions (1), (2), and (3) above hold.
Since the $\log$-$D$-characteristic form is only dependent 
on the $p$-part of the character $\chi$ associated to $\mf$ 
by Remark \ref{remcfnormal} (2),
we may assume that the order of $\chi$ is $p^{s}$ for $s\ge 0$.
Since the assertions are local, we may assume that $\chi$ is the image of 
a global section $a$ of $\fillog_{R_{\mf}^{D}}^{D}j_{*}W_{s}(\dvr_{U})\subset j_{*}W_{s}(\dvr_{U})$
by $\delta_{s,j}$ (\ref{deltassh}).
Then the image of $a$ 
by the composition
(\ref{compcffmwitt})
is $\cform^{D}(\mf)$ by Remark \ref{remcfnormal} (2), and
the image of $a$ by $-F^{s-1}d\colon j_{*}W_{s}(\dvr_{U})\rightarrow j_{*}\Omega_{U}^{1}$ (\ref{fdsh}) is of the form
\begin{equation}
\label{eqcompfsa}
-F^{s-1}da=\frac{\alpha_{1}d\log t_{1}+\beta_{2}t_{2}d\log t_{2}+t_{1}t_{2}\gamma}{t_{1}^{n_{1}}t_{2}^{n_{2}}},
\end{equation}
where $\alpha_{1}$ (resp.\ $\beta_{2}$) is a lift of $\alpha_{1}\in A/t_{1}t_{2}A$ (resp.\ $\beta_{2}\in A/t_{1}t_{2}A$) in $A$ by abuse of notation
and $\gamma$ is a global section of $\Omega_{X}^{1}(\log D)$.
Since $n>0$ is the valuation of the image of $\alpha_{1}$ in $\dvr_{D_{2},x}=A_{(t_{1},t_{2})}/t_{2}A_{(t_{1},t_{2})}$,
we can put $\alpha_{1}=ut_{1}^{n}+vt_{2}^{m}$ for $u,v\in A$ and $m\in \mathbf{Z}_{> 0}$ such that $u$ is invertible in $\dvr_{D_{2},x}$. 

Let $j^{(1)}\colon f_{1}^{*}U\rightarrow X_{1}$ denote the canonical open immersion.
Since $\alpha_{1}=0$ in $A/(t_{1},t_{2})A$, 
both $f_{1}^{*}\alpha_{1}|_{D_{0}^{(1)}}$ and $f_{1}^{*}(\beta_{2}t_{2})|_{D_{0}^{(1)}}$ are $0$ and we have
$df_{1}^{D}(f_{1}^{*}\cform^{D}(\mf))|_{D_{0}^{(1)}}=0$. 
Therefore $f_{1}^{*}a\in j^{(1)}_{*}W_{s}(\dvr_{f_{1}^{*}U})$ is a section of $\fillog^{f_{1}^{*}D}_{f_{1}^{*}R_{\mf}^{D}-D_{0}^{(1)}}j^{(1)}_{*}W_{s}(\dvr_{f_{1}^{*}U})$
by Lemma \ref{lemrsw} (1). 
We denote $f^{*}t_{i}$ by $t_{i}$ in $U_{i}^{(1)}$ for $i=1,2$ and we put
$f^{*}t_{2}=t_{1}t_{2}$ (resp.\ $f^{*}t_{1}=t_{1}t_{2}$) in $U_{1}^{(1)}$ (resp.\ in $U_{2}^{(1)}$) by abuse of notation.
Then the image $-F^{s-1}df_{1}^{*}a$ of $f_{1}^{*}a$  
by $-F^{s-1}d\colon j^{(1)}_{*}W_{s}(\dvr_{f_{1}^{*}U})\rightarrow j_{*}^{(1)}\Omega_{f_{1}^{*}U}^{1}$ is of the form
\begin{equation}
-F^{s-1}df_{1}^{*}a=\frac{(t_{1}^{-1}f_{1}^{*}\alpha_{1}+t_{2}f_{1}^{*}\beta_{2})d\log t_{1}+t_{2}f^{*}\beta_{2}d\log t_{2}+t_{1}t_{2}f_{1}^{*}\gamma}{t_{1}^{n_{1}+n_{2}-1}t_{2}^{n_{2}}} \notag
\end{equation}
in $U_{1}^{(1)}$
and 
\begin{equation}
-F^{s-1}df_{1}^{*}a=\frac{t_{2}^{-1}f_{1}^{*}\alpha_{1}d\log t_{1}+(t_{2}^{-1}f_{1}^{*}\alpha_{1}+f_{1}^{*}\beta_{2})d\log t_{2}+t_{1}t_{2}f_{1}^{*}\gamma}{t_{1}^{n_{1}}t_{2}^{n_{1}+n_{2}-1}} \notag
\end{equation}
in $U_{2}^{(1)}$ by (\ref{eqcompfsa}) and
is a section of $\Omega_{X_{1}}^{1}(\log (f_{1}^{*}D)_{\red})(f_{1}^{*}R_{\mf}^{D}-D_{0}^{(1)})$.
By the assumption that $n>0$, we have
\begin{equation}
\label{fonealphauone}
t_{1}^{-1}f_{1}^{*}\alpha_{1}=f^{*}ut_{1}^{n-1}+f^{*}vt_{1}^{m-1}t_{2}^{m}
\end{equation}
in $U_{1}^{(1)}$ and we have
\begin{align}
t_{2}^{-1}f_{1}^{*}\alpha_{1}&=
f^{*}ut_{1}^{n}t_{2}^{n-1}+f^{*}vt_{2}^{m-1} \notag
\end{align}
in $U_{2}^{(1)}$.
Since $u$ is invertible in $A_{(t_{1},t_{2})}/t_{2}A_{(t_{1},t_{2})}$ and
$\beta_{2}$ is invertible in $A/t_{1}t_{2}A$,
we have $-F^{s-1}df_{1}^{*}a|_{D_{i}^{(1)}}$ is not $0$ in 
$\Omega_{X_{1}}^{1}(\log (f_{1}^{*}D)_{\red})(f_{1}^{*}R_{\mf}^{D}-D_{0}^{(1)})|_{D_{i}^{(1)}}$ for $i=0,1,2$.
Hence we have
\begin{equation}
\label{exaeqrfof}
R_{f_{1}^{*}\mf}^{(f_{1}^{*}D)_{\red}}=f_{1}^{*}R_{\mf}^{D}-D_{0}^{(1)}
\end{equation}
by Lemma \ref{lemrsw} (1),
and we have $Z_{f_{1}^{*}\mf}=(f_{1}^{*}Z_{\mf})_{\red}$.
Since the pull-back $f_{1}^{*}a$ is a section of $\fillog^{f_{1}^{*}D}_{f_{1}^{*}R_{\mf}^{D}-D_{0}^{(1)}}j^{(1)}_{*}W_{s}(\dvr_{f_{1}^{*}U})$, 
the section $-F^{s-1}df^{*}_{1}a|_{(f_{1}^{*}Z_{\mf})_{\red}}$ of 
$\Omega_{X_{1}}^{1}(\log (f_{1}^{*}D)_{\red})(f_{1}^{*}R_{\mf}^{D}-D_{0}^{(1)})|_{(f_{1}^{*}Z_{\mf})_{\red}}$ is equal to the $\log$-$(f_{1}^{*}D)_{\red}$-characteristic form
$\cform^{(f_{1}^{*}D)_{\red}}(f_{1}^{*}\mf)$ of $f_{1}^{*}\mf$ 
by Remark \ref{remcfnormal} (2).
Thus we have
\begin{equation}
\label{cfexanones}
\cform^{(f_{1}^{*}D)_{\red}}(f_{1}^{*}\mf)=\frac{(t_{1}^{-1}f_{1}^{*}\alpha_{1}+t_{2}f_{1}^{*}\beta_{2})d\log t_{1}+t_{2}f^{*}\beta_{2}d\log t_{2}}{t_{1}^{n_{1}+n_{2}-1}t_{2}^{n_{2}}}
\end{equation}
in $U_{1}^{(1)}$ and
\begin{equation}
\cform^{(f_{1}^{*}D)_{\red}}(f^{*}\mf)=\frac{t_{2}^{-1}f_{1}^{*}\alpha_{1}d\log t_{1}+(t_{2}^{-1}f_{1}^{*}\alpha_{1}+f_{1}^{*}\beta_{2})d\log t_{2}}{t_{1}^{n_{1}}t_{2}^{n_{1}+n_{2}-1}} \notag
\end{equation}
in $U_{2}^{(1)}$. 

We then prove that the three conditions (1), (2), and (3) hold.
Since $f_{1}\colon X_{1}\rightarrow X$ is an isomorphism outside of $x$ and since
$u$ and $\beta_{2}$ are invertible in $A_{(t_{1},t_{2})}/t_{2}A_{(t_{1},t_{2})}$ 
and $A/t_{1}t_{2}A$, respectively,
the condition (1) holds by Remarks \ref{remlogdpcl} (1) and (2) and Lemma \ref{lemeqtologdpcl} (1).
The condition (2) follows from (\ref{cfexanones}), since $f^{*}\beta_{2}$ is invertible in 
$Z_{f^{*}_{1}\mf}=(f_{1}^{*}Z_{\mf})_{\red}$.
The first equation in the condition (3) follows from
(\ref{exaeqrfof}).
Since $f_{1}^{*}u$ is invertible in $\dvr_{D_{2}^{(1)},x^{(1)}}$, 
the second equation in (3) holds by (\ref{fonealphauone}) and (\ref{cfexanones}).
\end{exa}

\subsection{Computation in codimension 2}
\label{sscompccvar}

In this subsection, we prove the following main theorem:

\begin{thm}[{cf.\ \cite[Theorems 4.2 (ii), 6.1]{yacc}}]
\label{mainthmosct}
Conjecture \ref{conjcc} holds 
if the bases of irreducible components
of the inverse image $\tau_{D'}^{-1}(S_{D'}^{\log}(j_{!}\mf))\subset T^{*}X$
of the $\log$-$D'$-singular support $S_{D'}^{\log}(j_{!}\mf)\subset T^{*}X(\log D')$ (Definition \ref{deflifdp} (3)) by the canonical morphism $\tau_{D'}$ (\ref{deftaue})
are of codimension $\le 2$ in $X$.
\end{thm}

In order to reduce the proof of Theorem \ref{mainthmosct}
to the case where $X$ is a surface, 
we first prove the local existence of a $C_{D}$-transversal
and properly $S_{D'}(j_{!}\mf)$-transversal regular closed immersion
$g\colon S\rightarrow X$ from a smooth surface $S$ over $k$,
where $C_{D}$ is as in (\ref{defcdp}).

\begin{lem}
\label{lembertini}
Suppose that $k$ is algebraically closed
and that $Y$ is a connected smooth quasi-projective scheme over $k$
of dimension $d\ge 2$.
Let $i'\colon Y\rightarrow \mathbf{P}$ be an immersion to
a projective space $\mathbf{P}$ over $k$
and let $\mathbf{P}^{\vee}$ denote the dual of $\mathbf{P}$.
Let $\{Y_{i}\}_{i=1}^{m}$ be irreducible closed subsets of $Y$. 
Then the set of hyperplanes $H\in \mathbf{P}^{\vee}$
satisfying both of the following conditions is a dense subset of $\mathbf{P}^{\vee}$:
\begin{enumerate}
\item $H$ meets $Y$ transversally and the intersection $H\cap Y$ is connected and 
is smooth over $k$. 
\item $H$ meets $Y_{i}$ properly for every $i=1,2,\ldots,m$.
\end{enumerate}
\end{lem}

\begin{proof}
By Bertini's theorem, the set $V$ of hyperplanes $H\in \mathbf{P}^{\vee}$
such that $H$ meets $Y$ transversally and that the intersection $H\cap Y$ is connected and is smooth over $k$
is a dense subset of $\mathbf{P}^{\vee}$.
Since the image of the fiber $Q_{y_{i}}$ of the universal hyperplane $Q=\{(y,H)\; |\; y\in H\}\subset \mathbf{P}\times_{k}\mathbf{P}^{\vee}$
at a closed point $y_{i}\in Y_{i}$ by the projection
$p^{\vee}\colon Q\rightarrow \mathbf{P}^{\vee}$ is a divisor $E_{y_{i}}$,
the union of the intersections $V\cap (\mathbf{P}^{\vee}-\bigcup_{i=1}^{n}E_{y_{i}})$ for 
the $m$-tuples $(y_{i})_{i=1}^{m}$ of the closed points
$y_{i}\in Y_{i}$ for $i=1,2,\ldots,m$ satisfies the desired conditions.
\end{proof}

\begin{lem}
\label{lemexistsh}
Suppose that $k$ is algebraically closed
and that $Y$ is a connected smooth quasi-projective scheme over $k$
of dimension $d\ge 2$.
Let $C\subset T^{*}Y$ be a closed conical subset of pure dimension $d$. 
Then there exists a properly $C$-transversal regular closed immersion $Y'\rightarrow Y$
of codimension $1$
from a connected smooth scheme $Y'$ over $k$.
\end{lem}

\begin{proof}
By Lemma \cite[Lemma 3.19]{sacc},
we can take a very ample invertible 
$\dvr_{Y}$-module $\mathcal{L}$ 
and a $k$-linear mapping $E\rightarrow \Gamma(Y,\mathcal{L})$
from a $k$-vector space $E$ of finite dimension
defining an immersion $i'\colon Y\rightarrow \mathbf{P}=\mathbf{P}(E^{\vee})
=\Spec S^{\bullet} E$ 
satisfying the following two conditions:
\begin{itemize}
\item[(E)] The composition 
$E\rightarrow \Gamma(Y,\mathcal{L})\rightarrow \mathcal{L}_{u}/\mathfrak{m}_{u}^{2}\mathcal{L}_{u}\oplus 
\mathcal{L}_{v}/\mathfrak{m}_{v}^{2}\mathcal{L}_{v}$ 
is a surjection for every pair $(u,v)$ of distinct closed points of $Y$.
\item[(C)] None of the inverse images $\tilde{C}_{a}$ of irreducible components $C_{a}$ of $C$ by the morphism
$di'\colon T^{*}\mathbf{P}\times_{\mathbf{P}}Y\rightarrow T^{*}Y$ (\ref{defdh})
is contained in the zero section $T^{*}_{\mathbf{P}}\mathbf{P}\times_{\mathbf{P}}Y$.
\end{itemize}
Here the condition (C) is always satisfied unless the immersion $i'\colon Y'\rightarrow \mathbf{P}$ is an open immersion.
We identify the universal hyperplanes $Q=\{(y,H)\; |\; y\in H\}\subset \mathbf{P}\times_{k}\mathbf{P}^{\vee}$ with the covariant projective space bundle $\mathbf{P}(T^{*}\mathbf{P})$
and the fiber product $Y\times_{\mathbf{P}}Q$ with $\mathbf{P}(T^{*}\mathbf{P}\times_{\mathbf{P}}Y)$
as in \cite[Subsection 3.2]{sacc}.
Then the closure of the image of the projectivization $\mathbf{P}(\tilde{C}_{a})\subset \mathbf{P}(T^{*}\mathbf{P}\times_{\mathbf{P}}Y)$
by the projection $p^{\vee}\colon Y\times_{\mathbf{P}}Q\rightarrow \mathbf{P}^{\vee}$
is a divisor on $\mathbf{P}^{\vee}$ by \cite[Corollary 3.21.1]{sacc}.
Therefore we can take $t\in \mathbf{P}^{\vee}-p^{\vee}(\mathbf{P}(\tilde{C}))$, where
$\mathbf{P}(\tilde{C})\subset \mathbf{P}(T^{*}\mathbf{P}\times_{\mathbf{P}}Y)$ is the projectivization of the inverse image $\tilde{C}$ of 
$C$ by $di'$,
such that the hyperplane
$H_{t}$ corresponding to $t$ satisfies the conditions 
(1) and (2) in Lemma \ref{lembertini} for $Y$ and 
the bases $\{Y_{i}\}_{i=1}^{m}$ of irreducible components of $C$
by loc.\ cit..
Let $h\colon H_{t}\cap Y\rightarrow Y$ be the canonical closed immersion.
Then the pull-back $h^{*}C\subset T^{*}X\times_{X}(H_{t}\cap Y)$ is purely of dimension $d-1$
by Lemma \ref{lemlcidim}. 
Since $H_{t}\cap Y$ and $Y$ are smooth over $k$,
the closed immersion $h\colon H_{t}\cap Y\rightarrow Y$ is a regular immersion,
and it is sufficient to prove that $h$ is $C$-transversal.

Since the projection $p\colon Y\times_{\mathbf{P}}Q\rightarrow Y$
is $C$-transversal at every point on  $Y\times_{\mathbf{P}}Q-\mathbf{P}(\tilde{C})$ by \cite[Lemma 3.10]{sacc},
the closed immersion $h\colon H_{t}\cap Y\rightarrow Y$ is $C$-transversal if and only if
the canonical closed immersion $h'\colon H_{t}\cap Y\rightarrow Y\times_{\mathbf{P}}Q$
is $p^{\circ}C$-transversal by \cite[Lemma 3.4.3]{sacc}.
We consider the cartesian diagram
\begin{equation}
\xymatrix{
Y\times_{\mathbf{P}}Q  \ar[d]_-{p^{\vee}} \ar@{}[dr] | {\square}
& H_{t}\cap Y \ar[d] \ar[l]_-{h'}\\
\mathbf{P}^{\vee} & t, \ar[l]
} \notag
\end{equation}
where the horizontal arrows are the canonical closed immersions.
Since the horizontal arrows are regular closed immersions of
the same codimension
and since the right vertical arrow is $C'$-transversal 
for every closed conical subset $C'\subset T^{*}(H_{t}\cap Y)$ by \cite[Lemma 3.6.1]{sacc},
the closed immersion $h'\colon H_{t}\cap Y\rightarrow Y\times_{\mathbf{P}}Q$ is $p^{\circ}C$-transversal
if and only if $p^{\vee}$ is $p^{\circ}C$-transversal
on a neighborhood of $H_{t}\cap Y\subset Y\times_{\mathbf{P}}Q$ by \cite[Lemma 3.9.1]{sacc}.
Since $Y\times_{\mathbf{P}}Q-\mathbf{P}(\tilde{C})$ is a neighborhood of
$H_{t}\cap Y$ and since the projection $p^{\vee}$ is $p^{\circ}C$-transversal
at every point on $Y\times_{\mathbf{P}}Q-\mathbf{P}(\tilde{C})$
by \cite[Lemma 3.10]{sacc},
the closed immersion $h\colon H_{t}\cap Y\rightarrow Y$ is $C$-transversal.
\end{proof}

\begin{cor}
\label{coreximpc}
Let the notation and the assumptions be as in Lemma \ref{lemexistsh}.
Let $e\le d$ be a positive integer.
Then there exists a properly $C$-transversal regular closed immersion 
$Y'\rightarrow Y$ from a connected smooth scheme $Y'$ over $k$ of dimension $e$.
\end{cor}

\begin{proof}
We prove the assertion by the induction on $d-e$.
If $d=e$, then the identity mapping of  
$Y$ satisfies the desired condition.
Suppose that $e<d$.
By the induction hypothesis, there exists a properly 
$C$-transversal regular closed immersion
$h''\colon Y''\rightarrow Y$ from a connected smooth 
scheme $Y''$ over $k$ of dimension $e+1$.
Then $h''^{*}C$ is purely of dimension $e+1$,
and so is $h''^{\circ}C$ by Remark \ref{remctr} (1). 
Since $Y''$ is quasi-projective over $k$,
we can take a properly $h''^{\circ}C$-transversal
regular closed immersion $h'\colon Y'\rightarrow Y''$
of codimension $1$ from a connected smooth scheme $Y'$ over $k$ by Lemma \ref{lemexistsh}.
Then the composition $h=h''\circ h'$ satisfies the desired conditions by \cite[Lemma 7.2.2]{sacc}.
\end{proof}

\begin{cor}
\label{corexistsri}
Suppose that $k$ is algebraically closed, that
$X$ is a connected smooth quasi-projective scheme over $k$ of dimension $d\ge 2$, and
that $D$ has simple normal crossings.
Let $D'$ be a union of irreducible components of $D$.
Assume that the ramification of $\mf$ is $\log$-$D'$-clean along $D$
and that the inverse image $\tau_{D'}^{-1}(S_{D'}^{\log}(j_{!}\mf))\subset T^{*}X$ 
of the $\log$-$D'$-singular support $S_{D'}^{\log}(j_{!}\mf)\subset T^{*}X(\log D')$ (Definition \ref{deflifdp} (3)) by $\tau_{D'}$ (\ref{deftaue}) 
is purely of dimension $d$.
Let $e\le d$ be a positive integer.
Then there exists a $C_{D}$-transversal, where $C_{D}$ is as in (\ref{defcdp}), and properly $\tau_{D'}^{-1}(S_{D'}^{\log}(j_{!}\mf))$-transversal regular closed immersion $Y\rightarrow X$
from a connected smooth scheme $Y$ over $k$ of dimension $e$.
\end{cor}

\begin{proof}
We put $C=C_{D}\cup \tau_{D'}^{-1}(S_{D'}^{\log}(j_{!}\mf))$. 
Then $C$ is purely of dimension $d$, and
there exists a properly $C$-transversal regular closed immersion $g\colon Y\rightarrow X$
from a connected smooth scheme $Y$ of dimension $e$
by Corollary \ref{coreximpc}.
By Remark \ref{remctr} (3),
the immersion $g$ satisfies the desired conditions.
\end{proof}

We then prove Theorem \ref{mainthmosct}.

\begin{proof} [Proof of Theorem \ref{mainthmosct}]
We may assume that $k$ is algebraically closed, by replacing $k$ by 
an algebraic closure of $k$ if necessary.
By Propositions \ref{proptowild} (2) and \ref{proptauinvtw},
we may assume that $I_{\mT,\mf}=\emptyset$.
Since the inverse image $\tau_{D'}^{-1}(S_{D'}^{\log}(j_{!}\mf))$ is 
assumed to be of dimension $d$, 
the inverse image $\tau_{D'}^{-1}(S_{D'}^{\log}(j_{!}\mf))$ is purely of dimension $d$ and
we have $E_{\mf}^{D'}=\emptyset$ (Definition \ref{defbcf} (2)) by Corollary \ref{corinvli}.
By Remark \ref{remdefbcf} (3), we have $I'\cap I_{\mII,\mf}=\emptyset$ (Definition \ref{defindsub} (2)).
Since the bases of irreducible components of the inverse image
$\tau_{D'}^{-1}(S_{D'}^{\log}(j_{!}\mf))$ 
are assumed to be of codimension $\le 2$ in $X$,
the bases $D_{I''}$ (\ref{defdipp}) of $T^{*}_{D_{I''}}X$ 
is of codimension $\le 2$ for any $I''\subset I'$ by Corollary \ref{corcompcc} (1),
and the cardinality of $I'$ is locally $\le 2$.
Since the inverse image $\tau_{D'}^{-1}(S_{D'}^{\log}(j_{!}\mf))$ 
is purely of dimension $d$,
the support of $\tau_{D'}^{!}CC_{D'}^{\log}(j_{!}\mf)$ is
a union of irreducible components of $\tau_{D'}^{-1}(S_{D'}^{\log}(j_{!}\mf))$ by Remark \ref{remlifdp} (4).
By Corollary \ref{corssccsm} (2) and Corollary \ref{corsdpdim},
the support of $CC(j_{!}\mf)$ is equal to the singular support $SS(j_{!}\mf)$
and is a union of irreducible components
of $\tau_{D'}^{-1}(S_{D'}^{\log}(j_{!}\mf))$. 
Therefore we may shrink $X$ to a neighborhood in $X$ of each generic point
of the bases of irreducible components of $\tau_{D'}^{-1}(S_{D'}^{\log}(j_{!}\mf))$,
since the assertion is local.
Then we may assume that the cardinality of $I_{x}$ (\ref{defix}) is $\le 2$ for
every point $x$ on $X$.
Let $(B_{I'',\mf}^{D'})^{0}$ for $I''\subset I'$ be as in Proposition \ref{propinvli}.
Since the bases of irreducible components of $\tau_{D'}^{-1}(S_{D'}^{\log}(j_{!}\mf))$
are of codimension $\le 2$ in $X$ and since $E_{\mf}^{D'}$ and
$I'\cap I_{\mII,\mf}$ are empty, 
the cardinalities of the subsets $I''\subset I'$ such that
$(B_{I'',\mf}^{D'})^{0}\neq \emptyset$ are $\le 1$ 
by Remark \ref{remdefbcf} (2) and Corollary \ref{corcompcc} (1).
Hence we may assume that 
there exists at most one irreducible component $C_{a}$ of 
$\tau_{D'}^{-1}(S_{D'}^{\log}(j_{!}\mf))$ whose base 
$C_{a}\cap T^{*}_{X}X$ is of codimension $2$ in $X$ by loc.\ cit..

Let $x$ be a closed point of $X$.
Since the assertion is local,
we may shrink $X$ to a neighborhood of $x$.
Suppose that $x\in X-D'$.
Then we have $I'=\emptyset$ in a neighborhood of $x$,
and we may assume that $I'=\emptyset$.
Since $\tau_{\emptyset}$ is the identity mapping
of $T^{*}X$ and since
$CC_{\emptyset}^{\log}(j_{!}\mf)=CC(j_{!}\mf)$ by Remark \ref{remlifdp} (5),
the assertion holds.

Suppose that $x\in D'$.
By Remark \ref{rempropinvli}, Corollary \ref{corcompcc} (2), and Theorem \ref{thmccndeg} (2),
the equality (\ref{eqdesired}) holds except the terms whose
supports have the bases of codimension $2$ in $X$.
We take a $C_{D}$-transversal and properly $\tau_{D'}^{-1}(S_{D'}^{\log}(j_{!}\mf))$-transversal
regular closed immersion $g\colon S\rightarrow X$ 
from a smooth surface $S$ over $k$ by Corollary \ref{corexistsri}. 
Let $j'\colon g^{*}U\rightarrow S$ denote the base change of $j$ by $g$.
By Proposition \ref{propinvcc} (1),
the ramification of $g^{*}\mf$ is $\log$-$g^{*}D'$-clean along $g^{*}D$,
we have $D_{\mT,g^{*}\mf}\subset g^{*}D'$ (\ref{defd*mf}),
and $\tau_{g^{*}D'}^{-1}(S_{g^{*}D'}^{\log}(j'_{!}g^{*}\mf))$ is of dimension $2$.
By Proposition \ref{propinvcc} (2),
we have $\tau_{g^{*}D'}^{!}CC_{g^{*}D'}^{\log}(j'_{!}g^{*}\mf)=g^{!}\tau_{D'}^{!}CC_{D'}^{\log}(j_{!}\mf)$. 
Since the singular support $SS(j_{!}\mf)$ is a union of irreducible components of 
$\tau_{D'}^{-1}(S_{D'}^{\log}(j_{!}\mf))$,
the immersion $g$ is properly $SS(j_{!}\mf)$-transversal by Remark \ref{remctr} (3),
and we have $CC(j'_{!}g^{*}\mf)=g^{!}CC(j_{!}\mf)$ by Theorem \ref{thmccpb}.
Since there is at most one irreducible component of $\tau_{D'}^{-1}(S_{D'}^{\log}(j_{!}\mf))$
whose base is of codimension $2$ in $X$, it suffices to prove the equality
\begin{equation}
CC(j'_{!}g^{*}\mf)=\tau_{g^{*}D'}^{!}CC_{g^{*}D'}^{\log}(j'_{!}g^{*}\mf) \notag
\end{equation}
for the sheaf $j'_{!}g^{*}\mf$ on the surface $S$
in order to obtain the desired equality (\ref{eqdesired}).
Since $\tau_{g^{*}D'}^{-1}(S_{g^{*}D'}^{\log}(j'_{!}g^{*}\mf))$ is of dimension $2$,
we may assume that $X$ is a surface, namely $d=2$.

Suppose that $X$ is a surface.
If $I'=I$, then the assertion holds by Remark \ref{remlifdp} (5) and \cite[Theorems 4.2 (ii), 6.1]{yacc}.
Hence we may assume that $I'=\{1\}$ and that $I=\{1,2\}$.
Since $I'\cap I_{\mII,\mf}$ and $I_{\mT,\mf}$ are empty, we have $1\in I_{\mI,\mf}$.
We regard the image $\Image \xi_{1}^{D'}(\mf)$ of $\xi_{1}^{D'}(\mf)$ (Definition \ref{defximf} (2)) as an ideal sheaf of $\dvr_{D_{1}}$
by Remark \ref{remximf} (2).
Then we have $B_{\mf}^{D'}=B_{I',\mf}^{D'}=V(\Image\xi_{1}^{D'}(\mf))$
(Definition \ref{defbcf} (1)). 
It is sufficient to prove the equality (\ref{eqdesired}) in the following two cases:
\begin{itemize}
\item[(I)] $x\notin B_{\mf}^{D'}$.
\item[(II)] $x\in B_{\mf}^{D'}$.
\end{itemize}

In the case (I), we may assume that $B_{\mf}^{D'}=\emptyset$.
Then the ramification of $\mf$ is $\log$-$\emptyset$-clean along $D$ 
by Remark \ref{remdefbcf} (4),
and we obtain the desired equality in a neighborhood of $x$
by Corollary \ref{corcompcc} (2) and Theorem \ref{thmccndeg} (2).

In the case (II), we have $(B_{I',\mf}^{D'})^{0}=\{x\}$ and the special fiber $T^{*}_{x}X$ at $x$ is the unique irreducible component of $\tau_{D'}^{-1}(S_{D'}^{\log}(j_{!}\mf))$
whose base is of codimension $2$ in $X$
by Corollary \ref{corcompcc} (1).
By Lemma \ref{lemcleanatx}, Remark \ref{rempropinvli}, and Corollary \ref{corcompcc} (2),
we have \begin{align}
\tau_{D'}^{!}CC_{D'}^{\log}(j_{!}\mf)=
[T^{*}_{X}X]+ \sum_{i=1}^{r}\dt(\chi|_{K_{i}})[L_{i,\mf}^{\emptyset}] + t_{x}[T^{*}_{x}X]
\notag
\end{align}
with
\begin{align}
t_{x}=\sw(\chi|_{K_{1}})\ord^{\emptyset}(\mf;x,D_{1})+\dt(\chi|_{K_{2}})\ord^{\emptyset}(\mf;x,D_{2}) \notag
\end{align}
and $\ord^{\emptyset}(\mf;x,D_{2})=1$. 
If $2\in I_{\mI,\mf}$,
then the ramification of $\mf$ is $\log$-$D$-clean along $D$
by Lemma \ref{lemcleanatx}, and we have $\lambda_{x}=0$ by Remark \ref{remsx} (1).
Hence 
the equality (\ref{eqdesired}) holds by Theorem \ref{thmcalsf},
since $\dt(\chi|_{K_{2}})=\sw(\chi|_{K_{2}})+1$.
Suppose that $2\in I_{\mII,\mf}$.
Then the $\log$-$D$-characteristic form $\cform^{D}(\mf)$ is of the form (\ref{formcformexa}) in Example \ref{exalamdax} 
by Lemma \ref{lemclrelfirf} applied to the case where $(I',I'')$ is $(I,I')$
and by Lemma \ref{lemcleanatx}.
Thus we have $\lambda_{x}=\ord^{D}(\mf;x,D_{2})$ by Example \ref{exalamdax}, 
and the equality (\ref{eqdesired}) holds by Theorem \ref{thmcalsf},
since $\dt(\chi|_{K_{2}})=\sw(\chi|_{K_{2}})$.
\end{proof}

\begin{cor}
\label{corcurvcc}
Suppose that $X$ is a curve, namely is purely of dimension $1$, 
and that $D$ has simple normal crossings.
Let $I'\subset I$ be a subset such that $I_{\mT,\mf}\subset I'$ (Definition \ref{defindsub} (1)) and let
$D'=\bigcup_{i\in I'}D_{i}$.
Then the assumptions in Conjecture \ref{conjcc} are satisfied and
Conjecture \ref{conjcc} holds.
\end{cor}

\begin{proof}
By Remark \ref{remlogdpcl} (2), the ramification of $\mf$ is $\log$-$D'$-clean along $D$.
By Remark \ref{remtameloc} (2), we have $I_{\mII,\mf}=\emptyset$ (Definition \ref{defindsub} (2)).
Since the dimension of $X$ is $1$,
we have $E_{\mf}^{D'}=\emptyset$ by 
Remark \ref{remdefbcf} (2),
and the inverse image $\tau_{D'}^{-1}(S_{D'}^{\log}(j_{!}\mf))\subset T^{*}X$
of the $\log$-$D'$-singular support $S_{D'}^{\log}(j_{!}\mf)\subset T^{*}X(\log D')$ 
(Definition \ref{deflifdp} (3)) by the canonical morphism $\tau_{D'}$ (\ref{deftaue})
is purely of dimension $1$ by Corollary \ref{corinvli}.
Thus the assumptions in Conjecture \ref{conjcc} are satisfied.
Since the base of any closed conical subset of $T^{*}X$ is of codimension $\le 1$ in $X$,
Conjecture \ref{conjcc} holds by Theorem \ref{mainthmosct}.
\end{proof}

\begin{cor}
\label{corosct}
Suppose that $X$ is purely of dimension $d$ and that $D$ has simple normal crossings.
Let $I'\subset I$ be a subset containing $I_{\mT,\mf}$ (Definition \ref{defindsub} (1)) and let $D'=\bigcup_{i\in I'}D_{i}$.
Let the assumptions be as in Theorem \ref{mainthmosct}
(including the assumptions of Conjecture \ref{conjcc})
or let $d=1$.
Then we have
\begin{equation}
SS(j_{!}\mf)=\tau_{D'}^{-1}(S_{D'}^{\log}(j_{!}\mf)). \notag
\end{equation}
\end{cor}

\begin{proof}
By Theorem \ref{mainthmosct} and Corollary \ref{corcurvcc}, we have $CC(j_{!}\mf)=\tau_{D'}^{!}CC_{D'}^{\log}(j_{!}\mf)$.
Since the singular support $SS(j_{!}\mf)$ is the support of the characteristic cycle $CC(j_{!}\mf)$
by Corollary \ref{corssccsm} (2),
it suffices to show that the support of $\tau_{D'}^{!}CC_{D'}^{\log}(j_{!}\mf)$ is 
the inverse image $\tau_{D'}^{-1}(S_{D'}^{\log}(j_{!}\mf))$.
Since $S_{D'}(j_{!}\mf)$ is of dimension $d$, the inverse image
$\tau_{D'}^{-1}(S_{D'}^{\log}(j_{!}\mf))$ is purely of dimension $d$
by Corollary \ref{corinvli}.
Since the support of $CC_{D'}^{\log}(j_{!}\mf)$ is 
$S_{D'}^{\log}(j_{!}\mf)$ 
and we have $(-1)^{d}CC_{D'}^{\log}(j_{!}\mf)> 0$ by
Remark \ref{remlifdp} (4),
the assertion holds.
\end{proof}

\begin{cor}[Index formula, {cf.\ \cite[Corollary 3.8]{sawild}}]
Suppose that $X$ is purely of dimension $d$ and that $D$ has simple normal crossings.
Let $I'\subset I$ be a subset containing $I_{\mT,\mf}$ (Definition \ref{defindsub} (1)) and let $D'=\bigcup_{i\in I'}D_{i}$.
Let the assumptions be as in Corollary \ref{corosct}. 
If $X$ is projective over an algebraically closed field, then we have
\begin{equation}
\chi(X,j_{!}\mf)=(CC_{D'}^{\log}(j_{!}\mf), T^{*}_{X}X(\log D'))_{T^{*}X(\log D')},\notag
\end{equation}
where the left-hand side is the Euler-Poincar\'{e} characteristic of $j_{!}\mf$ and
the right-hand side is the intersection number of the $\log$-$D'$-characteristic cycle 
$CC_{D'}^{\log}(j_{!}\mf)$ of $j_{!}\mf$ (Definition \ref{deflifdp} (4))
with the zero section $T^{*}_{X}X(\log D')$ of $T^{*}X(\log D')$.
\end{cor}

\begin{proof}
The assertion follows from Theorem \ref{mainthmosct}, Corollary \ref{corcurvcc}, and the index formula \cite[Theorem 7.13]{sacc} for the characteristic cycle $CC(j_{!}\mf)$.
\end{proof}

We give examples of computations of the singular support $SS(j_{!}\mf)$ and 
the characteristic cycle $CC(j_{!}\mf)$.

\begin{exa}
\label{examain}
We put $X=\mathbf{A}_{k}^{3}=\Spec k[t_{1},t_{2},t_{3}]$ and
$D=D_{1}=(t_{1}=0)$.
Let $\mf$ be a smooth sheaf of $\Lambda$-modules of rank 1 on $U=X-D=\Spec k[t_{1}^{\pm1},t_{2},t_{3}]$ defined by 
the Artin-Schreier-Witt equation 
\begin{equation}
F(t)-t= \left(\frac{t_{2}}{t_{1}^{n}},\frac{t_{3}}{t_{1}^{pn}}\right), \notag
\end{equation}
where $F$ denotes the Frobenius, the right-hand side is an element of the Witt ring $W_{2}(k[t_{1}^{\pm1},t_{2},t_{3}])$,
and $n\in\mathbf{Z}_{\ge 1}$ is prime to $p$.
Then we have $\sw(\chi|_{K_{1}})=pn$, $\dt(\chi|_{K_{1}})=pn+1$, and
\begin{equation}
\cform^{D}(\mf)=\frac{nt_{2}^{p}\dlog t_{1}-t_{2}^{p-1}d t_{2}-dt_{3}}{t_{1}^{pn}}. \notag
\end{equation}
Hence the ramification of $\mf$ is $\log$-$D$-clean along $D$
and we have $D_{\mI,\mf}=D$ (\ref{defd*mf}) and
$B_{\mf}^{D_{\mI,\mf}}=B_{\{1\},\mf}^{D_{\mI,\mf}}=V(t_{1},t_{2}^{p})$ (Definition \ref{defbcf} (1)).
If we regard the image $\Image \xi_{1}^{D_{\mI,\mf}}(\mf)$ of $\xi_{1}^{D_{\mI,\mf}}(\mf)$
(Definition \ref{defximf} (2)) as an ideal sheaf of $\dvr_{D_{1}}$ by Remark \ref{remximf} (2) and if $V(\Image \xi_{1}^{D_{\mI,\mf}}(\mf))$ denotes the closed subscheme
of $D_{1}$ defined by the ideal sheaf $\Image \xi_{1}^{D_{\mI,\mf}}(\mf)\subset \dvr_{D_{1}}$,
then we have $V(\Image \xi_{1}^{D_{\mI,\mf}}(\mf))=V(t_{1},t_{2}^{p})$.
Hence we have
\begin{equation}
\tau_{D_{\mI,\mf}}^{-1}(S_{D_{\mI,\mf}}^{\log}(j_{!}\mf))
= T^{*}_{X}X\cup T^{*}_{D_{1}}X\cup \langle t_{2}^{p-1}dt_{2}+dt_{3}/V(t_{1},t_{2}^{p})\rangle \notag
\end{equation}
by Corollary \ref{corcompcc} (1).
Since  the dimension of $\tau_{D_{\mI,\mf}}^{-1}(S_{D_{\mI,\mf}}^{\log}
(j_{!}\mf))$ is $3$ and the bases of irreducible components of $\tau_{D_{\mI,\mf}}^{-1}(S_{D_{\mI,\mf}}^{\log}
(j_{!}\mf))$ are of codimension $\le 2$ in $X$,
we have 
\begin{equation}
\tau_{D_{\mI,\mf}}^{!}CC^{\log}_{D_{\mI,\mf}}(j_{!}\mf)=
-([T^{*}_{X}X]+(1+pn)[T^{*}_{D_{1}}X]+p^{2}n [\langle dt_{1}, t_{2}^{p-1}dt_{2}+dt_{3}/V(t_{1},t_{2})\rangle]) \notag
\end{equation}
by Corollary \ref{corcompcc} (2).
By Theorem \ref{mainthmosct} and Corollary \ref{corosct},
the singular support $SS(j_{!}\mf)$ and the characteristic cycle $CC(j_{!}\mf)$ are equal to $\tau_{D_{\mI,\mf}}^{-1}(S_{D_{\mI,\mf}}^{\log}(j_{!}\mf))$ and $\tau_{D_{\mI,\mf}}^{!}CC_{D_{\mI,\mf}}^{\log}(j_{!}\mf)$, respectively.
\end{exa}

\subsection{Computation in codimension 2 in remaining case}
\label{sscompccvarrem}

We consider computations of the singular support $SS(j_{!}\mf)$ and the characteristic cycle $CC(j_{!}\mf)$ in codimension $2$
without the assumption on the dimension of
the inverse image $\tau_{D'}^{-1}(S_{D'}^{\log}(j_{!}\mf))\subset T^{*}X$
of the $\log$-$D'$-singular support $S_{D'}^{\log}(j_{!}\mf)\subset T^{*}X(\log D')$.
By admitting blowing up $X$ along a closed subscheme of $D$,
we can compute them in codimension $2$: 

\begin{thm}
\label{thmcompccssblup}
Suppose that $X$ is purely of dimension $d$ and that $D$ has simple normal crossings.
Assume that the ramification of $\mf$ is $\log$-$D'$-clean along $D$
for some union $D'$ of irreducible components of $D$.
Let $f\colon X'\rightarrow X$ be the composition of successive blow-ups satisfying the conditions (1)--(3) in Proposition \ref{propblupcth}
and let $j'\colon f^{*}U\rightarrow X'$ be the base change of $j$ by $f$.
Then we have
\begin{equation}
SS(j'_{!}f^{*}\mf)=\tau_{D_{\mI,f^{*}\mf}\cup D_{\mT,f^{*}\mf}}^{-1}(S_{D_{\mI,f^{*}\mf}\cup D_{\mT,f^{*}\mf}}^{\log}(j'_{!}f^{*}\mf)) \notag
\end{equation}
and 
\begin{equation}
CC(j'_{!}f^{*}\mf)=\tau_{
D_{\mI,f^{*}\mf}\cup D_{\mT,f^{*}\mf}}^{!}
CC_{D_{\mI,f^{*}\mf}\cup D_{\mT,f^{*}\mf}}^{\log}(j'_{!}f^{*}\mf) \notag 
\end{equation}
in $Z_{d}(S_{D_{\mI,f^{*}\mf}\cup D_{\mT,f^{*}\mf}}(j'_{!}f^{*}\mf))$
(Definition \ref{defsdpf}) outside a closed subscheme of $X'$ of codimension $\ge 3$.
Here $D_{*,f^{*}\mf}$ for $*=\mI,\mT$ are as in (\ref{defd*mf}),
$\tau_{D_{\mI,f^{*}\mf}\cup D_{\mT,f^{*}\mf}}$ is as in (\ref{deftaue}),
and $\tau_{D_{\mI,f^{*}\mf}\cup D_{\mT,f^{*}\mf}}^{!}$ is as in (\ref{taudpgysin}).
\end{thm}

\begin{proof}
By the condition (2) in Proposition \ref{propblupcth}, 
the ramification of $f^{*}\mf$ is $\log$-$
D_{\mI,f^{*}\mf}\cup D_{\mT,f^{*}\mf}$-clean along $(f^{*}D)_{\red}$.
By Corollary \ref{corsddim}, the dimension of the inverse image
$\tau_{D_{\mI,f^{*}\mf}\cup D_{\mT,f^{*}\mf}}^{-1}(S_{D_{\mI,f^{*}\mf}\cup D_{\mT,f^{*}\mf}}^{\log}(j'_{!}f^{*}\mf))$ 
of the $\log$-$D_{\mI,f^{*}\mf}\cup D_{\mT,f^{*}\mf}$-singular
support $S_{D_{\mI,f^{*}\mf}\cup D_{\mT,f^{*}\mf}}^{\log}(j'_{!}f^{*}\mf)$ by
the canonical morphism $\tau_{D_{\mI,f^{*}\mf}\cup D_{\mT,f^{*}\mf}}\colon T^{*}X'
\rightarrow T^{*}X'(\log D_{\mI,f^{*}\mf}\cup D_{\mT,f^{*}\mf})$
is $d$.
Then we obtain the desired equalities
outside a closed subscheme of $X'$ of codimension $\ge 3$
by Theorem \ref{mainthmosct} and Corollary \ref{corosct}. 
\end{proof}

If the assumptions and the notation are in Theorem \ref{thmcompccssblup} and
if the dimension of $f_{\circ}SS(f^{*}j_{!}\mf)$ 
is $\le d$, then $CC(j_{!}\mf)$ is equal to the image $f_{!}CC(j'_{!}f^{*}\mf)$ of $CC(j_{!}\mf)$ by the morphism $f_{!}\colon Z_{d}(SS(j'_{!}f^{*}\mf))\rightarrow Z_{d}(f_{\circ}SS(j'_{!}f^{*}\mf))$ (\ref{fexpzd})
by \cite[Theorem 2.2.5]{sacond}.
However $f_{\circ}SS(j'_{!}f^{*}\mf)$ is not always of dimension $\le d$
as is seen in the following example:

\begin{exa}
\label{exapushdim}
Let $X=\mathbf{A}^{3}_{k}=\Spec k[t_{1},t_{2},t_{3}]$ and
$D=(t_{1}t_{2}=0)$.
We put $D_{i}=(t_{i}=0)$ for $i\in I=\{1,2\}$.
Let $\mf$ be a smooth sheaf of $\Lambda$-modules of rank $1$ on $U=X-D=\Spec 
k[t_{1}^{\pm 1}, t_{2}^{\pm 2}, t_{3}]$ 
defined by the Artin-Schreier-Witt equation
\begin{equation}
F(t)-t=\left(\frac{t_{1}}{t_{2}}, \frac{t_{2}}{t_{1}}, \frac{t_{3}}{t_{1}^{p}t_{2}^{p^{2}}}\right), \notag
\end{equation}
where $F$ denotes the Frobenius and the right-hand side is an element
of the Witt ring $W_{3}(k[t_{1}^{\pm 1}, t_{2}^{\pm 2},t_{3}])$.
Then we have $(\sw(\chi|_{K_{1}}),\sw(\chi|_{K_{2}}))=(p, p^{2})$, $(\dt(\chi|_{K_{1}}),\dt(\chi|_{K_{2}}))=(p+1,p^{2}+1)$, and
\begin{equation}
\cform^{D}(\mf)=\frac{(-t_{1}^{p^{2}+p}+t_{2}^{p^{2}+p})\dlog t_{1}+(t_{1}^{p^{2}+p}-t_{2}^{p^{2}+p})\dlog t_{2}-dt_{3}}{t_{1}^{p}t_{2}^{p^{2}}}. \notag
\end{equation}
Hence we have $D_{\mI,\mf}=D=D_{1}\cup D_{2}$ (\ref{defd*mf}),
the ramification of $\mf$ is $\log$-$D_{\mI,\mf}$-clean along $D$,
and we have $E_{\mf}^{D_{\mI,\mf}}=V(t_{1},t_{2})$
(Definition \ref{defbcf} (2)).

Let $f\colon X'\rightarrow X$ be the blow-up of $X$ along $D_{I_{\mI,\mf}}=D_{1}\cap D_{2}$ and let $j'\colon f^{*}U\rightarrow X$ denote the base change of $j$ by $f$.
We denote $f^{*}t_{i}$ in the complement $U_{i}'$ of the proper transform $D_{i}'$ of $D_{i}$
by $t_{i}'$ for $i=1,2$ and
we put $f^{*}t_{2}=t_{1}'t_{2}'$ (resp.\ $f^{*}t_{1}=t_{1}'t_{2}'$) and $f^{*}t_{3}=t_{3}'$ in
$U_{1}'$ (resp.\ $U_{2}'$). 
Then the sheaf $f^{*}\mf$ is defined by the  Artin-Schreier-Witt equation
\begin{equation}
F(t)-t=\left(\frac{1}{t_{2}'}, t_{2}', \frac{t_{3}'}{t_{1}'^{p+p^{2}}t_{2}'^{p^{2}}}\right) \notag
\end{equation}
in $U'_{1}$ and
\begin{equation}
F(t)-t=\left(t_{1}', \frac{1}{t_{1}'}, \frac{t_{3}'}{t_{1}'^{p}t_{2}'^{p^{2}+p}}\right) \notag
\end{equation}
in $U_{2}'$.
By Lemma \ref{lemblupbth} (1),
the exceptional divisor $D_{0}'=f^{-1}(D_{\mI,\mf})$ is contained in $D_{\mII,f^{*}\mf}$
(\ref{defd*mf}) and we have $D_{\mI,f^{*}\mf}=D_{1}'\cup D_{2}'$.
By Lemma \ref{lemblupbth} (2),
the ramification of $f^{*}\mf$ is $\log$-$D_{\mI,f^{*}\mf}$-clean along $(f^{*}D)_{\red}$.
Then we have $\sw(f^{*}\chi|_{K_{i}'})=p^{i}$ and $\dt(f^{*}\chi|_{K_{i}'})=p^{i}+1$ for $i=1,2$,
where $K_{i}'$ denotes the local field at the generic point of $D_{i}'$, 
and we have
\begin{equation}
\cform^{D_{\mI,f^{*}\mf}}(f^{*}\mf)=\frac{t_{1}'^{p^{2}+p}\dlog t_{2}'-dt_{3}'}{t_{1}'^{p^{2}+p}t_{2}'^{p^{2}}} \notag
\end{equation}
in $U_{1}'$ and
\begin{equation}
\cform^{D_{\mI,f^{*}\mf}}(f^{*}\mf)=\frac{t_{2}'^{p^{2}+p}\dlog t_{1}' -dt_{3}'}{t_{1}'^{p}t_{2}'^{p^{2}+p}} \notag
\end{equation}
in $U_{2}'$.
Hence we have 
\begin{equation}
\tau_{D_{\mI,f^{*}\mf}}^{-1}(S_{D_{\mI,f^{*}\mf}}^{\log}(j'_{!}f^{*}\mf))=
T^{*}_{X'}X'\cup 
\langle dt_{3}'/ V(t_{1}')\rangle \cup 
\langle dt_{2}'/ V(t_{2}')\rangle \cup
\langle dt_{2}',dt_{3}'/V(t_{1}',t_{2}')\rangle
\notag
\end{equation}
on $U_{1}'$ and
\begin{equation}
\tau_{D_{\mI,f^{*}\mf}}^{-1}(S_{D_{\mI,f^{*}\mf}}^{\log}(j'_{!}f^{*}\mf))=
T^{*}_{X'}X'\cup \langle dt_{1}'/ V(t_{1}')\rangle \cup
\langle dt_{3}'/ V(t_{2}')\rangle \cup \langle dt_{1}',dt_{3}'/V(t_{1}',t_{2}')\rangle
\notag 
\end{equation}
on $U_{2}'$ as sets
by Corollary \ref{corcompcc} (1).
Since $\tau_{D_{\mI,f^{*}\mf}}^{-1}(S_{D_{\mI,f^{*}\mf}}^{\log}(j'_{!}f^{*}\mf))$ is of dimension $3$ and the bases of irreducible components of $\tau_{D_{\mI,f^{*}\mf}}^{-1}(S_{D_{\mI,f^{*}\mf}}^{\log}(j'_{!}f^{*}\mf))$ are of codimension $\le 2$ in $X'$,
we have
\begin{equation}
SS(j'_{!}f^{*}\mf)=\tau_{D_{\mI,f^{*}\mf}}^{-1}(S_{D_{\mI,f^{*}\mf}}^{\log}(j'_{!}f^{*}\mf)) \notag
\end{equation}
by Corollary \ref{corosct}.

Let $df|_{U_{1}'}\colon T^{*}X\times_{X}U_{1}'\rightarrow T^{*}U_{1}'$ be the base change of $df$ (\ref{defdh}) by the canonical open immersion $U_{1}'\rightarrow X'$.
Then $\langle t_{2}'dt_{1}-dt_{2},dt_{3}/V(t_{1}')\rangle\subset T^{*}X\times_{X}U_{1}'$ is contained in
$df|_{U_{1}'}^{-1}(\langle dt_{3}'/V(t_{1}')\rangle)\subset df^{-1}(SS(j'_{!}f^{*}\mf))$.
We put $A=k[t_{1},t_{2},t_{3}]=\Gamma (X,\dvr_{X})$
and $B=k[t_{1}',t_{2}',t_{3}']=\Gamma (U_{1}',\dvr_{X'})$.
Let $(\partial/\partial t_{1},\partial/\partial t_{2},\partial/\partial t_{3})$
be the dual basis of the basis $(dt_{1},dt_{2},dt_{3})$
of the free $A$-module $\Omega_{A}^{1}$.
Then the prime ideal of the symmetric algebra $S^{\bullet}\Omega_{A}^{1\vee}\otimes_{A}B$ corresponding
to the generic point of $\langle t_{2}'dt_{1}-dt_{2},dt_{3}/V(t_{1}')\rangle\subset T^{*}X\times_{X}U_{1}'$ is generated by 
$t_{1}'\in B=S^{0}\Omega_{A}^{1\vee}\otimes_{A}B$ and $t_{2}'\partial/\partial t_{1}+\partial/\partial t_{2}\in S^{1}\Omega_{A}^{1\vee}\otimes_{A}B$.
Since the inverse image of the prime ideal $(t_{1}', t_{2}'\partial/\partial t_{1}+\partial/\partial t_{2})\subset S^{\bullet}\Omega_{A}^{1\vee}\otimes_{A}B$ by the canonical morphism
$S^{\bullet}\Omega_{A}^{1\vee}\rightarrow S^{\bullet}\Omega_{A}^{1\vee}\otimes_{A}B$ is 
$(t_{1},t_{2})\subset S^{\bullet}\Omega_{A}^{1\vee}$,
the closed conical subset $T^{*}X\times_{X}(D_{1}\cap D_{2})\subset T^{*}X$, which is of dimension $4$, is contained in $f^{\circ}SS(j'_{!}f^{*}\mf)$ so that $f^{\circ}SS(j'_{!}f^{*}\mf)$ is of dimension $>3$.
\end{exa}


\vspace{2\baselineskip}
\noindent
Yuri YATAGAWA

\noindent
Department of Mathematics

\noindent
Tokyo Institute of Technology

\noindent
Tokyo, 152-8551, Japan

\noindent
yatagawa@math.titech.ac.jp

\end{document}